\documentclass[11pt,english,letterpaper]{amsart}
\usepackage{amssymb,bm,mathrsfs,mathabx,stmaryrd,enumerate}
\usepackage[all]{xy}
\usepackage{hyperref}
\newcommand{\map}[1]{\xrightarrow{#1}}
\newcommand{\iso}{\cong}
\newcommand{\define}{\stackrel{\mathrm{def}}{=}}

\newcommand{\imes}{\ltimes}
\newcommand{\normal}{\lhd}

\newcommand{\Gal}{\mathrm{Gal}}
\newcommand{\Hom}{\mathrm{Hom}}
\newcommand{\Aut}{\mathrm{Aut}}
\newcommand{\End}{\mathrm{End}}
\newcommand{\Spec}{\mathrm{Spec}}
\newcommand{\Spf}{\mathrm{Spf}}

\newcommand{\Q}{\mathbb Q}
\newcommand{\Z}{\mathbb Z}
\newcommand{\R}{\mathbb R}
\newcommand{\C}{\mathbb C}
\newcommand{\F}{\mathbb F}
\newcommand{\A}{\mathbb A}
\newcommand{\co}{\mathcal O}

\newcommand{\Lie}{\mathrm{Lie}}

\newcommand{\Sg}{\mathrm{Sg}}

\newcommand{\GSpin}{\mathrm {GSpin}}
\newcommand{\SO}{\mathrm {SO}}
\newcommand{\GSp}{\mathrm {GSp}}
\newcommand{\GL}{\mathrm {GL}}
\newcommand{\SL}{\mathrm {SL}}
\newcommand{\Sh}{\mathrm {Sh}}

\newcommand{\rat}{\mathrm {rat}}
\newcommand{\unit}{\mathrm {unit}}
\newcommand{\spl}{ s }
\newcommand{\weyl}{\mathscr{W}}
\newcommand{\beef}{\diamond}
\newcommand{\DD}{\mathcal{D}}
\newcommand{\sS}{S}
\newcommand{\naive}{\mathrm{naive}}

\newcommand{\regtheta}{\Theta^\mathrm{reg}}

\theoremstyle{plain}
\newtheorem{theorem}{Theorem}[subsection]
\newtheorem{proposition}[theorem]{Proposition}
\newtheorem{lemma}[theorem]{Lemma}
\newtheorem{corollary}[theorem]{Corollary}

\newtheorem{bigtheorem}{Theorem}[section]

\theoremstyle{definition}
\newtheorem{definition}[theorem]{Definition}
\newtheorem{hypothesis}[theorem]{Hypothesis}

\theoremstyle{remark}
\newtheorem{remark}[theorem]{Remark}

\numberwithin{equation}{subsection}

\author{Benjamin Howard}
\address{Department of Mathematics, Boston College, 140 Commonwealth Ave, Chestnut Hill, MA 02467, USA}
\email{howardbe@bc.edu}

\author{Keerthi Madapusi Pera}
\address{Department of Mathematics, Boston College, 140 Commonwealth Ave, Chestnut Hill, MA 02467, USA}
\email{madapusi@bc.edu}

\title{Arithmetic  of Borcherds products}

\begin{document}

\begin{abstract}
We compute the divisors of Borcherds products on integral models of orthogonal Shimura varieties.
As an application, we obtain an integral version of a theorem of Borcherds on the modularity of a generating series of special divisors.
\end{abstract}



\thanks{B.H.~was supported in part by NSF grants DMS-1501583 and DMS-1801905.}
\thanks{K.M.P.~was supported in part by NSF grants DMS-1502142 and DMS-1802169. }

\maketitle

\setcounter{tocdepth}{1}
\tableofcontents



\section{Introduction}


In the series of papers \cite{Bor95, Bor98, Bor:GKZ}, Borcherds introduced a family of meromorphic modular forms on orthogonal Shimura varieties, whose zeroes and poles are prescribed linear combinations of special divisors arising from embeddings of smaller orthogonal Shimura varieties.  These meromorphic modular forms are the Borcherds products of the title.

After work of Kisin \cite{KisinJAMS} on integral models of general Hodge and abelian type Shimura varieties, the theory of integral models of orthogonal  Shimura varieties and their special divisors was developed further in \cite{hor:thesis, hor:book} and \cite{mp:spin,AGHMP-1,AGHMP-2}.

The goal of   this paper is to combine the above theories to compute the divisor of a Borcherds product on the integral model of an orthogonal Shimura variety.   We  show that such a divisor is given as a prescribed linear combination of special divisors, exactly as in the generic fiber.

The first such results were obtained by Bruinier, Burgos Gil, and K\"uhn \cite{BBK}, who worked on Hilbert modular surfaces (a special type of signature $(2,2)$ orthogonal Shimura variety).
Those results were later extended  to more general orthogonal Shimura varieties by H\"ormann \cite{hor:thesis, hor:book}, but with some  restrictions.

Our results extend   H\"ormann's,  but with substantially weaker hypotheses. 
 For example, our results include cases where the integral model is not smooth, cases where the divisors in question may have irreducible components supported in nonzero characteristics, and even cases where the  Shimura variety is compact (so that one has no theory of $q$-expansions with which to analyze the arithmetic properties of Borcherds products).


\subsection{Orthogonal Shimura varieties}
\label{ss:intro general}


Given an integer $n\ge 1$ and a     quadratic space $(V,Q)$   over $\Q$ of signature $(n,2)$,  one can construct  a Shimura datum  $(G,\DD)$ with reflex field $\Q$.

The group $G=\GSpin(V)$ is a subgroup of the group of units in the Clifford algebra $C(V)$, and sits in a short exact sequence
\[
1 \to \mathbb{G}_m \to G \to \SO(V) \to 1.
\]
The hermitian symmetric domain is
\[
\DD = \{ z \in V_\C : [z,z] =0 ,\, [z,\overline{z} ] < 0 \} / \C^\times   \subset \mathbb{P}(V_\C),
\]
where the bilinear form 
\begin{equation}\label{bilinear}
[x,y]=Q(x+y) - Q(x) -Q(y)
\end{equation}
  on $V$ has been extended $\C$-bilinearly to $V_\C$.

To define a Shimura variety, fix a   $\Z$-lattice $V_\Z \subset V$ on which the quadratic form is $\Z$-valued, and  a  compact open subgroup $K\subset G(\A_f)$ such that
\begin{equation}\label{K choice}
K \subset G(\A_f) \cap C(  V_{ \widehat{\Z}} )^\times.
\end{equation}
Here $C( V_{\widehat{\Z}} )$ is the Clifford algebra of the $\widehat{\Z}$-quadratic space 
$
V_{\widehat{\Z}}=V_\Z \otimes\widehat{\Z}.
$
The canonical model  of the  complex orbifold  
\[
\Sh_K(G,\DD)(\C)  = G(\Q) \backslash \big( \DD \times G(\A_f) / K \big)
\]
 is a  smooth  $n$-dimensional Deligne-Mumford stack 
 \[
 \Sh_K(G,\DD) \to \Spec(\Q) .
 \]

As in work of Kudla \cite{Ku97, Ku-MSRI}, our Shimura variety  carries a family of effective Cartier divisors
\[
Z(m,\mu) \to  \Sh_K(G,\DD)
\]
indexed by positive $m\in \Q$ and $\mu \in V^\vee_\Z / V_\Z$, and a  metrized line bundle 
\[
\widehat{\bm{\omega}} \in \widehat{\mathrm{Pic}}( \Sh_K(G,\DD) )
\]
of weight one modular forms.
Under the complex uniformization of the Shimura variety,  this line bundle pulls back to the tautological bundle on $\mathcal{D}$, with the metric defined by (\ref{better metric}).

We say that  $V_\Z$  is \emph{maximal} if there is no proper superlattice in $V$ on which $Q$ takes integer values, and is 
 \emph{maximal at  $p$} if the $\Z_p$-quadratic space  $V_{\Z_p}=V_\Z \otimes \Z_p$ has the analogous property.
It is clear that  $V_\Z$ is maximal at every prime not dividing the discriminant  $[V_\Z^\vee : V_\Z]$.

Let  $\Omega$ be a finite set of rational primes containing all primes at which $V_\Z$ is not maximal, and abbreviate
\[
\Z_\Omega = \Z [ 1/p : p\in \Omega ].
\]
Assume  that (\ref{K choice}) factors as $K = \prod_p K_p$,  in such a way that 
\[
K_p= G(\Q_p) \cap C( V_{\Z_p} )^\times
\]
for every prime $p\not\in \Omega$. 
For such  $K$ there is a  flat and normal integral model 
\[
\mathcal{S}_K(G,\DD) \to \Spec( \Z_\Omega )
\]
 of $\Sh_K(G,\DD)$.  It is a Deligne-Mumford stack of finite type over $\Z_\Omega$, and is a scheme if  $K$ is sufficiently small.
 At any prime $p\not\in \Omega$, it satisfies the following properties: 
 \begin{enumerate}
 \item
  If the lattice $V_\Z$ is self-dual at a prime $p$ (or even \emph{almost self-dual} in the sense of Definition \ref{defn:almost self-dual})   then the restriction of the integral model to  $\Spec(\Z_{(p)})$ is smooth.
  \item
If $p$ is odd and $p^2$ does not divide   the discriminant $[V_\Z^\vee : V_\Z]$, then
the restriction of the integral model to  $\Spec(\Z_{(p)})$ is regular.
\item
If  $n\ge 6$ then the reduction mod $p$ is  geometrically normal.
\end{enumerate}

The integral model carries over it a  metrized line bundle 
\[
\widehat{\bm{\omega}} \in \widehat{\mathrm{Pic}}( \mathcal{S}_K(G,\DD) )
\]
of weight one modular forms, extending the one already available in the generic fiber, and a family  of effective Cartier divisors
\[
\mathcal{Z}(m,\mu) \to  \mathcal{S}_K(G,\DD)
\]
indexed by positive $m\in \Q$ and $\mu \in V^\vee_\Z / V_\Z$.

\begin{remark}
If  $V_\Z$ is itself  maximal, one can   take $\Omega=\emptyset$, choose 
\[ 
K = G(\A_f) \cap C( V_{ \widehat{\Z}} )^\times
\] 
for the level subgroup,  and  obtain an integral model of $\Sh_K(G,\DD)$  over $\Z$.
\end{remark}


\subsection{Borcherds products}


In \S \ref{ss:weakly holomorphic forms}, we recall the  Weil representation
\[
\rho_{V_\Z }  : \widetilde{\SL}_2(\Z) \to \Aut_\C( \sS_{V_\Z} )
 \] 
of the metaplectic double cover of $\SL_2(\Z)$ on the  $\C$-vector space 
\[
\sS_{V_\Z  } = \C [ V^\vee_\Z / V_\Z ].
\]

Any weakly holomorphic form 
\[
f(\tau)  = \sum_{  \substack{ m\in \Q \\ m \gg -\infty} } c(m) \cdot q^m   \in M^!_{1- \frac{n}{2} }( \overline{\rho}_{V_\Z } )
\]
valued in the complex-conjugate representation has Fourier coefficients  
\[
c(m) \in \sS_{V_\Z},
\]
 and we denote by $c(m,\mu)$ the value of $c(m)$ at the coset $\mu \in V_\Z^\vee / V_\Z$.
Fix such an $f$,  assume that $f$ is \emph{integral} in the sense that  $c(m,\mu)\in \Z$ for all $m$ and $\mu$.

Using the theory of regularized theta lifts, 
Borcherds  \cite{Bor98}   constructs a Green function $\regtheta(f)$ for the analytic divisor
\begin{equation}\label{divisor combo}
 \sum_{  \substack{  m > 0 \\ \mu\in V_\Z^\vee / V_\Z }  }  c(-m,\mu) \cdot Z(m,\mu)(\C)
\end{equation}
on $\Sh_K(G,\DD)(\C)$,  and shows  (after possibly replacing $f$ by a suitable multiple) that some  power of  $\bm{\omega}^{an}$ admits a  meromorphic section $\psi(f)$ satisfying
\begin{equation}\label{intro norm}
  -2\log \| \psi(f) \|       =  \regtheta (f)  .
\end{equation}
This implies that the divisor of $\psi(f)$ is  (\ref{divisor combo}).
These meromorphic sections are the \emph{Borcherds products} of the title.

Our main result, stated in the text as Theorem \ref{thm:main borcherds}, asserts that the Borcherds product $\psi(f)$ is algebraic, defined over $\Q$, and has the expected divisor when viewed as a rational section over the integral model.

\begin{bigtheorem}\label{bigthm:integral borcherds}
After possibly replacing $f$ by a positive integer multiple, 
there is a rational section $\psi(f)$   of the line bundle $\bm{\omega}^{c(0,0)}$ on  $\mathcal{S}_K(G,\DD)$ 
whose norm under the  metric  (\ref{better metric}) satisfies (\ref{intro norm}),  and whose divisor is 
 \[
  \mathrm{div}( \psi (f) ) =  \sum_{  \substack{  m > 0 \\ \mu\in V_\Z^\vee / V_\Z }  }  c(-m,\mu) \cdot \mathcal{Z}(m,\mu).
  \]
\end{bigtheorem}

The unspecified positive integer by which one must multiply $f$ can be made at least somewhat more explicit.  
For example, it depends only on the lattice $V_\Z$, and not on the  form $f$.  See the discussion of \S \ref{ss:divisible remark}.

As  noted earlier, similar results can be found in the work of H\"ormann \cite{hor:thesis, hor:book}.   
H\"ormann only considers self-dual lattices, so that the corresponding integral model $\mathcal{S}_K(G,\DD)$ is smooth,
and always assumes that the quadratic space $V$ admits an isotropic line.  
This allows him to prove the flatness of $\mathrm{div}(\psi(f))$ by examining the $q$-expansion of $\psi(f)$ at a cusp.  
As H\"ormann's special divisors $\mathcal{Z}(m,\mu)$, unlike ours, are  defined as the Zariski closures of their generic fibers, the equality of divisors stated in Theorem \ref{bigthm:integral borcherds} is then a formal consequence of the same equality in the generic fiber.

In contrast, we can prove Theorem \ref{bigthm:integral borcherds} even in cases where the divisors in question may not be flat, and in cases where $V$ is anisotropic, so no theory of  $q$-expansions is available.

The reader may be surprised to learn that even the descent of $\psi(f)$  to $\Q$ was not previously known in full generality.
Indeed, there is a product formula for the Borcherds product giving  its  $q$-expansions at every cusp, and so 
one should be able to detect the field of definition of $\psi(f)$  from a suitable $q$-expansion principle.

If $V$ is anisotropic then $\Sh_K(G,\DD)$ is proper over $\Q$,  no theory of $q$-expansions exists, and the above strategy fails completely.   But even when $V$ is isotropic there is a serious technical obstruction to this argument. 
The product formula of Borcherds  is not completely precise, in that the $q$-expansion of $\psi(f)$ at a given cusp is only specified  up to multiplication by an unknown constant of absolute value $1$, and there is no a priori relation between the different constants at different cusps.  These constants are the $\kappa^{(a)}$ appearing in Proposition \ref{prop:product expansion}.

If $\Sh_K(G,\DD)$ admits (a toroidal compactification with) a cusp defined over $\Q$ there is no problem: simply rescale the Borcherds product by a constant of absolute value $1$ to remove the mysterious constant at that cusp, and now $\psi(f)$ is defined over $\Q$.  
But if $\Sh_K(G,\DD)$ has no rational cusps, then to prove that $\psi(f)$ descends to $\Q$ one must compare the $q$-expansions of $\psi(f)$ at all points in a  Galois orbit of cusps. 
 One can rescale the Borcherds product to trivialize the constant at one cusp, but then one has no control over the constants at other cusps in the Galois orbit.

Using the $q$-expansion principle alone, is seems that the best one can  prove is that $\psi(f)$ descends to the minimal field of definition of a cusp.  Our strategy to improve on this is sketched in \S \ref{ss:intro proof} below.

\begin{remark}
As in the statement and proof of \cite[Theorem 10.4.12]{hor:thesis},
there is an elementary argument using Hilbert's Theorem 90 that allows one to rescale the Borcherds product so that it descends to $\Q$, but in this argument one has no control over  the scaling factor, and it need not have absolute value $1$.  In particular this rescaling may  destroy the norm relation (\ref{intro norm}).
Even worse, rescaling by such factors may introduce unwanted and unknown vertical components into the divisor of the Borcherds product on the integral model of the Shimura variety, and understanding  what's happening on the  integral model is the central concern of this work.
\end{remark}


\subsection{Modularity of generating series}


The family of special divisors determines a family of line bundles
\[
\mathcal{Z}(m,\mu) \in \mathrm{Pic}(\mathcal{S}_K(G,\DD) )
\]
indexed by positive $m\in \Q$ and $\mu \in V_\Z^\vee / V_\Z$.
We extend the definition to $m=0$  by setting 
\[
\mathcal{Z}(0,\mu) = \begin{cases}
\bm{\omega}^{-1} & \mbox{if } \mu=0 \\
\co_{\mathcal{S}_K(G,\DD)} & \mbox{if }\mu \neq 0.
\end{cases}
\] 

Exactly as in the work of Borcherds \cite{Bor:GKZ}, Theorem \ref{bigthm:integral borcherds}  produces enough  relations in the Picard group   to prove the  modularity of the generating series of these line bundles.
Let 
\[
\phi_\mu \in \sS_{V_\Z}= \C [ V^\vee_\Z / V_\Z ]
\] 
denote the characteristic function of the coset $\mu\in V_\Z^\vee / V_\Z$.

\begin{bigtheorem}\label{bigthm:modularity}
The formal  $q$-expansion
\[
 \sum_{ \substack{ m \ge 0 \\   \mu \in V_\Z^\vee / V_\Z   }  } \mathcal{Z}(m ,\mu ) \otimes \phi_\mu   \cdot q^m
\]
 is a modular form valued in $\mathrm{Pic}( \mathcal{S}_K(G,\DD) ) \otimes \sS_{V_\Z}$.
More precisely, we have  
\[
 \sum_{ \substack{ m \ge 0 \\   \mu \in V_\Z^\vee / V_\Z   }  }  \alpha ( \mathcal{Z}(m ,\mu ) ) \cdot  \phi_\mu   \cdot q^m \in M_{ 1+ \frac{n}{2} } ( \rho_{V_\Z} )
\]
  for any $\Z$-linear map $\alpha: \mathrm{Pic}( \mathcal{S}_K(G,\DD) ) \to \C$.
\end{bigtheorem}

Theorem \ref{bigthm:modularity} is stated in the text as Theorem \ref{thm:naked modularity}. 
After endowing the special divisors with Green functions as in \cite{Bruinier}, we also prove a modularity result in the group of metrized line bundles.  See Theorem \ref{thm:metrized modularity}.


\subsection{Idea of the proof}
\label{ss:intro proof}


We first prove  Theorem  \ref{bigthm:integral borcherds} assuming that   $n\ge 6$, and that $V_\Z$ splits an integral hyperbolic plane.
This assumption has three crucial consequences. 
First, it guarantees the existence of cusps of $\Sh_K(G,\DD)$ defined over $\Q$.  
Second, it guarantees that our integral model has geometrically normal  fibers, so that we may use the results of \cite{mp:compactification} to fix a toroidal compactification  in such a way that every irreducible component of every mod $p$  fiber of $\mathcal{S}_K(G,\DD)$ meets a cusp.
Finally, it guarantees the flatness of all special divisors $\mathcal{Z}(m,\mu)$.  

As noted above, the existence of cusps over $\Q$ allows us to deduce the descent of $\psi(f)$ to $\Q$ using the $q$-expansion principle.
Moreover, by examining the $q$-expansions of $\psi(f)$ at the cusps, one can show that its divisor is flat over $\Z_\Omega$, and  the equality of divisors in Theorem \ref{bigthm:integral borcherds}  then follows from the known equality in the generic fiber.

\begin{remark}
In fact, we prove that  our divisors are  flat over $\Z$ as soon as $n\ge 4$. 
When $n\in \{1,2,3\}$  the orthogonal Shimura varieties and their special divisors can be interpreted as a moduli space of abelian varieties with additional structure, as in the work of Kudla-Rapoport \cite{KR-hilbert,KR-siegel,KR-drinfeld}.     
Already in the case of $n=1$, Kudla and Rapoport \cite{KR-drinfeld} provide examples in which the special divisors are not flat.
\end{remark}

To understand how to deduce the general case from the special case above, we  first recall how Borcherds constructs $\psi(f)$ in the complex fiber.  
 If $V$ contains an isotropic line, the construction boils down to explicitly writing down its $q$-expansion as an infinite product.  
 This gives the desired $\psi(f)$,  along with the norm relation (\ref{intro norm}),  on the region of convergence.
 The right hand side of (\ref{intro norm}) is a  pluriharmonic function defined on the complement of the support of (\ref{divisor combo}),  and the  meromorphic continuation of $\psi(f)$ follows more-or-less formally from this.

  Suppose now that $V$ is anisotropic.  The idea of Borcherds  is to fix isometric embeddings of $V$ into two 
 (very particular) quadratic spaces $V^{ [1]}$ and $V^{ [2] }$ of signature $(n+24,2)$.
  From this   one can construct  morphisms of orthogonal Shimura varieties
\[
\xymatrix{
& {\Sh_K(G,\DD)} \ar[dl]_{ j^{ [1]} }  \ar[dr]^{ j^{ [2]} } \\
{ \Sh_{K^{[1]} } (G^{[1]},\DD^{[1]}) }  & & { \Sh_{K^{[2]} } (G^{[2]},\DD^{[2]}) } .
}
\] 
As both $V^{[1]}$ and $V^{[2]}$ contain isotropic lines, one already has Borcherds products on their associated Shimura varieties.

The next step should be to define 
\begin{equation}\label{borcherds quotient}
\psi(f) =  \frac{  ( j^{ [2] }  )^* \psi(f ^{ [2]} ) }  {   ( j^{ [1] }  )^* \psi(f ^{ [1]} )  } 
\end{equation}
for (very particular) weakly holomorphic forms $f^{[1]}$ and $f^{[2]}$.
The problem is that the quotient on the right hand side  is nearly always either $0/0$ or $\infty/\infty$, and so doesn't really make sense.

Borcherds gets around this via an analytic construction on the level of hermitian domains.
On the hermitian domain 
\[
\DD^{[i]} = \{ z\in V_\C^{[i]} : [z,z]=0 ,\, [z,\overline{z}] <0 \} / \C^\times \subset \mathbb{P}(V^{[i]}_\C),
\]
every irreducible component of every special divisor has the form
\[
\DD^{[i]}(x) = \{ z\in \DD^{[i]} : z\perp x \}
\]
for some $x\in V^{[i]}$, and the dual of the tautological line bundle $\bm{\omega}_{ \DD^{[i]}} $ on $\DD^{[i]}$  admits a canonical section 
\[
\mathrm{obst}^{an}_x \in H^0 ( \DD^{[i]} , \bm{\omega}_{ \DD^{[i]}} ^{-1})
\]
with zero locus $\DD^{[i]}(x)$.  See the discussion at the beginning of \S \ref{ss:deformation}.

Whenever there is an $x\in V^{[i]}$ such that 
$
\DD \subset \DD^{[i]}(x),
$
Borcherds  multiplies $\psi(f ^{ [i]} )$ by a suitable power of $\mathrm{obst}^{an}_x$ in order to remove the component $\DD^{[i]}(x)$ from $\mathrm{div}(  \psi(f ^{ [i]} ) $.  
After modifying both $\psi(f ^{ [1]} )$ and $\psi(f ^{ [2]} )$ in this way, the quotient (\ref{borcherds quotient}) is defined.
This process is what Borcherds calls the \emph{embedding trick}  in \cite{Bor98}.
As understood by Borcherds, the embedding trick is a purely analytic construction.  The sections $\mathrm{obst}^{an}_x$ over $\DD^{[i]}$ do not descend to the Shimura varieties, and have no obvious algebraic properties.  In particular, even if one knows that the $\psi(f^{[i]})$ are defined over $\Q$, it is not obvious that the renormalized quotient  (\ref{borcherds quotient}) is defined over $\Q$.

One of the main  contributions of this paper is an algebraic analogue of the embedding trick, 
which  works  even on the level of integral models.
This is based on the methods used to compute improper intersections in \cite{BHY,AGHMP-1,Ho3}.  
The idea  is to use deformation theory to construct an analogue of the section  $\mathrm{obst}^{an}_x$, not over all of $\Sh_{K^{[i]}}( G^{[i]} ,\DD^{[i]}) $, but only over the first order infinitesimal neighborhood of 
$\Sh_K(G,\DD)$ in $\Sh_{K^{[i]}}( G^{[i]} ,\DD^{[i]})$.  
This section is the \emph{obstruction to deforming $x$} appearing in \S \ref{ss:deformation}.

With this algebraic analogue of the embedding trick in hand, we can make sense of the quotient (\ref{borcherds quotient}), and compute the divisor of the left hand side in terms of the divisors of the numerator and denominator on the right.
This allows us to deduce the general case of Theorem \ref{bigthm:integral borcherds} from the special case explained above.


\subsection{Organization of the paper}


Ultimately, all arithmetic information about Borcherds products comes from their $q$-expansions, and so  we must make heavy use of the arithmetic theory of  toroidal compactifications of Shimura varieties of \cite{pink, mp:compactification}.  This theory requires introducing a substantial amount of notation just to state the main results. 
Also,  because  Borcherds products are rational sections of powers of the line bundle $\bm{\omega}$, we need the theory of automorphic vector bundles on toroidal compactifications.  This theory is distributed across a series of papers of Harris \cite{harris:automorphic_0, harris:automorphic_1, harris:automorphic_2, harris:automorphic_3}   and Harris-Zucker \cite{HZ1,HZ2,HZ3}.

Accordingly,  before we even begin to talk about orthogonal Shimura varieties, we first  recall in \S \ref{s:pink review} the main results on toroidal compactification from Pink's thesis \cite{pink}, and in \S \ref{s:AVB}  the results of Harris and Harris-Zucker on automorphic vector bundles.  
All of this is in the generic fiber of fairly general Shimura varieties. 

 Beginning in \S \ref{s:GSpin} we specialize to  case of orthogonal Shimura varieties. 
 We  consider their toroidal compactifications, and give a purely algebraic definition of $q$-expansions of modular forms on them.
 In particular, we prove the $q$-expansion principle Proposition \ref{prop:q principle}, which can be used to  detect  their fields of definition.

In \S \ref{s:borcherds} we introduce  Borcherds products and, when $V$ admits an isotropic line, describe their $q$-expansions.

In  \S \ref{s:integrality I}  we  introduce   integral models of  orthogonal Shimura varieties over $\Z_{(p)}$, along with their line bundles of modular forms and special divisors.
This material is drawn from \cite{mp:spin, AGHMP-1, AGHMP-2}, although we work here in slightly more generality.
The main new result in  \S \ref{s:integrality I}   is the pullback formula of Proposition \ref{prop:pullback}, which explains how the special divisors behave under pullback via embeddings of orthogonal Shimura varieties.  This formula, whose proof is similar to calculations of improper intersections found in \cite{BHY,AGHMP-1,Ho3},  is essential to our  algebraic version of the embedding trick.

In  \S \ref{s:integrality II}  we prove some technical properties of the integral models over $\Z_{(p)}$.
We show that the special divisors are flat when $n\ge 4$, and the integral model has geometrically normal fibers when $n \ge 6$.
When $p\neq 2$ these results already appear in \cite{AGHMP-2}. The methods here are similar, except that we appeal to the work of Ogus~\cite{Ogus1979} instead of \cite{Howard-Pappas} (which excludes $p=2$) to control the dimension of the supersingular locus.

In \S \ref{s:integral q} we extend the theory of toroidal compactifications and $q$-expansions to our integral models,
making use of the general theory of  toroidal compactifications of Hodge type Shimura varieties found in \cite{mp:compactification}.   
The culmination of the discussion is Corollary \ref{cor:flatness by q}, which allows one to use $q$-expansions to detect the flatness of  divisors of  rational sections of $\bm{\omega}$ and its powers.

Finally, in \S \ref{s:integral borcherds} we put everything together to prove  Theorem \ref{bigthm:integral borcherds}.  
The modularity result of Theorem \ref{bigthm:modularity} (and its extension to the group of metrized line bundles) 
 follows  immediately from Theorem \ref{bigthm:integral borcherds} and the modularity criterion of Borcherds.


\subsection{Notation and conventions}


 For every $a\in \A_f^\times$ there is a unique factorization 
\[
a= \rat(a) \cdot \unit(a) 
\]
in which $\rat(a)$ is a positive rational number and $\unit(a) \in \widehat{\Z}^\times$.

Class field theory provides us with a reciprocity map
\[
\mathrm{rec} :   \Q^\times_{>0} \backslash  \A_f^\times \iso  \Gal(\Q^{ab} / \Q),
\]
which we normalize as follows.
  Let $\mu_\infty$ be the set of all roots of unity in $\C$, so that  $\Q^{ab}=\Q(\mu_\infty)$ is the maximal abelian extension of $\Q$. The group  $(\Z/M\Z)^\times$ acts on the set of  $M^\mathrm{th}$ roots of unity in the usual way, by letting  $u \in ( \Z/M\Z)^\times$ act by  $\zeta \mapsto \zeta^u$.  Passing to the limit yields an action of $\widehat{\Z}^\times$ on  $\mu_\infty$, and the reciprocity map is characterized by 
\[
 \zeta^{ \mathrm{rec}(a) } = \zeta^{\unit(a)}
\]
for all $a\in \A_f^\times$ and  $\zeta\in \mu_\infty$.

We follow the conventions of \cite{Deligne:Shimura2}  and \cite[Chapter 1]{pink} for Hodge structures and  mixed Hodge structures.  As usual, 
$
\mathbb{S} = \mathrm{Res}_{\C/\R} \mathbb{G}_{m\C}
$ 
is Deligne's torus, so that $\mathbb{S}(\C) =\C^\times\times\C^\times$, with complex conjugation acting by $(t_1,t_2) \mapsto ( \bar{t}_2,\bar{t}_1)$. In particular, $\mathbb{S}(\R)\iso \C^\times$ by   $(t,\bar{t}) \mapsto t$.   If $V$ is a rational vector space endowed with a  Hodge structure $\mathbb{S} \to \GL(V_\R)$, then  $V^{(p,q)} \subset V_\C$ is the subspace on which $(t_1,t_2) \in \C^\times \times \C^\times=\mathbb{S}(\C)$ acts via $t_1^{-p} t_2^{-q}$.   There is a distinguished cocharacter 
\[
\mathrm{wt} : \mathbb{G}_{m\R} \to \mathbb{S}
\]
 defined on complex points by $t\mapsto ( t^{-1},t^{-1})$.  The composition
\[
\mathbb{G}_{m\R} \map{\mathrm{wt}} \mathbb{S} \to \GL(V_\R)
\]
encodes the weight grading on $V_\R$, in the sense that 
\[
 \bigoplus_{p+q=k} V^{(p,q)} = \{ v\in V_\C : \mathrm{wt}(z) \cdot v = z^k \cdot v,\, \forall z\in \C^\times \}.
\]

Now suppose that $V$ is endowed with a mixed Hodge structure.  This consists of an increasing  weight filtration $\mathrm{wt}_\bullet V$ on $V$,  and a decreasing Hodge filtration $F^\bullet V_\C$  on $V_\C$, whose induced filtration on every graded piece 
\begin{equation}\label{wt grading}
\mathrm{gr}_k (V) = \mathrm{wt}_k V/ \mathrm{wt}_{k-1} V
\end{equation}
is a pure Hodge structure of weight $k$.  By \cite[Lemma-Definition 3.4]{PS} there is a canonical bigrading
$
V_\C =  \bigoplus V^{(p,q)}
$
with the property that 
\[
\mathrm{wt}_k V_\C = \bigoplus_{ p+q \le k } V^{(p,q)},\quad F^i V_\C = \bigoplus_{ p \ge i} V^{(p,q)}.
\]
This bigrading is  induced by a morphism $\mathbb{S}_\C  \to \GL(V_\C)$.


\section{Toroidal compactification}
\label{s:pink review}


This section is a (relatively) short summary of Pink's thesis \cite{pink} on  toroidal compactifications of canonical models of Shimura varieties.  See also \cite{hor:thesis} and \cite{HZ1,HZ3}. We limit ourselves to what is needed in the sequel, and simplify the discussion somewhat by only dealing with those mixed Shimura varieties that appear at the boundaries of pure Shimura varieties.


\subsection{Shimura varieties}


Throughout  \S \ref{s:pink review} and \S \ref{s:AVB} we let $(G,\DD)$ be a (pure) Shimura datum in the sense of \cite[\S 2.1]{pink}.  Thus $G$ is a reductive group over $\Q$, and $\DD$ is a $G(\R)$-homogeneous space equipped with a finite-to-one $G(\R)$-equivariant map
\[
 \mathtt{h} :  \DD \to \Hom(\mathbb{S} , G_\R )
\]
 such that the pair $(G , \mathtt{h} ( \DD ) )$ satisfies Deligne's axioms   \cite[(2.1.1.1)-(2.1.1.3)]{Deligne:Shimura2}.    We often abuse notation and confuse $z\in \DD$ with its image $\mathtt{h}(z)$.

 The \emph{weight cocharacter}  
\begin{equation}\label{pure weight}
w\define \mathtt{h}(z) \circ \mathrm{wt} : \mathbb{G}_{m\R} \to G_\R
\end{equation}
 of $(G,\DD)$ is independent of $z\in\DD$, and takes values in the center of $G_\R$.

\begin{hypothesis}\label{hyp:motivic}
Because it will simplify much of what follows, and because it is assumed throughout \cite{HZ3}, we always assume that our Shimura datum $(G,\DD)$ satisfies:
\begin{enumerate}
\item
The weight cocharacter (\ref{pure weight}) is defined over $\Q$.
\item
The connected center of $G$ is isogenous to the product of a $\Q$-split torus with  a torus whose group of real points is compact. 
\end{enumerate}
\end{hypothesis}

Suppose $K\subset G(\A_f)$ is any compact open subgroup.  The associated Shimura variety
\[
\Sh_K(G,\DD)(\C) = G(\Q) \backslash \big(  \DD \times G(\A_f) / K \big)
\]
is a complex orbifold.  Its canonical model $\Sh_K(G,\DD)$  is a Deligne-Mumford stack over the reflex field $E(G,\DD)\subset \C$.   If $K$ is neat in the sense of \cite[\S 0.6]{pink}, then $\Sh_K(G,\DD)$ is a quasi-projective scheme.  By slight abose of notation, the image of a point $(z, g) \in \DD \times G(\A_f)$ is again denoted 
\[
 (z, g)  \in \Sh_K(G,\DD)(\C).
 \]

\begin{remark}\label{rem:little shimura}
Let $\mathbb{G}_m(\R) = \R^\times$ act on the two point set
\[
\mathcal{H}_0 \define \{ 2\pi \epsilon  \in \C : \epsilon^2=-1 \}
\]  
via the unique continuous transitive action: positive real numbers act trivially, and negative real numbers swap the two points. If we define \[\mathcal{H}_0 \to \Hom(\mathbb{S} , \mathbb{G}_{m \R})\] by sending both points  to the norm map $\C^\times \to \R^\times$, then $(\mathbb{G}_m , \mathcal{H}_0)$ is a  Shimura datum in the sense of \cite{pink}.
\end{remark}


\subsection{Mixed Shimura varieties}
\label{ss:mixed}


Toroidal compactifications of Shimura varieties are obtained by gluing together certain mixed Shimura varieties, which we now define.

Recall from  \cite[Definition 4.5]{pink} the notion of an \emph{admissible} parabolic subgroup  $P\subset G$.
  If   $G^{ad}$ is simple, this just means that $P$ is either a maximal proper parabolic subgroup, or is all of $G$.  In general, it means that  if we write $G^{ad} = G_1 \times\cdots \times G_s$ as a product of simple groups, then $P$ is the preimage of a subgroup $P_1 \times \cdots \times P_s$, where each $P_i \subset G_i$ is an admissible parabolic.

\begin{definition}\label{def:clr}
A \emph{cusp label representative} 
$
\Phi=( P, \DD^\circ , h)
$
for $(G,\DD)$ is a triple consisting of an admissible parabolic subgroup $P$, a connected component $\DD^\circ \subset \DD$,   and an $h\in G(\A_f)$.  
\end{definition}

 As in  \cite[\S 4.11 and \S 4.12]{pink}, any cusp label representative $\Phi=(P,\DD^\circ, h)$ determines a mixed Shimura datum  $(Q_\Phi , \DD_\Phi)$, whose construction we now recall.

Let $W_\Phi\subset P$ be the unipotent radical, and let $U_\Phi$ be the center of $W_\Phi$.   According to \cite[\S 4.1]{pink} there is a distinguished central  cocharacter 
$
\lambda: \mathbb{G}_{m} \to P / W_\Phi .
$
 The weight cocharacter $w : \mathbb{G}_m \to G$ is central, so takes values in  $P$, and therefore determines a new central cocharacter
\begin{equation}\label{preweight}
w \cdot \lambda^{-1} : \mathbb{G}_m \to P/W_\Phi.
\end{equation}

Suppose $G \to \GL(N)$ is a faithful  representation on a finite dimensional $\Q$-vector space.
Each point $z\in \DD$  determines a Hodge filtration $F^\bullet N_\C$     on $N$.    
Any  lift  of (\ref{preweight}) to a cocharacter  $\mathbb{G}_{m} \to P$    determines a grading $N=\bigoplus N^k$, and the associated  \emph{weight filtration} 
\[ 
\mathrm{wt}_\ell N  = \bigoplus_{ k \le \ell} N ^k
\] 
is independent of the lift.  
The triple $(N, F^\bullet N_\C , \mathrm{wt}_\bullet N)$ is a mixed Hodge structure \cite[\S 4.12, Remark (i)]{pink}, and the associated bigrading of $N_\C$  determines a morphism $\mathtt{h}_\Phi(z) \in \Hom( \mathbb{S}_\C ,  P_\C  )$ independent of the choice of faithful representation $N$.

 As in \cite[\S 4.7]{pink}, define  $Q_\Phi \subset P$  to be the smallest closed normal subgroup through which every such $\mathtt{h}_\Phi(z)$ factors.  Thus we have normal subgroups
\[
U_\Phi \normal W_\Phi \normal Q_\Phi \normal P,
\]
and  a map
\[
\mathtt{h}_\Phi: \DD \to  \Hom( \mathbb{S}_\C ,  Q_{\Phi \C}  ).
\]
The cocharacter (\ref{preweight}) takes values in $Q_\Phi/W_\Phi$,  defining  the \emph{weight cocharacter} 
\begin{equation}\label{weight cocharacter}
w_\Phi  : \mathbb{G}_m \to Q_\Phi / W_\Phi.
\end{equation}

\begin{remark}
 Being an abelian unipotent group, $\Lie(U_\Phi) \iso U_\Phi$ has the structure of  a $\Q$-vector space.   By \cite[Proposition 2.14]{pink}, the conjugation action of $Q_\Phi$ on $U_\Phi$ is through a character
\begin{equation}\label{unipotent character}
\nu_\Phi : Q_\Phi \to \mathbb{G}_m.
\end{equation}
\end{remark}

By \cite[Proposition 4.15(a)]{pink}, the map $\mathtt{h}_\Phi$  restricts to an open immersion on every connected component of $\DD$, and so  the diagonal map 
\[
\DD \to \pi_0(\DD) \times \Hom( \mathbb{S}_\C ,  Q_{\Phi \C}  )
\]  
is a $P(\R)$-equivariant open immersion.  The action of the subgroup $U_\Phi(\R)$ on $\pi_0(\DD)$ is trivial, and we extend it to the trivial action of $U_\Phi(\C)$ on $\pi_0(\DD)$.  Now define 
\[
\DD_\Phi =  Q_\Phi(\R) U_\Phi(\C) \DD^\circ \subset \pi_0(\DD) \times \Hom( \mathbb{S}_\C ,  Q_{\Phi\C}  ).
\]
Projection to the second factor defines a finite-to-one map
  \[
\mathtt{h}_\Phi :   \DD_\Phi \to \Hom( \mathbb{S}_\C ,  Q_{\Phi\C}  ),
  \]
and we usually abuse notation and confuse $z\in \DD_\Phi$ with its image $\mathtt{h}_\Phi(z)$.


Having now defined the mixed Shimura datum $(Q_\Phi , \DD_\Phi)$,  the compact open subgroup  
\[
K_\Phi \define  hKh^{-1} \cap Q_\Phi(\A_f) 
\]
determines a mixed Shimura variety
\begin{equation}\label{complex mixed}
\Sh_{K_\Phi}( Q_\Phi , \DD_\Phi )(\C) = Q_\Phi(\Q) \backslash  \big( \DD_\Phi \times Q_\Phi(\A_f) / K_\Phi  \big),
\end{equation}
which has a canonical model $\Sh_{K_\Phi}( Q_\Phi , \DD_\Phi )$  over its reflex field.  Note that the reflex field is again $E(G,\DD)$,  by   \cite[Proposition 12.1]{pink}.    The canonical model is a  quasi-projective scheme if $K$ (hence $K_\Phi$) is neat.

\begin{remark}
If we choose our cusp label representative to have the form $\Phi = (G, \DD^\circ  , h)$, then  $(Q_\Phi , \DD_\Phi) = (G,\DD)$  and 
\[
\Sh_{K_\Phi}(Q_\Phi ,\DD_\Phi) = \Sh_{ hKh^{-1} }(G,\DD) \iso \Sh_{ K }(G,\DD).
\]
As a consequence, all of our statements about the mixed Shimura varieties  $\Sh_{K_\Phi}(Q_\Phi ,\DD_\Phi)$ include the Shimura variety $\Sh_{ K }(G,\DD)$ as a special case.
\end{remark}


\subsection{The torsor structure}
\label{ss:torsor}


Define $\bar{Q}_\Phi =  Q_\Phi/U_\Phi$ and    $\bar{\DD}_\Phi = U_\Phi(\C) \backslash \DD_\Phi$.  The pair 
\[
(  \bar{Q}_\Phi  ,  \bar{\DD}_\Phi )   =   (Q_\Phi , \DD_\Phi) / U_\Phi 
\]
is the quotient mixed Shimura datum in the sense of \cite[\S2.9]{pink}.
 Let  $\bar{K}_\Phi$ be the image of $K_\Phi$ under  the quotient map $Q_\Phi(\A_f) \to \bar{Q}_\Phi(\A_f)$, so that we have a canonical morphism 
 \begin{equation}\label{torus torsor}
\Sh_{K_\Phi}( Q_\Phi , \DD_\Phi )  \to \Sh_{\bar{K}_\Phi}( \bar{Q}_\Phi , \bar{\DD}_\Phi ),
\end{equation}
where the target mixed Shimura variety is defined in the same way as (\ref{complex mixed}).

\begin{proposition}\label{prop:torsor def}
Define a $\Z$-lattice in $U_\Phi(\Q)$ by 
$
\Gamma_\Phi =   K_\Phi \cap    U_\Phi(\Q).
$
The morphism (\ref{torus torsor}) is canonically a torsor for the relative torus
\[
T_\Phi \define \Gamma_\Phi(-1) \otimes  \mathbb{G}_m
\] 
with cocharacter group $\Gamma_\Phi(-1)=(2\pi i )^{-1}\Gamma_\Phi$.
\end{proposition}

\begin{proof}
This is proved in  \cite[\S 6.6]{pink}.  In what follows we only want to make the torsor structure explicit on the level of complex points.

The character (\ref{unipotent character}) factors through a character $\bar{\nu}_\Phi : \bar{Q}_\Phi \to \mathbb{G}_m$.    
A pair  $(z,g) \in \DD_\Phi \times Q_\Phi(\A_f)$ determines  points  
\[
 ( z, g)  \in \Sh_{K_\Phi}(Q_\Phi , \DD_\Phi)(\C), \quad   ( \bar{z} , \bar{g} ) \in \Sh_{\bar{K}_\Phi}(\bar{Q}_\Phi , \bar{\DD}_\Phi ) (\C),
\]
and  we define  
$
\mathbf{T}_\Phi(\C) \to  \Sh_{\bar{K}_\Phi}( \bar{Q}_\Phi , \bar{\DD}_\Phi ) (\C)
$ 
as the relative torus with  fiber 
 \begin{equation}\label{torus fiber}
   U_\Phi(\C) /   (    g K_\Phi g^{-1} \cap  U_\Phi(\Q)   )   =U_\Phi(\C) /  \mathrm{rat}( \bar{\nu}_\Phi( \bar{g} ) )  \cdot \Gamma_\Phi  
\end{equation}
at   $( \bar{z} , \bar{g} )$.   
There is a natural action of $\mathbf{T}_\Phi(\C)$ on (\ref{complex mixed})  defined as follows: using the natural action of $U_\Phi(\C)$ on $\DD_\Phi$,  a  point   $u$ in the fiber  (\ref{torus fiber}) acts as $ ( z,g) \mapsto  ( uz, g) $.

It now suffices to construct an isomorphism  
\begin{equation}\label{torus trivialization}
\mathbf{T}_\Phi(\C) \iso T_\Phi(\C) \times     \Sh_{\bar{K}_\Phi}( \bar{Q}_\Phi , \bar{\DD}_\Phi ) (\C),
\end{equation}
and this is essentially \cite[\S 3.16]{pink}.  First choose a morphism 
\begin{equation}\label{base morphism}
\bar{\DD}_\Phi \map{\bar{z}\mapsto 2\pi \epsilon(\bar{z})} \mathcal{H}_0
\end{equation}
 in such a way that it, along with the character $\bar{\nu}_\Phi$, induces a morphism of mixed Shimura data
$
( \bar{Q}_\Phi , \bar{\DD}_\Phi) \to (\mathbb{G}_m , \mathcal{H}_0) .
$
Such a morphism always exists, by the Remark of \cite[\S 6.8]{pink}.  The fiber (\ref{torus fiber}) is 
\[
 U_\Phi(\C) /  \mathrm{rat}( \bar{\nu}_\Phi( \bar{g} ) )  \cdot \Gamma_\Phi 
   \map{  2\pi \epsilon( \bar{z} ) /   \mathrm{rat}(  \bar{\nu}_\Phi( \bar{g} ) )   }   U_\Phi(\C) /   \Gamma_\Phi (1)   ,
\]
and this identifies $\mathbf{T}_\Phi(\C)$  fiber-by-fiber with the constant torus 
\begin{equation}\label{minus!}
U_\Phi(\C) /   \Gamma_\Phi (1)    \iso \Gamma_\Phi  \otimes \C/\Z(1)  \iso \Gamma_\Phi  \otimes \C^\times \map{ (-2\pi \epsilon^\circ)^{-1}} \Gamma_\Phi (-1) \otimes \C^\times.
\end{equation}
Here $2\pi \epsilon^\circ$ is the image of   $\DD^\circ$ under  $\DD_\Phi  \to \bar{\DD}_\Phi \to \mathcal{H}_0$, and the minus sign is included  so that (\ref{closure orientation}) holds below; compare with the definition of the function ``$\mathrm{ord}$" in \cite[\S 5.8]{pink}.  

One can easily check that the  trivialization (\ref{torus trivialization}) does not depend on the choice of (\ref{base morphism}).
 \end{proof}

\begin{remark}
Our $\Z$-lattice $\Gamma_\Phi \subset U_\Phi(\Q)$ agrees with the seemingly more complicated lattice of  \cite[\S 3.13]{pink}, defined as the image of 
\[
\{ (c, \gamma ) \in Z_\Phi(\Q)_0 \times U_\Phi(\Q) : c\gamma  \in K_\Phi \} \map{ (c,\gamma ) \mapsto \gamma} U_\Phi(\Q).
\]
Here  $Z_\Phi$ is the center of $Q_\Phi$, and  $Z_\Phi(\Q)_0 \subset Z_\Phi(\Q)$ is the largest subgroup  acting trivially on $\DD_\Phi$ (equivalently, acting trivially on $\pi_0(\DD_\Phi)$). This follows from  the final comments of [\emph{loc.~cit.}] and the simplifying Hypothesis \ref{hyp:motivic}, which implies that the connected center of $Q_\Phi/U_\Phi$ is isogenous to the product of a $\Q$-split torus and a torus whose group of real points is compact (see the proof of \cite[Corollary 4.10]{pink}).
\end{remark}

Denoting by $\langle - ,- \rangle : \Gamma^\vee_\Phi(1) \times \Gamma_\Phi(-1) \to \Z$ the tautological pairing,  define an isomorphism
\[
\Gamma^\vee_\Phi(1) \map{\alpha \mapsto q_\alpha}  \Hom( \Gamma_\Phi(-1) \otimes \mathbb{G}_m , \mathbb{G}_m) =  \Hom(T_\Phi,   \mathbb{G}_m ) 
\]
by  $q_\alpha( \beta \otimes z ) = z^{  \langle \alpha , \beta \rangle }$.   This determines an isomorphism
\[
T_\Phi \iso \Spec\Big(  \Q[q_\alpha]_{ \alpha \in \Gamma^\vee_\Phi(1) } \Big),
\]
and hence,  for any rational  polyhedral cone\footnote{By which we mean a \emph{convex rational polyhedral cone} in the sense of \cite[\S 5.1]{pink}.  In particular, each $\sigma$ is a closed subset of the real vector space $U_\Phi(\R)(-1)$.}   $\sigma \subset  U_\Phi(\R)(-1)$, a partial compactification
\begin{equation}\label{q torus}
T_\Phi(\sigma) \define \Spec\Big(  \Q[q_\alpha]_{  \substack{  \alpha \in \Gamma^\vee_\Phi(1)  \\  \langle \alpha , \sigma \rangle \ge 0  } } \Big).
\end{equation}

More generally, the $T_\Phi$-torsor structure on (\ref{torus torsor}) determines,   by the general theory of torus embeddings \cite[\S 5]{pink},    a partial compactification 
\begin{equation}\label{mixed compactification}
\xymatrix{
{ \Sh_{K_\Phi} ( Q_\Phi ,\DD_\Phi) }  \ar@{->}[r] \ar[d] &   {   \Sh _{K_\Phi} (Q_\Phi ,\DD_\Phi ,\sigma  )  }  \ar[dl] \\
{  \Sh_{\bar{K}_\Phi} ( \bar{Q}_\Phi , \bar{\DD}_\Phi)   }
}
\end{equation}
 with a stratification by locally closed substacks
\begin{equation}\label{mixed stratification}
\Sh_{K_\Phi} (Q_\Phi ,\DD_\Phi ,\sigma  ) = \bigsqcup_\tau Z^\tau_{K_\Phi}(Q_\Phi ,\DD_\Phi ,\sigma )
\end{equation}
indexed by the faces $\tau \subset \sigma$.    The unique open stratum  
\[
Z^{  \{ 0\}  }_{K_\Phi}(Q_\Phi ,\DD_\Phi ,\sigma ) = \Sh_{K_\Phi} (Q_\Phi ,\DD_\Phi  )
\]
corresponds to $\tau = \{0\}$.  The unique closed stratum corresponds to $\tau=\sigma$.


\subsection{Rational polyhedral cone decompositions}
\label{ss:cones}


Let  $\Phi = (P,\DD^\circ, h)$ be a cusp label representative for $(G,\DD)$, with associated mixed Shimura datum $(Q_\Phi , \DD_\Phi)$. We  denote by $\DD_\Phi^\circ = U_\Phi(\C) \DD^\circ$  the connected component of $\DD_\Phi$ containing $\DD^\circ$.

Define the \emph{projection to the imaginary part}  $c_\Phi : \DD_\Phi \to U_\Phi(\R)(-1)$ by 
\[
c_\Phi(z)^{-1} \cdot z \in \pi_0(\DD) \times \Hom(\mathbb{S} , Q_{\Phi \R} )
\]
for every   $z\in \DD_\Phi $.   By  \cite[Proposition 4.15]{pink} there is an open convex cone   
\begin{equation}\label{convex cone}
C_\Phi \subset U_\Phi(\R)(-1)
\end{equation}
characterized by 
$
\DD^\circ = \{ z\in \DD_\Phi^\circ : c_\Phi(z)\in C_\Phi  \}.
$


\begin{definition}\label{def:cusp morphism}
Suppose $\Phi = (P,\DD^\circ,h)$ and $\Phi_1=(P_1,\DD_1^\circ,h_1)$ are cusp label representatives. A $K$-\emph{morphism} 
\begin{equation}\label{K-morphism}
\Phi  \map{ ( \gamma , q ) } \Phi_1
\end{equation}
is a pair $(\gamma,q) \in G(\Q) \times Q_{\Phi_1 }(\A_f)$, such that 
\[
\gamma Q_{\Phi} \gamma^{-1} \subset Q_{\Phi_1}, \quad   \gamma \DD^\circ = \DD^\circ_1 , \quad  \gamma h \in q h_1 K.
\]
A  $K$-morphism is a \emph{$K$-isomorphism} if $\gamma Q_{\Phi } \gamma^{-1}=  Q_{\Phi_1}$.
\end{definition}

\begin{remark}
The Baily-Borel compactification of $\Sh_K(G,\DD)$ admits a stratification by locally closed substacks, defined over the reflex field,  whose strata are indexed by the $K$-isomorphism classes of  cusp label representatives.  Whenever there is a $K$-morphism $\Phi \to \Phi_1$, the stratum indexed by $\Phi$ is ``deeper into the boundary" than the stratum indexed by $\Phi_1$, in the sense that the $\Phi$-stratum is contained in the closure of the $\Phi_1$-stratum.  The unique open stratum, which is just the Shimura variety $\Sh_K(G,\DD)$,  is indexed by the $K$-isomorphism class consisting of all cusp label representatives of the form $(G,\DD^\circ, h)$ as $\DD^\circ$ and $h$ vary.
\end{remark}

Suppose we have a $K$-morphism (\ref{K-morphism}) of cusp label representatives.  It follows from  \cite[Proposition 4.21]{pink} that  $ U_{\Phi_1}  \subset  \gamma U_\Phi \gamma^{-1}$,  and the image of the open convex cone $ C_{\Phi_1 }$ under
\begin{equation}\label{morphism on unipotent}
U_{\Phi_1}(\R)(-1)  \map{ u \mapsto \gamma^{-1} u \gamma } U_{\Phi} (\R)(-1) 
\end{equation}
 lies in the closure of the open convex cone $C_{\Phi }$.  Define, as in \cite[Definition-Proposition 4.22]{pink},
\[
C_\Phi^* = \bigcup_{ \Phi \to  \Phi_1}  \gamma^{-1} C_{\Phi_1} \gamma  \subset U_\Phi(\R)(-1),
\]
where the union is over all $K$-morphisms with source $\Phi$.  This is a convex cone lying between  $C_\Phi$ and its  closure, but in general $C_\Phi^*$ is neither open nor closed.    For every $K$-morphism $\Phi\to \Phi_1$ as above, the injection (\ref{morphism on unipotent}) identifies
$
C^*_{\Phi_1} \subset C^*_\Phi.
$

\begin{definition}
A  \emph{(rational polyhedral)   partial cone decomposition} of $C_\Phi^*$ is a collection $\Sigma_\Phi = \{ \sigma\}$ of rational polyhedral cones $\sigma \subset U_\Phi(\R)(-1)$  such that
\begin{itemize}
\item
each $\sigma \in \Sigma_\Phi$ satisfies $\sigma \subset C_\Phi^*$,
\item
every face of every  $\sigma \in \Sigma_\Phi$ is again an element of $\Sigma_\Phi$, 
 \item
  the intersection of any $\sigma,\tau \in \Sigma_\Phi$ is a face of both $\sigma$ and $\tau$,
   \item
   $\{ 0 \} \in \Sigma_\Phi$.
   \end{itemize}
We say that  $\Sigma_\Phi$ is \emph{smooth} if it is smooth, in the sense of \cite[\S 5.2]{pink}, with respect to the lattice $\Gamma_\Phi(-1) \subset U_\Phi(\R)(-1)$.   It is \emph{complete} if  \[ C_\Phi^* = \bigcup_{\sigma \in \Sigma_\Phi } \sigma.\]
   \end{definition}

\begin{definition}
A \emph{$K$-admissible (rational polyhedral) partial cone decomposition}  $\Sigma = \{ \Sigma_\Phi \}_\Phi$  for $(G,\DD)$ is a collection of  partial  cone decompositions $\Sigma_\Phi$ for $C_\Phi^*$, one for every cusp label representative $\Phi$,  such that for any $K$-morphism $\Phi \to \Phi_1$, the induced inclusion $C^*_{\Phi_1} \subset C^*_\Phi$ identifies  
\[
\Sigma_{\Phi_1} = \{ \sigma \in \Sigma_\Phi : \sigma \subset C_{\Phi_1}^* \}.
\]
We say that $\Sigma$ is \emph{smooth}  if every $\Sigma_\Phi$ is smooth, and \emph{complete} if every $\Sigma_\Phi$ is complete.
\end{definition}

Fix   a  $K$-admissible complete cone decomposition $\Sigma$ of $(G,\DD)$.

\begin{definition}
A \emph{toroidal stratum representative} for $(G,\DD,\Sigma)$ is a pair $(\Phi ,\sigma)$ in which $\Phi$ is a cusp label representative  and $\sigma \in \Sigma_\Phi$ is a rational polyhedral cone whose interior is contained in $C_\Phi$.  In other words, $\sigma$ is not contained in any proper subset $C^*_{\Phi_1}  \subsetneq C^*_\Phi$ determined by a $K$-morphism $\Phi \to \Phi_1$.
\end{definition}

We  now extend Definition \ref{def:cusp morphism} from  cusp label representatives to toroidal stratum representatives.

\begin{definition}\label{def:K-morphism}
A \emph{$K$-morphism} of toroidal stratum representatives
\[
(\Phi ,\sigma)  \map{ (\gamma,q) } (\Phi_1,\sigma_1)
\]
consists of a pair $( \gamma , q) \in G(\Q) \times Q_{\Phi_1}(\A_f)$  such that 
\[
\gamma Q_{\Phi} \gamma^{-1} \subset Q_{\Phi_1}, \quad   \gamma \DD^\circ = \DD^\circ_1 , \quad  \gamma h \in q h_1 K,
\]
and such that the injection (\ref{morphism on unipotent})  identifies $\sigma_1$ with a face of $\sigma$.  Such a  $K$-morphism is a \emph{$K$-isomorphism} if $\gamma Q_{\Phi} \gamma^{-1} = Q_{\Phi_1}$ and  $\gamma^{-1} \sigma_1 \gamma = \sigma$.
 \end{definition}

The set of $K$-isomorphism classes of toroidal stratum representatives will be denoted $\mathrm{Strat}_K (G,\DD,  \Sigma  )$.

\begin{definition}\label{def:finite decomposition}
We say that  $\Sigma$  is \emph{finite} if 
$
\#  \mathrm{Strat}_K (G,\DD,  \Sigma  )  < \infty.
$
\end{definition}

\begin{definition}
We say that  $\Sigma$  has the \emph{no self-intersection property} if the following  holds: whenever we are given toroidal stratum representatives $(\Phi,\sigma)$ and $(\Phi_1,\sigma_1)$, and two $K$-morphisms
\[
\xymatrix{
{ ( \Phi ,\sigma)  } \ar@/^/ [r] \ar@/_/[r]  & { (\Phi_1,\sigma_1) },
}
\]
the two  injections
\[
\xymatrix{
{ U_{\Phi_1}(\R)(-1)  } \ar@/^/ [r] \ar@/_/[r]  & { U_\Phi (\R)(-1)}
}
\]
of (\ref{morphism on unipotent}) send  $\sigma_1$ to the same face of $\sigma$.
 \end{definition}

The no self-intersection property is just a rewording of the condition of  \cite[\S 7.12]{pink}.
If $\Sigma$ has the no self-intersection property then so does any refinement (in the sense of \cite[\S 5.1]{pink}).

\begin{remark}\label{rem:K shrink}
Any   finite and $K$-admissible cone decomposition $\Sigma$ for $(G,\DD)$ acquires the no self-intersection property after possibly replacing $K$ by a smaller compact open subgroup \cite[\S 7.13]{pink}.  Moreover, by examining the proof one can see that if $K$ factors as $K=K_\ell K^{\ell}$ for some prime $\ell$ with $K_\ell\subset G(\Q_\ell)$ and $K^{\ell} \subset G(\A_f^{\ell})$, then it suffices to shrink $K_\ell$ while holding $K^{\ell}$ fixed.
\end{remark}


\subsection{Functoriality of cone decompositions}
\label{ss:cone functoriality}


Suppose that we have an embedding $(G,\DD)\to (G^\prime,\DD^\prime)$  of Shimura data.

As explained in~\cite[(2.1.28)]{mp:compactification}, every cusp label representative 
\[
\Phi = ( P,\DD^\circ,g)
\] 
for $(G,\DD)$ determines a cusp label representative
\[
\Phi^\prime = (P^\prime,\DD^{\prime,\circ},g^\prime) 
\]
for $(G^\prime,\DD^\prime)$.  More precisely, we define  $g^\prime = g$, let $\DD^{\prime,\circ} \subset \DD^\prime$ be the connected component  containing $\DD^\circ$,  and let $P^\prime\subset G^\prime$ be the smallest admissible parabolic subgroup containing $P$.
In particular, 
\[
Q_{\Phi}\subset Q_{\Phi^\prime} ,\quad U_{\Phi}\subset U_{\Phi^\prime}, \quad C_{\Phi} \subset C_{\Phi^\prime}.
\] 


If $K\subset G(\A_f)$ is a compact open subgroup contained in a compact open subgroup $K^\prime\subset G^\prime(\A_f)$, then every $K$-morphism
\[
\Phi\xrightarrow{(\gamma,q)}\Phi_1
\]
determines a $K^\prime$-morphism
\[
\Phi^\prime\xrightarrow{(\gamma,q)}\Phi^\prime_1 .
\]

Any $K^\prime$-admissible rational cone decomposition $\Sigma^\prime$ for $(G^\prime,\DD^\prime)$ 
pulls back to a  $K$-admissible rational cone decomposition $\Sigma$  for $(G,\DD)$, defined by 
\[
\Sigma_ \Phi= \{\sigma^\prime\cap C_{\Phi}^*: \sigma^\prime\in \Sigma^\prime_{\Phi^\prime}  \}
\]
for every cusp label representative $\Phi$ of $(G,\DD)$. 
It is shown in \cite[\S 3.3]{harris:automorphic_3} that $\Sigma$ is finite whenever $\Sigma^\prime$ is so. 
It is also not hard to check that $\Sigma$ has the no self-intersection property whenever $\Sigma^\prime$ does, 
and that it is complete when $\Sigma^\prime$ is so.

Given a cusp label representative $\Phi$ for $(G,\DD)$ and a $\sigma\in \Sigma_\Phi$, there is a unique  rational polyhedral cone $\sigma^\prime\in \Sigma^\prime_{\Phi^\prime}$ such that $\sigma\subset \sigma'$, but $\sigma$ is not contained in any proper face of $\sigma'$.
The assignment $(\Phi,\sigma)\mapsto (\Phi^\prime,\sigma^\prime)$ induces a function 
 \[
 \mathrm{Strat}_K(G,\DD,\Sigma)\to \mathrm{Strat}_{K^\prime}(G^\prime,\DD^\prime,\Sigma^\prime)
 \]
on $K$-isomorphism classes of toroidal stratum representatives.


\subsection{Compactification of canonical models}
\label{ss:compactification}


In this subsection we assume that $K\subset G(\A_f)$ is neat.   Suppose $\Sigma$ is a  finite and $K$-admissible complete cone decomposition  for $(G,\DD)$.

\begin{remark}\label{rem:nice decompositions}
A $\Sigma$ with the above properties always exists, and  may be refined,  in the sense of \cite[\S 5.1]{pink}, to make it smooth.
This is the content of \cite[Theorem 9.21]{pink}.
\end{remark}

The main result of \cite[\S 12]{pink} is the existence  of a  proper toroidal  compactification 
\[
\Sh_K(G,\DD) \hookrightarrow \Sh_K(G,\DD , \Sigma),
\]
in the category of algebraic spaces over $E(G,\DD)$, along  with  a stratification
\begin{equation}\label{generic stratification}
\Sh_K(G,\DD , \Sigma) = \bigsqcup_{ (\Phi,\sigma) \in \mathrm{Strat}_K (G,\DD,  \Sigma  )  } 
Z_K^{(\Phi,\sigma)} (G,\DD ,\Sigma )
\end{equation}
by locally closed subspaces indexed by the finite set $ \mathrm{Strat}_K (G,\DD,  \Sigma  )$ appearing in Definition \ref{def:finite decomposition}.  
The stratum indexed by $( \Phi,\sigma)$ lies in the closure of the stratum indexed by $(\Phi_1,\sigma_1)$ if and only if there is a $K$-morphism of toroidal stratum representatives $(\Phi ,\sigma ) \to (\Phi_1,\sigma_1)$.

 If $\Sigma$ is smooth then so is the toroidal compacification.

After possibly shrinking $K$, we may assume that $\Sigma$ has the no self-intersection property (see Remark \ref{rem:K shrink}).
The no self-intersection property guarantees that the strata appearing in (\ref{generic stratification})  have an especially simple shape.   
Fix one $( \Phi , \sigma) \in \mathrm{Strat}_K(G,\DD,\Sigma)$ and write $\Phi = (P ,\DD^\circ ,h)$.  Pink shows that there is a canonical isomorphism
\begin{equation}\label{boundary divisor iso}
\xymatrix{
{  Z^\sigma_{K_\Phi} (Q_\Phi ,\DD_\Phi ,\sigma )}  \ar[rrr]^{\iso}  \ar[d]
& & &  {   Z_K^{ (\Phi,\sigma) } (G,\DD ,\Sigma )   }  \ar[d] \\
{       \Sh_{K_\Phi} (Q_\Phi ,\DD_\Phi  ,\sigma )   }   
& & &  {  \Sh_K(G,\DD , \Sigma)  } 
}
\end{equation}
such that  the two algebraic spaces in the bottom row become isomorphic  after formal completion along their common locally closed subspace in the top row.  See \cite[Corollary 7.17]{pink} and \cite[Theorem 12.4]{pink}.

In other words, if we  abbreviate 
\[
\widehat{\Sh}_{K_\Phi} (Q_\Phi ,\DD_\Phi  ,\sigma) =
\Sh_{K_\Phi} (Q_\Phi ,\DD_\Phi ,\sigma  )^\wedge_{ Z^\sigma_{K_\Phi} (Q_\Phi ,\DD_\Phi  ,\sigma ) }
\] 
for the formal completion of   $\Sh_{K_\Phi} (Q_\Phi ,\DD_\Phi  ,\sigma )$ along  its closed stratum, 
and abbreviate\footnote{In order to limit the already burdensome notation, we choose to suppress the dependence on $(\Phi,\sigma)$ of the left hand side.  The meaning will always be clear from context.}
\[
\widehat{\Sh}_K(G,\DD , \Sigma)   = \Sh_K(G,\DD , \Sigma)^\wedge_{  Z_K^{(\Phi,\sigma)} (G,\DD,\Sigma  )  } 
\] 
for the formal completion of $ \Sh_K(G,\DD , \Sigma)$ along  the locally closed stratum  $Z_K^{(\Phi,\sigma)} (G,\DD ,\Sigma )$,  there is   an isomorphism of formal algebraic spaces
\begin{equation}\label{boundary iso}
  \widehat{ \Sh }_{K_\Phi} (Q_\Phi ,\DD_\Phi  ,\sigma )  \iso \widehat{ \Sh }_K(G,\DD , \Sigma).
\end{equation}

\begin{remark}
In \cite{pink} the isomorphism (\ref{boundary iso}) is constructed after the left hand side is replace by its quotient by a finite group action.  
Thanks to Hypothesis \ref{hyp:motivic} and the assumption that $K$ is neat, the finite group in question is trivial.  See \cite[Lemma 1.7 and Remark 1.8]{wildeshaus}.
\end{remark}

We can make the above more explicit on the level of complex points.  Suppose $(\Phi , \sigma)$ is  a toroidal stratum representative with underlying cusp label representative $\Phi = (P,\DD^\circ, h)$, and denote by  $Q_\Phi(\R)^\circ \subset Q_\Phi(\R)$  the stabilizer of the connected component $\DD^\circ \subset \DD$.  The complex manifold
\[
\mathscr{U}_{K_\Phi} (Q_\Phi , \DD_\Phi) = Q_\Phi(\Q)^\circ \backslash  ( \DD^\circ \times Q_\Phi(\A_f)  / K_\Phi )
\]
sits in a diagram
\begin{equation}\label{analytic nbhd}
\xymatrix{
{  \mathscr{U}_{K_\Phi} (Q_\Phi , \DD_\Phi)   } \ar[r]\ar[d]_{    (z,g) \mapsto   (z,gh)     }   & {    \Sh_{K_\Phi} (Q_\Phi , \DD_\Phi) (\C)     }  \\
{    \Sh_K(G,\DD)(\C)    }
}
\end{equation}
in which the horizontal arrow is an open immersion, and the vertical arrow is a local isomorphism.  This allows us to define a partial compactification 
\[
 \mathscr{U}_{K_\Phi} (Q_\Phi , \DD_\Phi) \hookrightarrow  \mathscr{U}_{K_\Phi} (Q_\Phi , \DD_\Phi ,\sigma)
\]
as the interior of the closure of $ \mathscr{U}_{K_\Phi} (Q_\Phi , \DD_\Phi)$ in $ \Sh_{K_\Phi} (Q_\Phi , \DD_\Phi,\sigma)(\C)$.

Any $K$-morphism as in Definition \ref{def:K-morphism} induces a morphism of complex manifolds
\[
\mathscr{U}_{K_\Phi} (Q_\Phi , \DD_\Phi) \map{ (z,g)  \mapsto  (  \gamma z , \gamma g \gamma^{-1} q   )  } \mathscr{U}_{K_{\Phi_1}} (Q_{\Phi_1} , \DD_{\Phi_1}),
\]
which extends uniquely to 
\[
\mathscr{U}_{K_\Phi} (Q_\Phi , \DD_\Phi, \sigma)\to  \mathscr{U}_{K_{\Phi_1}} (Q_{\Phi_1} , \DD_{\Phi_1} ,\sigma_1).
\]
Complex analytically, the toroidal compactification is defined as the quotient
\[
\Sh_K(G,\DD , \Sigma) (\C) =  \Big(  \bigsqcup_{  (\Phi ,\sigma)  \in \mathrm{Strat}_K( G, \DD,\Sigma) }\mathscr{U}_{K_\Phi} (Q_\Phi , \DD_\Phi, \sigma)  \Big)  \Big/\sim,
\]
where $\sim$ is the equivalence relation generated by the graphs of all such morphisms.

By  \cite[\S 6.13]{pink} the closed stratum  appearing in (\ref{mixed stratification}) satisfies
\begin{equation}\label{closure orientation}
Z_{K_\Phi}^\sigma(Q_\Phi , \DD_\Phi,\sigma)(\C) \subset  \mathscr{U}_{K_\Phi} (Q_\Phi , \DD_\Phi ,\sigma).
\end{equation}
The  morphisms in (\ref{analytic nbhd})  extend continuously to morphisms
\begin{equation}\label{topological comparison}
\xymatrix{
{  \mathscr{U}_{K_\Phi} (Q_\Phi , \DD_\Phi,\sigma)   } \ar[r]\ar[d]      & {    \Sh_{K_\Phi} (Q_\Phi , \DD_\Phi , \sigma) (\C)   }  \\
{    \Sh_K(G,\DD,\Sigma)(\C)    }
}
\end{equation}
in such a way that the vertical map identifies 
\[
Z_{K_\Phi}^\sigma(Q_\Phi , \DD_\Phi,\sigma)(\C) \iso Z_K^{ (\Phi,\sigma) } (G,\DD,\Sigma  )(\C).
\]
This agrees with the analytification of the  isomorphism (\ref{boundary divisor iso}).

Now pick any point $z\in Z_\Phi^\sigma(Q_\Phi , \DD_\Phi,\sigma)(\C)$.  Let $R$ be the completed local ring of  $\Sh_K(G,\DD,\Sigma)_{/\C}$ at $z$, and let  $R_\Phi$ be the completed local ring of  $\Sh_{K_\Phi}(Q_\Phi,\DD_\Phi,\sigma)_{/\C}$ at $z$.  Each completed local ring can be computed with respect to the \'etale or  analytic topologies, and the results are canonically identified.   Working in the analytic topology, the  morphisms in (\ref{topological comparison}) induce an isomorphism $R \iso R_\Phi,$ as they identify both  rings  with the completed local ring of $ \mathscr{U}_{K_\Phi} (Q_\Phi , \DD_\Phi,\sigma)$ at $z$.  This analytic isomorphism agrees with the one induced by the algebraic isomorphism (\ref{boundary iso}).


\section{Automorphic vector bundles}
\label{s:AVB}


Throughout \S \ref{s:AVB} we fix  a Shimura datum  $(G,\DD)$  satisfying Hypothesis \ref{hyp:motivic}, and  a compact open subgroup $K\subset G(\A_f)$.

We  recall the theory of automorphic vector bundles on  the Shimura variety  $\Sh_K(G,\DD)$, on its toroidal compactification, and on the mixed Shimura varieties appearing along the boundary.  The main reference is  \cite{HZ3}.


\subsection{Holomorphic  vector bundles}


Let  $\Phi = (P,\DD^\circ , h)$ be a cusp label representative for $(G,\DD)$.  As in \S \ref{s:pink review}, this determines a mixed Shimura datum $(Q_\Phi , \DD_\Phi)$ and a compact open subgroup $K_\Phi \subset Q_\Phi(\A_f)$.

Suppose we have a representation  $Q_\Phi \to \GL(N)$ on a finite dimensional $\Q$-vector space.  
Given a point $z\in \DD_\Phi$, its image under 
\[
 \DD_\Phi \to \Hom(\mathbb{S}_\C , Q_{\Phi\C})
\]  
determines a  mixed Hodge structure  $(N,F^\bullet N_\C, \mathrm{wt}_\bullet N)$.    The weight filtration is  independent of $z$, and is split  by any lift $\mathbb{G}_m \to Q_\Phi$ of the weight cocharacter (\ref{weight cocharacter}).

Denote by  $(\bm{N}^{an}_{dR} , F^\bullet \bm{N}^{an}_{dR} , \mathrm{wt}_\bullet \bm{N}^{an}_{dR})$ the  doubly filtered holomorphic vector bundle on  $\DD_\Phi \times Q_\Phi(\A_f) / K_\Phi$   whose fiber at  $(z,g)$ is the vector space $N_\C$ endowed with the Hodge and weight filtrations determined by $z$.    There is a natural action of $Q_\Phi(\Q)$ on this doubly filtered vector bundle,  covering the action on  the base.  By  taking the quotient, we obtain   a functor 
\begin{equation}\label{mixed analytic bundles}
N \mapsto (\bm{N}^{an}_{dR} , F^\bullet \bm{N}^{an}_{dR} , \mathrm{wt}_\bullet \bm{N}^{an}_{dR})
\end{equation}
from finite dimensional representations of $Q_\Phi$ to doubly filtered holomorphic vector bundles on $\Sh_{K_\Phi}(Q_\Phi, \DD_\Phi)(\C)$.
Ignoring the double filtration, this functor is simply
\begin{equation}\label{mixed analytic bundles simple}
N\mapsto \bm{N}^{an}_{dR}  = Q_\Phi(\Q) \backslash  \big( \DD_\Phi \times N_\C \times Q_\Phi(\A_f) / K_\Phi \big).
\end{equation}

Given  a $K_\Phi$-stable $\widehat{\Z}$-lattice  $N_{\widehat{\Z}} \subset N\otimes \A_f$, we may define  a  $\Z$-lattice
\[
gN_\Z = g N_{\widehat{\Z}} \cap N
\]
for every $g\in Q_\Phi(\A_f)$, along with a weight filtration 
\[
\mathrm{wt}_\bullet (gN_\Z) = g N_{\widehat{\Z}} \cap \mathrm{wt}_\bullet  N.
\]
Denote by $(\bm{N}_{Be}, \mathrm{wt}_\bullet\bm{N}_{Be})$ the filtered   $\Z$-local system on $\DD_\Phi\times Q_\Phi(\A_f) / K_\Phi$ whose fiber at $(z,g)$ is $(gN_\Z, \mathrm{wt}_\bullet (g N_\Z))$.  This local system has an obvious action of $Q_\Phi(\Q)$, covering the action on the base.  Passing to the  quotient, we obtain a functor
\[
N_{\widehat{\Z}} \mapsto  ( \bm{N}_{Be}, \mathrm{wt}_\bullet\bm{N}_{Be})
\]
from $K_\Phi$-stable $\widehat{\Z}$-lattices in $N\otimes \A_f$ to  filtered $\Z$-local systems on  (\ref{complex mixed}).

By construction there is a canonical isomorphism
\begin{equation}\label{general mixed betti-derham}
(\bm{N}^{an}_{dR} ,  \mathrm{wt}_\bullet \bm{N}^{an}_{dR}) \iso (\bm{N}_{Be}  \otimes \co^{an} ,\mathrm{wt}_\bullet \bm{N}_{Be}  \otimes \co^{an}),
\end{equation}
where $\co^{an}$ denotes the structure sheaf on $\Sh_{K_\Phi}(Q_\Phi, \DD_\Phi)(\C)$.


\subsection{The Borel morphism}


Suppose $G\to \GL(N)$ is any faithful representation of $G$ on a finite dimensional $\Q$-vector space.   A point $z\in \DD$ determines a Hodge structure  $\mathbb{S} \to \GL(N_\R)$  on $N$,  and we denote by  $F^\bullet N_\C$ the induced Hodge filtration.   
As in  \cite[\S III.1]{milne:canonical} and  \cite[\S 1]{HZ3},  define the \emph{compact dual}
\begin{equation}\label{compact dual}
\check{M}(G,\DD)(\C) = \left\{ \begin{array}{cc} \mbox{descending filtrations on $N_\C$} \\ \mbox{that are $G(\C)$-conjugate to $F^\bullet N_\C$} \end{array} \right\}.
\end{equation}
 By construction, there  is a canonical $G(\R)$-equivariant  finite-to-one \emph{Borel morphism}
\[
\DD \to \check{M}(G,\DD)(\C)
\]
sending a point of $\DD$ to the induced Hodge filtration on $N_\C$.  The compact dual is the space of complex points of  a smooth projective  variety $\check{M}(G,\DD)$ defined  over the reflex field $E(G,\DD)$,   and admitting an action of $G_{E(G,\DD)}$ inducing the natural action of $G(\C)$ on complex points.   It is independent of the choice of $z$, and of the  choice of faithful representation $N$.

More generally,  there is an analogue of (\ref{compact dual}) for the mixed Shimura datum $(Q_\Phi,\DD_\Phi)$, as in  \cite[Main Theorem 3.4.1]{hor:thesis} and \cite[Main Theorem 2.5.12]{hor:book}.   Let  $Q_\Phi \to \GL(N)$ be a faithful representation on a finite dimensional $\Q$-vector space.  Any point $z\in \DD_\Phi$ then determines a mixed Hodge structure  $(N, F^\bullet N_\C, \mathrm{wt}_\bullet N)$, and we define the \emph{dual} of $(Q_\Phi , \DD_\Phi)$ by 
\[
\check{M}(Q_\Phi,\DD_\Phi )(\C) = \left\{ \begin{array}{cc} \mbox{descending filtrations on $N_\C$} \\ \mbox{that are $Q_\Phi(\C)$-conjugate to $F^\bullet N_\C$} \end{array} \right\}.
\]
 It is the space of complex points of  an open   $Q_{\Phi , E(G,\DD)}$-orbit
\[
 \check{M}(Q_\Phi,\DD_\Phi ) \subset  \check{M}(G,\DD),
\]
independent of the choice of $z\in \DD_\Phi$ and $N$.    By construction, there is a $Q_\Phi(\C)$-equivariant \emph{Borel morphism}
 \[
\DD_\Phi \to \check{M}(Q_\Phi,\DD_\Phi )(\C).
\]


\subsection{The standard torsor}


We want to give a more algebraic interpretation of the functor (\ref{mixed analytic bundles}).

Harris and Zucker \cite[\S 1]{HZ3} prove that  the mixed Shimura variety (\ref{complex mixed}) carries a \emph{standard torsor}\footnote{a.k.a.~\emph{standard principal bundle}}.   This consists of a diagram of $E(G,\DD )$-stacks
\begin{equation}\label{general  mixed local model}
\xymatrix{
 {  J_{K_\Phi}( Q_\Phi ,\DD_\Phi)    }   \ar[d]_a \ar[r]^{ b }  &    {  \check{M}(Q_\Phi ,\DD_\Phi) }  \\
{ \Sh_{K_\Phi}(Q_\Phi,\DD_\Phi)  ,}  
}
\end{equation}
in which  $a$ is a relative $Q_\Phi$-torsor, and $b$ is $Q_\Phi$-equivariant.   
See also the papers of Harris \cite{harris:automorphic_0,harris:automorphic_1,harris:automorphic_2}, Harris-Zucker \cite{HZ1,HZ2}, and Milne \cite{milne:automorphic,milne:canonical}.   Complex analytically, the standard torsor  is the complex orbifold
\[
J_{K_\Phi}( Q_\Phi ,\DD_\Phi) (\C) = Q_\Phi(\Q) \backslash  \big( \DD_\Phi \times Q_\Phi(\C) \times Q_\Phi(\A_f) /K_\Phi \big),
\]
with   $Q_\Phi(\C)$  acting  by  $s \cdot  ( z,t,g)  = ( z,ts^{-1}, g)$.    The morphisms $a$ and $b$ are, respectively, 
\[
( z,t,g) \mapsto ( z, g)  \quad\mbox{and}\quad  ( z,t,g) \mapsto t^{-1} z.
\]

Exactly as in \cite{HZ3},  we can use the standard torsor to define models of the vector bundles (\ref{mixed analytic bundles}) over the reflex field.   First, we require a lemma.

\begin{lemma}\label{lem:dual filtrations}
Suppose  $\check{N}\to \check{M}(Q_\Phi ,\DD_\Phi) $ is a  $Q_\Phi$-equivariant vector bundle; that is, a finite rank vector bundle   endowed with an action  of $Q_{\Phi,E(G,\DD)} $ covering the action on the base.   There are canonical $Q_\Phi$-equivariant filtrations
$\mathrm{wt}_\bullet \check{N}$ and $F^\bullet \check{N}$ on $\check{N}$, and the construction
\[
\check{N} \mapsto ( \check{N} , F^\bullet \check{N} , \mathrm{wt}_\bullet \check{N} )
\]
is functorial in $\check{N}$.
\end{lemma}

\begin{proof}
Fix a faithful representation $Q_\Phi \to \GL(H)$.  Suppose we are given an \'etale neighborhood $U \to   \check{M}(Q_\Phi , \DD_\Phi)$ of some geometric point  $x$ of $\check{M}(Q_\Phi , \DD_\Phi)$.  By the very definition of $\check{M}(Q_\Phi , \DD_\Phi)$, $U$ determines a $Q_{ \Phi U}$-stable filtration $F^\bullet H_U $ on $H_U = H\otimes \co_U$.  After possibly shrinking $U$ we may choose a   cocharacter $\mu_x : \mathbb{G}_m \to Q_{\Phi  U}$ splitting this filtration.

As $Q_{\Phi U}$ acts on $\check{N}_U$, the cocharacter $\mu_x$ determines a filtration $F^\bullet \check{N}_U$, which does not depend on the choice of splitting.  Glueing over an \'etale cover of $\check{M}(Q_\Phi , \DD_\Phi)$ defines the desired filtration $F^\bullet \check{N}$. The definition of  $\mathrm{wt}_\bullet \check{N}$ is similar, but easier: it is the  filtration  split  by any lift $\mathbb{G}_m \to Q_\Phi$ of the weight cocharacter (\ref{weight cocharacter}).
\end{proof}

Now suppose we have a representation $Q_\Phi \to \GL(N)$ on a finite dimensional $\Q$-vector space.  Applying Lemma \ref{lem:dual filtrations} to the constant $Q_\Phi$-equivariant vector bundle
\[
\check{N} =  \check{M}(Q_\Phi ,\DD_\Phi)    \times_{\Spec(E(G,\DD ) )  } N_{E(G,\DD)}
\]
yields a $Q_\Phi$-equivariant doubly filtered vector bundle $( \check{N} , F^\bullet \check{N} , \mathrm{wt}_\bullet \check{N} )$ on $\check{M}(Q_\Phi ,\DD_\Phi)$.    The construction
 \begin{equation}\label{mixed bundles}
N\mapsto ( \bm{N}_{dR}   , F^\bullet\bm{N}_{dR}, \mathrm{wt}_\bullet\bm{N}_{dR} ) =  Q_\Phi\backslash  b^* (\check{N}, F^\bullet \check{N},  \mathrm{wt}_\bullet \check{N}) 
\end{equation}
defines a functor from  representations of $Q_\Phi$  to doubly filtered  vector bundles on  $\Sh_{K_\Phi}(Q_\Phi,\DD_\Phi)$.  Passing to the complex fiber  recovers the functor (\ref{mixed analytic bundles}).

 The following proposition extends the above functor to partial compactifications.

 \begin{proposition}\label{prop:canonical bundle machine}
For any rational polyhedral cone $\sigma \subset U_\Phi(\R)(-1)$  there is a functor 
 \[
N\mapsto ( \bm{N}_{dR}   , F^\bullet\bm{N}_{dR}, \mathrm{wt}_\bullet\bm{N}_{dR} ),
\] 
extending (\ref{mixed bundles}), from  representations of $Q_\Phi$ on finite dimensional $\Q$-vector spaces to doubly filtered  vector bundles on  $\Sh_{K_\Phi}(Q_\Phi,\DD_\Phi,\sigma)$. 
 \end{proposition}

\begin{proof}
This is part of \cite[Definition-Proposition 1.3.5]{HZ3}.  Here we sketch a different argument.

Recall the $T_\Phi$-torsor structure on (\ref{torus torsor}).  On complex points, this action was deduced from the natural  left action of $U_\Phi(\C)$ on $\DD_\Phi$.    Of course the group  $U_\Phi(\C)$  also acts on both factors of $\DD_\Phi \times Q_\Phi(\C)$  on the left, and imitating the proof of Proposition \ref{prop:torsor def} yields action of the relative torus $T_\Phi(\C)$ on the standard torsor $J_{K_\Phi}( Q_\Phi ,\DD_\Phi)(\C)$, covering the action  on $\Sh_{K_\Phi}(Q_\Phi,\DD_\Phi)(\C)$.

  To see that the action is algebraic and defined over the reflex field, one can reduce, exactly as in the proof of  \cite[Proposition 1.2.4]{HZ3}, to the case in which $(Q_\Phi , \DD_\Phi)$ is either a pure Shimura datum, or is a mixed Shimura datum associated with a Siegel Shimura datum.  The pure case is vacuous  (the relative torus is trivial).  The   Siegel mixed Shimura varieties are moduli spaces of polarized $1$-motives, and it is not difficult to give a moduli-theoretic interpretation of the torus action, along the lines of \cite[\S 2.2.8]{mp:compactification}. 
  From this interpretation the descent to the reflex field is obvious.

In the diagram (\ref{general  mixed local model}),  the arrow $a$ is $T_\Phi$-equivariant, and the arrow $b$ is constant on $T_\Phi$-orbits.  This is clear from the complex analytic description.

Taking the quotient of the standard torsor by this action, we obtain a diagram
\[
\xymatrix{
 {  T_\Phi  \backslash  J_{K_\Phi}( Q_\Phi ,\DD_\Phi)    }   \ar[d]_a \ar[rr]^{ \qquad b }  & &   {  \check{M}(Q_\Phi ,\DD_\Phi) }  \\
{\Sh_{\bar{K}_\Phi}( \bar{Q}_\Phi , \bar{\DD}_\Phi )  ,}  
}
\]
in which  $a$ is a relative $Q_\Phi$-torsor  and $b$ is $Q_\Phi$-equivariant.    Pulling back the quotient $T_\Phi  \backslash  J_{K_\Phi}( Q_\Phi ,\DD_\Phi) $  along the diagonal arrow in (\ref{mixed compactification}) defines the upper left entry in the diagram 
\[
\xymatrix{
 {  J_{K_\Phi}( Q_\Phi ,\DD_\Phi,\sigma)   }   \ar[d]_a \ar[rr]^{ b }  &  &    {  \check{M}(Q_\Phi ,\DD_\Phi)  }  \\
{ \Sh_{K_\Phi}(Q_\Phi,\DD_\Phi,\sigma )  } 
}
\]
extending (\ref{general mixed local model}), in which $a$ is a $Q_\Phi$-torsor, and $b$ is $Q_\Phi$-equivariant.  Now simply repeat the construction (\ref{mixed bundles}) to obtain the desired functor.
 \end{proof}

\begin{remark}
The proof actually shows more: because the standard torsor admits a canonical descent to $\Sh_{\bar{K}_\Phi}( \bar{Q}_\Phi , \bar{\DD}_\Phi )$, the same is true of  all doubly filtered  vector bundles (\ref{mixed bundles}).  Compare with \cite[(1.2.11)]{HZ3}.
\end{remark}


\subsection{Automorphic vector bundles on toroidal compactifications}


Assume that $K$ is neat, and that $\Sigma$ is a   finite $K$-admissible complete cone decomposition for $(G,\DD)$ having the no self-intersection property.   


By results of Harris and Harris-Zucker, see especially \cite{HZ3},  one can glue together the diagrams in the proof of Proposition \ref{prop:canonical bundle machine} as $(\Phi , \sigma)$ varies in order to obtain a diagram
\begin{equation}\label{general compact local model}
\xymatrix{
 {  J_K( G,\DD ,\Sigma)    }   \ar[d]_a \ar[r]^{ b }  &    {  \check{M}(G,\DD) }  \\
{ \Sh_K(G,\DD,\Sigma)  }  
}
\end{equation}
in which $a$ is a $G$-torsor and $b$ is $G$-equivariant.  This implies the following:

\begin{theorem}\label{thm:total vector bundles}
There is a  functor  $N \mapsto (\bm{N}_{dR} , F^\bullet \bm{N}_{dR} )$ from representations of $G$ on finite dimensional $\Q$-vector spaces  to filtered vector bundles on  $\Sh_K(G,\DD,\Sigma)$,  compatible, in the obvious sense, with the isomorphism 
\[
 \widehat{ \Sh }_{K_\Phi} (Q_\Phi ,\DD_\Phi  ,\sigma )  \iso \widehat{ \Sh }_K(G,\DD , \Sigma)
 \]
 of (\ref{boundary iso}) and the functor of Proposition \ref{prop:canonical bundle machine},  for every   toroidal stratum representative
\[
(\Phi , \sigma) \in \mathrm{Strat}_K(G,\DD,\Sigma).
\]
\end{theorem}

In other words, there is an arithmetic theory of  automorphic vector bundles on toroidal compactifications.   

\begin{remark}\label{rem:no weight}
Over the open Shimura variety $\Sh_K(G,\DD)$ there is also a weight filtration $\mathrm{wt}_\bullet \bm{N}_{dR}$ on $\bm{N}_{dR}$, but it is not  compatible with the weight filtrations along the boundary.  It is also not very interesting.  
On an irreducible representation $N$ the (central) weight cocharacter $w:\mathbb{G}_m \to G$ acts  through  $z\mapsto z^k$ for some $k$,  and  the weight filtration has a unique nonzero graded piece $\mathrm{gr}_k \bm{N}_{dR}$.
\end{remark}


\subsection{A simple Shimura variety}
\label{ss:simple}


Let $(\mathbb{G}_m,\mathcal{H}_0)$ be the Shimura datum of Remark \ref{rem:little shimura}.
For any compact open subgroup $K \subset \A_f^\times$, we obtain  a $0$-dimensional Shimura variety 
\begin{equation}\label{simple shimura}
\Sh_K(\mathbb{G}_m ,\mathcal{H}_0 ) (\C)= \Q^\times \backslash  (  \mathcal{H}_0 \times \A_f^\times / K  ),
\end{equation}
with a canonical model $\Sh_K(\mathbb{G}_m , \mathcal{H}_0)$ over $\Q$.

The action of $\Aut(\C)$  on its complex points satisfies 
\begin{equation}\label{simple reciprocity}
\tau\cdot ( 2\pi \epsilon , a) = ( 2\pi \epsilon , a a_\tau  )
\end{equation}
whenever $\tau\in \Aut(\C)$ and  $a_\tau \in \A_f^\times$ are related by   $\tau|_{\Q^{ab}} = \mathrm{rec}(a_\tau)$.  
This implies that
\[
\Sh_K(\mathbb{G}_m ,\mathcal{H}_0 ) \iso \Spec(F),
\]
where $F/\Q$ is the abelian extension characterized by
\[
\mathrm{rec} :    \Q^\times_{>0} \backslash  \A_f^\times / K  \iso \Gal( F / \Q).
\]

The following proposition  shows that all  automorphic   vector bundles on (\ref{simple shimura}) are canonically trivial.
The particular trivializations will be essential in our later discussion of $q$-expansions.  See especially Proposition \ref{prop:canonical sections}.

  \begin{proposition}
For any representation  $\mathbb{G}_m \to \GL(N)$ there is a canonical isomorphism 
\[
  N \otimes  \co_{  \Sh_K(\mathbb{G}_m , \mathcal{H}_0) }\map{n \otimes 1\mapsto \bm{n}  }   \bm{N}_{dR} 
\]
of vector bundles.  If $\mathbb{G}_m$ acts on $N$ through the character $z\mapsto z^k$, the global section $\bm{n}= n \otimes 1$  is given, in terms of the complex parametrization 
\[
\bm{N}^{an}_{dR} = \Q^\times \backslash ( \mathcal{H}_0 \times N_\C \times  \A_f^\times/K ) 
\]
of  (\ref{mixed analytic bundles simple}), by
\[
( 2\pi \epsilon, a)  \mapsto \left(  2\pi \epsilon,   \frac{\mathrm{rat}(a)^k }{(2\pi \epsilon)^k}   \cdot n , a    \right).
\]
\end{proposition}

\begin{proof}
 First set  $N=\Q$ with $\mathbb{G}_m$ acting via the identity character $z\mapsto z$, and set $N_{\widehat{\Z}}=\widehat{\Z}$.  Recalling (\ref{general mixed betti-derham}),  the quotient $\bm{N}_{Be}\backslash \bm{N}^{an}_{dR}$  defines an analytic family of rank one tori  over $\Sh_K(\mathbb{G}_m , \mathcal{H}_0)(\C)$, whose relative Lie algebra is the line bundle
\[
\Lie(\bm{N}_{Be}\backslash \bm{N}^{an}_{dR} )  = \bm{N}^{an}_{dR} = \Q^\times \backslash ( \mathcal{H}_0 \times \C \times  \A_f^\times/K).
\]
Using this identification,  we may  identify  the standard $\C^\times$-torsor
\begin{equation}\label{simple analytic torsor}
J_K(\mathbb{G}_m ,\mathcal{H}_0 ) (\C) = \Q^\times \backslash \big(  \mathcal{H}_0 \times \C^\times  \times \A_f^\times / K\big)
\end{equation}
 with the $\C^\times$-torsor  of trivializations of $\Lie(\bm{N}_{Be}\backslash \bm{N}^{an}_{dR} )$.

On the other hand, the isomorphisms
 \[
( N \cap  aN_{\widehat{\Z}} )  \backslash N_\C =  ( \Q \cap a\widehat{\Z} ) \backslash \C \map{  2\pi \epsilon /\mathrm{rat}(a)  } \Z(1) \backslash \C \map{\exp}\C^\times
 \]
 identify $\bm{N}_{Be}\backslash \bm{N}_{dR}^{an}$, fiber-by-fiber,  with the constant torus $\C^\times$, and so identify (\ref{simple analytic torsor}) with the $\C^\times$-torsor  of trivializations of $\Lie(\C^\times )$.   The canonical model of (\ref{simple analytic torsor})   is now concretely  realized as the $\mathbb{G}_m$-torsor
\[
J_K(\mathbb{G}_m ,\mathcal{H}_0 )  = \underline{\mathrm{Iso}} \big(   \Lie (\mathbb{G}_m)     , \co_{ \Sh_K(\mathbb{G}_m , \mathcal{H}_0) }   \big).
\]

For any ring $R$, the Lie algebra of $\mathbb{G}_m= \Spec( R[q,q^{-1}] )$ is canonically trivialized by the invariant derivation $q \cdot d/dq $.  Thus the standard torsor   admits a canonical section which, in terms of the  uniformization (\ref{simple analytic torsor}), is 
\[
( 2\pi \epsilon, a)  \mapsto   \left(  2\pi \epsilon,   \frac{\mathrm{rat}(a)}{2\pi \epsilon}   , a    \right) .
\]
This section trivializes  the  standard torsor, and induces the desired trivialization of any automorphic vector bundle.   
\end{proof}

\begin{remark}\label{rem:tate bundles}
Let $\mathbb{G}_m$ act on  $N$   via $z\mapsto z^k$.  What the above proof actually shows is that there are canonical isomorphisms
\[
 N \otimes  \co_{  \Sh_K(\mathbb{G}_m , \mathcal{H}_0) } \iso  N \otimes \Lie(\mathbb{G}_m)^{\otimes k} \iso \bm{N}_{dR}.
\]
\end{remark}


\section{Orthogonal Shimura varieties}
\label{s:GSpin}


In \S \ref{s:GSpin} we specialize the preceding theory to the case of Shimura varieties associated to the group  of spinor similitudes of  a quadratic space $(V,Q)$ over $\Q$ of signature $(n,2)$ with $n\ge 1$.
This will allow us to   define $q$-expansions of modular forms on such Shimura varieties,  and prove the $q$-expansion principle of Proposition \ref{prop:q principle}.


\subsection{The GSpin Shimura variety}
\label{ss:gspin data}


 Let $G=\GSpin(V)$ as in \cite{mp:spin}.  This is  a  reductive group over $\Q$ sitting in an exact sequence
\[
1 \to \mathbb{G}_m \to G \to \SO(V)\to 1.
\]
 There is a distinguished character   $\nu : G \to \mathbb{G}_m$, called the \emph{spinor similitude}.  
 Its  kernel is the  usual spin double cover of $\SO(V)$, and its restriction  to    $\mathbb{G}_m$ is $z\mapsto z^2$.

The group $G(\R)$ acts on the hermitian domain
\begin{equation}
\label{orthogonal domain}
 \DD = \big\{ z\in  V_\C : [z,z]=0 \mbox{  and  } [z,\overline{z}]<0 \big\} / \C^\times  \subset \mathbb{P}(V_\C)
\end{equation}
in the obvious way.  This hermitian  domain   has two connected components,  interchanged by the action of any $\gamma\in G(\R)$ with $\nu(\gamma)<0$.      The pair $(G,\DD)$ is the \emph{GSpin Shimura datum}.  Its reflex field is $\Q$.

By construction, $G$ is a subgroup of the multiplicative group of the Clifford algebra $C(V)$.  As such, $G$ has two distinguished representations.  One is the standard representation $G\to \SO(V)$, and the other is the faithful action  on $H=C(V)$ defined by left multiplication in the Clifford algebra.   These two representations are related  by a $G$-equivariant injection 
\begin{equation}\label{special injection}
V \to \End_\Q(H)
\end{equation}
defined by the left multiplication action of $V\subset C(V)$ on $H$. 
A point $z\in \DD$ determines a Hodge structure on any representation of $G$.  
For the representations $V$ and $H$, the induced  Hodge filtrations are
\begin{equation}\label{V hodge}
F ^2 V_\C =0, \quad F ^1 V_\C = \C z,\quad F ^0 V_\C =(\C z)^\perp ,\quad F ^{-1} V_\C =V_\C,
\end{equation}
and
\begin{equation}\label{H hodge}
F^{1} H_\C =0,\quad F^0 H_\C = z H_\C , \quad F ^{-1} H_\C = H_\C.
\end{equation}
Here we are using  (\ref{special injection})  to view $\C z\subset \End_\C(H_\C)$.

In order to obtain a Shimura variety $\Sh_K(G,\DD)$ as in \ref{ss:intro general},  
 we fix  a $\Z$-lattice $V_\Z \subset V$ on which $Q$ is  $\Z$-valued and assume that  the compact open subgroup $K \subset G(\A_f)$ is chosen as in (\ref{K choice}).
According to  \cite[Lemma 2.6]{mp:spin},
any such  $K$ stabilizes both $V_{\widehat{\Z}}$ and its dual, and acts trivially on the discriminant group 
\begin{equation}\label{disc quotient}
V_\Z^\vee / V_\Z \iso V_{\widehat{\Z}}^\vee / V_{\widehat{\Z}}.
\end{equation}


\subsection{The line bundle of modular forms}
\label{ss:modular forms}


Applying the functor of Proposition  \ref{prop:canonical bundle machine} to the standard representation $G \to \SO( V )$ yields  a filtered vector bundle $(\bm{V}_{dR}, F^\bullet \bm{V}_{dR})$ on    $\Sh_K(G , \DD )$.   
The filtration has the form 
\[
0= F^2 \bm{V}_{dR}  \subset F^1 \bm{V}_{dR}  \subset F^0 \bm{V}_{dR}\subset  F^{-1} \bm{V}_{dR} = \bm{V}_{dR},
\]
in which $F^1\bm{V}_{dR}$ is a line,   isotropic with respect to the bilinear form
\begin{equation}\label{de Rham bilinear}
[-,-]: \bm{V}_{dR} \otimes \bm{V}_{dR} \to \co_{ \Sh_K( G , \DD  ) }
\end{equation}
induced by (\ref{bilinear}), and $F^0\bm{V}_{dR} = (F^1\bm{V}_{dR})^\perp$.
These properties are clear from the complex analytic definition (\ref{mixed analytic bundles}) of $ \bm{V}_{dR}^{an}$, and the explicit  description of the Hodge filtration (\ref{V hodge}). 
In particular, the filtration on $\bm{V}_{dR}$ is completely determined by the isotropic line $F^1\bm{V}_{dR}$.

\begin{definition}
The \emph{line bundle of weight one modular forms} on  $\Sh_K(G, \DD )$ is defined by
\[
 \bm{\omega} = F^1 \bm{V}_{dR}.
 \]
\end{definition}

For any $g\in G(\A_f)$, the pullback of $\bm{\omega}$ via the complex uniformization
\[
\DD \map{  z\mapsto (z , g)  } \Sh_K(G,\DD)(\C)
\]
is just the tautological  bundle on the hermitian domain (\ref{orthogonal domain}).   In particular, the line bundle $\bm{\omega}$ carries a metric, inherited from the metric 
\begin{equation}\label{naive metric}
\| z \|^2_\naive = - [ z, \overline{z}]
\end{equation}
on the tautological bundle. 
We will more often use the rescaled metric
\begin{equation}\label{better metric}
\| z \|^2 = - \frac{  [ z, \overline{z}]}{ 4\pi e^\gamma}
\end{equation}
where $\gamma = -\Gamma'(1)$ is the Euler-Mascheroni constant.


\subsection{The Hodge embedding}
\label{ss:hodge embedding}


As above, let $H=C(V)$ viewed as a faithful $2^{n+2}$-dimensional representation of $G \subset C(V)^\times$ via left multiplication.  If we define a $\Z$-lattice in $H$ by
\[
H_{\Z} = C(V_\Z) ,
\]
the inclusion (\ref{K choice}) implies that $H_{\widehat{\Z}} = H_\Z \otimes_\Z \widehat{\Z}$  is $K$-stable.

The discussion of \S \ref{s:AVB} provides  a filtered vector bundle $(\bm{H}_{dR} , F^\bullet \bm{H}_{dR})$ on $\Sh_K(G,\DD)$, and a $\Z$-local system $\bm{H}_{Be}$ over the complex fiber  endowed with an isomorphism
\[
\bm{H}_{Be} \otimes \co_{\Sh_K(G,\DD)(\C)} \iso \bm{H}_{dR}^{an}.
\]
The double quotient 
\begin{equation}\label{analytic KS}
A(\C) = \bm{H}_{Be} \backslash \bm{H}_{dR}^{an} / F^0 \bm{H}_{dR}^{an}
\end{equation}
defines an analytic family of complex tori over $\Sh_K(G,\DD)(\C)$. 
 In fact,  this arises from an abelian scheme over $\Sh_K(G,\DD)$, as we now explain.

As in \cite[\S 2.2]{AGHMP-1},  one may choose a symplectic form $\psi$ on $H$ such that  the representation of $G$ factors through 
$
G^\Sg=\GSp(H),
$ 
and induces a  Hodge embedding
\[
(G,\DD) \to (G^\Sg , \DD^\Sg)
\]
into the Siegel Shimura datum determined by  $(H,\psi)$.   
Explicitly, choose any vectors  $v,w\in V$ of negative length with $[v , w ] =0$ 
and set
\[
\delta = vw\in C(V).
\] 
If we denote by  $c\mapsto c^*$  the $\Q$-algebra involution on $C(V)$ fixing pointwise the subset $V\subset C(V)$, then $\delta^* = -\delta$.  Denoting by $\mathrm{Trd}:C(V)\to \Q$ the reduced trace, the symplectic form
\[
\psi(x,y) = \mathrm{Trd}(x\delta y^*)
\]
has the desired properties.

As in (\ref{compact dual}), we may describe the compact dual $\check{M}(G,\DD)$ as a $G$-orbit of descending filtrations on the faithful representation $H$.  It is more convenient to characterize the compact dual as  the $\Q$-scheme with functor of points
\[
\check{M}(G,\DD) (S) = \{ \mbox{isotropic lines } z \subset V \otimes  \co_S \},
\]
where \emph{line} means a locally free $\co_S$-module direct summand of rank one.  
In order to realize $\check{M}(G,\DD)$ as a space of filtrations on $H$, first 
define 
\[
\check{M}(G^\Sg,\DD^\Sg) (S) = \{ \mbox{Lagrangian subsheaves } F^0   \subset H \otimes  \co_S \}
\]
and then use (\ref{special injection}) to define a closed immersion 
\[
\check{M}(G,\DD) \to \check{M}(G^\Sg,\DD^\Sg)
\]
 sending  the isotropic line $z \subset V$ to the Lagrangian  $z H\subset H$.  

By rescaling, we may assume that $\psi$ is $\Z$-valued on $H_\Z$, and so 
the Hodge embedding  defines a morphism from $\Sh_K(G,\DD)$ to a moduli stack of polarized abelian varieties of dimension $ 2^{n+1}$. 
Pulling back the universal object defines the \emph{Kuga-Satake abelian scheme}
\[
\pi: A \to \Sh_K(G,\DD).
\]
The Kuga-Satake abelian scheme does not depend on the choice of $\psi$, but the polarization on it does.
Passing to the complex analytic fiber recovers the family of complex tori defined by (\ref{analytic KS}).

The first relative de Rham homology  of $A$, with its Hodge filtration,  is related to the vector bundle $\bm{H}_{dR}$ by a canonical isomorphism of filtered vector bundles 
\[
\bm{H}_{dR} \iso \underline{\Hom}\big( R^1 \pi_* \Omega^\bullet_{A/ \Sh_K(G,\DD)} , \co_{\Sh_K(G,\DD)} \big).
\]


\subsection{Cusp label representatives: isotropic lines}
\label{ss:orthogonal clr lines}


We wish to make more explicit the structure of the mixed Shimura datum $(Q_\Phi , \DD_\Phi)$ 
associated to a cusp label representative 
\[
\Phi=( P, \DD^\circ, h)
\]  
for $(G,\DD)$.  See  \S \ref{ss:mixed} for the definitions.

The admissible parabolic $P \subset G$  is the stabilizer of a totally isotropic subspace  $I\subset V$ with $\mathrm{dim}(I)\in \{0,1,2\}$. 
In this subsection we  assume that $P \subset G$ is  the stabilizer of an isotropic line $I\subset V$.   
The case of isotropic planes will be considered in \S \ref{ss:orthogonal clr planes}.

The  $P$-stable weight filtration on $V$ defined by
\[
\mathrm{wt}_{-3}V=0,\quad \mathrm{wt}_{-2}V = \mathrm{wt}_{-1}V = I , \quad \mathrm{wt}_0V=\mathrm{wt}_1V=I^\perp, \quad \mathrm{wt}_2 V= V,
\]
and the Hodge filtration (\ref{V hodge}) determined by a point $z\in \mathcal{D}$,  together determine a mixed Hodge structure 
\[
 \mathbb{S}_\C \map{\mathtt{h}_\Phi(z)}  P_\C \to \SO(V_\C)
\]
on $V$ of type $\{ (-1,-1) , (0,0),(1,1) \}$.

Similarly,the $P$-stable weight filtration on $H$ defined by
\[
\mathrm{wt}_{-3}H=0,\quad \mathrm{wt}_{-2} H = \mathrm{wt}_{-1} H = I H ,\quad \mathrm{wt}_0 H = H,
\]
and the Hodge filtration (\ref{H hodge}) determined by a point $z\in \mathcal{D}$,   together determine a mixed Hodge structure 
\[
 \mathbb{S}_\C \map{\mathtt{h}_\Phi(z)}  P_\C \to \GSp(H_\C)
\]
on $H$ of type $\{ (-1,-1) , (0,0) \}$.
In the definition of the weight filtration we are using the inclusion  $I \subset \End_\Q(H)$ determined by (\ref{special injection}), and setting
\[
I H = \mathrm{Span}_\Q \{ \ell x : \ell \in I ,\, x\in H\} .
\]

The proof of the following lemma is left as an exercise to the reader.

\begin{lemma}
Recalling the notation (\ref{wt grading}), 
the largest closed normal subgroup $Q_\Phi \subset P$ through which every such $\mathtt{h}_\Phi(z)$ factors is
\[
Q_\Phi  = \mathrm{ker} \big(   P \to \GL (   \mathrm{gr}_0(H) )    \big). 
\]
\end{lemma}

%
%
%

The action  $Q_\Phi \to \SO(V)$ is faithful, and  is given on the graded pieces of $\mathrm{wt}_\bullet V$  by the commutative diagram
\begin{equation}\label{line diagram}
\xymatrix{
{ Q_\Phi } \ar[r]^{ \nu_\Phi } \ar[d]  &   { \mathbb{G}_m }  \ar[d]^{  t \mapsto ( t,1,t^{-1} ) }  \\
{ P }  \ar[r] &  { \GL(I) \times \SO(I^\perp / I) \times \GL( V/I^\perp ) , }
}
\end{equation}
in which $\nu_\Phi$ is the restriction to $Q_\Phi$ of the spinor similitude on $G$.  This agrees with the character (\ref{unipotent character}).
The groups $U_\Phi$ and $W_\Phi$ are 
\[
U_\Phi = W_\Phi = \mathrm{ker} ( \nu_\Phi: Q_\Phi \to \mathbb{G}_m  ),
\]
and there is an isomorphism of $\Q$-vector spaces
\begin{equation}\label{cusp unipotent}
(I^\perp / I) \otimes I \iso  U_\Phi(\Q)
\end{equation}
sending $v \otimes \ell \in (I^\perp / I)  \otimes I$ to the unipotent transformation of $V$ defined by 
\[
x\mapsto  x + [x,\ell] v  - [x,  v  ] \ell  -  Q(v)  [x,  \ell   ] \ell .
\]

The dual  of $(Q_\Phi , \DD_\Phi)$ is the $\Q$-scheme with functor of points
\begin{equation}\label{dual domain at line}
\check{M} (Q_\Phi,\DD_\Phi)(S) = 
\left\{ \begin{array}{c} \mbox{isotropic lines }  z \subset V \otimes  \co_S \mbox{ such that}  \\
V \to V/I^\perp \\ \mbox{ identifies }  z \iso (V/I^\perp) \otimes \co_S \end{array}  \right\}.
\end{equation}
Every point 
\[
z\in \DD_\Phi \subset \pi_0(\mathcal{D}) \times \Hom(  \mathbb{S}_\C ,  Q_{\Phi\C}  )
\]
 determines a mixed Hodge structure on $V$ of type $\{ (-1,-1), (0,0) , (1,1)\}$, and the Borel morphism
\[
\DD_\Phi \to \check{M} (Q_\Phi,\DD_\Phi)(\C)
\]
sends  $z$ to the isotropic line $ F^1 V_\C \subset V_\C$.  This induces  an isomorphism
\begin{equation}\label{line realization}
\DD_\Phi \iso \pi_0(\mathcal{D}) \times  \check{M} (Q_\Phi,\DD_\Phi)(\C) .
\end{equation}


\subsection{Cusp label representatives: isotropic planes}
\label{ss:orthogonal clr planes}


In this subsection we  fix a cusp label representative $\Phi=( P, \DD^\circ, h)$ with $P \subset G$ the stabilizer of an isotropic plane $I\subset V$.

The $P$-stable  weight filtrations on $V$ defined by  
\[
\mathrm{wt}_{-2}V=0,\quad \mathrm{wt}_{-1}V = I , \quad \mathrm{wt}_0V = I ^\perp, \quad \mathrm{wt}_1 V= V,
\]
and the Hodge filtration (\ref{V hodge}) determined by a point $z\in \mathcal{D}$,  together determine a mixed Hodge structure 
\[
 \mathbb{S}_\C \map{\mathtt{h}_\Phi(z)}  P_\C \to \SO(V_\C)
\]
on $V$ of type $\{ (-1,0) , (0,-1)  , (0,0),(1,0), (0,1) \}$.

Similarly, the $P$-stable weight filtration on $H$ defined by
\[
\mathrm{wt}_{-3}H=0,\quad \mathrm{wt}_{-2} H = I^2 H ,\quad  \mathrm{wt}_{-1} H = I H ,\quad \mathrm{wt}_0 H = H,
\]
and the Hodge filtration (\ref{H hodge}) determined by a point $z\in \mathcal{D}$,  together determine a mixed Hodge structure 
\[
 \mathbb{S}_\C \map{\mathtt{h}_\Phi(z)}  P_\C \to \GSp(H_\C)
\]
on $H$ of type $\{ (-1,-1) , (-1,0) , (0,-1), (0,0) \}$.
In the definition of the weight filtration we are using the inclusion  $I \subset \End_\Q(H)$ determined by (\ref{special injection}), and setting
\begin{align*}
IH &=  \mathrm{Span}_\Q   \{ \ell  x  : \ell   \in I ,\, x  \in H\}   \\
I^2 H &= \mathrm{Span}_\Q  \{ \ell \ell'  x : \ell , \ell' \in I ,\, x\in H\}  .
\end{align*}

The proof of the following lemma is left as an exercise to the reader.

\begin{lemma}
Recalling the notation (\ref{wt grading}), 
the largest closed normal subgroup $Q_\Phi \subset P$ through which every such $\mathtt{h}_\Phi(z)$ factors is
\[
Q_\Phi  = \ker \big(   P \to \GL(   \mathrm{gr}_0(H) )    \big). 
\]
\end{lemma}

The natural action $Q_\Phi \to \SO(V)$ is faithful, and is trivial on the quotient $I^\perp / I$.  The groups  $U_\Phi \normal W_\Phi \normal Q_\Phi$  are  
\[ 
W_\Phi = \ker (   Q_\Phi \to \GL(  I )   ),
\] 
and 
\[
U_\Phi \iso \bigwedge\nolimits^2 I,
\]
where we identify $a\wedge b \in \bigwedge\nolimits^2 I$ with the unipotent transformation of $V$ defined by
\[
x \mapsto x+ [ x,a] b - [ x,b]a .
\]

The dual of $(Q_\Phi , \DD_\Phi)$ is the $\Q$-scheme with functor of points
\[
\check{M} (Q_\Phi,\DD_\Phi)(S) = 
\left\{ \begin{array}{c} \mbox{isotropic lines }  z \subset V \otimes  \co_S \mbox{ such that}  \\
V \to V/I^\perp \mbox{ identifies }  z \mbox{ with a rank one } \\ \mbox{local direct summand of } (V/I^\perp) \otimes \co_S \end{array}  \right\}.
\]
Every point  $z\in \DD_\Phi$ determines a mixed Hodge structure on $V$ of type $\{ (-1,0) , (0,-1) , (0,0) , (0,1) , (1,0)\}$, and  again  the Borel morphism
\[
\DD_\Phi \to \check{M} (Q_\Phi,\DD_\Phi)(\C)
\]
sends $z\mapsto F^1 V_\C$.  It identifies $\DD_\Phi$ with the open subset
\[
\DD_\Phi = U_\Phi(\C) \DD \subset \pi_0(\DD) \times  \check{M} (Q_\Phi,\DD_\Phi)(\C).
\]


\subsection{The $q$-expansion principle}
\label{ss:qs}


Now suppose the compact open subgroup $K$ of (\ref{K choice})  is neat, and small enough that there exists a finite  $K$-admissible complete cone decomposition $\Sigma$  for $(G,\DD)$ having the no self-intersection property. 
See \S \ref{ss:cones} for the definitions.

The results of Pink recalled in \S \ref{ss:compactification} provide us with a toroidal compactification
\begin{equation}\label{ortho compact}
\Sh_K(G,\DD , \Sigma) = \bigsqcup_{  (\Phi,\sigma)  \in \mathrm{Strat}_K (G,\DD,  \Sigma  )  } 
Z_K^{(\Phi,.\sigma)} (G,\DD ,\Sigma ),
\end{equation}
and the result of Harris-Zucker recalled as Theorem \ref{thm:total vector bundles} gives a filtered vector bundle
\[
0= F^2 \bm{V}_{dR}  \subset F^1 \bm{V}_{dR}  \subset F^0 \bm{V}_{dR}\subset  F^{-1} \bm{V}_{dR} = \bm{V}_{dR}
\]
on the compactification, endowed with a quadratic form 
\[
[ - , - ] : \bm{V}_{dR}  \to \co_{ \Sh_K( G , \DD ,\Sigma  ) }
\]
induced by the bilinear form on $V$.  
Exactly as in  \ref{ss:modular forms}, the \emph{line bundle of weight one modular forms}
\[
\bm{\omega} = F^1\bm{V}_{dR}
\]
is isotropic with respect to this bilinear form, and $F^0\bm{V}_{dR} = ( F^1\bm{V}_{dR} )^\perp$.
These constructions extend those of \S \ref{ss:modular forms} from the open Shimura variety to its compactification.

 In order to define $q$-expansions of sections of  $\bm{\omega}^{\otimes k}$ on (\ref{ortho compact}), we need to make some additional choices.  
The first choice is a boundary  stratum 
\begin{equation}\label{cuspidal stratum}
 Z_K^{(\Phi,\sigma)} (G,\DD, \Sigma)_{/\C} \subset  \Sh_K(G,\DD,\Sigma)_{/\C}
\end{equation}
 indexed by a toroidal stratum representative  
$ (\Phi ,\sigma)$   in which the parabolic subgroup  appearing in  the underlying cusp label representative 
\[
\Phi = (P,\DD^\circ , h)
\]
 is the stabilizer of an isotropic line $I \subset V$.
The second choice is a nonzero vector $\ell \in I$, which will determine a trivialization of $\bm{\omega}$ in a formal neighborhood of the stratum (\ref{cuspidal stratum}).

As $\DD$ has two connected components, there are exactly two continuous surjections $\nu: \DD \to \mathcal{H}_0$.  Fix one of them.  It, along with the spinor similitude  $\nu:G\to\mathbb{G}_m$, induces a morphism of Shimura data
\[
( G, \DD) \map{\nu} (\mathbb{G}_m , \mathcal{H}_0).
\]
Denote by $2\pi \epsilon^\circ \in \mathcal{H}_0$ the image of the component $\DD^\circ$.  
There is a unique continuous extension of $\nu: \DD\to \mathcal{H}_0$ to  $\nu_\Phi:\DD_\Phi \to \mathcal{H}_0$, and this determines a morphism of mixed Shimura data
\begin{equation}\label{component morphism}
(Q_\Phi , \DD_\Phi) \map{\nu_\Phi}  (\mathbb{G}_m , \mathcal{H}_0),
\end{equation}
where $\nu_\Phi:Q_\Phi \to \mathbb{G}_m$ is the character of (\ref{line diagram}).

Applying the functor of Proposition \ref{prop:canonical bundle machine} to the $Q_\Phi$-representations $I\subset V$ determines vector bundles $\bm{I}_{dR} \subset \bm{V}_{dR}$ on   $\Sh_{K_\Phi}(Q_\Phi , \DD_\Phi , \sigma )$.  
The vector bundle $\bm{V}_{dR}$ is endowed with a filtration and a symmetric bilinear pairing, exactly as in the discussion following (\ref{ortho compact}), and restricting  the bilinear pairing yields a homomorphism
 \begin{equation}\label{omega triv}
 [\cdot,\cdot]  : \bm{I}_{dR} \otimes \bm{\omega}   \to \co_{ \Sh_{K_\Phi}( Q_\Phi , \DD_\Phi , \sigma)  } .
\end{equation}

The choice of  nonzero vector $\ell \in I$  defines a section 
\[
\bm{\ell}^{an}  \in H^0\big(  \Sh_{K_\Phi}(Q_\Phi , \DD_\Phi)(\C) , \bm{I}_{dR}^{an}  \big)
\]
of the line bundle 
\[
\bm{I}_{dR}^{an}  = Q_\Phi(\Q) \backslash \big( \DD_\Phi \times I_\C \times Q_\Phi(\A_f) / K_\Phi \big)
\]
by sending
\[
(z,g) \mapsto   \left(
z , \frac{\rat( \nu_\Phi (g))  }{ \nu_\Phi(z)  }   \cdot \ell   ,  g
\right)  .
\]

\begin{proposition}\label{prop:canonical sections}
The holomorphic section $\bm{\ell}^{an}$ extends uniquely to  the partial compactification  $  \Sh_{K_\Phi}(Q_\Phi , \DD_\Phi , \sigma )(\C)$.   This extension is algebraic and defined over  $\Q$, and so arises  from a unique global section 
\begin{equation}\label{magic section}
\bm{\ell } \in H^0 \big(  \Sh_{K_\Phi}(Q_\Phi , \DD_\Phi ,\sigma ) , \bm{I}_{dR}  \big) .
\end{equation}
Moreover, (\ref{omega triv}) is an isomorphism, and induces  an isomorphism
\[
 \bm{\omega} \map{   \psi \mapsto [ \bm{\ell} , \psi ]   }  \co_{  \Sh_{K_\Phi}(Q_\Phi , \DD_\Phi ,\sigma ) }.
\]
\end{proposition}

\begin{proof}
As the action of $Q_\Phi$ on $I$ is via $\nu_\Phi :Q_\Phi/U_\Phi \to \mathbb{G}_m$, the discussion of \S \ref{ss:simple} (see especially Remark \ref{rem:tate bundles}) identifies $\bm{I}_{dR}$ with the pullback of the line bundle 
$
I\otimes \Lie(\mathbb{G}_m) \iso  I \otimes \co_{\Sh_{\nu_\Phi (K_\Phi)}(\mathbb{G}_m ,\mathcal{H}_0 )} 
$ 
via 
\[
 \Sh_{K_\Phi}(Q_\Phi , \DD_\Phi,\sigma) \map{(\ref{mixed compactification})}  \Sh_{\bar{K}_\Phi}( \bar{Q}_\Phi , \bar{\DD}_\Phi) = \Sh_{\nu_\Phi (K_\Phi)}(\mathbb{G}_m ,\mathcal{H}_0 ) .
\]
The section (\ref{magic section}) is simply the pullback of the trivializing section 
\[
\ell \otimes 1  \in 
H^0\big(  \Sh_{\nu_\Phi (K_\Phi)}(\mathbb{G}_m ,\mathcal{H}_0 ) ,  I \otimes \co_{\Sh_{\nu_\Phi (K_\Phi)}(\mathbb{G}_m ,\mathcal{H}_0 )}   \big)  .
\]

It now suffices to  prove that (\ref{omega triv}) is an isomorphism. 
Recall from \S \ref{ss:orthogonal clr lines}  that  $\check{ M  }( Q_\Phi , \DD_\Phi )$ has functor of points
\[
\check{M} (Q_\Phi,\DD_\Phi)(S) = 
\left\{ \begin{array}{c} \mbox{isotropic lines }  z \subset V \otimes \co_S \mbox{ such that}  \\
V \to V/I^\perp \mbox{ identifies }  z \iso (V/I^\perp) \otimes \co_S \end{array}  \right\}.
\]
Let $\check{I}$ and $\check{V}$ be the (constant)  $Q_\Phi$-equivariant vector bundles on $\check{ M  }( Q_\Phi , \DD_\Phi )$  determined by the representations  $I$ and $V$.   In the notation of Lemma \ref{lem:dual filtrations}, the line bundle $\check{\omega}=F^1\check{V}$ is  the tautological bundle, and the bilinear form on $V$ determines a $Q_\Phi$-equivariant isomorphism
\[
\check{I} \otimes  \check{\omega} \to \check{V} \otimes  \check{V} \map{[ - , - ] }  \co_{\check{M}( Q_\Phi , \DD_\Phi)}.
\]
By examining the construction of the functor in Proposition \ref{prop:canonical bundle machine},  the induced morphism (\ref{omega triv}) is also an isomorphism.
\end{proof}

It follows from the analysis of \S \ref{ss:orthogonal clr lines} that the diagram (\ref{mixed compactification}) has the form 
\[
\xymatrix{
{ \Sh_{K_\Phi} ( Q_\Phi ,\DD_\Phi) }  \ar@{->}[r] \ar[d]_{\nu_\Phi} &   {   \Sh _{K_\Phi} (Q_\Phi ,\DD_\Phi ,\sigma  )  }  \ar[dl] \\
{  \Sh_{\nu_\Phi(K_\Phi)} ( \mathbb{G}_m , \mathcal{H}_0)  , }
}
\]
in which the arrow labeled $\nu_\Phi$ is a torsor for the $n$-dimensional torus
\[
T_\Phi =  \Spec \Big( \Q [ q_\alpha]_{  \alpha \in \Gamma_\Phi^\vee(1)      } \Big) 
\]
over the $0$-dimensional base $ \Sh_{\nu_\Phi(K_\Phi)} ( \mathbb{G}_m , \mathcal{H}_0)$.
To define $q$-expansions    we will trivialize this torsor over an \'etale extension of the base,
 effectively putting coordinates on  the mixed Shimura variety $\Sh_{K_\Phi} ( Q_\Phi ,\DD_\Phi)$.

Choose an auxiliary isotropic line $I_* \subset V$ with $[I,I_*]\neq 0$.  
This choice fixes a section 
 \[
\xymatrix{
(Q_\Phi , \DD_\Phi)    \ar[r]_{ \nu_\Phi } & ( \mathbb{G}_m , \mathcal{H}_0)  \ar@/_1pc/[l]_{\spl}.
}
\] 
 The underlying morphism of groups $\spl :  \mathbb{G}_m \to Q_\Phi$ sends, for any $\Q$-algebra $R$,  $ a \in R^\times$ to the  orthogonal transformation 
\begin{equation}\label{component section}
\spl(a) \cdot x = \begin{cases}
ax  & \mbox{if } x\in I_R \\
a^{-1} x &  \mbox{if }x\in I_{*,R} \\
x  &  \mbox{if } x\in (I\oplus I_*)_R^\perp.
\end{cases}
\end{equation}
To characterize $s:\mathcal{H}_0 \to \DD_\Phi$, we first use  (\ref{dual domain at line})   to view
\[
I_{*\C} \in  \check{M}(Q_\Phi , \DD_\Phi) (\C)  .
\]
Recalling the isomorphism (\ref{line realization}), the preimage of $I_{*\C}$ under the projection 
\[
\DD_\Phi \iso \pi_0(\mathcal{D}) \times \check{M}(Q_\Phi ,\mathcal{D}_\Phi)(\C) \to \check{M}(Q_\Phi ,\mathcal{D}_\Phi)(\C)
\]
consists of two points,  indexed by the two connected components of $\mathcal{D}$.
The function $s:\mathcal{H}_0 \to \DD_\Phi$ is  defined by  sending $2\pi \epsilon^\circ \in \mathcal{H}_0$ to the point indexed by $\mathcal{D}^\circ$, and the other element of $\mathcal{H}_0$  to the point indexed by the other connected component of $\mathcal{D}$.

 The section $ \spl$ determines a Levi decomposition $Q_\Phi = \mathbb{G}_m \imes U_\Phi$.   
 Choose a compact open subgroup $K_0 \subset \mathbb{G}_m( \A_f)$ small enough that its image under (\ref{component section}) is contained in $K_\Phi$, and set
\[
K_{\Phi 0} = K_0  \imes ( U_\Phi(\A_f) \cap K_\Phi)  \subset K_\Phi.
\]
Our hypothesis that $K$ is neat implies that $K_0 \subset K_{\Phi 0} \subset K_\Phi$ are also neat.

\begin{proposition}\label{prop:torsor splitting}
Assume that the rational polyhedral cone $\sigma  \subset U_\Phi(\R)(-1)$ has (top) dimension $n$.
 The above choices determine a commutative diagram
\[
\xymatrix{
{   \bigsqcup_{ a\in  \Q^\times_{>0} \backslash \A_f^\times / K_0  }   \widehat{T}_\Phi(\sigma)_{/\C}   } \ar[rr]^{\iso} \ar[d]  & &
 {   \widehat{\Sh}_{K_{\Phi 0}}   ( Q_\Phi ,  \DD_\Phi ,\sigma )_{/\C}  } \ar[d] \\
{   \widehat{ \Sh }_K(G,\DD , \Sigma)_{/\C} }   \ar[rr]^{\iso} & &     {  \widehat{ \Sh}_{K_{\Phi }}   ( Q_\Phi ,  \DD_\Phi ,\sigma )_{/\C}  }
}
\]
of formal algebraic spaces,   in which  the vertical arrows are formally \'etale surjections.
Here 
\begin{equation}\label{formal torus}
\widehat{T}_ { \Phi} (\sigma) \define \Spf \Big( \Q [[ q_\alpha]]_{  \substack{ \alpha \in \Gamma_\Phi^\vee(1)  \\    \langle \alpha, \sigma \rangle \ge 0  }  } \Big) 
\end{equation}
is the formal completion of (\ref{q torus}) along its closed stratum,
 the lower left corner is the formal completion of $\Sh_K(G,\DD,\Sigma)_{/\C}$ along the $0$-dimensional stratum (\ref{cuspidal stratum}), and the bottom isomorphism is (\ref{boundary iso}).
\end{proposition}

 \begin{proof}
Consider  the  diagram
\[
\xymatrix{
{  \Sh_{K_0} ( \mathbb{G}_m , \mathcal{H}_0 )   \times_{\Spec(\Q)  } T_\Phi  }   \ar@{=}[r]  \ar[dr] &   { \Sh_{K_{\Phi 0}}   ( Q_\Phi ,  \DD_\Phi ) }  \ar[r] \ar[d]_{ \nu_\Phi }  & { \Sh_{K_\Phi}   ( Q_\Phi ,  \DD_\Phi )   }   \ar[d]_{\nu_\Phi}  \\
  {   }   &   {  \Sh_{K_0} ( \mathbb{G}_m , \mathcal{H}_0 )   }  \ar[r]  \ar@/_1pc/[u]_{\spl}  &  {   \Sh_{\nu_\Phi(K_\Phi)}   ( \mathbb{G}_m , \mathcal{H}_0 )   } 
}
\]
in which   the  arrows labeled $\nu_\Phi$   are the  $T_\Phi$-torsors of (\ref{torus torsor}), and the isomorphism $``="$  is the trivialization induced by the section $\spl$.

There is  a canonical bijection
\begin{equation}\label{eqn:gm shimura points}
 \Q^\times_{>0} \backslash \A_f^\times / K_0 \xrightarrow{\simeq} \Sh_{K_0} ( \mathbb{G}_m , \mathcal{H}_0 ) (\C) 
\end{equation}
defined by  $a\mapsto [ ( 2\pi \epsilon^\circ , a) ]$.   We remind the reader that $2\pi \epsilon^\circ \in \mathcal{H}_0$
was fixed in the discussion preceding (\ref{component morphism}).

Using this, the top row of the  above diagram exhibits $\Sh_{K_\Phi}   ( Q_\Phi ,  \DD_\Phi )_{/\C}$ as an \'etale quotient 
\begin{equation}\label{torus cover}
 \bigsqcup_{ a\in  \Q^\times_{>0} \backslash \A_f^\times / K_0  }   T_{\Phi/\C} \iso\Sh_{K_{\Phi 0}}   ( Q_\Phi ,  \DD_\Phi )_{/\C}   \to \Sh_{K_\Phi}   ( Q_\Phi ,  \DD_\Phi )_{/\C}.
\end{equation}
This morphism extends to partial compactifications, and formally completing along the closed stratum yields a formally \'etale morphism
\[
 \bigsqcup_{ a\in  \Q^\times_{>0} \backslash \A_f^\times / K_0  }   \widehat{T}_\Phi(\sigma)_{/\C}  \iso \widehat{\Sh}_{K_{\Phi 0}}   ( Q_\Phi ,  \DD_\Phi ,\sigma )_{/\C}   \to  \widehat{\Sh}_{K_\Phi}   ( Q_\Phi ,  \DD_\Phi ,\sigma )_{/\C}.
\]

This defines the top horizontal arrow and the right vertical arrow in the diagram.  The  vertical arrow on the left is defined by the commutativity of the diagram. 
\end{proof}

%
%
%
%
%
%
%
%
%
%

Propositions \ref{prop:torsor splitting} and \ref{prop:canonical sections} now give us a working theory of $q$-expansions along the $0$-dimensional boundary stratum (\ref{cuspidal stratum}) determined by a top dimensional cone $\sigma \subset U_\Phi(\R)(-1)$.
Taking tensor powers in Proposition \ref{prop:canonical sections} determines an isomorphism
\[
[\bm{\ell}^{\otimes k} ,\, \cdot\, ] : \bm{\omega}^{\otimes k} \iso   \co_{  \Sh_{K_\Phi}(Q_\Phi , \DD_\Phi ,\sigma ) } ,
\]
and hence any global section 
\[
\psi \in H^0\big( \Sh_K ( G , \DD  ,\Sigma )_{/\C}  ,  \bm{\omega} ^{\otimes k}   \big) 
\]
determines a formal function $[\bm{\ell}^{\otimes k} , \psi]$ on 
\[
\widehat{ \Sh }_K(G,\DD , \Sigma)_{/\C}   \iso  \widehat{ \Sh}_{K_{\Phi }}   ( Q_\Phi ,  \DD_\Phi ,\sigma )_{/\C}  .
\]
Now pull this formal function back via the formally \'etale surjection
\[
  \bigsqcup_{ a\in  \Q^\times_{>0} \backslash \A_f^\times / K_0  }   \widehat{T}_\Phi(\sigma)_{/\C}   \to   \widehat{ \Sh }_K(G,\DD , \Sigma)_{/\C}
\]
of Proposition \ref{prop:torsor splitting}.  By restricting the pullback to the copy of $ \widehat{T}_\Phi(\sigma)_{/\C}$ indexed by $a$, we obtain a formal $q$-expansion (\emph{a.k.a.}~\underline{F}ourier \underline{J}acobi expansion)
\begin{equation}\label{basic q}
\mathrm{FJ}^{(a)} ( \psi )  = \sum_{  \substack{ \alpha \in \Gamma_\Phi^\vee(1)  \\    \langle \alpha, \sigma \rangle \ge 0  }  } \mathrm{FJ}_\alpha^{(a)} (\psi) \cdot q_\alpha
 \in    \C [[q_\alpha]]_{  \substack{    \alpha \in \Gamma_\Phi^\vee(1)     \\      \langle \alpha, \sigma \rangle \ge 0     }  }.
\end{equation}
We emphasize that (\ref{basic q}) depends on the choice of toroidal stratum representative $(\Phi,\sigma)$, as well as on the choices of $\nu:\DD \to \mathcal{H}_0$, $I_*$, and $\ell \in I$.  These will always be clear from context.

 For each $\tau \in \Aut(\C)$,  denote by  $a_\tau \in  \Q^\times_{>0} \backslash  \A_f^\times$ the unique element with  \[
 \mathrm{rec}(a_\tau) = \tau|_{\Q^{ab}} .
  \]
  The following is our  $q$-expansion principle; see also \cite[Theorem 2.8.7]{hor:book}.

\begin{proposition}[Rational $q$-expansion principle]\label{prop:q principle}
For any $a\in \A_f^\times$ and  $\tau\in \Aut(\C)$, the $q$-expansion coefficients of $\psi$ and $\psi^\tau$  are related by 
\[
\mathrm{FJ}_\alpha^{(  a a_\tau   )} (  \psi^\tau    )  =   \tau \big(  \mathrm{FJ}_\alpha^{(a  )}  (\psi ) \big) .
\]
Moreover,   $\psi$ is defined over a subfield $L\subset \C$ if and only if 
\[
\mathrm{FJ}_\alpha^{( a a_\tau   )} ( \psi     )  =   \tau \big(  \mathrm{FJ}_\alpha^{(a  )}  (\psi ) \big)
\]
for all $a\in \A_f^\times$, all $\tau \in \Aut(\C/L)$, and all $\alpha\in \Gamma_\Phi^\vee(1)$.
\end{proposition}

\begin{proof}
The formal scheme $\widehat{T}_\Phi(\sigma)$ of (\ref{formal torus}) has a distinguished $\Q$-valued point defined by $q_\alpha=0$ (\emph{i.e.}~the unique point of the underlying reduced $\Q$-scheme), and so has a distinguished $\C$-valued point. 
 Hence, using the morphism
 \[
  \bigsqcup_{ a\in  \Q^\times_{>0} \backslash \A_f^\times / K_0  }   \widehat{T}_\Phi(\sigma)_{/\C}   \to 
    \widehat{\Sh}_{K_{\Phi } }   ( Q_\Phi ,  \DD_\Phi ,\sigma)(\C)
 \]
  of Proposition \ref{prop:torsor splitting},  each $a\in \Q^\times_{>0} \backslash \A_f^\times$ determines a distinguished point 
\[
\mathrm{cusp}^{(a)} \in    \widehat{\Sh}_{K_{\Phi} }   ( Q_\Phi ,  \DD_\Phi ,\sigma)(\C).
\]
By examining the proof of Proposition   \ref{prop:torsor splitting},  the reciprocity law (\ref{simple reciprocity}) implies that 
\[
 \mathrm{cusp}^{(aa_\tau)}  = \tau ( \mathrm{cusp}^{(a)} ) 
\]
for any $\tau \in \Aut(\C)$, and the $q$-expansion (\ref{basic q}) is, tautologically, the image of the formal function $[ \bm{\ell}^{\otimes k} ,\psi ]$ in the completed local ring at $\mathrm{cusp}^{(a)}$.   The first claim is now a consequence of the equality
\[
[ \bm{\ell}^{\otimes k} ,\psi ] ^\tau=[ \bm{\ell}^{\otimes k} ,\psi^\tau  ]
\]
of formal functions on  
\[
\widehat{\Sh}_{K_\Phi} (   Q_\Phi ,  \DD_\Phi   ,\sigma  ) \iso  \widehat{\Sh}_K(G,\DD , \Sigma).
\]

The second claim follows from the first, and the observation that two rational sections $\psi_1$ and $\psi_2$ are equal if and only if 
$
\mathrm{FJ}^{(a)}(\psi_1) = \mathrm{FJ}^{(a)}(\psi_2)
$
for all $a$.  Indeed, to check that $\psi_1=\psi_2$, it suffices to check this in a formal neighborhood of one point on each connected component of  $\Sh_K(G,\DD,\Sigma)_{/\C}$.    Using strong approximation for the simply connected group
\[
\mathrm{Spin}(V) = \mathrm{ker} \big( \nu : G\to \mathbb{G}_m ),
\]
one can show that  the fibers of 
\[
\Sh_K(G,\DD)(\C) \to \Sh_{\nu(K)}(\mathbb{G}_m , \mathcal{H}_0)(\C)
\]
are connected.  This implies  that the images of the points $\mathrm{cusp}^{(a)}$ under 
\[
\widehat{\Sh}_{K_{\Phi 0} }   ( Q_\Phi ,  \DD_\Phi ,\sigma)(\C) \to \widehat{\Sh}_{K_{\Phi } }   ( Q_\Phi ,  \DD_\Phi ,\sigma)(\C)
\iso    \widehat{\Sh}_K(G,\DD , \Sigma)(\C)
\]
hit every connected component of $\Sh_K(G,\DD,\Sigma)(\C)$.
\end{proof}


\section{Borcherds products}
\label{s:borcherds}


Once again, we work with a fixed   $\Q$-quadratic space $(V,Q)$ of signature $(n,2)$ with $n\ge 1$, and denote by $(G,\DD)$  the associated GSpin Shimura datum of \S \ref{ss:gspin data}.  Fix  a $\Z$-lattice   $V_\Z \subset V$ on which the quadratic form is $\Z$-valued, and let $K$ be   as in  (\ref{K choice}).  Recalling the notation of Remark \ref{rem:little shimura},   fix a choice of 
\[
2\pi i \in \mathcal{H}_0.
\]

 We recall the analytic theory of Borcherds products \cite{Bor98, Bruinier} on $\Sh_K(G,\DD)(\C)$ using the adelic formulation as  in  \cite{KuBorcherds}.
Assuming that $V$ contains an isotropic line, we express their product expansions in  the algebraic language of  \S \ref{ss:qs}.


\subsection{Weakly holomorphic forms}
\label{ss:weakly holomorphic forms}


Let $S(V_{\A_f })$ be the Schwartz space of locally constant $\C$-valued compactly supported functions on $V_{\A_f} = V\otimes \A_f$. For any $g\in G(\A_f)$ abbreviate 
\[
g V_\Z = g V_{\widehat{\Z}} \cap V.
\]

Denote by  $\sS_{V_\Z  } \subset S( V_{ \A_f }  )$ the finite dimensional subspace of functions invariant under $V_{\widehat{\Z}}$, and supported on its dual lattice; we often identify it with the space
\[
\sS_{V_\Z  } = \C [ V^\vee_\Z / V_\Z ]
\]   
 of functions on $V^\vee_\Z / V_\Z$.  The metaplectic double cover $\widetilde{\SL}_2(\Z)$ of $\SL_2(\Z)$ acts via the Weil representation 
\[
\rho_{V_\Z }  : \widetilde{\SL}_2(\Z) \to \Aut_\C( \sS_{V_\Z} )
 \] 
 as in \cite{Bor98, Bruinier, BruinierFunke}.  Define the complex conjugate representation by
 \[
\overline{ \rho }_{V_\Z}( \gamma ) \cdot \varphi = \overline{ ( \rho _{V_\Z}( \gamma ) \cdot  \overline{\varphi } )} ,
 \]
 for $\gamma \in \widetilde{\SL}_2(\Z) $ and $\varphi \in \sS_{V_\Z} $.

\begin{remark}\label{rem:weil dual}
There is a canonical basis 
\[
\{ \phi_\mu : \mu \in V_\Z^\vee / V_\Z\} \subset \sS_{V_\Z},
\]
in which $\phi_\mu$ is the characteristic function of $\mu + V_\Z$.
This allows us to identify $\sS_{V_\Z}$ with its own $\C$-linear dual.  Under this identification, the
 complex conjugate representation $\overline{\rho}_{V_\Z}$ agrees with with  contragredient representation $\rho^\vee_{V_\Z}$. 
  It also agrees with  the representation denoted $\omega_{V_\Z}$ in \cite{AGHMP-1,AGHMP-2}.
 \end{remark}

Denote by $M^!_{1- n/2}(\overline{\rho}_{V_\Z })$  the space of weakly holomorphic forms for $\widetilde{\SL}_2(\Z)$ of weight $1-n/2$ and representation $\overline{\rho}_{V_\Z }$, as in \cite{Bor98, Bruinier, BruinierFunke}.   In particular, any 
\begin{equation}\label{input form}
f(\tau)  = \sum_{  \substack{ m\in \Q \\ m \gg -\infty} } c(m) \cdot q^m   \in M^!_{1- \frac{n}{2} }( \overline{\rho}_{V_\Z } )
\end{equation}
is an $\sS_{V_\Z }$-valued holomorphic function on the complex upper half-plane $\mathcal{H}$.  Each Fourier coefficient $c(m) \in \sS_{V_\Z }$ is determined by its values  $c(m,\mu)$ at the various cosets $\mu \in V^\vee_\Z/ V_\Z$.  Moreover,  $c(m,\mu) \neq 0$ implies  $m \equiv Q(\mu)$ modulo $\Z$.

\begin{definition}\label{def:integral form}
The weakly holomorphic form (\ref{input form}) is \emph{integral} if \[c(m,\mu) \in \Z\] for all $m\in \Q$ and all $\mu \in V_\Z^\vee / V_\Z$.
\end{definition}

It is a theorem of McGraw \cite{mcgraw} that the space of all forms (\ref{input form})  has a $\C$-basis of integral forms.


\subsection{Borcherds products and regularized theta lifts}
\label{ss:borcherds definition}


We now recall the meromorphic Borcherds products of \cite{Bor98, Bruinier, KuBorcherds}.

Write $\tau =u+i v \in \mathcal{H}$ for the variable on the complex upper half-plane.  
For each $\varphi\in \sS(V_{\A_f})$ there is a Siegel theta function
\[
\vartheta( \tau , z, g ; \varphi) : \mathcal{H} \times \DD \times G(\A_f)  \to \C ,
\]
as in \cite[(1.37)]{KuBorcherds}, satisfying the transformation law
\[
\vartheta( \tau , \gamma z,\gamma g h  ;  \varphi ) = \vartheta(\tau , z,g ;  \varphi \circ h^{-1} )
\]
for any $\gamma \in G(\Q)$ and any $h \in G(\A_f)$.  
If we use the basis of Remark \ref{rem:weil dual} to  define
\[
\vartheta( \tau , z, g ) =  \sum_{ \mu \in V_\Z^\vee / V_\Z } \vartheta( \tau , z, g ,\phi_\mu) \cdot \phi_\mu,
\]
we obtain a function 
\[
\vartheta( \tau , z, g ) : \mathcal{H} \times \DD \times G(\A_f)  \to \sS_{V_{\Z}},
\]
which  transforms in the variable $\tau$ like a modular form of weight  $\frac{n}{2}-1$ and representation $\rho_{V_\Z}$.

Given a  weakly holomorphic form (\ref{input form}) one can regularize the divergent integral 
\begin{equation}\label{theta integral}
\regtheta (f) (z, g) = \int_{   \SL_2(\Z) \backslash \mathcal{H}   }      f(\tau)   \vartheta(\tau,z,g)        \  \frac{ du\, dv}{v^2}
\end{equation}
as in  \cite{Bor98, Bruinier,  KuBorcherds}.  
Here we are using  the  map 
$
\sS_{V_\Z}  \otimes  \sS_{V_\Z} \to \C
$ 
defined by 
\[
\phi_{\mu} \otimes \phi_\nu \mapsto 
\begin{cases}
1 &\mbox{if }\mu=\nu \\
0 &\mbox{otherwise}
\end{cases}
\]
to obtain an $\SL_2(\Z)$-invariant scalar-valued integrand  $ f(\tau)   \vartheta(\tau,z,g)$. 

As the subgroup $K$ acts trivially on the quotient (\ref{disc quotient}),    the subspace  $\sS_{V_\Z}\subset S(V_{\A_f} )$ is $K$-invariant.  It follows that   (\ref{theta integral})  satisfies
\[
\regtheta(f) ( \gamma z, \gamma g k ) = \regtheta(f)(z,g)
\]
for any $\gamma \in G(\Q)$ and any $k \in K$.  This  allows us to view $\regtheta(f)$ as a function on $\Sh_K(G,\DD)(\C)$, which we call the \emph{regularized theta lift}.  Our $\regtheta(f)$ is usually denoted $\Phi(f)$ in the literature.

\begin{remark}\label{rem:miracle values}
The following fundamental theorem of Borcherds implies that the regularized theta lift is real analytic away from a prescribed divisor, with logarithmic singularities along that divisor.  
Remarkably, the regularization process gives $\regtheta(f)$ a meaningful value at \emph{every} point of $\Sh_K(G,\DD)(\C)$, including along the  singular divisor.
In the context of unitary Shimura varieties, this is  \cite[Theorem 4.1]{BHY} and \cite[Corollary 4.2]{BHY}, and the proof for orthogonal Shimura varieties is identical.
In other words,  $\regtheta(f)$ is a well-defined (but discontinuous) function on all of $\Sh_K(G,\DD)(\C)$.
Its values along the singular divisor will be  made more explicit  in \S \ref{ss:rational} when we use the embedding trick to complete the proof of Theorem \ref{thm:main borcherds}.
\end{remark}

\begin{theorem}[Borcherds]\label{thm:old borcherds}
Assume that $f$  is integral.  After multiplying $f$ by any sufficiently divisible positive integer\footnote{In particular, we may assume   $c(0,0)\in 2\Z$.},  there is a meromorphic section $\Psi(f)$ of the analytic line bundle $(\bm{\omega}^{an})^{\otimes c (0,0)/2}$  on $\Sh_K(G,\DD)(\C)$ such that, away from the support of $\mathrm{div}(\Psi(f))$, we have
\begin{equation}\label{naive borcherds}
     - 4 \log \| \Psi(f) \|_\naive  =  \regtheta (f) +  c (0,0)  \log(\pi) +  c (0,0) \Gamma'(1).
\end{equation}
Here $\Gamma'(s)$ is the derivative of the usual  Gamma function,  and $\| - \|_\naive$ is the metric of (\ref{naive metric}).
\end{theorem}

\begin{proof}
Choose a connected component $\DD^\circ \subset \DD$,  let $G(\R)^\circ \subset G(\R)$ be its stabilzer (this is just the subgroup of elements on which the spinor similitude $\nu : G\to \mathbb{G}_m$ is positive) and define $G(\Q)^\circ$ similarly.  Denote by  $\bm{\omega}_{\DD^\circ}$ the restriction to $\DD^\circ$ of the tautological line bundle on (\ref{orthogonal domain}).  It carries an action of $G(\R)^\circ$ covering the action on the base, and a $G(\R)^\circ$ invariant metric (\ref{naive metric}).  

For any $g\in G(\A_f)$,  denote by 
\[
\regtheta_g(f)(z) \define \regtheta(f)( z, g)
\]
the restriction of the regularized theta lift to the connected component 
\begin{equation}\label{connected component}
(  G(\Q)^\circ \cap g K g^{-1}  ) \backslash \DD^\circ \map{  z\mapsto ( z,g)    } \Sh_K( G,\DD)(\C).
\end{equation}

 Borcherds \cite{Bor98} proves the existence of a meromorphic section $\Psi_g(f)$ of $\bm{\omega}_{\DD^\circ}^{\otimes c (0,0)/2}$ satisfying 
\begin{equation}\label{naive norm}
 - 4 \log \| \Psi_g(f) \|_\naive  = \regtheta_g (f) +  c (0,0)  \log(\pi) +  c (0,0) \Gamma'(1).
\end{equation}
Note that Borcherds does not work adelically.  Instead, for every input form (\ref{input form}) he constructs a single
meromorphic section  $\Psi_{\mathrm{classical}} (f)$ over $\DD^\circ$.  However,    $g\in G(\A_f)$ determines an isomorphism 
$V_\Z ^\vee / V_\Z   \to g V_\Z ^\vee/ g V_\Z$, which induces an isomorphism
\[
M^!_{1- \frac{n}{2} }( \overline{\rho}_{V_\Z } ) \map{ f\mapsto g \cdot f} M^!_{1- \frac{n}{2} }( \overline{\rho}_{gV_\Z }).
\]
Replacing the pair $(V_\Z , f)$ by $( g V_\Z , g f)$  yields another meromorphic section $\Psi_\mathrm{classical}(g f)$ over $\DD^\circ$, and  
\[
 \Psi_g(f) = \Psi_\mathrm{classical}(g f)  .
\]

The relation (\ref{naive norm}) determines $\Psi_g(f)$ up to scaling by a complex number of absolute value $1$, and 
the linearity of $f \mapsto \regtheta_g(f)$ implies the multiplicativity 
\[
\Psi_g(f_1+f_2) = \Psi_g(f_1) \otimes \Psi_g(f_2)
\]
relation,  up to the ambiguity just noted.
As  $\regtheta_g(f)$ is invariant under translation by every $\gamma \in G(\Q)^\circ  \cap g K g^{-1}$, we must have
\[
 \gamma \cdot    \Psi_g(f ) (z)   = \xi_g(\gamma) \cdot \Psi_g(f) (  \gamma z  ) 
\]
for some unitary character   
\begin{equation}\label{multiplier}
\xi_g : G(\Q)^\circ  \cap g K g^{-1} \to \C^\times.
\end{equation}
The main result of  \cite{Bor:GKZcorrection}  asserts   that the character $\xi_g$ is of finite order, and so we may replace $f$ by a positive integer multiple in order to make it trivial.   The section $\Psi_g(f)$ now descends to the quotient (\ref{connected component}).

Repeating this procedure on every connected component of $\Sh_K(G ,\DD)(\C)$ yields a section $\Psi(f)$ satisfying (\ref{naive borcherds}).  
\end{proof}

The meromorphic section $\Psi(f)$ of the theorem is what is usually called the \emph{Borcherds product} (or Borcherds lift) of $f$.    We will use the same terminology to refer to the meromorphic section 
\[
\psi(f) =  (2\pi i)^{ c(0,0)} \Psi(2 f)
\]
of $(\bm{\omega}^{an})^{ \otimes c(0,0) }$, which has better arithmetic properties.  
We will see in \S \ref{s:integral borcherds}   that, after rescaling by a constant of absolute value $1$ on every connected component of $\Sh_K(G,\DD)(\C)$,  the section $\psi(f)$ is algebraic and defined over the reflex field $\Q$.

\begin{proposition}\label{prop:borcherds algebraic}
Assume that either  $n\ge 3$, or that $n=2$ and $V$ has Witt index $1$.
The    Borcherds product  $\Psi(f)$ of Theorem \ref{thm:old borcherds},  a priori a meromorphic section on $\Sh_K(G,\DD  )(\C)$, is the analytification of a rational section  on  $\Sh_K(G,\DD  )_{/\C}$.     
\end{proposition}

\begin{proof}
It suffices to prove this after shrinking $K$, so we may assume that $K$ is neat and $\Sh_K(G,\DD  )$ is a quasi-projective variety.
The hypotheses on $n$  imply that the  boundary of  the (normal and projective) Baily-Borel compacification  
\[
\Sh_K(G,\DD  ) \hookrightarrow \Sh_K(G,\DD  )^{\mathrm{BB}}
\]
lies in codimension $\ge 2$.

Let $D$ be the polar part of the divisor of  $\Psi(f)$, so that $D$ is an effective analytic divisor  on $\Sh_K(G,\DD)(\C)$ with $\mathrm{div}(\Psi(f)) + D$ effective.
The proof of Levi's generalization of  Hartogs'  theorem \cite[\S 9.5]{GR} 
shows that the topological closure of $D$ in $\Sh_K(G,\DD  )^{\mathrm{BB}}(\C)$ is again an analytic divisor.
By Chow's theorem on the algebraicity of analytic divisors on projective varieties, this closure is algebraic, and so $D$ itself was algebraic.

 Now view $\Psi(f)$ as a holomorphic section of the analytification of the line bundle
$
\bm{\omega}^{\otimes c(0,0)/2} \otimes \co(D)
$
on  $  \Sh_K(G,\DD  )_{/\C}$.   By Hartshorne's extension of GAGA \cite[Theorem VI.2.1]{hart} this section  is  algebraic, as desired.
\end{proof}


\subsection{The product expansion I}
\label{ss:first product}


As in the proof of Theorem \ref{thm:old borcherds}, fix a connected component $\DD^\circ\subset\DD$ and  an $h\in G(\A_f)$, and denote by $\Psi_h(f)$ the restriction of the Borcherds product to the connected component
 \[
(  G(\Q)^\circ \cap h K h^{-1}  ) \backslash \DD^\circ \map{  z\mapsto ( z,h)    } \Sh_K( G,\DD)(\C).
\]
In this subsection we recall the product expansion for $\Psi_h(f)$ due to Borcherds.
Let $\bm{\omega}_{\DD^\circ}$ be the restriction to $\DD^\circ$ of the tautological bundle on (\ref{orthogonal domain}).

Assume throughout \S \ref{ss:first product} that there exists an isotropic line $I\subset V$.
 Choose a second isotropic line $I_*\subset V$ with $[I,I_*]\neq 0$, but do this in a  particular way: first choose a $\Z$-module generator 
 \[
 \ell \in I\cap hV_{\Z},
 \]
  and then choose a  $k\in hV_{\Z}^\vee$ such that $[\ell,k]=1$.  Now  take  $I_*$ be the span of the isotropic vector
\begin{equation}\label{second isotropic}
\ell_* = k - Q(k) \ell.
\end{equation}
Obviously $[\ell,\ell_*]=1$, but we need not have $\ell_* \in h V_\Z^\vee$.

Abbreviate $V_0 = I^\perp / I$.  This is a $\Q$-vector space endowed with a quadratic form  of signature $(n-1,1)$, and a $\Z$-lattice 
\begin{equation}\label{middle lattice}
V_{0\Z}   = (I^\perp \cap hV_{\Z}) / (I \cap hV_{\Z})   \subset V_0.
\end{equation}
Denote by 
\[
\mathrm{LightCone}(V_{0\R}) =  \{ w \in V_{0\R} : Q(w) < 0 \}
\]
the light cone in $V_{0\R}$.  It is a disjoint union of two open convex cones.    Every $v\in  I^\perp_\C$ determines an isotropic vector 
\[
 \ell_* + v - [\ell_*,v] \ell - Q(v) \ell \in V_\C,
\]
depending only on  the image $v \in V_{0\C}$.  The resulting injection $V_{0\C} \to \mathbb{P}^1(V_\C)$ restricts to an isomorphism
\[
V_{0\R}  +  i \cdot  \mathrm{LightCone}(V_{0\R}) \iso \DD,
\]
and we let  
\begin{equation}\label{light component}
\mathrm{LightCone}^\circ (V_{0\R}) \subset  \mathrm{LightCone}(V_{0\R})
\end{equation}
 be the connected component with 
\[
V_{0\R}  + i \cdot   \mathrm{LightCone}^\circ (V_{0\R}) \iso \DD^\circ.
\]

There is an action $\rho_{V_{0\Z} }$ of $\widetilde{\SL}_2(\Z)$ on the finite dimensional $\C$-vector space $\sS_{V_{0\Z}}$, exactly  as in \S \ref{ss:weakly holomorphic forms}, and  a  weakly holomorphic modular form
\[
f_0(\tau) = \sum_{  \substack{ m\in \Q  \\  m\gg -\infty } } \sum_{\lambda\in V_{0\Z}^\vee / V_{0\Z} } c_0(m,\lambda) \cdot q^m \in 
M^!_{1-\frac{n}{2}}(\overline{\rho}_{V_{0\Z}})
\]
whose coefficients are defined by
\[
c_0(m,\lambda) = \sum_{ \substack{  \mu  \in  hV_{\Z}^\vee / hV_{\Z}   \\   \mu\sim  \lambda    }  } c(m, h^{-1}\mu).
\]
Here we understand $h^{-1} \mu $ to mean the image of $\mu$ under the isomorphism $hV_{\Z}^\vee / hV_{\Z} \to V_{\Z}^\vee /V_{\Z}$ defined by multiplication by $h^{-1}$. The notation  $\mu \sim   \lambda$ requires explanation:  denoting by 
\[
 p: ( I^\perp \cap  hV_{\Z}^\vee) /(I^\perp \cap  hV_{\Z})  \to V_{0\Z}^\vee  / V_{0\Z}
\]
the natural map,  $\mu \sim  \lambda$ means that there is a  
\begin{equation}\label{sim lift}
\tilde{\mu} \in  I^\perp \cap  (\mu+hV_{\Z})
\end{equation}
  such that  $p(\tilde{\mu}) = \lambda$.

Every vector $x\in V_0$ of positive length determines a hyperplane $x^\perp \subset V_{0\R}$. For each $m\in \Q_{>0}$ and $\lambda\in V_{0\Z}^\vee/V_{0\Z}$ define a formal  sum of hyperplanes
\[
H(m,\lambda) = \sum_{   \substack{  x\in \lambda+ V_{0\Z} \\  Q(x) =m }   }  x^\perp ,
\]
in $V_{0\R}$, and set 
\[
H(f_0) =  \sum_{  \substack{   m \in \Q_{>0} \\ \lambda \in V_{0\Z}^\vee / V_{0\Z}   }  }
 c_0(-m,\lambda) \cdot H(m,\lambda).
\]

\begin{definition}
A \emph{Weyl chamber} for $f_0$ is a connected component 
\begin{equation}\label{weyl chamber}
\weyl \subset   \mathrm{LightCone}^\circ(V_{0\R})  \smallsetminus \mathrm{Support}(H(f_0) ) .
\end{equation}
\end{definition}

Let $N$ be  the positive integer determined by $N\Z =[ hV_{\Z} , I \cap hV_{\Z} ]$, and note that $\ell / N \in h V_\Z^\vee$. Set 
  \begin{equation}\label{constant A}
A    = \prod_{ \substack{    x \in \Z/ N  \Z \\ x \neq 0       }  }
\left(
1-e^{ 2\pi  i x /N} \right)^{ c(0, x   h^{-1} \ell /N ) }.
\end{equation}

Tautologically, every fiber of $\bm{\omega}_{\DD^\circ}$ is a line in $V_\C$, and each such fiber pairs nontrivially with the isotropic line $I_\C$.  Using the nondegenerate pairing 
\[
[\, \cdot\, , \, \cdot\, ]  :  I_\C \otimes \bm{\omega}_{\DD^\circ} \to \co_{\DD^\circ},
\] 
the Borcherds product $\Psi_h(f)$ and the isotropic vector $\ell \in I$ determine a meromorphic function $[  \ell ^{\otimes c(0,0) /2}  , \Psi_h(f)    ]$ on $\DD^\circ$.    It is this function that Borcherds expresses as an infinite product.

\begin{theorem}[Borcherds \cite{Bor98,Bruinier}]\label{thm:first q-expansion}
For each Weyl chamber $\weyl$  there is a vector  $\varrho \in V_0$  with the following property:   For all 
\[
v  \in  V_{0\R} + i\cdot \mathscr{W} \subset V_{0\C}
\]
with  $| Q(  \mathrm{Im}(v)   ) | \gg 0$, the value of   $[  \ell ^{\otimes c(0,0) /2}  , \Psi_h(f)    ]$  at the isotropic line
\[
 \ell_* + v - [\ell_*,v] \ell - Q(v) \ell \in \DD^\circ
\]
is given by the (convergent)  infinite product 
\[
\kappa   A    \cdot e^{2\pi  i [\varrho ,v] }
\prod_{  \substack{   \lambda \in V_{0\Z}^\vee  \\  [\lambda , \weyl ] >0    }   }
\prod_{   \substack{  \mu \in  hV_{\Z}^\vee /  h V_{\Z} \\   \mu \sim   \lambda  }   }
\Big(
1- \zeta_\mu \cdot  e^{2\pi  i [\lambda, v ]}
\Big)^{  c( - Q(\lambda) , h^{-1} \mu)    }
\]
for some $\kappa\in \C$ of absolute value $1$.  Here,  recalling the vector $k \in h V_{\Z}^\vee$ appearing in  (\ref{second isotropic}),  we have set
 \[
 \zeta_\mu = e^{2\pi  i [ \mu, k  ]} .
  \] 
\end{theorem}

\begin{remark}
The vector $\varrho\in V_0$ of the theorem  is the \emph{Weyl vector}.  It is completely determined by the weakly holomorphic form $f_0$ and the choice of Weyl chamber $\mathscr{W}$.
\end{remark}


\subsection{The product expansion II}
\label{ss:second product}


We now  connect the product expansion of Theorem \ref{thm:first q-expansion} with the algebraic theory of $q$-expansions from \S \ref{ss:qs}.    Throughout \S \ref{ss:second product} we assume that $K$ is neat.

 The theory of \S \ref{ss:qs} applies to sections of the algebraic line bundle $\bm{\omega}_{/\C}$ on $\Sh_K(G,\DD)_{/\C}$ and at the moment we only know the algebraicity of  Borcherds products in special cases (Proposition \ref{prop:borcherds algebraic}).  
Throughout \S \ref{ss:second product}, we simply assume that our given Borcherds product $\Psi(f)$ is algebraic.

Begin by choosing a cusp label representative 
\[
\Phi =(P,\DD^\circ , h)
\]
 for which $P$ is the stabilizer of an isotropic line $I$.  Let $\ell \in I \cap hV_\Z$ be a generator, let $\ell_*$ be as in (\ref{second isotropic}), and let $I_* = \Q\ell_*$.   
 
Recall from the discussion surrounding (\ref{component section}) that the choice of $I_*$ determines morphisms of mixed Shimura data
 \[
\xymatrix{
(Q_\Phi , \DD_\Phi)    \ar[r]_{ \nu_\Phi } & ( \mathbb{G}_m , \mathcal{H}_0)  \ar@/_1pc/[l]_{\spl},
}
\] 
where we specify that  $\nu_\Phi: \DD_\Phi \to \mathcal{H}_0$ sends the connected component $\DD_\Phi^\circ \subset \DD_\Phi$ containing $\DD^\circ$ to the 
$2\pi i \in \mathcal{H}_0$ fixed at the beginning of \S \ref{s:borcherds}.

 Set $V_0 = I^\perp / I$ as before.  The connected component  (\ref{light component}) was chosen in such a way that the isomorphism
 \[
 V_{0\C} \map{ \otimes \ell }V_{0\C} \otimes I \map{(\ref{cusp unipotent})} U_\Phi(\C)
 \]
 identifies 
 \[
 V_{0\R} + i \cdot   \mathrm{LightCone}^\circ (V_{0\R}) \iso U_\Phi(\R) + C_\Phi,
 \]
  where
 $C_\Phi \subset U_\Phi(\R)(-1)$ is the open convex cone  (\ref{convex cone}). 
Equivalently,  the isomorphism
 \begin{equation}\label{lightcone match}
V_{0\R} \map{\otimes  ( - 2\pi i )^{-1} \ell  } V_{0\R} \otimes I(-1) \iso U_\Phi(\R)(-1)
\end{equation}
(note the minus sign!)
 identifies $\mathrm{LightCone}^\circ(V_{0\R}) \iso C_\Phi$.

\begin{lemma}\label{lem:good decomposition}
Fix a Weyl chamber $\mathscr{W}$ as in (\ref{weyl chamber}).  After possibly shrinking $K$, there exists a $K$-admissible,  complete cone decomposition $\Sigma$ of $(G,\DD)$  having the no self-intersection property, and such that the following holds: there is some  top-dimensional rational polyhedral cone  $\sigma \in \Sigma_\Phi$ whose interior  is identified with an  open subset of  $\weyl$ under the above isomorphism 
\[
C_\Phi \iso  \mathrm{LightCone}^\circ(V_{0\R})  .
\]
\end{lemma}

\begin{proof}
This is an elementary exercise.  Using Remark \ref{rem:nice decompositions}, we first  shrink $K$ in order to find some  $K$-admissible, complete cone decomposition $\Sigma$ of $(G,\DD)$  having the no self-intersection property.   We may furthermore choose $\Sigma$ to be smooth, and applying  barycentric subdivision \cite[\S 5.24]{pink} finitely many times yields a refinement of $\Sigma$ with the desired properties.
\end{proof}

For the remainder of  \S \ref{ss:second product} we assume that   $K$, $\Sigma$, $\weyl$, and $\sigma \subset U_\Phi(\R)(-1)$ are as in Lemma \ref{lem:good decomposition}.    As in \S \ref{ss:qs},  the line bundle $\bm{\omega}$ on $\Sh_K(G,\DD)$ has a canonical extension to $\Sh_K(G,\DD,\Sigma)$, and we view $\Psi(f)$ as a rational section over $\Sh_K(G,\DD,\Sigma)_{/\C}$.

The top-dimensional cone $\sigma$ singles out a $0$-dimensional stratum
\[
Z^{(\Phi,\sigma)}_K(G,\DD,\Sigma) \subset \Sh_K(G,\DD ,\Sigma )
\]
as in \S \ref{ss:compactification}.  Completing along this stratum,  Proposition \ref{prop:torsor splitting} provides us with a formally \'etale surjection
\[
\bigsqcup_{ a \in \Q^\times_{>0} \backslash \A_f^\times / K_0  }\Spf \Big(   \C[[ q_\alpha]]_{  \substack{  \alpha \in \Gamma_\Phi^\vee(1)  \\ \langle \alpha, \sigma \rangle \ge 0   }  }   \Big) \to  \widehat{\Sh}_K(G,\DD,\Sigma)_{/\C},
\]
where $K_0 \subset \A_f^\times$ is chosen small enough  that the section   (\ref{component section})  satisfies $\spl(K_0) \subset K_\Phi$.
As in  (\ref{basic q}), the Borcherds product $\Psi(f)$ and the isotropic vector $\ell$ determine a rational formal function $ [ \bm{\ell}^{  \otimes c(0,0) / 2} , \Psi  (f) ]$ on the target, which pulls back to a  rational formal function 
\begin{equation}\label{BFJ}
\mathrm{FJ}^{(a)} ( \Psi(f) ) \in   \mathrm{Frac}\Big(  \C[[ q_\alpha]]_{  \substack{  \alpha \in \Gamma_\Phi^\vee(1)  \\ \langle \alpha, \sigma \rangle \ge 0   }  } \Big)
\end{equation}
  for every index $a$.   The following proposition explains how this formal $q$-expansion  varies  with $a$.

\begin{proposition}\label{prop:product expansion}
Let $F \subset \C$ be the abelian extension of $\Q$ determined by 
\[
 \mathrm{rec} : \Q^\times_{>0} \backslash \A_f^\times/K_0  \iso \Gal(F  / \Q).
\]
The  rational formal function (\ref{BFJ}) has the form
\begin{equation}\label{trivialized B}
(2\pi  i )^{ c(0,0)/2 }  \cdot \mathrm{FJ}^{(a)}( \Psi(f) )  = 
  \kappa^{(a)} A^{\mathrm{rec}(a)}    q_{ \alpha( \varrho)   }  \cdot  
   \mathrm{BP}(f)^{ \mathrm{rec}(a) } .
\end{equation}
Here   $\kappa^{(a)}\in \C$ is some constant of absolute value $1$, and the power series 
\[
\mathrm{BP}(f)   \in \co_F [[ q_\alpha ]]_{  \substack{    \alpha \in \Gamma_\Phi^\vee(1) \\  \langle \alpha, \sigma\rangle \ge 0      }   }
\]
 (\underline{B}orcherds \underline{P}roduct) is  the infinite product
\[
\mathrm{BP} (f) = 
 \prod_{  \substack{   \lambda \in V_{0\Z}^\vee  \\  [\lambda ,  \weyl ] >0    }   }
\prod_{   \substack{  \mu \in  hV_{\Z}^\vee /  h V_{\Z} \\   \mu \sim   \lambda  }   }
\Big(
1-  \zeta_\mu    \cdot 
 q _{  \alpha(\lambda) } \Big)^{  c ( - Q(\lambda) ,   h^{-1}  \mu )    }    .
\]
The constant $A$ and the roots of unity $\zeta_\mu$ have the same meaning as in Theorem \ref{thm:first q-expansion}, and these constants lie in $\co_F$.
The meaning of $q_{\alpha(\lambda) }$ is as follows: dualizing the isomorphism (\ref{lightcone match})
 yields an isomorphism 
 \begin{equation}\label{q shift}
 V_{0\R}   \map{ \lambda\mapsto \alpha(\lambda) } U_\Phi(\R)^\vee(1),
 \end{equation}
 and the image of each $\lambda \in  V_{0\R}$ appearing in the product satisfies
 \[
 \alpha (\lambda) \in  \Gamma_\Phi^\vee(1).
 \]
 The condition $[\lambda,\weyl]>0$ implies $\langle \alpha(\lambda) ,\sigma  \rangle >0$.
Of course $q_{\alpha(\varrho)}$ has the same meaning, with $\varrho \in V_0$ the Weyl vector of Theorem \ref{thm:first q-expansion}.  Again $\alpha(\varrho)\in \Gamma_\Phi^\vee(1)$, 
but need not satisfy the positivity condition with respect to $\sigma$.
\end{proposition}

\begin{proof}
First we address the field of definition of the constants $A$ and $\zeta_\mu$.

\begin{lemma}
The constant $A$ of (\ref{constant A}) lies in $\co_F$, and $\zeta_\mu \in \co_F$  for every $\mu$ appearing in the above product.   
\end{lemma}

\begin{proof}
Suppose $a\in K_0$. It follows from the discussion preceeding (\ref{disc quotient}) that $\spl(a)\in h K h^{-1}$ stabilizes the lattice $h V_{\Z}$,  and acts trivially on the quotient $hV_{\Z}^\vee / hV_{\Z}$.   In particular, $\spl(a)$ acts trivially on the vector $\ell/N \in hV_{\Z}^\vee / hV_{\Z}$. On the other hand, by its very definition (\ref{component section}) we know that $\spl(a)$ acts by $a$ on this vector. It follows that  
$
(a-1) \ell / N \in h V_{\Z},
$
from which we deduce first $a-1 \in   N\widehat{\Z}$, and then $A^{\mathrm{rec}(a)} =A$.

Suppose  $\mu \in hV_\Z^\vee / h V_\Z$ satisfies $\mu\sim\lambda$  for some $\lambda \in V_{0\Z}^\vee$.  By (\ref{sim lift})  we may fix some $\tilde{\mu} \in I^\perp \cap (\mu + hV_\Z)$.  This allows us to compute, using  (\ref{second isotropic}),
\begin{align*}
\zeta_\mu^{\mathrm{rec}(a)} = e^{2\pi i [\tilde{\mu}, a k] } & =e^{2\pi i [ \tilde{\mu}, a\ell_* ] } e^{2\pi i  Q(k)  \cdot [\tilde{\mu}, a\ell  ] } \\
& = e^{2\pi i [ \tilde{\mu}, \spl(a)^{-1} \ell_* ] } e^{2\pi i Q(k) \cdot   [\tilde{\mu},  \spl(a) \ell  ] }.
\end{align*}
As $[\tilde{\mu} ,\ell]=0$, we have   $[\tilde{\mu},  \spl(a) \ell  ] = 0 = [\tilde{\mu},  \spl(a)^{-1} \ell  ]$.  Thus
\[
\zeta_\mu^{\mathrm{rec}(a)}    =  e^{2\pi i [ \tilde{\mu}, \spl(a)^{-1} \ell_* ] } e^{2\pi i Q(k) \cdot   [\tilde{\mu},  \spl(a)^{-1} \ell  ] }
= e^{2\pi i [  \tilde{\mu}, \spl(a)^{-1} k] } =  e^{2\pi i [ \spl(a)  \tilde{\mu}, k] }.
\]
As above, $\spl(a)$ acts trivially on $hV_{\Z}^\vee / hV_{\Z}$, and we conclude that
\[
\zeta_\mu^{\mathrm{rec}(a)}   =e^{2\pi i [  \tilde{\mu}, k] } = \zeta_\mu.
\]
\end{proof}

Suppose $a\in \A_f^\times$.   The image of the discrete group
\[
\Gamma_\Phi^{(a)} =   \spl(a)K_\Phi \spl(a)^{-1} \cap Q_\Phi(\Q)^\circ 
\]
under $\nu_\Phi : Q_\Phi \to \mathbb{G}_m$ is contained in $\widehat{\Z}^\times \cap \Q^\times_{>0}=\{1\}$, and hence  $\Gamma_\Phi^{(a)}$ is contained in $\ker(\nu_\Phi)=U_\Phi$.  Recalling that the conjugation action of $Q_\Phi$ on $U_\Phi$ is by $\nu_\Phi$, we find that
\[
\Gamma_\Phi^{(a)}  =  \mathrm{rat}( \nu_\Phi(s(a))) \cdot \big(   K_\Phi  \cap U_\Phi(\Q)  \big) =  \mathrm{rat}(a) \cdot \Gamma_\Phi
\]
as lattices in $U_\Phi(\Q)$.

Recalling (\ref{torus cover}) and (\ref{analytic nbhd}), consider the following commutative diagram of complex analytic spaces
\begin{equation}\label{nbhd diagram}
\xymatrix{
{  \bigsqcup_a  \Gamma_\Phi^{(a)} \backslash \DD^\circ  }  \ar@{-->}[r] \ar[d]_{\iso}^{    z \mapsto    ( z, s(a)  )   }  &    {  \bigsqcup_a T_\Phi(\C) }  \ar@{=}[r] \ar[d]^{\iso}   &   { \bigsqcup_a \Gamma_\Phi(-1) \otimes \C^\times }    \\
 { \mathscr{U}_{K_{ \Phi 0 } }  (Q_\Phi,\DD_\Phi) } \ar[r]    \ar[d] &  {\Sh_{K_{ \Phi 0 } } (Q_\Phi,\DD_\Phi)  (\C) }   \ar[d] \\
  { \mathscr{U}_{K_{ \Phi  } }  (Q_\Phi,\DD_\Phi) }  \ar[r]    \ar[d]^{    ( z,g  ) \mapsto    ( z,g h )   }  &  {\Sh_{K_{ \Phi  } } (Q_\Phi,\DD_\Phi)  (\C) } \\
  { \Sh_K(G,\DD)(\C) ,}
}
\end{equation}
in which all horizontal arrows are open immersions, all vertical arrows are  local isomorphisms on the source, and the disjoint unions are over  a set of coset representatives
$
a\in \Q^\times_{>0} \backslash \A_f^\times / K_0.
$  
The dotted arrow is, by definition, the unique open immersion making the upper left square commute.

\begin{lemma}
Fix a $\lambda \in V_{0\R}$ whose image under (\ref{q shift}) satisfies  
\[
\alpha(\lambda) \in  \Gamma_\Phi^\vee(1),
\]
and suppose \[ v\in V_{0\R} +  i\cdot  \mathrm{LightCone}^\circ(V_{0\R}) .\]   If we restrict the character
\[
q_{\alpha(\lambda)} : T_\Phi(\C)   \to \C^\times
\]
 to a function $\Gamma_\Phi^{(a)} \backslash \DD^\circ \to \C^\times$ via the open immersion in the top row of (\ref{nbhd diagram}), its value at the isotropic vector   
\[
 \ell_* + v - [\ell_*,v] \ell  - Q(v) \ell \in \DD^\circ
 \]
 is $e^{2\pi i [ \lambda, v]/\mathrm{rat}(a)}$.
\end{lemma}

\begin{proof}
The proof is a (rather tedious) exercise in tracing through the definitions.
   The dotted arrow in the diagram above is induced by the open immersion $\DD^\circ \subset U_\Phi(\C)\DD^\circ = \DD_\Phi^\circ$ and the isomorphisms
\begin{equation}\label{torus coord}
\bigsqcup_a  \Gamma_\Phi^{(a)} \backslash \DD_\Phi^\circ  \iso 
\Sh_{K_{ \Phi 0 } } (Q_\Phi,\DD_\Phi)  (\C)   \iso  \bigsqcup_a T_\Phi(\C) .
\end{equation}
The second isomorphism is the trivialization of the  $T_\Phi  (\C)$-torsor
\begin{equation}\label{tedious torsor}
\Sh_{K_{\Phi 0}}(Q_\Phi , \DD_\Phi)(\C) \to \Sh_{K_0}(\mathbb{G}_m , \mathcal{H}_0 )(\C)
\end{equation}
induced by the section $s:(\mathbb{G}_m , \mathcal{H}_0) \to (Q_\Phi , \DD_\Phi)$, as in the proof of Proposition \ref{prop:torsor splitting}.

Tracing through the proof of Proposition \ref{prop:torsor def}, this isomorphism is obtained by combining the  isomorphism
\begin{equation}\label{exponentiation}
 U_\Phi(\C) /  \Gamma_\Phi \iso \Gamma_\Phi (-1) \otimes \C/\Z(1)  \map{ \mathrm{id} \otimes \exp }
 \Gamma_\Phi(-1) \otimes \C^\times  = T_\Phi(\C)
 \end{equation}
with the isomorphism
\begin{equation}\label{exponentiation 2}
U_\Phi(\C) /  \Gamma_\Phi  \map{ -\mathrm{rat}(a) }U_\Phi(\C) /  \Gamma_\Phi^{(a)} \iso 
 \Gamma_\Phi^{(a)} \backslash \DD_\Phi^\circ
\end{equation}
obtained by  trivializing $ \Gamma_\Phi^{(a)} \backslash \DD_\Phi^\circ$ as  a $U_\Phi(\C)/  \Gamma_\Phi^{(a)} $-torsor using the point $\ell_* \in \DD_\Phi^\circ$.
Note the minus sign in (\ref{exponentiation 2}),  which arises from the minus sign in the isomorphism (\ref{minus!}) used to define the torsor structure on (\ref{tedious torsor}).

Denote by $\beta$ the composition
\[
V_{0\R} \map{ \otimes \ell} V_{0\R} \otimes I  \map{(\ref{cusp unipotent})}   U_\Phi(\R) .
\]
It is related to $\alpha(\lambda) \in U_\Phi(\R)^\vee(1)$ by
\[
\langle \alpha(\lambda) , \beta(v) \rangle = -2\pi i \cdot [\lambda, v],
\]
for all $v\in V_{0\R}$.
Extending $\beta$ complex linearly yields a commutative diagram
\[
\xymatrix{
& &  & &  {   T_\Phi(\C)   }  \ar@{=}[dd]^{ (\ref{torus coord}) }    \\
 {  V_{0\C}  }  \ar[rr]^{ \beta } &  & {     U_\Phi(\C)  } \ar[drr]_{ (\ref{exponentiation 2}) }   \ar[urr]^{(\ref{exponentiation})} \\
& & & &  {   \Gamma_\Phi^{(a)} \backslash \DD_\Phi^\circ ,}
  } 
\]
and going all the way back to the definitions preceding  (\ref{q torus}), we find that the 
pullback of  $q_{\alpha(\lambda)} : T_\Phi(\C) \to \C^\times$ to a function on $V_{0\C}$ is given by
\[
q_{\alpha(\lambda)} (v)  =e^{ -2\pi i  [ \lambda,v ] }.
\]
 On the other hand, the composition along the bottom row sends
 \[
\frac{ v } {-\rat(a) } \in  V_{0\C} 
\]
  to the point obtained by translating $\ell_* \in \DD_\Phi^\circ$ by the vector $v\otimes \ell \in V_{0\C}\otimes I$, viewed as an element of $U_\Phi(\C)$ using (\ref{cusp unipotent}).
  This translate is 
\[
 \ell_* + v - [\ell_*,v] \ell  - Q(v) \ell \in   \DD_\Phi^\circ,
 \]
and hence the value of $q_{\alpha(\lambda)}$ at this point is
\[
q_{\alpha(\lambda)} \left(  \frac{ v } {-\rat(a) }  \right) = e^{   2\pi i  [\lambda, v] /\rat(a)      }. \qedhere
\]
\end{proof}

\begin{lemma}
Suppose  $v\in V_{0\R} + i\cdot \mathscr{W} $ with $| Q(\mathrm{Im}(v)) | \gg 0$.
The value of the meromorphic function 
\[
\rat(a)^{c(0,0) /2 } \cdot  [ \ell^{ c(0,0) / 2 } , \Psi_{s(a)h}(f) ] 
\]
at the isotropic line $ \ell_* + v - [\ell_*,v] \ell  - Q(v) \ell \in \DD^\circ$ is 
 \begin{eqnarray*}\lefteqn{
A_\Phi^{ \mathrm{rec}(a) } \cdot e^{2\pi  i   [\varrho ,v]   / \rat(a)   } } \\
& &\times  \prod_{  \substack{   \lambda \in V_{0\Z}^\vee  \\  [\lambda , \weyl ] >0    }   }
\prod_{   \substack{  \mu \in  hV_{\Z}^\vee /  h V_{\Z} \\   \mu \sim    \lambda  }   }
\Big(
1-  \zeta_\mu^{\mathrm{rec}(a)}   \cdot   e^{2\pi  i     [\lambda, v ]    / \rat(a)    } \Big)^{  c( - Q(\lambda) ,   h^{-1} \mu)    },
\end{eqnarray*}
up to scaling by a complex number of absolute value $1$.
\end{lemma}

\begin{proof}
The proof amounts to carefully keeping track of how   Theorem \ref{thm:first q-expansion} changes when  $\Psi_h(f)$ is replaced by $\Psi_{\spl(a)h}(f)$. The main source of confusion is that the vectors $\ell$ and $\ell_*$ appearing in Theorem \ref{thm:first q-expansion} were  chosen to have nice properties with respect to the lattice $hV_{\Z}$, and so we must first pick new isotropic vectors $\ell^{(a)}$ and $\ell^{(a)}_*$ having  similarly nice properties with respect to  $\spl(a) hV_{\Z}$.

Set  $\ell^{(a)} = \rat(a) \ell$.  This is a generator of 
\[
I\cap \spl(a) hV_{\Z} = \rat(a)\cdot  ( I \cap hV_{\Z}).
\]
Now choose a $k^{(a)} \in \spl(a) hV_{\Z}^\vee$ such that $[ \ell^{(a)} , k^{(a)} ]=1$, and let $I_*^{(a)}\subset V$ be the  span of the isotropic vector
\[
\ell_*^{(a)} = k^{(a)} - Q(k^{(a)})  \ell^{(a)}.
\]

Using the fact that $Q_\Phi$ acts trivially on the quotient $I^\perp / I$, it is easy to see that the lattice
\[
V_{0\Z}^{(a)} = (I^\perp \cap \spl(a) hV_{\Z}) / (I\cap \spl(a) hV_{\Z}) \subset I^\perp / I
\]
is equal, as a subset of $I^\perp/I$, to the lattice $V_{0\Z}$ of (\ref{middle lattice}). Thus replacing $hV_{\Z}$ by $\spl(a) hV_{\Z}$ has no effect on the construction  of the modular form $f_0$, or on the formation of Weyl chambers or their corresponding Weyl vectors.

Similarly,  as $Q_\Phi$ stabilizes $I$, the ideal $N\Z = [ h V_\Z , I\cap h V_\Z]$ is unchanged if $h$ is replaced by $\spl(a) h$.  Replacing $h$ by $\spl(a) h$ in the definition of $A$ now determines a new constant
\begin{align*}
A^{(a)} & = 
 \prod_{ \substack{    x \in \Z/ N \Z \\ x \neq 0       }  }
\left(
1-e^{ 2\pi  i x /N}
\right)^{ c(0, x \cdot  h^{-1}  \spl(a)^{-1}  \ell^{(a)} /N) } \\
& = 
 \prod_{ \substack{    x \in \Z/ N \Z \\ x \neq 0       }  }
\left(
1-e^{ 2\pi  i x /N} \right)^{ c(0, x \cdot  \unit(a)^{-1}  h^{-1} \ell /N) } \\
& = 
 \prod_{ \substack{    x \in \Z/ N \Z \\ x \neq 0       }  }
\left(  1-e^{ 2\pi  i x \cdot  \unit(a)  /N} \right)^{ c(0, x   h^{-1} \ell /N) } \\
& =  A^{\mathrm{rec}(a) } .
\end{align*}

Citing Theorem \ref{thm:first q-expansion} with $h$ replaced by $\spl(a)h$ everywhere, and using the isomorphism 
\[
\spl(a) hV_{\Z}^\vee / \spl(a) hV_{\Z} \iso hV_{\Z}^\vee / hV_{\Z}
\] 
induced by the action of $\spl(a)^{-1}$, we find that   the value of 
\begin{equation}\label{b expansion function}
[ ( \ell^{(a)} )^{ c(0,0) / 2 } , \Psi_{s(a)h}(f) ] = \rat(a)^{ c(0,0) /2 }  [ \ell^{ c(0,0) / 2 } , \Psi_{s(a)h}(f) ] 
\end{equation}
 at the isotropic line 
\[
 \ell^{(a)}_* + v - [\ell^{(a)}_*,v] \ell ^{(a)} - Q(v) \ell^{(a)} \in \DD^\circ
\]
 is given by the infinite product 
\[
 A_\Phi^{(a)}  \cdot e^{2\pi  i [\varrho ,v] }
\prod_{  \substack{   \lambda \in V_{0\Z}^\vee  \\  [\lambda , \weyl ] >0    }   }
\prod_{   \substack{  \mu \in  hV_{\Z}^\vee /  h V_{\Z} \\   \mu \sim  \lambda  }   }
\big(
1-e^{2\pi  i [  \spl(a) \mu, k^{(a)} ]}  \cdot  e^{2\pi  i [\lambda, v]}
\big)^{  c( - Q(\lambda) , h^{-1} \mu)    }.
\]

Now make a change of variables.  If we set
$
v^{(a)} = \ell_* - \rat (a) \ell_*^{(a)}  \in V_0 ,
$
we find that the value of (\ref{b expansion function}) at the isotropic line 
\begin{eqnarray*}\lefteqn{
 \ell_* + v - [\ell_*,v] \ell  - Q(v) \ell   }    \\
& = &   \ell^{(a)}_* + \Big( \frac{  v+v^{(a)} } {  \rat(a)  }  \Big) - \Big[\ell^{(a)}_*, \Big( \frac{  v+v^{(a)} } {  \rat(a)  }  \Big)  \Big] \ell ^{(a)} - Q\Big( \frac{  v+v^{(a)} } {  \rat(a)  }  \Big) \ell^{(a)}
\end{eqnarray*}
is 
\begin{eqnarray*}\lefteqn{
A_\Phi^{(a)} \cdot e^{2\pi  i   [\varrho ,v+v^{(a)}]   / \rat(a)   } } \\
& &\times  \prod_{  \substack{   \lambda \in V_{0\Z}^\vee  \\  [\lambda , \weyl ] >0    }   }
\prod_{   \substack{  \mu \in  hV_{\Z}^\vee /  h V_{\Z} \\   \mu \sim    \lambda  }   }
\Big(
1-  e^{2\pi  i          [  \spl(a)   \mu ,      k^{(a)}  ]     }      
e^{2\pi  i   [\lambda, v+ v^{(a)} ]    / \rat(a)    }  \Big)^{  c( - Q(\lambda) ,   h^{-1} \mu)    }.
\end{eqnarray*}

Assuming that $\mu\sim \lambda$, we may  lift $\lambda\in I^\perp / I$ to $\tilde{\mu} \in I^\perp \cap( \mu+ hV_\Z)$.
As $s(a)\in Q_\Phi(\A_f)$ acts trivially on $(I^\perp / I) \otimes \A_f$, we have
\[
 [   \lambda,    v^{(a)}  ]  = [\lambda, \spl(a)^{-1} v^{(a)} ] =  [\tilde{\mu} , \spl(a)^{-1} v^{(a)} ] .
\]
Using (\ref{component section}) and the definition of $v^{(a)}$, we find
\[
\rat(a)^{-1}  \spl(a)^{-1} v^{(a)}  = \unit(a)  \ell_* -   \spl(a)^{-1} \ell_*^{(a)} .
\]
Combining these relations with  $[\tilde{\mu} , \ell_*] =[\tilde{\mu} ,k ]$ and $[\tilde{\mu} , \ell^{(a)}_*] =[\tilde{\mu} ,k^{(a)} ]$ shows that
\[
\frac{ [   \lambda,    v^{(a)}  ]  }{\rat(a) }
=  [\tilde{\mu} ,  \unit(a)  k -  \spl(a)^{-1} k^{(a)} ] .
\]
As  $\unit(a)   k - \spl(a^{-1}) k^{(a)}  \in     h V_{ \widehat{\Z}}^\vee$ and $\tilde{\mu}-\mu \in hV_\Z$, we deduce the equality
\[
\frac{ [   \lambda,    v^{(a)}  ]   }{ \rat(a) }   = [  \mu ,    \unit(a)   k   - \spl(a)^{-1}   k^{(a)}   ] 
\]
in $\Q/\Z \iso \widehat{\Q}/\widehat{\Z}$.
Thus 
\begin{align*}
e^{2\pi  i    [  \spl( a )   \mu ,      k^{(a)}  ]     }      e^{2\pi  i  [\lambda, v+ v^{(a)} ]    / \rat(a)    } 
&=
e^{2\pi  i   [   \mu ,    \spl(a)^{-1}  k^{(a)}  ]     }    e^{2\pi  i      [\lambda,  v  +  v^{(a)} ]    / \rat(a)    }   \\
&=
e^{2\pi  i   [   \mu ,    \unit(a)  k ]     }    e^{2\pi  i      [\lambda,  v  ]    / \rat(a)    }   \\
&=
\zeta_\mu^{\unit(a)}   \cdot   e^{2\pi  i     [\lambda, v ]    / \rat(a)    }.
\end{align*}

Finally, the equality 
\[
 e^{2\pi  i   [\varrho ,v+v^{(a)}]   / \rat(a)   }=  e^{2\pi  i   [\varrho ,v ]   / \rat(a)   } 
\]
holds up to a root of unity, simply because $ [\varrho , v^{(a)}] \in \Q$.
\end{proof}

Working on one connected component
\[
\Gamma_\Phi^{(a)} \backslash \DD^\circ \hookrightarrow \mathscr{U}_{K_{\Phi 0 } } ( Q_\Phi , \DD_\Phi),
\]
the pullback of $\Psi(f)$ is $\Psi_{ s(a) h }(f)$.  The pullback of the section $\bm{\ell}^{an}$ of the constant vector bundle $\bm{I}_{dR}^{an}$ determined by $I_\C$ is, by the definition preceding Proposition \ref{prop:canonical sections},  the constant section determined by 
\[
\frac{ \rat(a) }{  2\pi i   } \cdot  \ell \in I_\C .
\]
Thus on $\Gamma_\Phi^{(a)} \backslash \DD^\circ $ we have the equality of  meromorphic functions
\[
( 2\pi i  )^{c(0,0)/2}   \cdot  [ \bm{\ell}^{  \otimes c(0,0)/2 } , \Psi(f) ] =   \rat(a)^{c(0,0)/2}   \cdot   [  \ell^{ \otimes c(0,0)/2} , \Psi_{s(a)h} (f)] .
\]
Combining the two lemmas above, we see that the value of this meromorphic function at the isotropic line 
$ \ell_* + v - [\ell_*,v] \ell  - Q(v) \ell \in \DD^\circ$ is 
\[
A_\Phi^{ \mathrm{rec}(a) } \cdot q_{  \alpha( \varrho )}  \cdot
  \prod_{  \substack{   \lambda \in V_{0\Z}^\vee  \\  [\lambda , \weyl ] >0    }   }
\prod_{   \substack{  \mu \in  hV_{\Z}^\vee /  h V_{\Z} \\   \mu \sim    \lambda  }   }
\Big(
1-  \zeta_\mu^{\mathrm{rec}(a)}   \cdot   q_{ \alpha(\lambda)} \Big)^{  c( - Q(\lambda) ,   h^{-1} \mu)    },
\]
up to scaling by a complex number of absolute value $1$.
The stated $q$-expansion (\ref{trivialized B}) follows from this.

It remains to prove the integrality conditions $\alpha(\lambda) \in \Gamma_\Phi^\vee(1)$.
A priori, every $\alpha(\lambda)$ appearing in the product above (including $\lambda = \varrho$) lies in $U_\Phi(\Q)^\vee(1)$.
However, as the product itself is invariant under the action of 
\[
\Gamma^{(a)}_\Phi = \rat(a) \cdot  \Gamma_\Phi \subset U_\Phi(\Q)
\]
on $\mathcal{D}^\circ$,   the uniqueness of the $q$-expansion implies that only those terms  
\begin{equation}\label{wee q}
 q_{ \alpha(\lambda) }  = e^{2\pi i [\lambda,v] / \rat(a) }
\end{equation}
that are themselves invariant under $\Gamma^{(a)}_\Phi$ can appear.
Pullback by the action of  $u \in U_\Phi(\Q)$ sends
\[
q_{\alpha(\lambda)} \mapsto 
  q_{\alpha(\lambda)} 
  \cdot  e ^{ \langle \alpha( \lambda ) ,  u \rangle / \rat(a) },
\]
where  $\langle -,-\rangle : U_\Phi(\Q)^\vee (1)  \otimes U_\Phi(\Q) \to \Q(1)$ is the tautological pairing,
and it follows that the invariance of (\ref{wee q}) under $\Gamma^{(a)}_\Phi$
 is equivalent to the integrality condition $\alpha(\lambda ) \in \Gamma_\Phi^\vee(1)$.  

This completes the proof of  Proposition \ref{prop:product expansion}.
\end{proof}


\section{Integral models}
\label{s:integrality I}


As in \S \ref{s:borcherds}, we keep  $V_\Z \subset V$ of signature $(n,2)$ with $n\ge 1$.
Fix  a prime $p$  at which $V_\Z$ is  maximal in the sense of \S \ref{ss:intro general}, and 
 assume  that the compact open subgroup (\ref{K choice}) factors as $K = K_p K^p$ with $p$-component
\[
K_p = G(\Q_p) \cap C(V_{\Z_p})^\times.
\]   

Under this assumption,  we recall from \cite{AGHMP-1,AGHMP-2} the construction of an integral model 
\[
\mathcal{S}_K (G,\DD)\to \Spec(\Z_{(p)})
\]
 of $\Sh_K(G,\DD)$, and extensions to this model of the line bundle of weight one modular forms  and the special divisors.   In those references it is  assumed that $V_\Z$ is maximal at \emph{every} prime, but nearly everything extends verbatim to the  more general case considered here.  
 Indeed, one only has to be careful about the definitions of special divisors in \S \ref{ss:special divisors}. 
  Once the correct definitions are formulated  the proofs of [\emph{loc.~cit.}] go through without significant change, and we simply give the appropriate citations without further comment.

 The main new result  is  the pullback formula of Proposition \ref{prop:pullback},  which describes how special divisors restrict under embeddings between orthogonal Shimura varieties of different dimension.  This will be a crucial ingredient in our algebraic variant of the embedding trick of Borcherds.


\subsection{Almost self-dual lattices}


 The motivation for the following definition will become clear in \S \ref{ss:gspin integral}.
 
\begin{definition}
 \label{defn:almost self-dual}
 We  say that $V_{\Z_p}$ is \emph{almost self-dual} if it has one of the following (mutually exclusive) properties:
 \begin{itemize}
 	\item $V_{\Z_p}$ is self-dual;
 	\item $p=2$, $\dim_\Q(V)$ is odd,  and $[V^\vee_{\Z_2}:V_{\Z_2}]$ is not divisible by $4$.
 \end{itemize}
 \end{definition}


\begin{remark}
Almost self-duality is equivalent to the smoothness of the  quadric over $\Z_p$  parameterizing isotropic lines in $V_{\Z_p}$. 
Here, an \emph{isotropic line} in $V_R$ for an $\Z_p$-algebra $R$ is a local direct summand $I\subset V_R$ of rank $1$ that is locally generated by an element $v\in I$ satisfying $Q(v) = 0$.
\end{remark}

Recall from \S \ref{ss:hodge embedding} that $G$ acts on the $\Q$-vector space $H=C(V)$, and that one may choose  a $\Z$-valued symplectic form $\psi$ on 
\[
H_\Z = C(V_\Z)
\]   
in such a way that the action of $G$ induces a Hodge embedding into the Siegel Shimura datum determined by $(H,\psi)$. The following lemma  will be used in \S \ref{s:integral q}  to choose $\psi$ in a particularly nice way.

\begin{lemma}\label{lem:integral polarization}
Assume that $V$ has Witt index $2$ (this is automatic if $n\ge 5$).
If  $V_{\Z_p}$ is almost self-dual, then we may choose $\psi$ as above in such a way that $H_{\Z_p}$ is self-dual.\end{lemma}

\begin{proof}
Choose any isotropic line $I\subset V$, and let  $\ell \in I \cap V_\Z$ be a $\Z$-module generator.
Let $N$ be the positive integer defined by 
\[
N\Z = [ V_\Z , \ell].
\]
On the one hand,    $\ell / N \in V_{\Z}^\vee/V_{\Z}$ is isotropic under the $\Q/\Z$-valued quadratic form induced by $Q$.    On the other hand,   maximality of  $V_{\Z_p}$ implies that  $V_{\Z_p}^\vee /V_{\Z_p} $ has no nonzero isotropic vectors.  Thus $\ell/N \in V_{\Z_p}$, and  so $p\nmid N$.

It follows that there is some  $k\in V_\Z$ such that $p\nmid [k,\ell]$, and from  this it is easy to see that there exists a vector  $v\in \Z k + \Z \ell$    such that  $Q(v)$ is negative and prime to $p$.

The $\Q$-span of $k,\ell \in V$ is  a hyperbolic plane over $\Q$, and the $\Z_p$-span of $k,\ell\in V_{\Z_p}$ is an integral  hyperbolic plane over $\Z_p$.
It follows that the orthogonal complement
\[
W= (\Q k + \Q \ell)^\perp \subset V
\]  
has Witt index $1$, and that the $\Z$-lattice  $W_\Z = W\cap V_\Z$ satisfies
\[
V_{\Z_p} = \Z_p k  \oplus \Z_p \ell \oplus W_{\Z_p}.
\]
In particular  $W_{\Z_p}$ is again maximal.  Repeating the argument above with $V_{\Z}$ replaced by $W_{\Z}$,  we find another vector  $w\in V_\Z$ with $Q(w)$ negative and prime to $p$, and $[v,w]=0$.

We have now constructed an element $\delta = vw\in C(V_\Z)$ such that 
\[
\delta^2=-Q(v)Q(w) \in \Z_{(p)}^\times.
\]
 Set $\psi(x,y) = \mathrm{Trd}(x\delta y^*)$,  exactly as in \S \ref{ss:hodge embedding}.

It remains  to  prove  that  $H_{\Z_p}$ is self-dual.
We will use the decomposition 
\[
H_{\Z_p}=H_{\Z_p}^+ \oplus H_{\Z_p}^-
\]
 induced by the decomposition $C(V_{\Z_p})=C^+(V_{\Z_p}) \oplus C^-(V_{\Z_p})$ into even and odd parts.  It is not hard to see that these direct summands of $H_{\Z_p}$ are orthogonal to each other under $\psi$, and so it suffices to prove the self-duality of each summand individually.

According to \cite[\S C.2]{MR3362641}, the almost self-duality of $V_{\Z_p}$ implies that the even Clifford algebra $C^+(V_{\Z_p})$ is an Azumaya algebra  over  its center, and this center is itself a finite \'etale $\Z_p$-algebra.   Equivalently, $C^+(V_{\Z_p})$ is isomorphic \'etale locally on $\Spec(\Z_p )$ to a finite product of matrix algebras.
It follows from this that 
\[
C^+(V_{\Z_p}) \otimes C^+(V_{\Z_p}) \map{ x\otimes y \mapsto \mathrm{Trd}(xy)  } \Z_p
\]
is a perfect bilinear pairing.  The self-duality of $H^+_{\Z_p}$  under  $\psi$ follows easily from this.
The self-duality of $H^-_{\Z_p}$ then follows using the isomorphism \[H_{\Z_p}^- \iso H_{\Z_p}^+\] given by right multiplication by the $v\in C(V_\Z)$ chosen above, and the relation
\[
\psi(  x v,  y v ) =  -Q(v) \cdot  \psi ( x , y)
\]
for all $x,y\in H$.
\end{proof}


\subsection{Isometric embeddings}
\label{ss:embiggen}


We will repeatedly find ourselves in the following situation.
Suppose we have another  quadratic space $(V^\beef ,Q^\beef)$  of signature $(n^\beef , 2)$, and an isometric embedding $V \hookrightarrow V^\beef$.  
 This induces a morphism of Clifford algebras $C(V) \to C(V^\beef)$, which induces a morphism of GSpin Shimura data 
\[
(G,\DD) \to (G^\beef , \DD^\beef).
\]

Just as we assume for $(V,Q)$, suppose we are given a $\Z$-lattice $V^\beef_\Z \subset V^\beef$
on which $Q^\beef$ is integer valued, and which is maximal at $p$.
Let
\[
K^\beef   = K^\beef_p \cdot K^{\beef,p} \subset G^\beef(\A_f) \cap C( V_{\widehat{\Z}}^\beef) ^\times
\]
be a compact open subgroup with $p$-component
\[
K_p^\beef = G^\beef(\Q_p) \cap C(V^\beef_{\Z_p})^\times.
\]
Assume that $V_{ \Z} \subset  V^\beef_{\Z}$ and   $K\subset K^\beef$,  so that we have a  finite and unramified morphism 
\begin{equation}\label{embiggen morphism}
j : \Sh_K(G,\DD) \to \Sh_{K^\beef}(G^\beef , \DD^\beef)
\end{equation}
of canonical models.  
Our choices  imply (using the assumption that $V_{\Z_p}$ is maximal) that   
$V_{\Z_p} = V_{\Q_p} \cap V_{\Z_p}^\beef$ and 
$K_p = K_p^\beef \cap G(\Q_p)$.

\begin{lemma}\label{lem:good beef}
It is possible to choose $(V^\beef, Q^\beef)$ and $V^\beef_\Z$ as above in such a way that $V_\Z^\beef$ is self-dual. Moreover, we can ensure that $V\subset V^\beef$ has codimension at most $2$ if $n$ is even, and has codimension at most $3$ if $n$ is odd.
\end{lemma}

\begin{proof}
%
An exercise in the classification of quadratic spaces over $\Q$ shows that we may choose a positive definite quadratic space $W$ in such a way that the orthogonal direct sum $V^\beef=V\oplus W$  admits a self-dual lattice  locally at every finite prime (for example,  we may arrange for $V^\beef$ to be a sum of  hyperbolic planes locally at every finite prime).
From Eichler's theorem that any two maximal lattices in a $\Q_p$-quadratic space are isometric \cite[Theorem 8.8]{Ger}, it follows that any maximal lattice in $V^\beef$ is self-dual.   Enlarging $V_\Z$ to a maximal lattice  $V^\beef_\Z \subset V^\beef$ proves the first claim.

A more careful analysis, once again using the classification of quadratic spaces, also yields the second claim.
\end{proof}

%
%
%
%
%


\subsection{Definition of the integral model}
\label{ss:gspin integral}


We now define our integral model of the Shimura variety $\Sh_K(G,\DD)$.

Assume  first that $V_{\Z_p}$ is almost self-dual.
This implies, by \cite[\S C.4]{MR3362641},  that 
\[
\mathcal{G} = \GSpin(V_{\Z_{(p)}})
\]
 is a reductive group scheme over $\Z_{(p)}$, and hence that $K_p=\mathcal{G}(\Z_p)$ is a hyperspecial compact open subgroup of $G(\Q_p)$. 
Thus $\Sh_K(G,\DD)$ admits a canonical smooth integral model $
\mathcal{S}_K(G,\DD)$ over $\Z_{(p)}$ by the results of Kisin \cite{KisinJAMS} (and \cite{KMP} if $p=2$).

\begin{remark}
The notion of almost self-duality does not appear anywhere in our main references \cite{mp:spin,AGHMP-1,AGHMP-2} on integral models of $\Sh_K(G,\DD)$.  This is due to an oversight on the authors' part: we did not realize that one could obtain smooth integral models even if $V_{\Z_p}$ fails to be self-dual.
\end{remark}

According to \cite[Proposition 3.7]{KMP}  there is a  functor
\begin{equation}\label{int bundle machine}
N  \mapsto ( \bm{N}_{dR} , F^\bullet  \bm{N}_{dR}  )
\end{equation}
 from representations  $\mathcal{G} \to \GL(N)$ on free $\Z_{(p)}$-modules of finite rank  to filtered vectors bundles on $\mathcal{S}_K(G,\DD)$, restricting to the functor (\ref{mixed bundles}) in the generic fiber.\footnote{There is also a weight filtration on $\bm{N}_{dR}$, but, as noted in Remark \ref{rem:no weight}, it is not very interesting over the pure Shimura variety.}

 Applying this functor to the representation $V_{\Z_{(p)}}$ yields a filtered vector bundle $(\bm{V}_{dR} ,F^\bullet \bm{V}_{dR})$.
 Applying the  functor to the representation $H_{\Z_{(p)} }=C(V_{\Z_{(p)}})$  yields a filtered vector bundle $(\bm{H}_{dR} ,F^\bullet \bm{H}_{dR})$.   
The inclusion (\ref{special injection}) restricts to $V_{\Z_{(p)}} \to \End( H_{\Z_{(p)}} )$, which  determines an injection 
\[
\bm{V}_{dR} \to \underline{\End}(\bm{H}_{dR})
\]
onto a local direct summand.

 For any local section $x$ of $\bm{V}_{dR}$, the composition $x\circ x$ is a local section of the subsheaf
$
 \co_{ \mathcal{S}_K(G,\DD) } \subset \underline{\End}(\bm{H}_{dR}).
$
This defines a quadratic form 
\[
 \bm{Q} : \bm{V}_{dR} \to \co_{ \mathcal{S}_K(G,\DD) },
 \]
  with an  associated bilinear form
\[
[-,-] : \bm{V}_{dR} \otimes \bm{V}_{dR} \to  \co_{ \mathcal{S}_K(G,\DD) } 
\]
related  as in (\ref{bilinear}).
The filtration on $\bm{V}_{dR}$ has the form 
\[
0= F^2 \bm{V}_{dR}  \subset F^1 \bm{V}_{dR}  \subset F^0 \bm{V}_{dR}\subset  F^{-1} \bm{V}_{dR} = \bm{V}_{dR},
\]
in which $F^1\bm{V}_{dR}$ is an isotropic  line,  and   $F^0\bm{V}_{dR} = (F^1\bm{V}_{dR})^\perp$.  
As in \S \ref{ss:modular forms},  the \emph{line bundle of weight one modular forms} on $\mathcal{S}_K(G,\DD)$ is 
\[
\bm{\omega} =F^1 \bm{V}_{dR}.
\]

If $V_{\Z_p}$ is not almost self-dual then  choose auxiliary data $(V^\beef, Q^\beef)$ as in \S \ref{ss:embiggen} in such a way that  $V^\beef_{\Z_p}$ is almost self-dual.  This determines a commutative diagram 
\begin{equation}\label{integral definition}
\xymatrix{
{  \mathcal{S}_K (G,\DD)} \ar[d]  & {  \Sh _K(G,\DD)  }   \ar[d]  \ar[l] \\
 { \mathcal{S}_{K^\beef}(G^\beef,\DD^\beef)  }  & {  \Sh_{K^\beef}(G^\beef , \DD^\beef)  }  \ar[l]  \\
}
\end{equation}
in which the lower left corner is the canonical integral model of $ \Sh_{K^\beef}(G^\beef , \DD^\beef)$, and the upper right corner  is defined as its normalization in $\Sh _K(G,\DD)$, in the sense of  \cite[Definition 4.2.1]{AGHMP-2}.  By construction, $\mathcal{S}_K(G,\DD)$ is a normal Deligne-Mumford stack,  flat and of finite type  over $\Z_{(p)}$.

Define the \emph{line bundle of weight one modular forms} on $\mathcal{S}_K(G,\DD)$ by
\begin{equation}\label{integral omega}
\bm{\omega}= \bm{\omega}^\beef|_{\mathcal{S}_K(G,\DD)},
\end{equation}
where $\bm{\omega}^\beef$ is the line bundle  on $\mathcal{S}_{K^\beef}(G^\beef,\DD^\beef)$ constructed in the almost self-dual case above.
The line bundle (\ref{integral omega}) extends the line bundle of the same name previously constructed on the generic fiber.

The following is \cite[Proposition 4.4.1]{AGHMP-2}.

\begin{proposition}\label{prop:integral model ind}
The $\Z_{(p)}$-stack $\mathcal{S}_K(G,\DD)$ and the line bundle $\bm{\omega}$ are independent of the auxiliary choices of $(V^\beef, Q^\beef)$, $V_{\widehat{\Z}}^\beef$,  and $K^\beef$ used in their construction, and the Kuga-Satake abelian scheme of \S \ref{ss:hodge embedding} extends uniquely to an abelian scheme
$\mathcal{A} \to \mathcal{S}_K(G,\DD)$. 
\end{proposition}

The following is a restatement of the main result of \cite{mp:spin}. 

\begin{proposition}
 If $p>2$ and $p^2$ does not divide $[V_\Z^\vee : V_\Z]$,  then $\mathcal{S}_K(G,\DD)$ is regular.
\end{proposition}

\begin{remark}\label{rem:different models}
Our $\mathcal{S}_K(G,\DD)$  is not quite the same as the integral model of  \cite{AGHMP-1}.  That integral model is obtained from $\mathcal{S}_K(G,\DD)$ by deleting certain closed substacks supported in characteristics $p$ for which $p^2$ divides $[V^\vee_\Z : V_\Z]$.  
The point of deleting such substacks is that the vector bundle $\bm{V}_{dR}$ on  $\Sh_K(G,\DD)$ of  \S \ref{ss:modular forms} then extends canonically to the remaining open substack.  In the present work, as in  \cite{AGHMP-2}, the only automorphic vector bundle required on $\mathcal{S}_K(G,\DD)$ is the line bundle of modular forms $\bm{\omega}$ just constructed; we have no need of an extension of $\bm{V}_{dR}$ to $\mathcal{S}_K(G,\DD)$.
\end{remark}

\subsection{Special divisors}
\label{ss:special divisors}


For $m\in \Q_{>0}$ and $\mu \in V_\Z^\vee / V_\Z$ there is a  Cartier divisor $\mathcal{Z}(m,\mu)$ on $\mathcal{S}_K(G,\DD)$, defined in \cite{AGHMP-1,AGHMP-2} in the case where $V_\Z$ is maximal.
As we now assume only the weaker hypothesis that $V_\Z$ is maximal at $p$, 
the definition requires minor adjustment.

We first define the divisors in the generic fiber, where they were originally constructed by Kudla \cite{Ku97}.
Our construction is different, and has a more moduli-theoretic flavor.

By the theory of automorphic vector bundles described in \S \ref{s:AVB}, the $G$-equivariant inclusion (\ref{special injection}) determines an inclusion 
\begin{equation}\label{de Rham special inclusion}
\bm{V}_{dR} \subset \underline{\End}(\bm{H}_{dR})
\end{equation}
of vector bundles on $\Sh_K(G,\DD)$,  respecting the Hodge filtrations.   
Recall from \S \ref{ss:hodge embedding} that the filtered vector bundle $\bm{H}_{dR}$ is canonically identified with the first relative de Rham homology of the Kuga-Satake abelian scheme \[ \pi: A \to \Sh_K(G,\DD).\]

The compact open subgroup $K\subset G(\A_f)$ appears as a quotient of the \'etale fundamental group of 
 $\Sh_K(G,\DD)$, and hence representations of $K$ give rise to \'etale local systems.
In particular, for any prime $\ell$  the $\Z_\ell$-lattice $H_{\Z_\ell}$  determines an \'etale sheaf of $\Z_\ell$-modules $\bm{H}_\ell$ on $\Sh_K(G,\DD)$.  This is just the relative $\ell$-adic Tate module 
\[
\bm{H}_\ell \iso \underline{\Hom}( R^1\pi_{et,*} \underline{\Z}_\ell , \underline{\Z}_\ell)
\]
of the Kuga-Satake abelian scheme.

 As in the discussion preceding (\ref{disc quotient}),  $K$ also acts on both
\[
V_{\Z_\ell} = V_\Z \otimes \Z_\ell \qquad\mbox{and}\qquad V^\vee_{\Z_\ell} = V^\vee_\Z \otimes \Z_\ell,
\]
and the induced action on the quotient $V^\vee_{\Z_\ell} / V_{\Z_\ell}$ is trivial.  
These representations of $K$ determine \'etale sheaves of $\Z_\ell$-modules
$
\bm{V}_{\ell} \subset \bm{V}_\ell^\vee,
$
along with an inclusion of \'etale $\Z_\ell$-sheaves
\begin{equation}\label{etale special inclusion}
\bm{V}_\ell  \subset \underline{\End}( \bm{H}_\ell)  .
\end{equation}
and a canonical trivialization
$
\bm{V}_\ell^\vee / \bm{V}_{\ell}  \iso V^\vee_{\Z_\ell} / V_{\Z_\ell} .
$
In particular, each $\mu_\ell \in V^\vee_{\Z_\ell} / V_{\Z_\ell}$ determines a subsheaf of sets 
\begin{equation}\label{mu sheaf}
 \mu_\ell + \bm{V}_{\ell} \subset \bm{V}_\ell \otimes \Q_\ell.
\end{equation}

Suppose we are given a $\Q$-scheme $S$ and a morphism $S \to \Sh_K(G,\DD)$.
Denote by  $A_S \to S$ the  pullback of the Kuga-Satake abelian scheme.
A  quasi-endomorphism\footnote{A quasi-endomorphism should really be defined as global section of the Zariski sheaf $\underline{\End}(A_S) \otimes \Q$ on $S$.  If $S$ is not of finite type over $\Q$, the space of such global sections can be strictly larger than $\End(A_S) \otimes \Q$.  For simplicity of notation, we ignore this minor technical point.}  $x\in \End(A_S)\otimes \Q$ is \emph{special} if
\begin{itemize}
\item
 its de Rham realization 
\[
x_{dR} \in H^0 (S , \underline{\End}(\bm{H}_{dR})|_S )
\]
lies in the subsheaf $\bm{V}_{dR} |_S$, and
\item
  its $\ell$-adic realization
\[
x_\ell \in H^0(S , \underline{\End}( \bm{H}_\ell )|_S \otimes \Q_\ell )
\]
lies in the subsheaf $\bm{V}_{\ell} |_S  \otimes \Q_\ell$ for every prime $\ell$.
\end{itemize}

The space of all special quasi-endomorphisms of $A_S$ is a $\Q$-subspace
\[
V(A_S) \subset \End(A_S) \otimes \Q.
\]
Under the inclusion $V \subset \End(H)$,  the quadratic form on $V$ becomes  $Q(x) = x\circ x$.
Similarly, the square of any  $x\in V(A_S)$  lies in $\Q \subset \End(A_S) \otimes \Q$, 
and  $V(A_S)$ is endowed with the positive definite quadratic form $Q(x) = x\circ x$. 
For each   $\mu \in V_\Z^\vee / V_\Z$,  we now define 
\begin{equation}\label{generic special mu}
V_\mu(A_S) \subset V(A_S)
\end{equation}
to be the set of all special quasi-endomorphisms  whose $\ell$-adic realization  lies in the subsheaf (\ref{mu sheaf}) for every prime $\ell$, and set
\[
Z( m , \mu)(S) \define \{ x \in V_\mu(A_S) : Q(x) = m\}.
\]

We now explain how to extend this definition to the integral model.  

First assume that $V_\Z$ is self-dual at $p$.
As in the discussion following (\ref{int bundle machine}), the inclusion of vector bundles (\ref{de Rham special inclusion}) has a canonical extension to the integral model $\mathcal{S}_K(G,\DD)$.
Directly from the definitions, so does the inclusion of \'etale $\Q_\ell$-sheaves (\ref{etale special inclusion}) for any $\ell\neq p$.  As a substitute for $p$-adic \'etale cohomology, we use the inclusion 
\begin{equation}\label{crystalline special inclusion}
\bm{V}_{crys} \subset \underline{\End}(\bm{H}_{crys})
\end{equation}
of locally free crystals  on $\mathcal{S}_K(G,\DD)_{/\F_p}$ as in \cite[Proposition 4.2.5]{AGHMP-2}.
There is a canonical isomorphism
\[
\bm{H}_{crys} \iso 
\underline{\Hom}(R^1 \pi_{crys,*} \co^{crys}_{ \mathcal{A}_{\F_p} / \Z_p } ,  \co^{crys}_{ \mathcal{M}_{\F_p} / \Z_p }  )
\]
between $\bm{H}_{crys}$ and the first relative crystalline homology of the reduction  of the 
Kuga-Satake abelian scheme  $\pi:\mathcal{A} \to \mathcal{S}_K(G,\DD)$ of Proposition \ref{prop:integral model ind}.

Still assuming that $V_\Z$ is self-dual at $p$, suppose we are given a $\Z_{(p)}$-scheme $S$ and a morphism $S\to \mathcal{S}_K(G,\DD)$, and let 
 $\mathcal{A}_S$ be the pullback of the Kuga-Satake abelian scheme.
We call $x\in \End(A_S)\otimes \Z_{(p)}$  \emph{special} if 
\begin{itemize}
\item
 its de Rham realization 
\[
x_{dR} \in H^0 (S , \underline{\End}(\bm{H}_{dR})|_S )
\]
lies in the subsheaf $\bm{V}_{dR} |_S$, 
\item
 its $\ell$-adic realization
\[
x_\ell \in H^0(S , \underline{\End}( \bm{H}_\ell )|_S \otimes \Q_\ell )
\]
lies in the subsheaf $\bm{V}_{\ell} |_S  \otimes \Q_\ell$ for every prime $\ell\neq p$,
\item 
its $p$-adic realization
\[
x_p\in H^0(S_{\Q} , \underline{\End}( \bm{H}_p)|_{S_\Q} )
\]
over the generic fiber $S_\Q$ lies in the subsheaf $\bm{V}_p |_{S_\Q}$, and 
\item
its crystalline realization 
\[
x_{crys} \in H^0( S_{\F_p} , \underline{\End}( \bm{H}_{crys})|_{S_{\F_p}}  )
\]
over the special fiber $S_{\F_p}$ lies in the subcrystal $\bm{V}_{crys}|_{S_{\F_p}}$.
\end{itemize}
The space of all such  $x\in \End(A_S)\otimes \Z_{(p)}$ is denoted
\[
V(\mathcal{A}_S)_{\Z_{(p)}}  \subset \End( \mathcal{A}_S) \otimes \Z_{(p)},
\]
and tensoring with $\Q$ defines the subspace of all special quasi-endomorphisms
\[
V(\mathcal{A}_S)   \subset \End( \mathcal{A}_S) \otimes \Q.
\]
It  endowed with a positive definite quadratic form $Q(x) = x\circ x$ exactly as above.
For any $\mu \in V_\Z^\vee/V_\Z$ we define
\begin{equation}\label{self-dual mu}
V_\mu(\mathcal{A}_S) \subset V(\mathcal{A}_S)_{\Z_{(p)}}
\end{equation}
as the subset of elements whose $\ell$-adic realization lies in (\ref{mu sheaf}) for every prime $\ell \neq p$.

Now consider the general case in which $V_\Z \subset V$ is only assumed to be maximal at $p$.  
In this generality we still have the \'etale $\Q_\ell$-sheaves (\ref{etale special inclusion}) for $\ell\neq p$.
However,  there is no adequate theory of automorphic vector bundles or crystals on $\mathcal{S}_K(G,\DD)$; 
compare with Remark \ref{rem:different models}.  In particular, we have 
 no adequate substitute for the sheaves in  (\ref{crystalline special inclusion}).

So that we may apply the results of \cite{AGHMP-1,AGHMP-2}, 
enlarge $V_\Z$ to a lattice $V_\Z' \subset V$ that is maximal at every prime.
This choice determines a second  $\Z$-lattice $H'_\Z  \subset H_\Q$, and hence a second Kuga-Satake abelian scheme 
\[
\mathcal{A}' \to \mathcal{S}_K(G,\DD)
\]
endowed with an isogeny 
$
\mathcal{A} \to \mathcal{A}'
$
 of degree prime to $p$.     Choose a larger quadratic space $V^\beef$ as in \S \ref{ss:embiggen} admitting a maximal lattice $V_\Z^\beef \subset V^\beef$ that is self-dual at $p$, and  an isometric embedding $V_\Z' \to V_\Z^\beef$.

By the very construction of the integral model, there is a finite morphism
\[
\mathcal{S}_K(G,\DD) \to \mathcal{S}_{K^\beef}(G^\beef,\DD^\beef).
\] 
According to \cite[Proposition 2.5.1]{AGHMP-1}, the abelian schemes $\mathcal{A}'$ and 
\[
\mathcal{A}^\beef\to \mathcal{S}_{K^\beef} (G^\beef,\DD^\beef)
\]
 carry right actions of the integral Clifford algebras
$C(V_\Z')$ and $C(V_\Z^\beef)$, respectively, and are related by a canonical isomorphism
\begin{equation}\label{clifford serre}
\mathcal{A}'\otimes_{C(V_\Z') } C(V_\Z^\beef) \iso \mathcal{A}^\beef |_{ \mathcal{S}_K(G,\DD) }.
\end{equation}
Note that the Serre tensor construction on the left is  defined because the maximality of $V_\Z'$ implies that $V_\Z' \subset V_\Z^\beef$ as a $\Z$-module direct summand, which implies that the natural map $C(V_\Z') \to C(V_\Z^\beef)$ makes $C(V_\Z^\beef)$ into a free $C(V_\Z')$-module.

\begin{definition}
Suppose we are given a morphism $S\to \mathcal{S}_K(G,\DD)$.  A quasi-endomorphism 
\[
x\in \End(\mathcal{A}_S) \otimes \Q
\]
 is \emph{special} if the induced quasi-endomorphism of $\mathcal{A}_S'$ commutes with the action of  $C(V_\Z')$, and its image under the map 
\[
\End_{ C(V_\Z') } (\mathcal{A}' _S) \otimes \Q \to \End_{ C(V_\Z^\beef) } (\mathcal{A}^\beef_S ) \otimes \Q
\]
induced by (\ref{clifford serre}) is a special quasi-endomorphism of $\mathcal{A}^\beef_S$
(in the sense already defined for the self-dual-at-$p$ lattice $V_\Z^\beef$).
\end{definition}

The following is \cite[Proposition 4.3.4]{AGHMP-2}.

\begin{proposition}\label{prop:special rigidity}
If $S$ is connected,  then $x\in \End(\mathcal{A}_S) \otimes \Q$ is special if and only if the restriction  $x_s \in \End(\mathcal{A}_s) \otimes \Q$  is special for some (equivalently, every) geometric point $s\to S$.
\end{proposition}

Once again, the space of all special quasi-endomorphisms 
\[
V(\mathcal{A}_S) \subset \End(\mathcal{A}_S) \otimes \Q
\]
carries a positive definite quadratic form  $Q(x)=x\circ x$.  By construction it comes with an isometric embedding
\begin{equation}\label{embiggen special isometry}
V(\mathcal{A}_S) \subset V(\mathcal{A}^\beef_S).
\end{equation}

It remains to define a subset 
\begin{equation}\label{special cosets}
V_\mu(\mathcal{A}_S) \subset V(\mathcal{A}_S)
\end{equation}
for each  coset $\mu \in V_\Z^\vee / V_\Z$.
Let $\mu_\ell \in V_{\Z_\ell}^\vee / V_{\Z_\ell}$ be the $\ell$-component.
If $\ell\neq p$ let
\[
V_{\mu_\ell}(\mathcal{A}_S) \subset V(\mathcal{A}_S)
\]
be  the subset of elements whose $\ell$-adic realization lies in the subsheaf (\ref{mu sheaf}). 
To treat the $p$-part of $\mu$, define
\[
\Lambda = \{ x\in V_\Z^\beef : x\perp V_\Z\}.
\]
The maximality of $V_\Z$ at $p$ implies that $V_{\Z_p} \subset V_{\Z_p}^\beef$ is a $\Z_p$-module direct summand.  From this and the self-duality of $V_\Z^\beef$ at $p$ it is easy to see that the projections to the two factors in 
\[
V^\beef =  V \oplus \Lambda_\Q 
\]
induce bijections
\begin{equation}\label{coset swap}
V^\vee_{\Z_p} / V_{\Z_p}  \iso (V_{\Z_p}^\beef)^\vee / V_{\Z_p}^\beef  \iso \Lambda^\vee_{\Z_p}/\Lambda_{\Z_p}.
\end{equation}
The image of $\mu_p$ under this bijection is denoted $\bar{\mu}_p \in  \Lambda^\vee_{\Z_p}/\Lambda_{\Z_p}$.
As in  \cite[Proposition 2.5.1]{AGHMP-1}, there is a canonical isometric embedding
\[
\Lambda \to V(\mathcal{A}_S^\beef)_{\Z_{(p)} } 
\]
whose image is orthogonal to that of (\ref{embiggen special isometry}).  In fact, we have an orthogonal decmposition 
\[
V(\mathcal{A}_S^\beef) = V(\mathcal{A}_S) \oplus \Lambda_\Q,
\]
which allows us to define
\begin{equation}\label{sneaky coset}
V_{\mu_p}(\mathcal{A}_S) = \{ x \in V(\mathcal{A}_S) : x + \bar{\mu}_p \in V(\mathcal{A}_S^\beef)_{\Z_{(p)} } \}.
\end{equation}
Finally, define (\ref{special cosets}) by
\[
V_\mu(\mathcal{A}_S) = \bigcap_\ell V_{\mu_\ell} (\mathcal{A}_S).
\]
This  set   is independent of the choice of auxiliary data  $V_\Z' \subset V_\Z^\beef \subset V^\beef$ used in its definition, and agrees with the definition (\ref{generic special mu})  if $S$ is a $\Q$-scheme.  
See \cite[Proposition 4.5.3]{AGHMP-2}.

The following is \cite[Proposition 2.7.2]{AGHMP-1}.

\begin{proposition}\label{prop:divisor represent}
Given a positive $m\in \Q$  and  a $\mu \in V_\Z^\vee / V_\Z$, the functor sending an $\mathcal{S}_K(G,\DD)$-scheme $S$ to 
\[
 \mathcal{Z}( m , \mu)(S)  = \{ x\in V_\mu(\mathcal{A}_S) : Q(x) =m \}
\]
is represented by a finite, unramified, and relatively representable morphism of Deligne-Mumford stacks
\begin{equation}\label{special divisors}
\mathcal{Z}( m , \mu) \to \mathcal{S}_K(G,\DD).
\end{equation}
\end{proposition}

 In the next subsection we will justify in what sense the morphisms (\ref{special divisors}), which are not even closed immersions, deserve the name \emph{special divisors}.  
 
 We end this section by describing what the morphism (\ref{special divisors}) looks like in the complex fiber.
 For each $g\in G(\A_f)$, the pullback of  (\ref{special divisors}) via the  complex uniformization
\[
\DD \map{z\mapsto (z,g)}  \Sh_K(G,\DD)(\C)
\]
can be made explicit.  
 Each $x\in V$ with $Q(x) >0$ determines an analytic subset 
\[
\DD(x) = \{ z\in \DD : [z,x] =0 \} 
\]
of the hermitian domain (\ref{orthogonal domain}).

From the discussion of \S \ref{ss:hodge embedding}, we see that  the fiber of the Kuga-Satake abelian scheme 
at a point $z\in \DD$ is the complex torus
\[
A_z(\C) = g H_\Z \backslash H_\C / z H_\C.
\]
The action of $x\in V \subset \End(H)$ by left multiplication in the Clifford algebra $C(V)$ defines a quasi-endomorphism of $A_z(\C)$ if and only if it 
preserves the subspace $z H_\C \subset H_\C$, and a linear algebra  exercise shows that this condition is equivalent to $z\in \DD(x)$.
Using this, one can check that the pullback of (\ref{special divisors}) via the above complex uniformization is \begin{equation}\label{complex divisor}
 \bigsqcup_{  \substack{ x\in g\mu + g V_\Z \\ Q(x) = m}  } \DD(x)  \to   \DD.
\end{equation}
Here, by mild abuse of notation, $g\mu$ is the image of $\mu$ under the action-by-$g$ isomorphism
$
V^\vee_\Z / V_\Z \to gV^\vee_\Z / gV_\Z
$


\subsection{Deformation theory}
\label{ss:deformation}


We need to explain the sense in which the morphism (\ref{special divisors}) merits the name \emph{special divisor}.
This is closely tied up with the deformation theory of special endomorphisms, which will also be needed in the proof of the pullback formula of Proposition \ref{prop:pullback} below.

It is enlightening to first consider the  complex analytic situation of (\ref{complex divisor}).
Each subset $\DD(x) \subset \DD$ is not only an analytic divisor, but arises as the $0$-locus of a canonical section 
\begin{equation}\label{analytic obstruction}
\mathrm{obst}_x^{an} \in H^0( \DD , \bm{\omega}_\DD^{-1} )
\end{equation}
of the inverse of the tautological bundle $\bm{\omega}_{\DD}$ on (\ref{orthogonal domain}).   
Indeed, recalling that the fiber of $\bm{\omega}_\DD$ at $z\in \DD$ is the isotropic line $\C z \subset V_\C$, 
we define (\ref{analytic obstruction})  as the linear functional
\[
\C z \map{ z \mapsto   [z,x] }  \C .
\]
This is the \emph{analytic obstruction to deforming $x$}.

Returning to the algebraic world, suppose 
\begin{equation}\label{deformation diagram}
\xymatrix{
{   S } \ar[r]  \ar[d]   &   { \mathcal{Z}(m,\mu) }  \ar[d] \\
  { \widetilde{S} }   \ar[r]  &  {   \mathcal{S}_K(G,\DD) }
}
\end{equation}
is  a commutative diagram of stacks in which  $S \to \widetilde{S}$ is a closed immersion of schemes defined by an ideal sheaf $J \subset \co_{\widetilde{S}}$ with $J^2=0$.   
After pullback  to $S$, the Kuga-Satake abelian scheme $\mathcal{A} \to  \mathcal{S}_K(G,\DD)$ acquires a tautological special quasi-endomorphism
\[
x \in V_\mu(\mathcal{A}_S),
\]
and we want to know when this lies in the image of the (injective) restriction map
\begin{equation}\label{special deformation}
V_\mu( \mathcal{A}_{\widetilde{S}}) \to V_\mu(\mathcal{A}_S).
\end{equation}
Equivalently, when there is a (necessarily unique) dotted arrow 
\[
\xymatrix{
{   S } \ar[r]  \ar[d]   &   { \mathcal{Z}(m,\mu) }  \ar[d] \\
  { \widetilde{S} }   \ar@{.>}[ur] \ar[r]  &  {   \mathcal{S}_K(G,\DD) }
}
\]
making the diagram commute.

\begin{proposition}\label{prop:obstruction}
In the situation above, there is a canonical section
\begin{equation}\label{obstruction}
\mathrm{obst}_x \in H^0 \big( \widetilde{S} , \bm{\omega}|^{-1}_{\widetilde{S} } \big),
\end{equation}
called the \emph{obstruction to deforming $x$}, such that 
$x$ lies in the image of (\ref{special deformation}) if and only if $\mathrm{obst}_x=0$.
\end{proposition}

\begin{proof}
Suppose first that $V_\Z$ is self-dual at $p$, so that we have an inclusion 
\[
\bm{V}_{dR} \to \underline{\End}(\bm{H}_{dR})
\]
as a local direct summand of vector bundles on $\mathcal{S}_K(G,\DD)$.
The vector  bundle  $\bm{H}_{dR}$  is identified with the first relative de Rham homology of the Kuga-Satake abelian scheme $\mathcal{S}$.  As such, it is is endowed with its Gauss-Manin connection,  which restricts to a flat connection 
\[
\nabla : \bm{V}_{dR} \to \bm{V}_{dR} \otimes \Omega^1_{  \mathcal{S}_K(G,\DD) / \Z_{(p)} }.
\]
Indeed, one can check this in the complex fiber, over which the connection becomes identified, using (\ref{general mixed betti-derham}), with 
\[
 \bm{V}_{Be} \otimes_\Z \co_{ \Sh_K(G,\DD)(\C)} \map{ 1 \otimes d} \bm{V}_{Be} \otimes_\Z \Omega^1_{ \Sh_K(G,\DD)(\C)}.
\]

 The  de Rham realization 
\begin{equation}\label{deform x}
x_{dR} \in H^0(S , \bm{V}_{dR}|_S )
\end{equation}
 is parallel, and therefore admits parallel transport (the algebraic theory of parallel transport can be extracted from \cite[\S 2]{BO}, for example) to $\widetilde{S}$: there is a unique parallel extension of $x_{dR}$ to 
 \begin{equation}\label{extend x}
\widetilde{x}_{dR} \in H^0( \widetilde{S} , \bm{V}_{dR}|_{\widetilde{S}} ).
\end{equation}
We now define $\mathrm{obst}_x$ be the image of $\widetilde{x}_{dR}$ under
$
\bm{V}_{dR} \to \bm{V}_{dR}/ F^0 \bm{V}_{dR},
$
and use the perfect bilinear pairing (\ref{de Rham bilinear}) to identify 
\[
\bm{V}_{dR}/ F^0 \bm{V}_{dR}\iso ( F^1 \bm{V}_{dR} )^{-1}=\bm{\omega}^{-1} .
\]

The local sections of $F^0 \bm{V}_{dR}$ are precisely those local sections of $\bm{V}_{dR} \subset \underline{\End}(\bm{H}_{dR})$ which preserve the Hodge filtration $F^0 \bm{H}_{dR} \subset \bm{H}_{dR}$.
The vanishing of $\mathrm{obst}_x$ is equivalent to 
\[
\widetilde{x}_{dR} \in H^0(\widetilde{S} ,  F^0 \bm{V}_{dR}|_{\widetilde{S}} ),
\]
which is therefore equivalent to the endomorphism
\[
\widetilde{x}_{dR} \in \End(  \bm{H}_{dR}|_{\widetilde{S}}  ) 
\]
respecting the Hodge filtration.  Using the deformation theory of abelian schemes described in  \cite[Chapter 2]{Lan}, this is equivalent to 
\[
x \in V_\mu(\mathcal{A}_S) \subset \End(\mathcal{A}_S)\otimes  \Z_{(p)}
\]
admitting an extension to  
\[
\widetilde{x} \in \End(\mathcal{A}_{\widetilde{S}} )\otimes  \Z_{(p)}.
\] 
Using Proposition \ref{prop:special rigidity} it is easy to see that when such an  extension exists it must lie in  $V_\mu(\mathcal{A}_S)$.  This proves the claim when $V_\Z$ is self-dual at $p$.

We now explain how to construct the section (\ref{obstruction}) in general.
 Fix an isometric embedding $V_\Z \subset V_\Z^{\beef}$ as in \S \ref{ss:embiggen}, and assume that $V_\Z^\beef$ is  self-dual at $p$, so that  we have morphisms of integral models
\[
\mathcal{S}_K(G,\DD) \to \mathcal{S}_{K^\beef}(G^\beef,\DD^\beef). 
\]
In the notation of (\ref{sneaky coset}), the special quasi-endomorphism $x\in V_\mu( \mathcal{A}_S)$ determines  another special quasi-endomorphism 
\[
x^\beef = x + \bar{\mu}_p \in V( \mathcal{A}_S)_{ \Z_{(p)} },
\]
and $x$ extends  to $V_\mu( \mathcal{A}_{\widetilde{S} } )$ and only if $x^\beef$ extends to $V( \mathcal{A}_{\widetilde{S} } )_{ \Z_{(p)} }$.   

The self-dual-at-$p$ case considered above determines an obstruction to deforming $x^\beef$, denoted
\[
\mathrm{obst}_{x^\beef} \in H^0 \big( \widetilde{S} , \bm{\omega}^\beef|^{-1}_{\widetilde{S} } \big).
\]
Recalling that $\bm{\omega}|_{\widetilde{S}} = \bm{\omega}^\beef|_{\widetilde{S}}$ by definition, we now define (\ref{obstruction}) by
\[
\mathrm{obst}_x=\mathrm{obst}_{x^\beef}.
\]
It is easy to check that this does not depend on the auxiliary choice of $V_\Z^\beef$ used in its construction, and has the desired properties.
\end{proof}

\begin{proposition}\label{prop:sufficiently small}
Every geometric point of $\mathcal{S}_K(G,\DD)$ admits an \'etale neighborhood $U \to \mathcal{S}_K(G,\DD)$ such that 
\[
\mathcal{Z}(m,\mu)_{/U} \to U
\] 
restricts to a closed immersion on every connected component of its domain.
Each such closed immersion is an effective Cartier divisor on $U$.
\end{proposition}

\begin{proof}
The first claim is a formal consequence of Proposition \ref{prop:divisor represent}, and holds for any finite, unramified, relatively representable morphism of Deligne-Mumford stacks.
Indeed, if $\co_s$ denotes the \'etale local ring at a geometric point $s\to \mathcal{S}_K(G,\DD)$, then finiteness and relative representability imply that 
\[
 \Spec(\co_s) \times_{ \mathcal{S}_K(G,\DD)  } \mathcal{Z}(m,\mu) \iso \bigsqcup_t \Spec(\co_t)
\]
where $t$ runs over  the geometric points $t \to \mathcal{Z}(m,\mu)$ above $s$, and unramifiedness implies that 
each morphism  $\co_s \to \co_t$  is surjective.

Fix one such $t$,  set $J = \mathrm{ker}( \co_s \to \co_t)$, and consider the nilpotent thickening
\[
\Spec(\co_t) = \Spec(\co_s /  J ) \hookrightarrow \Spec(\co_s / J^2).  
\] 
In particular, we have a diagram
\[
\xymatrix{
{   \Spec(\co_t)   } \ar[r]  \ar[d]   &   { \mathcal{Z}(m,\mu) }  \ar[d] \\
  { \Spec(\co_s / J^2) }   \ar[r]  &  {   \mathcal{S}_K(G,\DD) }
}
\]
exactly as in (\ref{deformation diagram}).
The pullback of the Kuga-Satake abelian scheme to $\Spec(\co_t)$ acquires a tautological special quasi-endomorphism $x$.  
The obstruction to deforming $x$ is, after choosing a trivialization of $\bm{\omega}|_{ \Spec(\co_t)}$, an element 
\[
\mathrm{obst}_x \in   \co_s / J^2
\]
that generates $J/J^2$ as an $\co_s$-module.  
Nakayama's lemma now implies that  $J \subset \co_s$ is a principal ideal, and so 
$\Spec(\co_t) \hookrightarrow \Spec(\co_s)$ is an effective Cartier divisor.

This proves the claim on the level of \'etale local rings, and the extension to \'etale neighborhoods is routine.
\end{proof}

Proposition \ref{prop:sufficiently small} is what justifies referring to the morphisms (\ref{special divisors}) as divisors, even though they are not closed immersions. 
 In the notation of that proposition, every connected component of the source of  
\[
\mathcal{Z}(m,\mu)_{/U} \to U
\] 
determines a  Cartier divisor on $U$.  Summing over all such components and then glueing as $U$ varies over an \'etale cover defines an effective Cartier divisor on $\mathcal{S}_K(G,\DD)$ in the usual sense.   When no confusion can arise (and perhaps even when it can), we denote this Cartier divisor again by $\mathcal{Z}(m, \mu)$.

We end this subsection by explaining the precise relation between the analytic obstruction (\ref{analytic obstruction}) and the algebraic obstruction  (\ref{obstruction}).

Fix a $g\in G(\A_f)$.  If we pull back the diagram (\ref{deformation diagram}) via the morphism
\begin{equation}\label{deformation uniform}
\DD \map{ z\mapsto (z,g) } \Sh_K(G,\DD) (\C) 
\end{equation}
 we obtain (at least if $\widetilde{S}$ is of finite type over $\Q$) a diagram 
\[
\xymatrix{
  {   \mathcal{S}  = S^{an}  \times_{ \Sh_K(G,\DD)^{an}  } \DD  } \ar[r]  \ar[d]   &   { \bigsqcup \DD(x)    }  \ar[d] \\
   { \widetilde{\mathcal{S}}  = \widetilde{S}^{an}  \times_{ \Sh_K(G,\DD)^{an}  } \DD}   \ar[r]  &  {   \DD, }
}
\]
of complex analytic spaces, in which  the disjoint union is as in (\ref{complex divisor}), and the vertical arrow on the left is defined by a coherent  sheaf of ideals whose square is $0$.  In particular $\mathcal{S} \to \widetilde{\mathcal{S}}$ induces an isomorphism of underlying topological spaces.

  For a fixed $x$, let $\mathcal{S}(x)\subset \mathcal{S}$ be the union of those connected components of  whose image under the top horizontal arrow lies in the factor $\DD(x)$. This determines a union of connected components
   $\widetilde{\mathcal{S}}(x) \subset\widetilde{\mathcal{S}}$, and gives us a  diagram of complex analytic spaces
\[
\xymatrix{
{  S^{an}   } \ar[d]   & {   \mathcal{S}(x) } \ar[r]  \ar[d]   \ar[l]&   {  \DD(x)    }  \ar[d] \\
{   \widetilde{S}^{an}  }  &  { \widetilde{\mathcal{S}}(x) }   \ar[r]  \ar[l] &  {   \DD. }
}
\]

\begin{proposition}\label{prop:two obstructions}
There is an equality of sections
\[
\mathrm{obst}_x^{an}|_{ \widetilde{\mathcal{S}}(x) }  =\mathrm{obst}_x|_{ \widetilde{\mathcal{S}}(x) },
\]
 where the left hand side is the pullback of (\ref{analytic obstruction})  via 
 $
\widetilde{ \mathcal{S} } (x) \to \DD
 $
  and the right hand side is the pullback of  (\ref{obstruction}) via 
 $\widetilde{\mathcal{S}}(x)\to  \widetilde{S}^{an}$.
 \end{proposition}

 \begin{proof}
The pullback of $\bm{V}_{dR}$ via (\ref{deformation uniform}) is canonically identified with  the constant vector bundle   
 \[
 \bm{V}_{dR}|_{\DD} = V\otimes \co_{\DD},
 \]
  and under this identification the pullback of the connection $\nabla$ is the  induced by the usual $d : \co_{\DD} \to \Omega^1_{ \DD/\C}$.

By the discussion leading to (\ref{complex divisor}), the pullback of (\ref{deform x}) via $\mathcal{S}(x) \to S^{an}$ is identified with the constant section 
 \[
 x\otimes 1 \in H^0 \big( \mathcal{S}(x) ,  \bm{V}_{dR}|_{ \mathcal{S}(x) } \big),
 \]
and  the pullback of (\ref{extend x}) via $\tilde{\mathcal{S}}(x) \to\widetilde{S}^{an}$ is its unique parallel extension 
 \[
 x\otimes 1 \in H^0 \big( \widetilde{\mathcal{S}}(x) ,  \bm{V}_{dR}|_{ \widetilde{\mathcal{S}}(x) } \big).
 \] 
 Thus $\mathrm{obst}_x|_{ \widetilde{\mathcal{S}}(x) }$ is the image of $x\otimes 1$ under
 \[
 V\otimes \co_{\widetilde{\mathcal{S}}(x) } \iso   \bm{V}_{dR}|_{\widetilde{\mathcal{S}}(x)}  \to
    \big( \bm{V}_{dR}  / F^0 \bm{V}_{dR} \big)|_{\widetilde{\mathcal{S}}(x)} 
   \iso \bm{\omega}_{\widetilde{\mathcal{S}}(x)}^{-1} .
 \]

On the other hand,  the analytically  defined obstruction (\ref{analytic obstruction}) is, essentially by construction,  the image of the constant section $x\otimes 1$ under 
   \[
V \otimes \co_\DD \iso   \bm{V}_{dR}|_{\DD}  \to  \big( \bm{V}_{dR}/ F^0 \bm{V}_{dR} \big) |_{\DD} 
\iso \bm{\omega}_{\DD}^{-1} .
 \]
The stated equality of sections over $\widetilde{\mathcal{S}}(x)$ follows immediately.
 \end{proof}


\subsection{The pullback formula for special divisors}
\label{ss:integral pullbacks}


Suppose we are in the general situation of \S \ref{ss:embiggen}  (in particular, we impose no assumption of self-duality on $V_\Z^\beef$),  
so that we have a morphism  (\ref{embiggen morphism}) of Shimura varieties  
\[
\Sh_K(G,\DD) \to\Sh_{K^\beef}(G^\beef,\DD^\beef).
\]

The larger Shimura variety $\Sh_{K^\beef}(G^\beef , \DD^\beef)$ has its own integral model 
\[
\mathcal{S}_{K^\beef}(G^\beef , \DD^\beef) \to \Spec( \Z_{(p)}),
\]
 obtained by repeating the construction of  \S \ref{ss:gspin integral} with $(G,\DD)$ replaced by $(G^\beef , \DD^\beef)$. 
  That is, choose an isometric embedding  $V^\beef \subset V^{\beef \beef}$ into a larger quadratic space that admits an almost self-dual lattice at $p$, and define $\mathcal{S}_{K^\beef}(G^\beef , \DD^\beef)$ as a normalization. 
       Of course  $\mathcal{S}_{K^\beef}(G^\beef , \DD^\beef)$  has its own line bundle $\bm{\omega}^\beef$, its own Kuga-Satake abelian scheme, and its own collection of special divisors $\mathcal{Z}^\beef (m,\mu)$.

 \begin{proposition}\label{prop:integral bundle pullback}
 The above morphism of canonical models extends uniquely to a finite morphism 
\begin{equation}\label{integral embiggen morphism}
  \mathcal{S}_K(G,\DD) \to \mathcal{S}_{K^\beef}(G^\beef , \DD^\beef)
\end{equation}
 of integral models.
The line bundles of weight one modular forms on the source and target of (\ref{integral embiggen morphism}) are related by a canonical isomorphism
\begin{equation}\label{integral line pullback}
 \bm{\omega}^\beef|_{ \mathcal{S}_K(G,\DD)} \iso  \bm{\omega}  .
\end{equation}
\end{proposition}

\begin{proof}
The existence and uniqueness of  (\ref{integral embiggen morphism}) is proved in \cite[Proposition 2.5.1]{AGHMP-1}.

If $V_\Z^\beef$ is almost self-dual at $p$ then (\ref{integral line pullback}) is just a restatement of the definition of $\bm{\omega}$.
For the general case, one embeds $V^\beef$ into a quadratic space $V^{\beef \beef}$ admitting a lattice that is almost self-dual  at $p$.  This allows one to identify both sides of (\ref{integral line pullback}) with the pullback of  $\bm{\omega}^{\beef\beef}$ for some morphisms
\[
 \mathcal{S}_K(G,\DD) \to \mathcal{S}_{K^\beef}(G^\beef,\DD^\beef)  \to \mathcal{S}_{K^{\beef \beef}}(G^{\beef \beef},\DD^{\beef \beef}) 
 \]
 into the larger Shimura variety determined by $V^{\beef\beef}$.
\end{proof}

Define a quadratic space
\[
\Lambda = \{ x \in L_\Z^\beef : x\perp L \}
\]
over $\Z$ of signature $(n^\beef - n ,0)$.
There are natural inclusions
\[
V_\Z \oplus \Lambda \subset V_\Z^\beef \subset   (V_\Z^\beef)^\vee \subset V_\Z^\vee \oplus \Lambda^\vee
\] 
all of finite index, from which it follows that the orthogonal decomposition
\[
V^\beef = V \oplus \Lambda_\Q
\]
identifies 
\[
\mu + V_\Z^\beef = \bigsqcup_{ \mu_1+\mu_2 \in \mu  } ( \mu_1 + V_\Z) \times (\mu_2 +\Lambda).
\]
Here the disjoint union over $\mu_1+\mu_2\in \mu$ is understood to mean   the union   over all pairs 
\[
( \mu _1, \mu_2) \in (V_\Z^\vee / V_\Z) \oplus (\Lambda^\vee / \Lambda) 
\]
satisfying $\mu_1  + \mu_2 \in ( \mu+ V_\Z^\beef) / (  V_\Z \oplus \Lambda  )$.

The following lemma  gives a corresponding decomposition  of special quasi-endomorphisms.
For the proof see   \cite[Proposition 2.6.4]{AGHMP-1}.

\begin{proposition}\label{prop:special decomp cosets}
For any scheme $S$ and any morphism $S \to \mathcal{S}_K(G,\DD)$ there is a canonical isometry
\[
V(\mathcal{A}_S^\beef ) \iso   V(\mathcal{A}_S) \oplus \Lambda_\Q   ,
\]
which restricts to a bijection
\begin{equation}\label{special decomp cosets}
V_\mu(\mathcal{A}_S^\beef) \iso \bigsqcup_{   \mu_1 + \mu_2 \in \mu  } V_{\mu_1}(\mathcal{A}_S) \times (\mu_2 + \Lambda).
\end{equation}
\end{proposition}

The relation between special divisors on the source and target of (\ref{integral embiggen morphism})
 is most easily expressed in terms of the line bundles associated to the divisors, rather than the divisors themselves.
By abuse of notation, we now use $\mathcal{Z}(m,\mu)$ to denote also the line bundle on $\mathcal{S}_K(G,\DD)$ determined by the Cartier divisor of the same name, extend the definition to  $m\le 0$ by
\[
\mathcal{Z}(m,\mu) = \begin{cases}
\bm{\omega}^{-1} & \mbox{if } (m,\mu) = (0,0) \\
\co_{\mathcal{S}_K(G,\DD)} & \mbox{otherwise,}
\end{cases}
\]
and use similar conventions for $\mathcal{S}_{K^\beef}(G^\beef,\DD^\beef)$.

\begin{proposition}\label{prop:pullback}
  For any rational number $m\ge 0$ and any $\mu \in (V^\beef_\Z )^\vee / V^\beef_\Z$, there is a  canonical isomorphism of line bundles  
\[
 \mathcal{Z}^\beef (m,\mu)|_{\mathcal{S}_K(G,\DD)} \iso  \bigotimes_{ \substack{ m_1 + m_2 = m  \\  \mu_1 + \mu_2\in   \mu   } }  \mathcal{Z}(m_1,\mu_1)^{ \otimes r_\Lambda(m_2,\mu_2) }
\]
on $\mathcal{S}_K(G,\DD)$.  Here we have set
\[
R_\Lambda(m,\mu) =  \{  \lambda \in \mu +\Lambda : Q (\lambda) = m \} 
\]
and $r_\Lambda(m,\mu) = \# R_\Lambda(m,\mu)$.
\end{proposition}

\begin{proof}
If $m<0$, or if $m=0$ and $\mu\neq 0$,   the tensor product on the right is empty,  and both sides of the desired isomorphism are canonically trivial.
If $(m,\mu) = (0,0)$ the claim is just a restatement of Proposition \ref{prop:integral bundle pullback}.
  Thus we may assume that $m>0$.

The decomposition (\ref{special decomp cosets}) induces an isomorphism 
\begin{align}\label{special stack decomp}
 \mathcal{Z}^\beef (m,\mu)_{/\mathcal{S}_K(G,\DD)}  
 & \iso   \bigsqcup\limits_{ \substack{ m_1 + m_2 = m  \\  \mu_1 + \mu_2 \in  \mu \\m_1 >0 \\ \lambda \in R_\Lambda(m_2,\mu_2) } }  \mathcal{Z} (m_1,\mu_1) 
 \sqcup  \bigsqcup\limits_{  \substack{  \mu_2 \in  \mu \\ \lambda \in R_\Lambda(m,\mu_2) } } 
  \mathcal{S}_K (G,\DD)
\end{align}
of  $\mathcal{S}_K(G,\DD)$-stacks,
where the condition $\mu_2\in \mu$ means that 
\[
0+\mu_2   \in  ( V_\Z^\vee / V_\Z)  \oplus ( \Lambda^\vee/ \Lambda) 
\]
 lies in the subset  $( \mu+V_\Z^\beef) /  ( V_\Z \oplus \Lambda)$.
Explicitly, given any connected scheme $S$ and a morphism
\[
S \to \mathcal{S}_K(G,\DD),
\]
a lift of the morphism to the  first disjoint union on the right hand side of (\ref{special stack decomp}) determines a pair 
\[
(x,\lambda)  \in V_{\mu_1}(\mathcal{A}_S) \times( \mu_2+\Lambda )
\]  
satisfying $m_1=Q(x)$ and $m_2=Q(\lambda)$.
Using (\ref{special decomp cosets}) we obtain a special quasi-endomorphism
\[
x^\beef = x+\lambda \in V_{\mu}(\mathcal{A}^\beef_S).
\]
 Similarly, a lift to the second disjoint union determines a vector $\lambda\in \mu_2+\Lambda$ satisfying $m=Q(\lambda)$, which determines a special quasi-endomorphism
\begin{equation}\label{degenerate special}
x^\beef =0+ \lambda \in V_{\mu}(\mathcal{A}^\beef_S).
\end{equation}
 In either case  $Q(x^\beef)=m$, and so our lift determines an $S$-point of the left hand side of (\ref{special stack decomp}).

If $\Lambda^\vee$ does not represent $m$, then $R_\Lambda(m,\mu_2)=\emptyset$ for all choices of $\mu_2$, and the desired isomorphism of line bundles 
\begin{align*}
 \mathcal{Z}^\beef (m,\mu)|_{\mathcal{S}_K(G,\DD)}
& \iso  \bigotimes_{ \substack{ m_1 + m_2 = m  \\  \mu_1 + \mu_2 \in \mu \\ m_1>0  } } 
 \mathcal{Z}(m_1,\mu_1)^{ \otimes r_\Lambda(m_2,\mu_2) } \\
& \iso  \bigotimes_{ \substack{ m_1 + m_2 = m  \\  \mu_1 + \mu_2 \in  \mu  } } 
\mathcal{Z}(m_1,\mu_1)^{ \otimes r_\Lambda(m_2,\mu_2) },
\end{align*}
on $\mathcal{S}_K(G,\DD)$ follows immediately from (\ref{special stack decomp}). 
In general, the decomposition (\ref{special stack decomp}) shows that the support of  $\mathcal{Z}^\beef(m,\mu)$ contains the image of (\ref{integral embiggen morphism}) as soon as there is some  $\mu_2\in \mu$ for which   $R_\Lambda(m,\mu_2)$ is nonempty.  Thus we must compute an improper intersection.

Fix a geometric point $s\to \mathcal{S}_K(G,\DD)$ and, as in Proposition \ref{prop:sufficiently small},  an \'etale neighborhood 
\[
U^\beef \to \mathcal{S}_{K^\beef}(G^\beef,\DD^\beef)
\]
of $s$ small enough that the morphism
\[
 \mathcal{Z}^\beef(m,\mu)_{/U^\beef} \to U^\beef
\]
restricts to a closed immersion on every connected component of the domain.
By shrinking  $U^\beef$  we may assume that these connected components are in bijection with the set of lifts
\[
\xymatrix{
& \mathcal{Z}^\beef(m,\mu)_{/U^\beef} \ar[d] \\
s \ar@{.>}[ur] \ar[r]  & U^\beef.
}
\]

Having so chosen $U^\beef$, we then choose a connected \'etale neighborhood 
\[
U \to  \mathcal{S}_ K (G,\DD) 
\]
of $s$ small enough that there exists a lift
\[
\xymatrix{
{ U } \ar[d]\ar@{.>}[r]  &  { U^\beef } \ar[d]  \\
{  \mathcal{S}_ K (G,\DD)  }  \ar[r]  &  {   \mathcal{S}_{K^\beef}(G^\beef,\DD^\beef)  ,}  
}
\]
and so that in the cartesian diagram
\[
\xymatrix{
{  \mathcal{Z}^\beef(m,\mu)_{/U}  } \ar[r] \ar[d]  &   { \mathcal{Z}^\beef(m,\mu)_{/U^\beef} } \ar[d] \\
 {U }  \ar[r]  &  { U^\beef }
 }
\]
each  of the vertical arrows restricts to  a closed immersion on every connected component of its source, and  the top horizontal arrow induces a bijection on connected components.

The decomposition (\ref{special stack decomp}) induces a decomposition of $U$-schemes
\begin{align*}
\mathcal{Z}^\beef (m,\mu)_{/U}    \iso   \bigsqcup\limits_{ \substack{ m_1 + m_2 = m  \\  \mu_1 + \mu_2 \in  \mu \\m_1 >0 \\ \lambda \in R_\Lambda(m_2,\mu_2)} }  \mathcal{Z}(m_1,\mu_1)_{/U} 
 \sqcup  \bigsqcup\limits_{    \substack { \mu_2 \in  \mu \\    \lambda\in R_\Lambda(m,\mu_2)   }}  U . \nonumber
\end{align*}
The first disjoint union defines a Cartier divisor on $U$.
In the second disjoint union, the copy of $U$ indexed by $\lambda\in R_\Lambda(m,\mu_2)$ is the image of the open and closed immersion
\[
f_\lambda : U \to \mathcal{Z}^\beef (m,\mu)_{/U}  
\]
obtained by endowing the Kuga-Satake abelian scheme $\mathcal{A}^\beef_U$ with the special quasi-endomorphism
$0 + \lambda \in   V_\mu(\mathcal{A}_U^\beef)$ of (\ref{degenerate special}).

There is a corresponding  canonical decomposition of $U^\beef$-schemes
\begin{equation}\label{special stack components}
\mathcal{Z}^\beef(m,\mu)_{/ U^\beef} =  
\mathcal{Z}^\beef_{\mathrm{prop}} \sqcup  \bigsqcup_{  \substack{ \mu_2 \in \mu  \\  \lambda \in R_\Lambda(m,\mu_2)  }  }\mathcal{Z}^\beef_\lambda,
\end{equation}
in which  
\begin{equation}\label{lambda component}
\mathcal{Z}^\beef_\lambda \subset \mathcal{Z}^\beef(m,\mu)_{/ U^\beef}
\end{equation}
 is the connected component containing  the image of 
\[
U \map{f_\lambda} \mathcal{Z}^\beef(m,\mu)_{/U} \to \mathcal{Z}^\beef(m,\mu)_{/U^\beef}
\]
and
 \[
 \mathcal{Z}^\beef_{\mathrm{prop}}\subset \mathcal{Z}^\beef(m,\mu)_{/ U^\beef}
 \] 
 is the union of all  connected components not of this form.

 It is clear from the definitions that  the image of $U \to U^\beef$  intersects the Cartier divisor   $\mathcal{Z}^\beef_{\mathrm{prop}} \to U^\beef$  properly, and in fact
 \begin{equation}\label{proper intersection}
 \mathcal{Z}^\beef_{\mathrm{prop} / U } \iso
  \bigsqcup\limits_{ \substack{ m_1 + m_2 = m  \\  \mu_1 + \mu_2 \in  \mu \\m_1 >0 \\ \lambda \in R_\Lambda(m_2,\mu_2)} }  \mathcal{Z}(m_1,\mu_1)_{/U}.
\end{equation}
On the other hand,  the image of $U \to U^\beef$  is completely contained within the support of every  $\mathcal{Z}^\beef_\lambda \to U^\beef$.

By mild abuse of notation, we denote again by $\mathcal{Z}_\mathrm{prop}^\beef$ and $\mathcal{Z}_\lambda^\beef$ the line bundles on $U^\beef$ determined by the Cartier divisors of the same name.

\begin{lemma}\label{lem:full adjunction}
There is a canonical isomorphisms of line bundles 
\[
 \mathcal{Z}^\beef_{\mathrm{prop}}| _U \iso
\bigotimes_{ \substack{ m_1 + m_2 = m  \\   \mu_1 + \mu_2 \in  \mu\\ m_1 >0   } }  
\mathcal{Z}(m_1,\mu_1)|_U^{ \otimes r_\Lambda(m_2,\mu_2) },
\]   
and a canonical isomorphism
\begin{equation}\label{bundle adjunction}
\mathcal{Z}^\beef_\lambda |_U\iso \bm{\omega}|_U^{-1}.
\end{equation}
\end{lemma}

\begin{proof}
The first isomorphism is clear from the isomorphism of $U$-schemes (\ref{proper intersection}).
The second isomorphism is more subtle, and is based on similar calculations  in the context of unitary Shimura varieties; see especially  \cite[Theorem 7.10]{BHY} and \cite{Ho3}.

Our \'etale neighborhood $U^\beef \to \mathcal{S}_{K^\beef}(G^\beef,\DD^\beef)$ was chosen in such a way that 
$\mathcal{Z}_\lambda^\beef \to U^\beef$ is a closed immersion defined by a locally principal sheaf of ideals 
$J_\lambda \subset \co_{U^\beef}$. 
The closed subscheme 
\[
\widetilde{\mathcal{Z}}_\lambda^\beef  \subset  U^\beef
\]
defined by $J_\lambda^2$ is called the \emph{first order tube} around $Z_\lambda^\beef$.
We now have morphisms
\[
U  \map{f_\lambda}  \mathcal{Z}_\lambda^\beef \hookrightarrow \widetilde{\mathcal{Z}}_\lambda^\beef  \hookrightarrow U^\beef \to \mathcal{S}_{K^\beef}(G^\beef,\DD^\beef).
\]

Tautologically, $J_\lambda^{-1}$ is the line bundle on $U^\beef$ determined by the Cartier divisor $\mathcal{Z}^\beef_\lambda$.   Denote by $\sigma_\lambda$ the constant function $1$, viewed as a section of 
$ \co_{U^\beef}  \subset J_\lambda^{-1}$, so that $\mathrm{div}(\sigma_\lambda) = \mathcal{Z}^\beef_\lambda$.

On the other hand, after restriction to the connected component (\ref{lambda component}) the Kuga-Satake abelian scheme $\mathcal{A}^\beef$ acquires a tautological 
\[
x^\beef \in V_\mu(\mathcal{A}^\beef_{\mathcal{Z}_\lambda^\beef}).
\]
The discussion of  \S \ref{ss:deformation} then provides us with a canonical section
\[
\mathrm{obst}_{x^\beef} \in H^0 \big( \widetilde{\mathcal{Z}}_\lambda^\beef , \bm{\omega}|^{-1}_{   \widetilde{\mathcal{Z}}_\lambda^\beef }  \big) 
\]
whose zero locus is the closed subscheme $\mathcal{Z}_\lambda^\beef$.

The idea is roughly that the equality of divisors 
\[
\mathrm{div}(\sigma_\lambda) = \mathrm{div}( \mathrm{obst}_{x^\beef} ) 
\]
 should imply that there is  a  unique  isomorphism of line bundles (\ref{bundle adjunction}) over the first order tube  sending $\sigma_\lambda \mapsto \mathrm{obst}_{x^\beef}$,   which we would then pull back via $f_\lambda$.
This is a bit too strong.  Instead, we argue  that such an  isomorphism exists Zariski locally on the first order tube, and that any two such local isomorphisms restrict to the same  isomorphism over $U$.

Indeed, working Zariski locally, we can assume that 
\[
U = \Spec(R), \quad U^\beef = \Spec ( R^\beef)
\]
for integral domains $R$ and $R^\beef$, and 
\[
\mathcal{Z}_\lambda^\beef = \Spec ( R^\beef/J)  , \quad  \widetilde{\mathcal{Z}}_\lambda^\beef  = \Spec ( R^\beef/J^2).
\]
The morphisms $U \to \mathcal{Z}_\lambda^\beef \to  U^\beef$ then correspond to  homomorphisms 
\[
R^\beef\to R^\beef/J \to R.
\] 
Let  $\mathfrak{p} \subset R^\beef$ be the kernel of this composition, so that $J\subset \mathfrak{p}$.
Note that $p \not\in \mathfrak{p}$, as the flatness of $\mathcal{S}_K(G,\DD)$ over $\Z_{(p)}$ implies that $R$ has no $p$-torsion.

Assume that we have chosen trivializations of the line bundles $\bm{\omega}|_{U^\beef}$ and $Z_\lambda^\beef$ on $U^\beef$, so that  our sections  $\mathrm{obst}_\lambda$ and $\sigma_\lambda$ are identified with elements  
\[
a \in R^\beef/J^2   \qquad \mbox{and} \qquad b\in R^\beef,
\]  
respectively.  Each of these elements  generates the  ideal $J / J^2 \subset R^\beef / J^2$.

\begin{lemma}
There exists  $u\in R^\beef/J^2$ such that $ua=b$.   The image of any such $u$ in $R^\beef/\mathfrak{p}\subset R$ is a unit.
If also $u'a=b$, then  $u=u'$  in $R^\beef/\mathfrak{p}\subset R$.
\end{lemma}

\begin{proof}
Suppose we are given any $x\in R^\beef / J^2$ with  $bx=0$.  We claim that 
\[
x\in \mathfrak{p}/J^2 .
\]
If not, then any lift $x\in R^\beef$ becomes a unit in the localization $R^\beef_\mathfrak{p}$.  
As $bx \in J^2$, we obtain 
\begin{equation}\label{little b}
b \in \mathfrak{p}^2 R^\beef_\mathfrak{p}.
\end{equation}

We have noted above that $p\not\in \mathfrak{p}$, and so $R_\mathfrak{p}^\beef$ is a $\Q$-algebra.
The source and target of 
\[
\mathcal{Z}^\beef( m,\mu) \to \mathcal{S}_{K^\beef}(G^\beef,\DD^\beef)
\]
have smooth generic fibers,   and so   
$
R^\beef_\mathfrak{p} \to R^\beef_\mathfrak{p} / b R^\beef_\mathfrak{p}
$
is a morphism of regular local rings.   By (\ref{little b}),   this morphism induces an isomorphism on tangent spaces, and so is itself an isomorphism.  Thus $b=0$ in  $R^\beef_\mathfrak{p}$, and hence also in $R^\beef$.
This contradicts the fact that  $b$ generates the ideal $J$.

As $a$ and $b$ generate both generate $J / J^2$,  there exist $u,v\in R^\beef/J^2$ such that $ua=b$ and $vb=a$. 
Obviously $b\cdot (1-uv)=0$, and taking $x= 1-uv$  the paragraph above  implies $1-uv \in \mathfrak{p} / J^2$.  
Thus the image of $u$ in $R^\beef / \mathfrak{p}$ is a unit with inverse $v$.
If also $u' a =b$, the same argument shows that the image of $u'$ in $R^\beef /\mathfrak{p}$ is a unit with inverse $v$, and hence $u=u'$ in $R^\beef/\mathfrak{p}$.
\end{proof}

The discussion above provides us with a canonical isomorphism  
\[
\mathcal{Z}^\beef_\lambda |_U\iso \bm{\omega}|_U^{-1}
\] 
Zariski locally on $U$, and  glueing over an open cover completes the proof of Lemma \ref{lem:full adjunction}.
\end{proof}

We now complete the proof of Proposition \ref{prop:pullback}.
If we interpret the isomorphism of $U^\beef$-schemes (\ref{special stack components}) as an isomorphism
\[
\mathcal{Z}^\beef(m,\mu)|_{U^\beef} \iso 
\mathcal{Z}^\beef_{\mathrm{prop}} \otimes  \bigotimes_{  \substack{ \mu_2 \in \mu  \\  \lambda \in R_\Lambda(m,\mu_2)  }  } \mathcal{Z}^\beef_\lambda,
\]
of line bundles on $U^\beef$, pull back via $U\to U^\beef$, and use Lemma \ref{lem:full adjunction}, we obtain  canonical isomorphisms
\begin{align*}
\mathcal{Z}^\beef (m,\mu)|_U 
 &  \iso  \Bigg(  \bigotimes_{ \substack{ m_1 + m_2 = m  \\  \mu_1 + \mu_2\in  \mu \\ m_1>0 } } 
  \mathcal{Z}(m_1,\mu_1)|_U^{ \otimes r_\Lambda(m_2,\mu_2) } \Bigg)
  \otimes \Bigg(
   \bigotimes_{ \substack{  \mu_2 \in  \mu  } }  \bm{\omega}|_U^{-r_\Lambda(m,\mu_2)} 
  \Bigg)
   \\
&  \iso  \bigotimes_{ \substack{ m_1 + m_2 = m  \\  \mu_1 + \mu_2= \mu  } }  \mathcal{Z}(m_1,\mu_1)|_U^{ \otimes r_\Lambda(m_2,\mu_2) }
\end{align*}
of line bundles over the \'etale neighborhood  $U$ of $s\to \mathcal{S}_K(G,\DD)$.  
Now let   $U$ vary over an \'etale cover and apply  descent.
\end{proof}


\section{Normality and flatness}
\label{s:integrality II}


Keep $V_\Z \subset V$ and $K \subset G(\A_f)$ as in \S \ref{s:integrality I}, and once again fix a prime $p$ at which $V_\Z$ is maximal. 
 After some technical preliminaries in \S \ref{ss:local properties}, we prove in \S \ref{ss:flatness}
 that the special fiber of the integral model 
 \[
 \mathcal{S}_K(G,\DD) \to \Spec(\Z_{(p)})
 \]
 is geometrically normal if $n \ge 6$, and that the special divisors are flat if $n \ge 4$.  
 When $p\neq 2$ these results already appear\footnote{with the sharper bounds  $n \ge 5$ and $n \ge 3$, respectively} in \cite{AGHMP-2}.  Here we use similar ideas, but employ the methods of Ogus~\cite{Ogus1979} to control the dimension of the supersingular locus, as these apply even when $p=2$.


\subsection{Local properties of special cycles}
\label{ss:local properties}


Suppose in this subsection that $V_\Z$ is self-dual at $p$.
As in the discussion of \S \ref{ss:gspin integral}, the smooth integral model 
$\mathcal{S}_{K}(G,\DD)$ comes with  filtered vector bundles
$
0 \subset F^0\bm{H}_{{dR}}  \subset \bm{H}_{{dR}}
$
and
\[
0 \subset F^1 \bm{V}_{dR}\subset F^0\bm{V}_{dR} \subset \bm{V}_{dR},
\]
along with an injection 
$
\bm{V}_{dR} \to \underline{\End}(\bm{H}_{{dR}})
$
onto a local direct summand.
Composition in $\underline{\End}(\bm{H}_{{dR}})$ endows $\bm{V}_{{dR}}$ with a  quadratic form
\[
\bm{Q}:\bm{V}_{{dR}}\to \co_{\mathcal{S}_{K}(G,\DD) },
\]
under which $F^1\bm{V}_{{dR}}$ is an isotropic line with orthogonal subsheaf $F^0\bm{V}_{{dR}}$.

Recall from (\ref{self-dual mu}) the $\Z$-module (\ref{self-dual mu}) of special quasi-endomoprhisms
\[
V_0 ( \mathcal{A}_S)  \subset \End(\mathcal{A}_S) \otimes \Z_{(p)}
\]
determined by the trivial  coset $0 \in V_\Z^\vee / V_\Z$.
Any $x \in V_0(\mathcal{A}_S)$
has a \emph{de Rham realization} $\bm{x}_{dR}$, which is a global section of the subsheaf
\[
\bm{V}_{{dR},S} \subset \underline{\End}(\bm{H}_{{dR},S}).
\]
In particular, de Rham realization defines a morphism of $\co_S$-modules
\[
V_0(\mathcal{A}_S) \otimes \co_S \to \bm{V}_{{dR},S}.
\]
compatible with the quadratic forms on source and target.
In fact, as is clear from the proof of Proposition \ref{prop:obstruction}, the image is contained in $F^0 \bm{V}_{{dR},S}$.

Fix a  positive definite quadratic space $\Lambda$ over $\Z$, and consider the stack
\begin{equation}\label{general cycle}
\mathcal{Z}(\Lambda) \to \mathcal{S}_{K}(G,\DD)
\end{equation}
with functor of points 
\[
\mathcal{Z}(\Lambda)(S) = \{\mbox{isometric embeddings }\iota:\Lambda\to V_0(\mathcal{A}_S) \}
\]
for any morphism $S\to \mathcal{S}_{K}(G,\DD)$.  
As observed in~\cite[\S 4.4]{AGHMP-2} (see also Lemma \ref{lem:zlamb smooth} below), this is a Deligne-Mumford stack over $\Z_{(p)}$ whose generic fiber is smooth  of dimension  $n-\mathrm{rank}(\Lambda)$.
Moreover, the morphism (\ref{general cycle})   is  finite and unramified.

We now briefly recall the deformation theory of these stacks. 
As in the proof of Proposition \ref{prop:obstruction}, we have a canonical flat connection
\[
\nabla:\bm{V}_{dR}\to \bm{V}_{dR}\otimes  \Omega^1_{\mathcal{S}_K(G,\DD)/\Z_{(p)}}.
\]
This connection satsfies Griffiths's transversality with respect to the Hodge filtration, and the Kodaira-Spencer map associated with it induces an isomorphism
\[
F^1\bm{V}_{dR}\otimes\bigl(\Omega^1_{\mathcal{S}_K(G,\DD)/\Z_{(p)}}\bigr)^\vee
\iso F^0 \bm{V}_{dR} / F^1 \bm{V}_{dR}.
\]
Dualizing, and using the bilinear pairing on $\bm{V}_{dR}$, we obtain an isomorphism
\[
F^0\bm{V}_{dR} / F^1 \bm{V}_{dR} \iso 
( \bm{V}_{dR} / F^0\bm{V}_{dR}  ) \otimes \Omega^1_{\mathcal{S}_K(G,\DD)/\Z_{(p)}}.
\]
This is \cite[Proposition 4.16]{mp:spin}, whose proof applies also when $p=2$; one only has to replace appeals to results from~\cite{KisinJAMS} with appeals  to the analogous results from \cite{KMP}.

Now, suppose that we have a point $s$ of $\mathcal{S}_K(G,\DD)$ valued in a field $k$. 
If $\tilde{s}$ is any lift of $s$ to the ring of dual numbers $k[\epsilon]$,  the connection $\nabla$ induces a canonical isomorphism
\[
\xi_{\tilde{s}}:\bm{V}_{dR,s}\otimes_k k[\epsilon]\iso \bm{V}_{dR,\tilde{s}},
\]
and thus gives rise to an isotropic line
\[
F^1_{\tilde{s}}(\bm{V}_{dR,s}\otimes_kk[\epsilon]) 
\define  \xi_{\tilde{s}}^{-1}(F^1\bm{V}_{dR,\tilde{s}})   \subset \bm{V}_{dR,s}\otimes_kk[\epsilon].
\]
By construction, this line lifts $F^1\bm{V}_{dR,s}$.

The properties of the Kodaira-Spencer map mentioned above can now be reinterpreted as saying that the association
\[
\tilde{s}\mapsto F^1_{\tilde{s}}(\bm{V}_{dR,s}\otimes_kk[\epsilon])
\]
is a bijection from the tangent space of $\mathcal{S}_K(G,\DD)$ at $s$ to the space of isotropic lines in $\bm{V}_{dR,s}\otimes_kk[\epsilon]$ lifting $F^1\bm{V}_{dR,s}$. 
This latter space can be canonically identified with the $k$-vector space
\[
\Hom_k(F^1\bm{V}_{dR,s}, F^0\bm{V}_{dR,s}    /   F^1\bm{V}_{dR,s} )
\]
as follows: Any lift $F^1_{\tilde{s}}(\bm{V}_{dR,t}\otimes_kk[\epsilon])$ will be contained in $F^0\bm{V}_{dR,s}\otimes_kk[\epsilon]$, and so we can consider the associated map
\[
F^1_{\tilde{s}}(\bm{V}_{dR,s}\otimes_kk[\epsilon])\to 
 (  F^0\bm{V}_{dR,s}    /   F^1\bm{V}_{dR,s} )  \otimes_k  k[\epsilon], 
\]
which will factor as
\begin{displaymath}
\xymatrix{
  F^1_{\tilde{s}}(\bm{V}_{dR,s}\otimes_kk[\epsilon]) \ar[r]\ar[d]_{\epsilon\mapsto 0}	& (  F^0\bm{V}_{dR,s}    /   F^1\bm{V}_{dR,s} )  \otimes_k  k[\epsilon]\\
  F^1\bm{V}_{dR,s}\ar[r]_{\varphi_{\tilde{s}}}&F^0\bm{V}_{dR,s}    /   F^1\bm{V}_{dR,s}\ar[u]_{1\otimes\epsilon}
}.
\end{displaymath}
The desired identification is now given by the assignment $F^1_{\tilde{s}}\mapsto \varphi_{\tilde{s}}$.

We can say more.  Suppose that $s$ lifts to a $k$-point of $\mathcal{Z}(\Lambda)$ corresponding to an embedding $\Lambda\hookrightarrow V(\mathcal{A}_s) $. 
We will continue to use $s$ to denote this lift as well. 
The de Rham realization of the embedding gives a map
\[
\Lambda\to  F^0 \bm{V}_{dR,s},
\]
and we let 
\[
\bm{\Lambda}_{dR,s}\subset  F^0 \bm{V}_{dR,s}
\]
  be the $k$-subspace generated by its image.
Now, the bijection from the previous paragraph identifies the tangent space of $\mathcal{Z}(\Lambda)$ at $s$ with the space of isotropic lines in $\bm{V}_{dR,s}\otimes_kk[\epsilon]$ that lift $F^1\bm{V}_{dR,s}$ and are also orthogonal to $\bm{\Lambda}_{dR,s}$. 
This space in turn can be identified with the $k$-vector space
\begin{equation}\label{better tangent}
\Hom_k(F^1\bm{V}_{dR,s},\overline{\bm{\Lambda}}^{\perp}_{dR,s}),
\end{equation}
where 
\[
\overline{\bm{\Lambda}}_{dR,s}\subset F^0\bm{V}_{dR,s}    /   F^1\bm{V}_{dR,s}
\]
 is the  the image of  $\bm{\Lambda}_{dR,s}$
and $\overline{\bm{\Lambda}}^{\perp}_{dR,s}$ is its orthogonal complement.

For proofs of the above statements, which use the explicit description of the complete local rings of $\mathcal{S}_K(G,\DD)$, see~\cite[Prop. 5.16]{mp:spin}. As observed there, they also apply more generally to arbitrary nilpotent divided power thickenings. We record some immediate consequences.

\begin{lemma}
\label{lem:zlamb smooth}
Let the notation be as above, and set $r = \mathrm{rank}(\Lambda)$. 
\begin{enumerate}
	\item The completed \'etale local ring $\widehat{\co}_{\mathcal{Z}(\Lambda),s}$  is a quotient of $\widehat{\co}_{\mathcal{S}_K(G,\DD),s}$ by an ideal generated by $ \mathrm{rank}(\Lambda)$ elements.
	\item $\mathcal{Z}(\Lambda)$ is smooth at $s$ if and only if  $\overline{\bm{\Lambda}}_{dR,s}$ has $k$-dimension  $ \mathrm{rank}(\Lambda)$. In particular, the generic fiber of $\mathcal{Z}(\Lambda)$ is smooth.
	\item Suppose that $k$ has characteristic $p$, and that the Krull dimension of  $\widehat{\co}_{\mathcal{Z}(\Lambda),s}/ (p) $ is  $n- \mathrm{rank}(\Lambda)$. Then $\mathcal{Z}(\Lambda)$ and $\mathcal{Z}(\Lambda)_{\F_p}$ are local complete intersections at $s$. Moreover, $\mathcal{Z}(\Lambda)$ is flat over $\Z_{(p)}$ at $s$.
	\end{enumerate}
\end{lemma}

\begin{proof}
The first claim is a consequence of the deformation theory explained above (more precisely, of its generalization to arbitrary square-zero thickenings) and Nakayama's lemma. See~\cite[Corollary 5.17]{mp:spin}.

For the second claim, note that $\mathcal{Z}(\Lambda)$ will be smooth at $s$ if and only if its tangent space at $s$ has dimension  $n-\mathrm{rank}(\Lambda)$.
As we have identified the tangent space with (\ref{better tangent}), this is equivalent to  $\overline{\bm{\Lambda}}_{dR,s}$ having dimension $\mathrm{rank}(\Lambda)$. For the assertion about the generic fiber, it suffices to check the criterion for smoothness at every $\C$-valued point. Now, note that the de Rham realization
\[
V_0(\mathcal{A}_s)\otimes_\Z\C \to \bm{V}_{dR,s},
\]
is injective, and also that the image of this realization is precisely the weight $(0,0)$ part of the Hodge structure on $\bm{V}_{dR,s}$, and hence is complementary to $F^1\bm{V}_{dR,s}$. This implies that $\overline{\bm{\Lambda}}_{dR,s}$ has dimension $r$ over $\C$, and hence that $\mathcal{Z}(\Lambda)$ is smooth at $s$.

Now we come to the third claim. Note that $\widehat{\co}_{\mathcal{S}_K(G,\DD),s}$ is formally smooth over $W(k)$ of Krull dimension $n+1$. Hence, $\widehat{\co}_{\mathcal{S}_K(G,\DD),s}/(p)$ is also formally smooth over $k$ of Krull dimension $n$, and $\widehat{\co}_{\mathcal{Z}(\Lambda),s}/(p)$, which is its quotient by an ideal generated by $\mathrm{rank}(\Lambda)$ element, is a complete intersection as soon as
\[
\dim\big(\widehat{\co}_{\mathcal{Z}(\Lambda),s}/(p) \big) = n - \mathrm{rank}(\Lambda).
\]
This is precisely our hypothesis.

Now, note that we have
\[
n - \mathrm{rank}(\Lambda) + 1\leq \mathrm{dim}(\widehat{\co}_{\mathcal{Z}(\Lambda),s})
\le \mathrm{dim}\big(\widehat{\co}_{\mathcal{Z}(\Lambda),s}/(p)\big) + 1 = n - \mathrm{rank}(\Lambda) + 1.
\]
Here, the first two inequalities follow from Krull's Hauptidealsatz. This shows 
\[
\dim ( \widehat{\co}_{\mathcal{Z}(\Lambda),s} )= n - \mathrm{rank}(\Lambda) + 1,
\]
and implies that $\widehat{\co}_{\mathcal{Z}(\Lambda),s}$ is a complete intersection ring.

Finally, to see that $\mathcal{Z}(\Lambda)$ is flat over $\Z_{(p)}$ at $s$, note that $p$ cannot be a zero divisor in $\widehat{\co}_{\mathcal{Z}(\Lambda),s}$: Indeed, the equality
\[
\dim~(\widehat{\co}_{\mathcal{Z}(\Lambda),s}/(p)) = n - \mathrm{rank}(\Lambda) = \dim~\widehat{\co}_{\mathcal{Z}(\Lambda),s} - 1
\]
implies that $p$ is not contained in any minimal prime of $\widehat{\co}_{\mathcal{Z}(\Lambda),s}$. Since $\widehat{\co}_{\mathcal{Z}(\Lambda),s}$ is a complete intersection ring and hence Cohen-Macaulay, this implies that $p$ is not a zero divisor.
\end{proof}

For any morphism $S\to \mathcal{Z}(\Lambda)$,  de Rham realization defines a morphism
\[
\Lambda\otimes\co_S \to \bm{V}_{dR,S}.
\]
Let $\bm{\Lambda}_{dR,S}\subset \bm{V}_{dR,S}$ be the image of this morphism.

We will consider the canonical open substack
\begin{equation}\label{nice part}
\mathcal{Z}^{\mathrm{pr}}(\Lambda)\hookrightarrow \mathcal{Z}(\Lambda)
\end{equation}
 characterized by the property that a morphism $S\to \mathcal{Z}(\Lambda)$ factors through $\mathcal{Z}^{\mathrm{pr}}(\Lambda)$ if and only if $\bm{\Lambda}_{dR,S} \subset \bm{V}_{dR,S}$ is a local direct summand  of rank equal to $\mathrm{rank}(\Lambda)$.

\begin{proposition}
\label{prop:z pr gen smooth}
Consider the following assertions:  
\begin{enumerate}
	\item 
	For any generic geometric point $\eta$ of $\mathcal{Z}^{\mathrm{pr}}(\Lambda)_{\F_p}$, the Kuga-Satake abelian scheme $\mathcal{A}_\eta$ is ordinary, and the tautological map
$\Lambda\to V_0(\mathcal{A}_\eta)$ is an isomorphism.
	\item 
	The special fiber $\mathcal{Z}^{\mathrm{pr}}(\Lambda)_{\F_p}$ is a generically smooth local complete intersection of dimension $n-\mathrm{rank}(\Lambda)$.
	\item
	The special fiber $\mathcal{Z}^{\mathrm{pr}}(\Lambda)_{\F_p}$ is smooth outside of a codimension $2$ subspace. 
\end{enumerate}
Then (1) and (2) hold whenever $\mathrm{rank}(\Lambda)\le n/2$, and (3) holds whenever $\mathrm{rank}(\Lambda) \le  ( n-1 )/2$.
\end{proposition}

\begin{proof}
We will prove the proposition by induction on the rank of $\Lambda$.
For any integer $r\geq 0$ and $i\in\{1,2,3\}$, let $P_i(r)$ be the statement that assertion (i) is valid whenever $\mathrm{rank}(\Lambda)=r$. 
We claim 
\begin{enumerate}[(i)]
	\item if  $0\leq r\le ( n-1) / 2 $ then  $P_2(r)$ implies $P_1(r)$,
	\item if  $r\leq  ( n-2) /2 $ then $P_1(r)$ and $P_2(r)$ together imply $P_2(r+1)$,
	\item if  $r\leq  ( n-3) /2$ then $P_1(r)$ and $P_2(r)$ together imply $P_3(r+1)$.
\end{enumerate}
Once the claims are proved, the lemma will follow by induction.
Indeed, the base case $P_2(0)$ is implied by the smoothness of $\mathcal{S}_K(G,\DD)$. 

The claims themselves follow from an argument derived from~\cite{Ogus1979}, which was used in~\cite[Proposition 6.17]{mp:spin}, and exploits the following simple lemma. 

\begin{lemma}
\label{lem:parallel dimension}
Let $Z$ be an $\F_p$-scheme admitting an unramified map $Z\to \mathcal{S}_K(G,\DD)$. Suppose that we have a local direct summand $\bm{N}\subset \bm{V}_{dR}\vert_Z$ that is horizontal for the integrable connection
\[
\bm{V}_{dR,Z}\to \bm{V}_{dR,Z}\otimes_{\co_Z}\Omega^1_{Z/\F_p}
\]
induced from the one on $\bm{V}_{dR}$. Suppose also that
\[
F^1\bm{V}_{dR,Z}\subset \bm{N}.
\]
Then $\dim(Z) \leq \mathrm{rank}(\bm{N})-1$. If, in addition $\bm{N}\cap F^0\bm{V}_{dR,Z}$ is a local direct summand of $\bm{V}_{dR,Z}$, then we in fact have
\[
\dim(Z) \leq \mathrm{rank}(\bm{N}\cap F^0\bm{V}_{dR,Z})-1.
\]
\end{lemma}
\begin{proof}
For the first assertion, it is enough to show that, at any point $z\in Z(k)$ valued in a field $k$, the tangent space of $Z$ at $z$ has dimension at most $\mathrm{rank}(\bm{N})-1$. But our hypotheses imply that, if $\tilde{z}\in Z(k[\epsilon])$ is any lift of $Z$, then we must have
\[
F^1_{\tilde{z}}(\bm{V}_{dR,Z}\otimes_kk[\epsilon])\subset (\bm{N}_z\otimes_kk[\epsilon])\cap (F^0\bm{V}_{dR,z}\otimes_kk[\epsilon]).
\]
This, combined with the fact that $Z$ is unramified over $\mathcal{S}_K(G,\DD)$,  implies that the tangent space of $Z$ at $z$ can be identified with a subspace of
\[
\Hom_k(F^1\bm{V}_{dR,z},\overline{\bm{N}}_z)\subset 
\Hom_k(F^1\bm{V}_{dR,z}, F^0\bm{V}_{dR,z}  /  F^1\bm{V}_{dR,z}  ),
\]
where $\overline{\bm{N}}_z$ is the image of $\bm{N}_z\cap F^0\bm{V}_{dR,z}$ in $F^0\bm{V}_{dR,z}  /  F^1\bm{V}_{dR,z}$. 
We are now done, since $\overline{\bm{N}}_z$ has dimension at most $\mathrm{rank}(\bm{N})-1$.

The second assertion is immediate from the proof of the first.
\end{proof}

We begin with claim (i). Assume $P_2(r)$, and suppose  $\mathrm{rank}(\Lambda)=r$. 
Fix a geometric generic point $\eta$ of $\mathcal{Z}^{\mathrm{pr}}(\Lambda)_{\F_p}$. 
Then $P_2(r)$ implies that there is a smooth $\F_p$-scheme $U$, equidimensional of dimension $n-r$, and an \'etale map $U\to \mathcal{Z}^{\mathrm{pr}}(\Lambda)_{\F_p}$, whose image contains $\eta$. 
 

As explained in the proof of~\cite[Proposition 6.17]{mp:spin}, there is a canonical isotropic line 
\[
\bm{C}\subset \bm{V}_{dR,U},
\]
called the \emph{conjugate filtration}, which is horizontal for the connection on $\bm{V}_{dR,U}$, is contained in $\bm{\Lambda}_{dR,U}^\perp$, and is such that a point $t\in U(k)$ is non-ordinary if and only if $\bm{C}_{t}\subset F^0\bm{V}_{dR,t}$,  or, equivalently, if and only if 
\[
F^1\bm{V}_{dR,t}\subset \bm{C}_t^\perp\cap \bm{\Lambda}_{dR,t}^\perp.
\]

Now, we have
\[
\bm{C}_{t}\subset \bm{\Lambda}_{dR,t}
\]
only if $F^1\bm{V}_{dR,t}\subset \bm{\Lambda}_{dR,t}$. See for instance~\cite[Lemma 4.20]{mp:spin}. Therefore, since we are assuming that $U$ is smooth, the subsheaf
\[
\bm{C}_U + \bm{\Lambda}_{dR,U}\subset \bm{V}_{dR,U}
\]
is a horizontal local direct summand of rank $r+1$.


By Lemma~\ref{lem:parallel dimension}, if $Z\subset U$ is a closed subscheme with 
\[
F^1\bm{V}_{dR,Z}\subset \bm{C}_{Z}+\bm{\Lambda}_{dR,Z},
\]
then $\dim Z\leq r$. Using $r\leq (n-1)/2$, we see that $r = \dim Z < \dim U = n-r$. 
Therefore, after shrinking $U$ if necessary, we can assume that
\[
F^1\bm{V}_{dR,U}+\bm{C}_{U} +\bm{\Lambda}_{dR,U}\subset \bm{V}_{dR,U}
\]
is a direct summand of rank $r+2$, or, equivalently, that
\[
F^0\bm{V}_{dR,U}\cap \bm{C}_U^\perp \cap \bm{\Lambda}_{dR,U}^\perp\subset \bm{V}_{dR,U}
\]
is a direct summand of rank $n-r$. Therefore, once again by Lemma~\ref{lem:parallel dimension}, the locus in $U$ where $F^1\bm{V}_{dR,U}$ is contained in this direct summand has dimension at most $n-r-1$.
But this is precisely the non-ordinary locus in $U$. As $\dim(U) = n-r$, this shows the first part of $P_1(r)$.

Suppose now that the map $\Lambda\to V_0(\mathcal{A}_\eta)$ is not a bijection, so that there exists $x\in V_0(\mathcal{A}_\eta)$ such that
\[
\widetilde{\Lambda} = \Lambda +\langle x\rangle \subset V_0(\mathcal{A}_\eta)
\]
is a direct summand of rank $r+1$, and its de Rham realization
\[
\widetilde{\bm{\Lambda}}_{dR,\eta} = \bm{\Lambda}_{dR,\eta} + \langle \bm{x}_{dR,\eta}\rangle \subset \bm{V}_{dR,\eta}
\]
is a $k(\eta)$-vector subspace of dimension $r+1$.

After shrinking $U$ if necessary, we can assume that $x\in V_0(\mathcal{A}_U)$, and that de Rham realization gives us a local direct summand
\[
\widetilde{\bm{\Lambda}}_{dR,U} = \bm{\Lambda}_{dR,U} + \langle \bm{x}_{dR,U}\rangle \subset \bm{V}_{dR,U}
\]
of rank $r+1$ that is horizontal for the connection. However, the discussion of the deformation theory above Lemma~\ref{lem:zlamb smooth} implies that, over $U$, the Kodaira-Spencer map factors through an isomorphism
\[
( F^0\bm{V}_{dR,U} /  F^1\bm{V}_{dR,U} ) /\overline{\bm{\Lambda}}_{dR,U}
\iso  (  \bm{V}_{dR,U} / F^1\bm{V}_{dR,U}  ) \otimes_{\co_U}\Omega^1_{U/\F_p}.
\]
However,  the horizontality of $\widetilde{\bm{\Lambda}}_{dR,U}$ guarantees that its (non-trivial) image on the left-hand side is in the kernel of the Kodaira-Spencer map. This contradiction finishes the proof of claim (i).

We will prove claims (ii) and (iii). Suppose that $P_1(r)$ and $P_2(r)$ hold and that $\mathrm{rank}(\Lambda) = r+1$. Write $\Lambda = \Lambda_1\oplus \Lambda_0$, where $\mathrm{rank}(\Lambda_0) = 1$. Then we have an obvious factorization
\[
\mathcal{Z}^{\mathrm{pr}}(\Lambda)\to \mathcal{Z}^{\mathrm{pr}}(\Lambda_1)\to \mathcal{S}_K(G,\DD).
\]
The first arrow exhibits $\mathcal{Z}^{\mathrm{pr}}(\Lambda)_{\F_p}$  as a divisor on $\mathcal{Z}^{\mathrm{pr}}(\Lambda_1)_{\F_p}$  (\'etale locally on the source, in the sense of Proposition \ref{prop:sufficiently small}).
  Indeed, the complete local rings of the former are cut out by one equation in those of the latter, and $P_1(r)$ shows that $\mathcal{Z}^{\mathrm{pr}}(\Lambda)_{\F_p}$ does not contain any generic points of $\mathcal{Z}^{\mathrm{pr}}(\Lambda_1)_{\F_p}$. Therefore, by Lemma~\ref{lem:zlamb smooth} and $P_2(r)$, we find that $\mathcal{Z}^{\mathrm{pr}}(\Lambda)_{\F_p}$ is a local complete intersection of dimension $n-(r+1)$.

Let $W\subset \mathcal{Z}^{\mathrm{pr}}(\Lambda)_{\F_p}$ be the nonsmooth locus, with its reduced substack structure.  We  find from Lemma~\ref{lem:zlamb smooth} that 
\[
F^1\bm{V}_{dR}\vert_W\subset \bm{\Lambda}_{dR}\vert_W.
\] 
By Lemma~\ref{lem:parallel dimension}, this implies that $\mathrm{dim}(W)\leq r$. This is bounded by $n-r-2$ under the hypothesis $r\leq  ( n-2 ) /2$, and by $n-r-3$ if $r\leq ( n-3 )/2$.
This  proves (ii) and (iii), and completes the proof of Proposition  \ref{prop:z pr gen smooth}.
\end{proof}

It will be useful to recall some bounds on the dimension of the supersingular locus in the mod-$p$ fiber of $\mathcal{S}_K(G,\DD)$ under the assumption that $V_{\Z_p}$ is almost self-dual. 

\begin{proposition}
\label{prop:supersing dim bound}
Suppose that $V_{\Z_p}$ is almost self-dual of rank $n+2$, and suppose that $Z\to \mathcal{S}_K(G,\DD)$ is an unramified morphism from an $\F_p$-scheme $Z$ such that, for all points $z\in Z(k)$ valued in a field $k$, the abelian variety $\mathcal{A}_z$ is supersingular. Then $\dim (Z)  \leq  n/2$. 
If  $V_{\Q_p}$ is an orthogonal sum of hyperbolic planes, we have the sharper bound
\[
\dim(Z) \leq \frac{n}{2}-1.
\]
\end{proposition}
\begin{proof}
If $V_{\Q_p}$ is not an orthogonal sum of hyperbolic planes, then we can find an embedding
\[
V_{\Z}\hookrightarrow V^\beef_{\Z},
\]
where $V^\beef_{\Q_p}$ is of this form, and where the codimension of $V \subset V^\beef$ is $1$ if $n$ is odd and $2$ if $n$ is even. Using such an embedding, the proposition can be reduced to proving the final assertion, and so we may assume that $V_{\Q_p}$ (and hence $V_{\Z_p}$) is an orthogonal sum of hyperbolic planes. 

When $p>2$, the proposition follows from the much finer results of~\cite{Howard-Pappas}, which give a complete description of the supersingular locus of $\mathcal{S}_K(G,\DD)_{\F_p}$. However, if one is only interested in upper bounds, one can appeal to the methods of \cite{Ogus2001-wy}, which apply even when $p=2$ and $V_{\Z_p}$ is self-dual. See in particular Proposition 14 of [\emph{loc.~cit.}]

For the convenience of the reader, we sketch the basic idea here. First, we can replace $Z$ with its underlying reduced scheme.  Second, we can throw away its singular part, and assume that $Z$ is smooth. 

If $z\in Z(k)$ is a geometric point, then the \emph{Artin invariant} of $z$ is the $k$-codimension of the image of $V_0(\mathcal{A}_z)\otimes_\Z k \to \bm{V}_{dR,z}$.  This is an integer between $1$ and $n/2$. 
Ogus's argument shows that there is a canonical filtration of $F^0\bm{V}_{dR,Z}$ by coherent, isotropic, horizontal coherent subsheaves
\[
\bm{E}_1 \subset \cdots \bm{E}_i\ \subset\cdots \subset  \bm{E}_{n/2}\subset F^0\bm{V}_{dR,Z}
\]
with the following properties:
\begin{itemize}
	\item A geometric point $z\in Z(k)$ has Artin invariant  $ \le j$ if and only if
	\[
      F^1\bm{V}_{dR,z}\subset \bm{E}_{j,z} . 
	\]
	\item If $Z_{\geq j}\subset Z$ is the open subscheme where the Artin invariant is  $ \ge j$, 
	then $\bm{E}_{j,Z_{\geq j}}$ is a rank $j$ local direct summand of $\bm{V}_{dR,Z_{\geq j}}$.
\end{itemize}
Note that the first condition ensures that locus where the Artin invariant is bounded below by $j$ is indeed an open subscheme of $Z$. 

Given these two properties, it is immediate from Lemma~\ref{lem:parallel dimension} that the dimension of $Z$ is bounded above by $r-1$, where $r$ is the maximal Artin invariant attained by a geometric point of $Z$. This proves the proposition.

The construction of $\bm{E}_{j}$ is as follows.
For $j=1$, $\bm{E}_1$ is just the conjugate filtration $\bm{C}\subset \bm{V}_{dR,Z}$ already encountered in the proof of Proposition~\ref{prop:z pr gen smooth}. 
The crystalline Frobenius on the crystalline realization of $\mathcal{A}_Z$ induces an isometry
\[
\gamma: \mathrm{Fr}_Z^* (    F^0 \bm{V}_{dR,Z}  / F^1 \bm{V}_{dR,Z}) 
\iso
\bm{C}^\perp / \bm{C},
\]
where $\mathrm{Fr}_Z$ is the absolute Frobenius on $Z$.
Now  inductively define $\bm{E}_j\subset \bm{C}^\perp$ as the pre-image of the image of $\bm{E}_{j-1}$ under the composition
\[
\mathrm{Fr}_Z^*\bm{E}_{j-1}\hookrightarrow \mathrm{Fr}_Z^*F^0\bm{V}_{dR,Z}
 \map{\gamma}
\bm{C}^\perp / \bm{C} .
\]

It  follows from the argument in \cite[Lemma 5]{Ogus2001-wy} that $\bm{E}_j$ is a subsheaf of $F^0\bm{V}_{dR,Z}$ for all $j$, so that the inductive procedure is well-defined. That it is isotropic, coherent and horizontal follows from the construction. That the filtration thus obtained has the desired properties follows from the arguments in Proposition 6 and Lemma 9 of [\emph{loc.~cit.}].
\end{proof}

\begin{lemma}
\label{lem:zpr complement}
Suppose that $\Lambda$ is maximal at $p$. 
The complement of $\mathcal{Z}^{\mathrm{pr}}(\Lambda)$ in $\mathcal{Z}(\Lambda)$ lies above the supersingular locus of $\mathcal{S}_K(G,\DD)_{\F_p}$. 
If we let
\[
m = \begin{cases}
\frac{n}{2} & \mbox{if $V_{\Z_p}$ is an orthogonal sum of hyperbolic planes}\\
\lfloor\frac{n}{2}  \rfloor   & \mbox{if $n$ is odd}\\
\frac{n}{2}-1 & \mbox{otherwise,}
\end{cases}
\]
then the following properties hold.
\begin{enumerate}
\item If $\mathrm{rank}(\Lambda) \leq m$ then $\mathcal{Z}^{\mathrm{pr}}(\Lambda)_{\F_p}$ is dense in $\mathcal{Z}(\Lambda)_{\F_p}$.
\item If $\mathrm{rank}(\Lambda) \leq m-1$  then the complement of $\mathcal{Z}^{\mathrm{pr}}(\Lambda)_{\F_p}$ in $\mathcal{Z}(\Lambda)_{\F_p}$ has codimension at least $2$.
\end{enumerate}
\end{lemma}
\begin{proof}
Once we know that the complement is supported above the supersingular locus of the mod-$p$ fiber, the rest will follow from the bounds in Proposition~\ref{prop:supersing dim bound}.

To prove the assertion on the complement, we first note that the open immersion~\eqref{nice part} induces an isomorphism of the generic fibers; see~\cite[Prop. 6.16]{mp:spin}. Therefore, we only have to show that the mod-$p$ fiber of the complement is supported on the supersingular locus. 
Equivalently, we must show that, for any non-supersingular point $s\in \mathcal{Z}(\Lambda)(k)$ valued in a field $k$ of characteristic $p$, the subspace $\bm{\Lambda}_{dR,s}\subset \bm{V}_{dR,s}$ has $k$-dimension $\mathrm{rank}(\Lambda)$.

Arguing as in~\cite[\S 6.27]{mp:spin}, we find that, for such a point $s$, the de Rham realization map
\[
V_0(\mathcal{A}_s)\otimes k \to \bm{V}_{dR,s}
\]
is injective. 
Moreover, by the maximality of $\Lambda$ at $p$, the image of 
\[
\Lambda \otimes\Z_{(p)} \to V_0(\mathcal{A}_s)\otimes\Z_{(p)}
\]
 is a $\Z_{(p)}$-module direct summand of rank $\mathrm{rank}(\Lambda)$.
Combining these two observations shows that the subspace  $\bm{\Lambda}_{dR,s}\subset \bm{V}_{dR,s}$ has $k$-dimension $\mathrm{rank}(\Lambda)$, and completes the proof of the lemma. 
\end{proof}

\begin{proposition}
\label{prop:zlamb local props}
Suppose that $\Lambda$ is maximal at $p$, and let $m$  be defined as in Lemma~\ref{lem:zpr complement}. 
\begin{enumerate}
\item If $\mathrm{rank}(\Lambda)\leq m$ then  $\mathcal{Z}(\Lambda)_{\F_p}$ is a generically smooth local complete intersection
of dimension $n-\mathrm{rank}(\Lambda)$. Moreover, $\mathcal{Z}(\Lambda)$ is normal and flat over $\Z_{(p)}$.
\item If $\mathrm{rank}(\Lambda)\leq m-1$ then $\mathcal{Z}(\Lambda)_{\F_p}$ is geometrically normal.
\end{enumerate}
\end{proposition}

\begin{proof}
Note that we always have
\[
m\leq  \frac{n}{2} \quad\mbox{ and }\quad  m-1 \leq   \frac{n-1}{2}.
\]

First suppose $\mathrm{rank}(\Lambda) \le m$. Combining Proposition \ref{prop:z pr gen smooth} and Lemma \ref{lem:zpr complement} shows that
$\mathcal{Z}^{\mathrm{pr}}(\Lambda)_{\F_p}$ is a generically smooth local complete intersection of dimension $n-\mathrm{rank}(\Lambda)$, and is dense in $\mathcal{Z}(\Lambda)_{\F_p}$.
Hence $\mathcal{Z}(\Lambda)_{\F_p}$ is itself generically smooth of dimension $n-\mathrm{rank}(\Lambda)$.

It now follows from claim (3) of Lemma~\ref{lem:zlamb smooth} that $\mathcal{Z}(\Lambda)$ is a local complete intersection, flat over $\Z_{(p)}$.
 In particular, it is Cohen-Macaualy and so satisfies Serre's property $(S_k)$ for all $k\ge 1$. 
 Recall from claim (2) of Lemma~\ref{lem:zlamb smooth} that the generic fiber of $\mathcal{Z}(\Lambda)$ is smooth over $\Q$.
 As we have already proved that the special fiber is generically smooth,  $\mathcal{Z}(\Lambda)$ is  regular in codimension one, and  hence satisfies Serre's property $(R_1)$.  Claim (2) now follows from Serre's criterion for normality. 

Now suppose $\mathrm{rank}(\Lambda) \le m-1$.  We have already shown that the geometric fiber of $\mathcal{Z}(\Lambda)_{\F_p}$ is a local complete intersection. So, just as above, to show that it is normal  it is enough to show that it is regular in codimension one. This follows by combining Proposition \ref{prop:z pr gen smooth} and Lemma \ref{lem:zpr complement}, which shows that
$\mathcal{Z}^{\mathrm{pr}}(\Lambda)_{\F_p}$ is smooth outside of a codimension two subspace, and that its complement in  $\mathcal{Z}(\Lambda)_{\F_p}$ has codimension at least $2$.

\end{proof}


\subsection{Normality of the fibers, and flatness of  divisors}
\label{ss:flatness}


We return to the general setting in  which $V_\Z \subset V$ is any maximal lattice, and deduce two important consequences from the results of \S \ref{ss:local properties}.

\begin{proposition}\label{prop:normal fiber}
If $n\geq 6$, the special fiber of  $\mathcal{S}_K(G,\DD)$ is geometrically normal.
\end{proposition}

\begin{proof}
When $p>2$, this is part of \cite[Theorem 4.4.5]{AGHMP-2}.  
 The same idea of proof works in general, bolstered now by Proposition~\ref{prop:zlamb local props}

Using Lemma~\ref{lem:good beef}, we may choose an embedding $V_\Z\hookrightarrow V^\beef_\Z$ as in \S \ref{ss:embiggen} in such a way that  $V^\beef_{\Z}$ is self-dual at $p$, and 
\[
\Lambda = \{ x\in V_\Z^\beef : x\perp V_\Z \}
\]
has rank at most $r$, where $r=2$ if $n$ is even and $r=3$ otherwise.\footnote{If $p\neq 2$  we can choose $V^\beef_\Z$ to be self-dual at $p$ with $r=2$. In this case, we can improve the bound to $n\geq 5$ as in \cite[Theorem 4.4.5]{AGHMP-2}.}

There is a commutative diagram
\[
\xymatrix{
 &   {  \mathcal{Z}^\beef(\Lambda) }   \ar[d] \\
 {  \mathcal{S}_K(G,\DD)  } \ar[ur]\ar[r]  &   {   \mathcal{S}_{K^\beef}(G^\beef,\DD^\beef) } 
}
\]
in which the vertical morphism is defined as in (\ref{general cycle}), the horizontal morphism is 
(\ref{integral embiggen morphism}), and the diagonal arrow is  induced by the isometric embedding
\[
\Lambda \to V_0( \mathcal{A}^\beef_{ \mathcal{S}_K(G,\DD)} ).
\]
determined by  (\ref{special decomp cosets}).

The self-duality of $V^\beef_{\Z_p}$ gives us an isomorphism
\[
V^\vee_{\Z_p} / V_{\Z_p} \iso \Lambda^\vee_{\Z_p}/\Lambda_{\Z_p}
\]
of quadratic spaces over $\Q_p/\Z_p$, as in (\ref{coset swap}).  The maximality of $V_\Z$ at $p$ implies that the left hand side contains no 
nonzero isotropic vectors, and so neither does the right hand side.  This implies the maximality of $\Lambda$ at $p$.
With this in hand, we may apply  Proposition~\ref{prop:zlamb local props} and the inequality
\[
r\leq \frac{n+r}{2}-2,
\]
which holds as $n\geq 6$, to see that  $\mathcal{Z}^\beef(\Lambda)$ has geometrically normal fibers.

Thus it suffices to  show that the diagonal arrow is an open and closed immersion.
This holds in the generic fiber by  \cite[Lemma 7.1]{mp:spin}, and hence also on the level of integral models as the source and target are both normal.
\end{proof}

\begin{proposition}\label{prop:nice divisors}
Assume that $n\ge 4$.  For every positive $m\in \Q$ and $\mu \in V_\Z^\vee /V_\Z$,  the special divisor $\mathcal{Z}(m,\mu)$ is flat over $\Z_{(p)}$.
\end{proposition}

\begin{proof}
When $p>2$ this is \cite[Proposition 4.5.8]{AGHMP-2}.  We explain how to extend the proof to the general case.

As in the proof of Proposition~\ref{prop:normal fiber} fix an embedding $V_\Z\hookrightarrow V^\beef_\Z$ with $V^\beef_{\Z}$ self-dual at $p$, and so that 
\[ 
\Lambda = \{ x\in V_\Z^\beef : x\perp V_\Z \}
\]
  is maximal of rank at most $r$ with $r=2$ when $n$ is even and $r=3$ otherwise.\footnote{Once again, if $p>2$, then we can always take $r=2$ and the result can be strengthened to only require $n\ge 3$.}

Consider again the finite unramified morphism
\[
\mathcal{Z}^\beef(\Lambda)\to \mathcal{S}_{K^\beef}(G^\beef,\DD^\beef).
\]
By Proposition~\ref{prop:zlamb local props}, this is normal and flat over $\Z_{(p)}$, as long as we have
\[
2 \le \frac{n+2}{2}-1,
\]
for $n$ even and
\[
3 \le \frac{n+3}{2}-1,
\]
for $n$ odd. These inequalities hold for $n \ge 4$.

Using the decomposition (\ref{special stack decomp}), we may choose a positive 
$m^\beef\in \Q$ and a $\mu^\beef\in (V^\beef_\Z)^\vee /V^\beef_\Z$ in such a way that 
\[
\mathcal{Z}(m,\mu) \subset 
\mathcal{Z}^\beef(m^\beef,\mu^\beef) \times_{ \mathcal{S}_{K^\beef}(G^\beef,\DD^\beef) }   \mathcal{S}_K(G,\DD)
\]
as an open and closed substack.
Now use the open and closed immersion
\[
\mathcal{S}_K(G,\DD) \to \mathcal{Z}^\beef(\Lambda)
\]
from the proof of Proposition \ref{prop:normal fiber} to identify
\begin{equation}\label{nonstandard divisor 1}
\mathcal{Z}(m,\mu) \subset  
\mathcal{Z}^\beef(m^\beef,\mu^\beef)\times_{\mathcal{S}_{K^\beef}(G^\beef,\DD^\beef)}\mathcal{Z}^\beef(\Lambda)
\end{equation} 
as a union of connected components.   In particular, by Proposition \ref{prop:sufficiently small}, the projection
\begin{equation}\label{nonstandard divisor 2}
\mathcal{Z}(m,\mu)  \to \mathcal{Z}^\beef(\Lambda)
\end{equation}
is,  \'etale locally on the target, a disjoint union of closed immersions each defined by a single equation.

\begin{lemma}
The image of (\ref{nonstandard divisor 2}) contains no irreducible component of  $\mathcal{Z}^\beef(\Lambda)_{\F_p}$.
\end{lemma}
\begin{proof}
An  $S$-point of  $\mathcal{Z}(m,\mu)$ determines a special quasi-endomorphism
$x\in V(\mathcal{A}_S)$ with $Q(x)=m$.
The image of such an $S$-point under the inclusion (\ref{nonstandard divisor 1}) determines an
$x^\beef \in V(\mathcal{A}^\beef_S)$, as well as an  isometric embedding $\iota : \Lambda \to V_0(\mathcal{A}_S)$.    Unpacking the construction of the inclusion (\ref{nonstandard divisor 1}), we find that 
 the orthogonal decomposition
\[
V(\mathcal{A}_S^\beef ) =    V(\mathcal{A}_S) \oplus \Lambda_\Q   ,
\]
of Proposition \ref{prop:special decomp cosets} identifies 
$
x^\beef = x + \iota(\lambda)
$
for some $\lambda \in \Lambda_\Q$.   In particular, $x$ determines a nonzero element of  $V(\mathcal{A}_S^\beef )$ orthogonal to $\iota ( \Lambda_\Q )$, and 
\[
\iota : \Lambda_\Q   \to  V(\mathcal{A}_S^\beef)
\]
is not surjective.

In contrast,   for every generic point $\eta$ of $\mathcal{Z}^\beef(\Lambda)_{\F_p}$ we have
\[
\iota_\eta:\Lambda\iso V_0(\mathcal{A}_\eta^\beef).
\]
Indeed, this follows from the density $\mathcal{Z}^{\mathrm{pr}}(\Lambda)_{\F_p} \subset \mathcal{Z}(\Lambda)_{\F_p}$  proved in Lemma \ref{lem:zpr complement}, and assertion (1) of Proposition \ref{prop:z pr gen smooth}. It can be checked that the numerical hypotheses hold under our hypothesis $n\geq 4$.

Thus the image of (\ref{nonstandard divisor 2}) cannot contain the generic point of any irreducible component of $\mathcal{Z}^\beef(\Lambda)_{\F_p}$, completing the proof of the lemma.
\end{proof}

To complete the proof of  Proposition \ref{prop:nice divisors}, we  apply the following lemma to the complete local ring of the local complete intersection (and hence Cohen-Macaulay) stack $\mathcal{Z}^\beef(\Lambda)$ at a point in the image of~\eqref{nonstandard divisor 2}, and taking $a$ to be the equation defining the complete local ring of $\mathcal{Z}(m,\mu)$ at a point in the pre-image.
\begin{lemma}
Let $R$ be a complete local flat $\Z_{(p)}$-algebra that is Cohen-Macaulay. Suppose that $a\in R$ is such that $\Spec(R/aR)\subset \Spec~R$ does not contain any irreducible component of $\Spec(R\otimes_{\Z_{(p)}}\F_p)$. Then $R/aR$ is also flat over $\Z_{(p)}$.
\end{lemma}
\begin{proof}
Since $R$ is $\Z_{(p)}$-flat, $R/pR$ is once again Cohen-Macaulay. Our hypotheses imply that the image $\overline{a}\in R/pR$ of $a$ is not contained in any minimal prime of $R/pR$, which means that $\overline{a}$ is a non-zero divisor in $R/pR$. Since $R$ is local, this is equivalent to saying that $p$ is a non-zero divisor in $R/aR$, which shows that $R/aR$ is $\Z_{(p)}$-flat.
\end{proof}
This completes the proof of Proposition \ref{prop:nice divisors}
\end{proof}

\section{Integral theory of $q$-expansions}
\label{s:integral q}


Keep the hypotheses and notation of \S \ref{s:integrality I} and \S \ref{s:integrality II}.
 In particular, we fix a prime $p$ at which $V_\Z \subset V$ is  maximal.
We now consider  toroidal compactifications of the integral model
\[
\mathcal{S}_K (G,\DD)\to \Spec(\Z_{(p)})
\]
If $V$ is anisotropic then  \cite[Corollary 4.1.7]{mp:compactification} shows that the integral model  is already proper. 
Therefore, in this subsection, we  assume that $V$ admits an isotropic vector.


\subsection{Toroidal compactification}
\label{ss:integral toroidal}


Fix auxiliary data $V_\Z^\beef \subset V^\beef$ and $K^\beef$  as in \S \ref{ss:embiggen}, 
and choose this in such a way that $V_\Z^\beef$ is almost self-dual at $p$.  
In particular, from (\ref{integral definition}) we have the  finite morphism 
\[
\mathcal{S}_K(G,\DD) \to \mathcal{S}_{K^\beef}(G^\beef,\DD^\beef)
\]
of integral models,  under which   $\bm{\omega}^\beef$ pulls back to $\bm{\omega}$.

 We may choose the auxiliary $V^\beef$ to have signature $(n^\beef ,2)$ with $n^\beef \ge 5$.
 By Lemma \ref{lem:integral polarization}, this allows us to choose a symplectic form $\psi^\beef$ on 
 \[
 H^\beef=C(V^\beef)
 \]
   in such a way that the $\Z$-lattice 
 $
H^\beef_\Z = C(V^\beef_\Z ) 
$
is self-dual at $p$.  As in \S \ref{ss:hodge embedding} we obtain an embedding 
 \[
(G^\beef, \DD^\beef) \to (G^\Sg , \DD^\Sg)
\]
into the Siegel Shimura datum determined by $(H^\beef,\psi^\beef)$. 
Recalling the Shimura datum $(\mathbb{G}_m , \mathcal{H}_0)$ of \S \ref{ss:simple}, 
this also fixes a morphism $(G^\beef,\DD^\beef) \to (\mathbb{G}_m , \mathcal{H}_0)$.

Define reductive groups over $\Z_{(p)}$by 
\[
\mathcal{G}^\beef = \GSpin(V^\beef_{\Z_{(p)}}),\quad \mathcal{G}^\Sg = \GSp(H^\beef_{\Z_{(p)}} ),
\]
so that $G^\beef\to G^\Sg$ extends to a closed immersion $\mathcal{G}^\beef \to \mathcal{G}^\Sg$. 
Fix a compact open subgroup 
\[
K^\Sg  =K^\Sg_p  K^{\Sg,p}  \subset G^\Sg(\A_f)
\] 
containing $K^\beef$ and satisfying $K^\Sg_p=\mathcal{G}^\Sg(\Z_p)$.
After shrinking the prime-to-$p$ parts of 
\[
K \subset K^\beef \subset K^\Sg,
\]
we assume that all three are neat.

We can construct a toroidal compactification of $\mathcal{S}_K(G,\DD)$ as follows.  
Fix a finite, complete $K^\Sg$-admissible  cone decomposition $\Sigma^\Sg$ for $(G^\Sg,\DD^\Sg)$.  
As explained in \S\ref{ss:cone functoriality}, it pulls back  to a finite, complete,  $K^\beef$-admissible polyhedral cone decomposition   $\Sigma^\beef$ for $(G^\beef,\DD^\beef)$, and   a finite, complete,  $K$-admissible polyhedral cone decomposition   $\Sigma$ for $(G,\DD)$. 
If $\Sigma^\Sg$ has the no self-intersection property, then so do the decompositions induced from it.

Assume that $K^\Sg$ and $\Sigma^\Sg$ are chosen so that $\Sigma^\Sg$ is smooth and satisfies the no self-intersection property.
We obtain a commutative diagram
\begin{equation}\label{eqn:comp maps}
\xymatrix{
{  \mathcal{S}_K(G,\DD,\Sigma)    } \ar[d]  & {  \Sh _K(G,\DD,\Sigma)  }   \ar[d] \ar[l] \\
 { \mathcal{S}_{K^\beef} (G^\beef , \DD^\beef,\Sigma^\beef) }  \ar[d]  & {  \Sh_{K^\beef}(G^\beef , \DD^\beef,\Sigma^\beef)  }  \ar[l]  \ar[d] \\
  { \mathcal{S}_{K^\Sg}(G^\Sg , \DD^\Sg,\Sigma^\Sg)  }    & {  \Sh_{K^\Sg}(G^\Sg , \DD^\Sg,\Sigma^\Sg), }  \ar[l]
}
\end{equation}
where $ \mathcal{S}_{K^\Sg}(G^\Sg , \DD^\Sg,\Sigma^\Sg)$ is the toroidal compactification of $\mathcal{S}_{K^\Sg}(G^\Sg , \DD^\Sg)$ constructed by Faltings-Chai.  Note that the neatness of $K^\Sg$ implies that it is an algebraic space,  rather than a stack, but does not guarantee that it is a scheme.
The two algebraic spaces above it are defined by normalization, exactly as in (\ref{integral definition}).

According to \cite[Theorem 4.1.5]{mp:compactification}, the algebraic space  $ \mathcal{S}_K(G,\DD,\Sigma) $ is proper over $\Z_{(p)}$  and admits a stratification
\begin{equation}\label{integral strata}
 \mathcal{S}_K(G,\DD,\Sigma)   =
\bigsqcup_{ (\Phi,\sigma) \in \mathrm{Strat}_{K}(G ,\DD,\Sigma) } \mathcal{Z}_{K}^{  (\Phi,\sigma) }(G,\DD,\Sigma)
\end{equation}
 by locally closed  subspaces, extending  (\ref{generic stratification}), in which every stratum is flat over $\Z_{(p)}$.  
 The unique open stratum  is  $\mathcal{S}_K(G,\DD)$, and its complement is a Cartier divisor.

Fix a toroidal stratum representative $(\Phi,\sigma) \in \mathrm{Strat}_K(G,\DD,\Sigma)$ in such a way  that the parabolic subgroup underlying $\Phi$ is the stabilizer of an isotropic line.   As in \S \ref{ss:torsor}, the cusp label representative $\Phi$ determines  a  $T_\Phi$-torsor
\[
\Sh_{K_\Phi}(Q_\Phi,\DD_\Phi) \to \Sh_{\nu_\Phi(K_\Phi)}(\mathbb{G}_m,\mathcal{H}_0),
\]
and the rational polyhedral cone $\sigma$ determines a partial compactification
 \[
\Sh_{K_\Phi}(Q_\Phi,\DD_\Phi)\hookrightarrow \Sh_{K_\Phi}(Q_\Phi,\DD_\Phi,\sigma).
 \]

The base $\Sh_{\nu_\Phi(K_\Phi)}(\mathbb{G}_m,\mathcal{H}_0)$ of the $T_\Phi$-torsor, being a zero dimensional \'etale  scheme over $\Q$, has a canonical finite normal integral model defined as the normalization of $\Spec(\Z_{(p)})$.  The picture is 
\[
\xymatrix{ 
{  \mathcal{S}_{\nu_\Phi(K_\Phi)}(\mathbb{G}_m,\mathcal{H}_0)  } \ar[d]  & { \Sh_{\nu_\Phi(K_\Phi)}(\mathbb{G}_m,\mathcal{H}_0) } \ar[d] \ar[l]  \\
{  \Spec(\Z_{(p)})  }  &  { \Spec(\Q)  .}  \ar[l]
}
\]

 \begin{proposition} \label{prop:strata descp}  
 Define an integral model 
  \[
 \mathcal{T}_\Phi =  \Spec\Big(  \Z_{(p)} [q_\alpha]_{ \alpha \in \Gamma^\vee_\Phi(1) } \Big)
 \]
of  the torus $T_\Phi$ of \S \ref{ss:torsor}.
 \begin{enumerate}
 	\item The $\Q$-scheme $\Sh_{K_\Phi}(Q_\Phi,\DD_\Phi)$ admits a canonical integral model 
	\[
	\mathcal{S}_{K_\Phi}(Q_\Phi,\DD_\Phi) \to \Spec( \Z_{(p)}),
	\] 
endowed with the structure of a  relative $\mathcal{T}_\Phi$-torsor 
\[
	\mathcal{S}_{K_\Phi}(Q_\Phi,\DD_\Phi)   \to \mathcal{S}_{\nu_\Phi(K_\Phi)}(\mathbb{G}_m,\mathcal{H}_0)
\] 
compatible with the torsor structure (\ref{torus torsor}) in the generic fiber.
	 	\item 
		There is a canonical isomorphism
 \[
 \widehat{\mathcal{S}}_{K_\Phi}(Q_\Phi,\DD_\Phi,\sigma)\iso\widehat{\mathcal{S}}_K(G,\DD,\Sigma)
 \]
 of formal algebraic spaces extending~\eqref{boundary iso}.
 \end{enumerate}
 Here $\mathcal{S}_{K_\Phi}(Q_\Phi,\DD_\Phi) \hookrightarrow \mathcal{S}_{K_\Phi}(Q_\Phi,\DD_\Phi,\sigma)$ is the partial compactification determined by the rational polyhedral cone 
 \[
 \sigma \subset U_\Phi(\R)(-1) = \Hom( \mathbb{G}_m,\mathcal{T}_\Phi )_\R
 \]
  and the formal scheme on the left hand side is its completion along its unique closed stratum.  On the right, 
\[
\widehat{\mathcal{S}}_K(G,\DD,\Sigma) =
 \mathcal{S}_K(G,\DD,\Sigma)^{\wedge}_{\mathcal{Z}^{  (\Phi,\sigma) }_K(G,\DD,\Sigma)}
\]
is the  formal  completion along the stratum indexed by $(\Phi,\sigma)$.
 \end{proposition}

\begin{proof}
This is a consequence of \cite[Theorem 4.1.5]{mp:compactification}.
\end{proof}

By \cite[Theorem 2]{mp:compactification} and \cite{FaltingsChai},
both  $\mathcal{S}_{K^\beef}(G^\beef,\DD^\beef,\Sigma^\beef)  $ and the Faltings-Chai compactification are  proper.
They admit  stratifications 
  \[
\mathcal{S}_{K^\beef}(G^\beef,\DD^\beef,\Sigma^\beef)  =
\bigsqcup_{ (\Phi^\beef,\sigma^\beef)  \in \mathrm{Strat}_{K^\beef}(G^\beef,\DD^\beef ,\Sigma^\beef) }
 \mathcal{Z}_{K^\beef}^{(\Phi^\beef,\sigma^\beef)}(G^\beef,\DD^\beef,\Sigma^\beef)  ,
\]
and
\[
 \mathcal{S}_{K^\Sg}(G^\Sg , \DD^\Sg,\Sigma^\Sg)  =
\bigsqcup_{ (\Phi^\Sg,\sigma^\Sg)  \in \mathrm{Strat}_{K^\Sg}(G^\Sg,\DD^\Sg) }
 \mathcal{Z}_{K^\Sg}^{  (\Phi^\Sg,\sigma^\Sg)  }(G^\Sg , \DD^\Sg,\Sigma^\Sg).
\]
analogous to (\ref{integral strata}).   
By \cite[(4.1.13)]{mp:compactification}, these stratifications satisfy a natural compatibility:  
if 
\[
( \Phi,\sigma) \in \mathrm{Strat}_K(G,\DD,\Sigma)
\]
 has images  $(\Phi^\beef,\sigma^\beef)$ and 
$(\Phi^\Sg,\sigma^\Sg)$, in the sense of \S \ref{ss:cone functoriality}, then the maps in~\eqref{eqn:comp maps} induce maps on strata
\[
\mathcal{Z}^{( \Phi,\sigma)}_K(G,\DD,\Sigma)
 \to \mathcal{Z}^{(\Phi^\beef,\sigma^\beef)}_{K^\beef}(G^\beef,\DD^\beef,\Sigma^\beef)
 \to \mathcal{Z}^{(\Phi^\Sg,\sigma^\Sg)}_{K^\Sg}(G^\Sg,\DD^\Sg,\Sigma^\Sg).
\]



Applying the functor of Proposition \ref{prop:total integral vector bundles} below to the $\mathcal{G}^\beef$-representation $V^\beef_{\Z_{(p)}}$ yields a line bundle $\bm{\omega}^\beef = F^1 \bm{V}^\beef_{dR}$ on $\mathcal{S}_{K^\beef} (G^\beef,\DD^\beef,\Sigma^\beef)$, which we pull back to a line bundle $\bm{\omega}$ on  $\mathcal{S}_K(G,\DD,\Sigma)$.   This gives an extension of   (\ref{integral omega})   to the toroidal compactification.

\begin{proposition}\label{prop:total integral vector bundles}
There is a  functor  
\[
N  \mapsto ( \bm{N}_{dR} , F^\bullet  \bm{N}_{dR} )
\]
 from representations  $\mathcal{G}^\beef \to \GL(N)$ on free   $\Z_{(p)}$-modules of finite rank to filtered vectors bundles on $\mathcal{S}_{K^\beef}(G^\beef,\DD^\beef,\Sigma^\beef)  $, extending the functor (\ref{int bundle machine}) on the open stratum, and the functor of Theorem \ref{thm:total vector bundles} in the generic fiber.
\end{proposition}

\begin{proof}
Consider the filtered vector bundle $(\bm{H}^\beef_{dR} , F^\bullet \bm{H}^\beef_{dR})$ over $\Sh_{K^\beef}(G^\beef,\DD^\beef)$ obtained by applying the functor (\ref{mixed bundles})  to the representation
\[
G^\beef \to G^\Sg=\mathrm{GSp}(H^\beef).
\]

Now let $\nu^\beef:G^\beef \to \mathbb{G}_m$ be the spinor similitude, and let $\Q(\nu^\beef)$ denote the corresponding one-dimensional representation of $G^\beef$. 
 It determines a line bundle on $\Sh_{K^\beef}(G^\beef,\DD^\beef)$, which is canonically a pullback via the morphism
\[
\Sh_{K^\beef}(G^\beef,\DD^\beef) \map{\nu^\beef} \Sh_{\nu^\beef( K^\beef) }  ( \mathbb{G}_m,\mathcal{H}_0).
\]
Combining this with Remark \ref{rem:tate bundles}, we see that the line bundle determined by $\Q(\nu^\beef)$  is canonically identified with $\Lie (\mathbb{G}_m)$, and hence the $G^\beef$-equivariant morphism  $\psi^\beef: H\otimes H \to \Q(\nu^\beef)$ induces an  alternating form 
\[
\psi^\beef : \bm{H}^\beef_{dR}  \otimes \bm{H}^\beef_{dR}  \to  \Lie (\mathbb{G}_m).
\]
The nontrivial step $F^0\bm{H}_{dR}^\beef$ in the filtration is a Lagrangian subsheaf with respect to this pairing.

The vector bundle $\bm{H}_{dR}^\beef$  is canonically identified with the pullback via
 \[
 \Sh_{K^\beef}(G^\beef,\DD^\beef) \to  \Sh_{K^{\Sg}} (G^\Sg , \DD^\Sg)
 \]
   of the first relative homology 
\[
\bm{H}^{\Sg}_{dR} = \underline{\Hom}  \big(R^1 \pi_*  \Omega^\bullet_{A^\Sg / \Sh_{K^{\Sg}} (G^\Sg , \DD^\Sg)} , \co_{\mathrm{Sh}_{K^{\Sg}} (G^\Sg , \DD^\Sg)} \big)
\]
of the universal polarized abelian scheme  $\pi : A^\Sg \to \Sh_{K^{\Sg}} (G^\Sg , \DD^\Sg)$.  
  As the universal abelian scheme extends canonically to the integral model,  so does $\bm{H}_{dR}^{\Sg}$.
    Its pullback defines  an extension of  $\bm{H}^\beef_{dR}$, along with its filtration and alternating form, to the integral model $\mathcal{S}_{K^\beef}(G^\beef , \DD^\beef)$.

Now fix a family of tensors 
\[
\{s_\alpha\} \subset  H_{\Z_{(p)}}^{\beef, \otimes}
\]
 that cut out the reductive subgroup 
 $
 \mathcal{G}^\beef \subset \mathcal{G}^\Sg .
 $   
 The functoriality of (\ref{mixed bundles}) implies  that these tensors define global sections $\{ \bm{s}_{\alpha,dR} \}$ of $\bm{H}_{dR}^{\beef,\otimes}$ over the generic fiber.  By  \cite[Corollary 2.3.9]{KisinJAMS}, they extend (necessarily uniquely) to sections over the integral model $\mathcal{S}_{K^\beef}(G^\beef , \DD^\beef)$.

By  \cite[Proposition 4.3.7]{mp:compactification}, the filtered vector bundle $(\bm{H}^\beef_{dR} , F^\bullet \bm{H}^\beef_{dR} )$  admits a canonical extension to  $\mathcal{S}_{K^\beef}(G^\beef , \DD^\beef,\Sigma^\beef)$.   
The alternating form $\psi^\beef$ and the sections $\bm{s}_{\alpha, dR}$  also extend (necessarily uniquely).

This allows us to define a $\mathcal{G}^\beef$-torsor
\[
\mathcal{J}_{K^\beef}(G^\beef , \DD^\beef,\Sigma^\beef) \map{a}\mathcal{S}_{K^\beef}(G^\beef , \DD^\beef,\Sigma^\beef)
\]
whose  functor of points assigns to a scheme $S\to \mathcal{S}_{K^\beef}(G^\beef , \DD^\beef,\Sigma^\beef)$ the set of all pairs $(f,f_0)$ of isomorphisms
\begin{equation}\label{torsor pair}
 f :  \bm{H}^\beef_{dR / S}  \iso  H^\beef_{\Z_{(p)}} \otimes \co_S ,\quad
 f_0: \Lie(\mathbb{G}_m)_{/S} \iso \co_S
\end{equation}
satisfying $f (   \bm{s}_{\alpha,dR}  ) = s_\alpha \otimes 1$ for all $\alpha$, and making the diagram
\[
\xymatrix{
{    \bm{H}^\beef_{dR / S}\otimes \bm{H}^\beef_{dR / S}      } \ar[rrr]^{\psi^\beef}  \ar[d]_{f\otimes f} &&& { \Lie(\mathbb{G}_m) }    \ar[d]^{f_0} \\
{   ( H^\beef_{\Z_{(p)}} \otimes \co_S ) \otimes (  H^\beef_{\Z_{(p)}} \otimes \co_S  ) } \ar[rrr]^{ \psi^\beef } &&& {  \co_S} 
}
\]
commute.

Define  smooth $\Z_{(p)}$-schemes  $\check{\mathcal{M}}^\beef$ and $\check{\mathcal{M}}^\Sg$ with  functors of points
\begin{align*}
\check{\mathcal{M}} (G^\beef,\DD^\beef) (S)  & = \{ \mbox{isotropic lines } z \subset V_{\Z_{(p)}}^\beef \otimes  \co_S \} \\
\check{\mathcal{M}} (G^\Sg,\DD^\Sg) (S)  & = \{ \mbox{Lagrangian subsheaves\ } F^0  \subset H_{\Z_{(p)}}^\beef \otimes  \co_S \}.
\end{align*}
These are integral models  of the  compact duals  $\check{M}(G^\beef,\DD^\beef)$ and  $\check{M}(G^\Sg,\DD^\Sg)$ of \S \ref{ss:hodge embedding},  and are related, using (\ref{special injection}), by a closed immersion
\begin{equation}\label{pre b dual}
\check{\mathcal{M}}  (G^\beef,\DD^\beef)   \to \check{\mathcal{M}} (G^\Sg,\DD^\Sg)
\end{equation}
sending the isotropic line $z\subset V_{\Z_{(p)}}^\beef$  to the Lagrangian 
$
zH_{\Z_{(p)}}^\beef \subset H_{\Z_{(p)}}^\beef.
$

We now have a diagram
\begin{equation}\label{integral local model}
\xymatrix{
 {  \mathcal{J}_{K^\beef}(G^\beef , \DD^\beef,\Sigma^\beef)}   \ar[d]_a \ar[r]^{ b }  &    { \check{\mathcal{M}}   (G^\beef,\DD^\beef)}  \\
{\mathcal{S}_{K^\beef}(G^\beef , \DD^\beef,\Sigma^\beef)}  
}
\end{equation}
in which $a$ is a $\mathcal{G}^\beef$-torsor and $b$ is $\mathcal{G}^\beef$-equivariant, extending the diagram  (\ref{general compact local model}) already constructed in the generic fiber.   To define the morphism $b$ we first define a morphism  
\[
 \mathcal{J}_{K^\beef}(G^\beef , \DD^\beef,\Sigma^\beef) \to  \check{\mathcal{M}} (G^\Sg,\DD^\Sg)
 \]
by sending an $S$-point $(f,f_0)$ to the Lagrangian subsheaf
\[
f( F^0 \bm{H}_{dR /S }  ) \subset  H_{\Z_{(p)}} \otimes \co_S.
\]
This morphism factors through (\ref{pre b dual}).   Indeed, as   (\ref{pre b dual}) is a closed immersion,  this is a formal consequence of the fact that we have such a factorization in the generic fiber, as can be checked using the analogous complex analytic construction.

With the diagram (\ref{integral local model}) in hand, the construction of the desired functor proceeds by simply imitating the construction (\ref{mixed bundles}) used in the generic fiber.
\end{proof}


\subsection{Integral $q$-expansions}


Continue with the assumptions of \S \ref{ss:integral toroidal}, and now fix a toroidal stratum representative 
\[
 ( \Phi , \sigma)  \in \mathrm{Strat}_K(G,\DD,\Sigma)
\] 
as in \S \ref{ss:qs}.  Thus $\Phi = ( P, \DD^\circ, h)$ with $P$ the stabilizer of an isotropic line $I\subset V$,  and $\sigma \in \Sigma_\Phi$ is a  top dimensional rational polyhedral cone.   Let 
\[
( \Phi^\beef , \sigma^\beef) \in \mathrm{Strat}_{K^\beef}(G^\beef,\DD^\beef,\Sigma^\beef)
\] 
be the image of $ ( \Phi , \sigma)$, in the sense of \S \ref{ss:cone functoriality}.

 The formal completions along the corresponding  strata 
 \begin{align}\label{integral cusp strata}
\mathcal{Z}^{ ( \Phi , \sigma) }_K(G,\DD,\Sigma) & \subset \mathcal{S}_K(G,\DD,\Sigma) \\ 
\mathcal{Z}^{ ( \Phi^\beef , \sigma^\beef) }_{K^\beef}(G^\beef , \DD^\beef , \Sigma^\beef)  & \subset \mathcal{S}_{K^\beef}(G^\beef , \DD^\beef , \Sigma^\beef), \nonumber
\end{align}
 are denoted 
 \begin{align*}
 \widehat{ \mathcal{S} }_K(G,\DD,\Sigma)  
 &= \mathcal{S}_K(G,\DD,\Sigma)^\wedge _{  \mathcal{Z}^{(\Phi,\sigma)}_K (G,\DD,\Sigma)   } \\
 \widehat{ \mathcal{S} }_{K^\beef}(G^\beef , \DD^\beef , \Sigma^\beef)  
 & =  \mathcal{S}_{K^\beef}(G^\beef , \DD^\beef , \Sigma^\beef)^\wedge   _{  \mathcal{Z}^{ (\Phi^\beef,\sigma^\beef)}_{K^\beef} (G^\beef,\DD^\beef,\Sigma^\beef)   }   .
\end{align*}
These are formal algebraic spaces over $\Z_{(p)}$ related by a finite morphism 
\begin{equation}\label{integral boundary map}
\widehat{ \mathcal{S} }_K(G,\DD,\Sigma)       \to \widehat{ \mathcal{S} }_{K^\beef}(G^\beef , \DD^\beef , \Sigma^\beef)  .
\end{equation}

Fix a $\Z_{(p)}$-module generator $\ell \in I \cap V_{\Z_{(p)}}$.  Recall from the discussion leading to (\ref{basic q}) that such an $\ell$ determines an isomorphism 
\[
 [\bm{\ell}^{\otimes k} , -   ]    :   \bm{\omega}^{\otimes k} \to \co_{  \widehat{\Sh}_K   ( G,  \DD ,\Sigma )  } 
\]
of line bundles on $  \widehat{\Sh}_K   ( G,  \DD ,\Sigma ) $.

\begin{proposition}\label{prop:integral canonical sections}
The above isomorphism   extends uniquely to an isomorphism 
\[
 [\bm{\ell}^{\otimes k} , -  ]    :   \bm{\omega}^{\otimes k} \to \co_{ \widehat{ \mathcal{S} }_K(G,\DD,\Sigma)    } 
\]
of line bundles on the integral model $ \widehat{ \mathcal{S} }_K(G,\DD,\Sigma) $.
\end{proposition}

\begin{proof}
The maximality of  $V_{\Z_p}$ implies that $V_{\Z_p} \subset V_{\Z_p}^\beef$  is a $\Z_p$-module direct summand.  In particular, 
\[
I_{\Z_{(p)} }= I \cap V_{\Z_{(p)} } = I \cap V^\beef_{\Z_{(p)} } \subset V^\beef_{\Z_{(p)}}
\]
is a $\Z_{(p)}$-module direct summand generated by $\ell$.  Because  $\bm{\omega}$ is defined as the pullback of $\bm{\omega}^\beef$, and because the uniqueness part of the claim is obvious, it suffices to construct an isomorphism
\begin{equation}\label{integral magic section}
[\bm{\ell} ,  -   ] : \bm{\omega}^\beef \to \co_{ \widehat{ \mathcal{S} }_{K^\beef}(G^\beef,\DD^\beef,\Sigma^\beef)    } 
\end{equation}
extending the one in the generic fiber,  and then pull back along (\ref{integral boundary map}).

We return to the notation of the proof of Proposition \ref{prop:total integral vector bundles}.
 Let $\mathcal{P}^\beef \subset \mathcal{G}^\beef$ be the stabilizer of the isotropic line  $I_{\Z_{(p)}} \subset V^\beef_{\Z_{(p)}}$,  define a
  $\mathcal{P}^\beef$-stable weight filtration
\[
\mathrm{wt}_{-3}H^\beef_{\Z_{(p)}}=0,\quad \mathrm{wt}_{-2}H^\beef_{\Z_{(p)}} = \mathrm{wt}_{-1}H^\beef_{\Z_{(p)}}= I_{\Z_{(p)}} H^\beef_{\Z_{(p)}},
\quad \mathrm{wt}_{0}H^\beef_{\Z_{(p)}} = H^\beef_{\Z_{(p)}},
\]
 and set 
\[
\mathcal{Q}^\beef_\Phi = \mathrm{ker} \big(   \mathcal{P}^\beef \to \GL (   \mathrm{gr}_0(H^\beef_{\Z_{(p)} }) )    \big).
\]
Compare with the discussion of \S \ref{ss:orthogonal clr lines}.

The   $\Z_{(p)}$-schemes  of (\ref{pre b dual})  sit in a commutative diagram
\[
\xymatrix{
{  \check{\mathcal{M}}^\beef_\Phi  }  \ar[r]  \ar[d]  &  {   \check{\mathcal{M}}^\Sg_\Phi  }  \ar[d]  \\
{  \check{\mathcal{M}}(G^\beef,\DD^\beef)  } \ar[r]   &   {  \check{\mathcal{M}}(G^\Sg , \DD^\Sg)  } 
}
\]
in which the horizontal arrows are closed immersions, and the vertical arrows are open immersions. 
 The $\Z_{(p)}$-schemes  in the top row are defined by their  functors of points, which are
\[
\check{\mathcal{M}}_\Phi^\beef (S) = 
\left\{ \begin{array}{c} \mbox{isotropic lines }  z \subset V^\beef_{\Z_{(p)}} \otimes  \co_S \mbox{ such that}  \\
V^\beef_{\Z_{(p)}} \to V^\beef_{\Z_{(p)}}/I_{\Z_{(p)}}^\perp \\ \mbox{ identifies }  z \iso (V^\beef_{\Z_{(p)}}/I_{\Z_{(p)}}^\perp) \otimes \co_S \end{array}  \right\}
\]
and
\[
\check{\mathcal{M}}_\Phi^\Sg (S) = 
\left\{ \begin{array}{c} \mbox{Lagrangian subsheaves }  F^0 \subset H^\beef_{\Z_{(p)}} \otimes  \co_S \mbox{ such that}  \\
H^\beef_{\Z_{(p)}} \to  \mathrm{gr}_0 ( H^\beef_{\Z_{(p)} }  )   \\ \mbox{ identifies }  F^0 \iso \mathrm{gr}_0 ( H^\beef_{\Z_{(p)} }  )  \otimes \co_S \end{array}  \right\}.
\]

Passing to formal completions, the diagram (\ref{integral  local model})  determines a diagram 
\begin{equation}\label{integral boundary torsor}
\xymatrix{
 {  \widehat{ \mathcal{J} } _{K^\beef} (G^\beef,\DD^\beef, \Sigma^\beef) }   \ar[d]_a \ar[rr]^{ b }  &  &    { \check{\mathcal{M}} (G^\beef,\DD^\beef) }  \\
{ \widehat{\mathcal{S}}_{K^\beef} (G^\beef , \DD^\beef , \Sigma^\beef) }  
}
\end{equation}
of formal algebraic spaces over $\Z_{(p)}$, in which $a$ is a $\mathcal{G}^\beef$-torsor and $b$ is $\mathcal{G}^\beef$-equivariant, and $\widehat{ \mathcal{J} } _{K^\beef} (G^\beef,\DD^\beef, \Sigma^\beef) $ is the formal completion of $\mathcal{J}_{K^\beef}(G^\beef,\DD^\beef, \Sigma^\beef)$ along the fiber over the stratum  (\ref{integral cusp strata}).

\begin{lemma}
The $\mathcal{G}^\beef$-torsor in  (\ref{integral boundary torsor}) admits a canonical reduction of structure to a $\mathcal{Q}^\beef_\Phi$-torsor $\mathcal{J}_\Phi ^\beef$, sitting in a diagram
\[
\xymatrix{
 {  \mathcal{J}_\Phi ^\beef}   \ar[d]_a \ar[r]^{ b }  &    { \check{\mathcal{M}}^\beef_\Phi    }  \\
{ \widehat{\mathcal{S}}_{K^\beef} (G^\beef , \DD^\beef , \Sigma^\beef)   }.
}
\]
\end{lemma}

\begin{proof}
The essential point is that the filtered vector bundle $(\bm{H}_{dR}^\beef, F^\bullet\bm{H}_{dR}^\beef)$ on  $\mathcal{S}_{K^\beef} (G^\beef , \DD^\beef , \Sigma^\beef)$ used in the construction of the $\mathcal{G}^\beef$-torsor 
\[
\mathcal{J}_{K^\beef} (G^\beef , \DD^\beef , \Sigma^\beef)  \to \mathcal{S}_{K^\beef} (G^\beef , \DD^\beef , \Sigma^\beef) 
\] 
acquires extra structure after restriction to   $\widehat{\mathcal{S}}_{K^\beef} (G^\beef , \DD^\beef , \Sigma^\beef) $.  Namely, it acquires a weight filtration 
\[
\mathrm{wt}_{-3}\bm{H}_{dR}^\beef=0,\quad \mathrm{wt}_{-2}\bm{H}_{dR}^\beef = \mathrm{wt}_{-1}\bm{H}_{dR}^\beef, \quad \mathrm{wt}_{0}\bm{H}_{dR}^\beef = \bm{H}_{dR}^\beef,
\]
along with distinguished isomorphisms
\begin{align*}
\mathrm{gr}_{-2}  ( \bm{H}_{dR}^\beef )   & \iso \mathrm{gr}_{-2} ( H^\beef_{\Z_{(p)}}  ) \otimes  \Lie (\mathbb{G}_m)  \\
\mathrm{gr}_{0}  ( \bm{H}_{dR}^\beef  )  & \iso \mathrm{gr}_0  (  H^\beef_{\Z_{(p)}} )\otimes \co_{  \widehat{\mathcal{S}}_{K^\beef}(G^\beef,\DD^\beef,\Sigma^\beef)}.
\end{align*}
This follows from the discussion of \cite[(4.3.1)]{mp:compactification}. The essential point is that over the formal completion $\widehat{\mathcal{S}}_{K^\beef} (G^\beef , \DD^\beef , \Sigma^\beef) $ there is a canonical degenerating abelian scheme, and the desired extension of $\bm{H}_{dR}^\beef$ is its de Rham realization. 
The extension of the weight and Hodge filtrations is also a consequence of this observation; see~\cite[\S 1]{mp:compactification}, and in particular \cite[Proposition 1.3.5]{mp:compactification}.

The desired reduction of structure 
$
\mathcal{J}_\Phi ^\beef   \subset  \widehat{ \mathcal{J} } _{K^\beef} (G^\beef,\DD^\beef, \Sigma^\beef)
$
is now defined as the closed formal algebraic subspace parametrizing pairs of isomorphisms $(f,f_0)$ as in (\ref{torsor pair}) that respect this additional structure.

Moreover, after restricting $\bm{H}_{dR}^\beef$ to  $ \widehat{\mathcal{S}}_{K^\beef} (G^\beef , \DD^\beef , \Sigma^\beef) $, the surjection $\bm{H}_{dR}^\beef \to \mathrm{gr}_0 \bm{H}_{dR}^\beef$  identifies $F^0 \bm{H}_{dR}^\beef \iso \mathrm{gr}_0 \bm{H}_{dR}^\beef$.  Indeed, in the language of~\cite[\S 1]{mp:compactification}, this just amounts to the observation that the de Rham realization of a $1$-motive with trivial abelian part has trivial weight and Hodge filtrations.

As the composition 
\[
\mathcal{J}_\Phi ^\beef    \subset \widehat{ \mathcal{J} } _{K^\beef} (G^\beef,\DD^\beef, \Sigma^\beef)  
 \map{b} \check{\mathcal{M}} (G^\beef , \DD^\beef) \subset \check{\mathcal{M}} (G^\Sg ,\DD^\Sg)
\]
sends $(f,f_0) \mapsto f( F^0 \bm{H}_{dR}^\beef)$, it takes values in the open subscheme 
\[
\check{\mathcal{M}}_\Phi^\Sg \subset \check{\mathcal{M}} ( G^\Sg , \DD^\Sg ).
\]  
It therefore take values in the closed subscheme $\check{\mathcal{M}}_\Phi^\beef \subset \check{\mathcal{M}}_\Phi^\Sg$, as this can be checked in the generic fiber, where it follows from the analogous complex analytic constructions.
\end{proof}

Returning to the main proof,  let $\check{I} \subset \check{V}^\beef$ be the constant $\mathcal{Q}_\Phi^\beef$-equivariant line bundles on  $\check{\mathcal{M}}^\beef_\Phi$ determined by the representations $I_{\Z_{(p)}} \subset V^\beef_{\Z_{(p)}}$, 
and let $\check{\omega}^\beef\subset \check{V}^\beef$ be the tautological line bundle.
The self-duality of $V^\beef_{\Z_{(p)}}$ guarantees that the bilinear pairing on $\check{V}^\beef$ restricts to an isomorphism
\[
[  -  ,  -   ] : \check{I}\otimes \check{\omega}^\beef\to \co_{\check{\mathcal{M}}^\beef_\Phi}.
\]
Pulling back these line bundles to $\mathcal{J}_\Phi^\beef$ and taking the quotient by $\mathcal{Q}_\Phi^\beef$, we obtain an isomorphism
\[
[   -   , - ] : \bm{I}_{dR} \otimes \bm{\omega}^\beef \to \co_{  \widehat{ \mathcal{S} }_{K^\beef} (G^\beef, \DD^\beef , \Sigma^\beef ) }
\]
of line bundles on $  \widehat{ \mathcal{S} }_{K^\beef} (G^\beef, \DD^\beef , \Sigma^\beef ) $.

On the other hand, the action of $\mathcal{Q}^\beef_\Phi$ on $I_{\Z_{(p)}}$ is through the character $\nu^\beef_\Phi$, which agrees with the restriction of $\nu^\beef:\mathcal{G}^\beef \to \mathbb{G}_m$ to $\mathcal{Q}_\Phi^\beef$.  The canonical morphism
\[
\mathcal{J}^\beef_\Phi   \to \widehat{\mathcal{J}}_{K^\beef}(G^\beef , \DD^\beef ,\Sigma^\beef)   \map{(f,f_0) \mapsto f_0 }
  \underline{\mathrm{Iso}} \big( \Lie(\mathbb{G}_m) , \co_{ \widehat{ \mathcal{S} }_{K^\beef} (G^\beef, \DD^\beef , \Sigma^\beef )  }  \big)
\]
of formal algebraic spaces over $ \widehat{ \mathcal{S} }_{K^\beef} (G^\beef, \DD^\beef , \Sigma^\beef )$ identifies
$\mathrm{ker}( \nu^\beef_\Phi )   \backslash  \mathcal{J}^\beef_\Phi$ with the trivial $\mathbb{G}_m$-torsor 
\[
 \underline{\mathrm{Iso}} \big( \Lie(\mathbb{G}_m) , \co_{ \widehat{ \mathcal{S} }_{K^\beef} (G^\beef, \DD^\beef , \Sigma^\beef )  }  \big)
 \iso 
  \underline{\mathrm{Aut}} \big(  \co_{ \widehat{ \mathcal{S} }_{K^\beef} (G^\beef, \DD^\beef , \Sigma^\beef )  }  \big) 
\]
over $ \widehat{ \mathcal{S} }_{K^\beef} (G^\beef, \DD^\beef , \Sigma^\beef )$.  As the action of $\mathcal{G}^\diamond$ on $I_{\Z_{(p)}}$ is via  $\nu_\Phi^\diamond$, this trivialization fixes an isomorphism
\begin{align*}
\bm{I}_{dR}    &  = \mathcal{Q}_\Phi^\beef \backslash \big( I_{\Z_{(p)}} \otimes \co_{  \mathcal{J}_\Phi^\beef  }  \big)   \\
&  = \mathbb{G}_m \backslash \big( I_{\Z_{(p)}} \otimes \co_{  \mathrm{ker}(\nu_\Phi^\beef) \backslash  \mathcal{J}_\Phi^\beef  }  \big)   \\
&  \iso   I_{\Z_{(p)}} \otimes \co_{   \widehat{ \mathcal{S} }_{K^\beef} (G^\beef, \DD^\beef , \Sigma^\beef )  }.
\end{align*}

The generator $\ell \in I_{\Z_{(p)}}$ now determines a trivializing section $\bm{\ell}=\ell\otimes 1$ of $\bm{I}_{dR}$, defining the desired  isomorphism (\ref{integral magic section}).  This completes the proof of Proposition \ref{prop:integral canonical sections}.
\end{proof}

Let   $I_*\subset V$ and 
 \[
\xymatrix{
(Q_\Phi , \DD_\Phi)    \ar[r]_{ \nu_\Phi } & ( \mathbb{G}_m , \mathcal{H}_0)  \ar@/_1pc/[l]_{\spl}.
}
\] 
 be as in the discussion preceding Proposition \ref{prop:torsor splitting}.
 Choose a compact open subgroup $K_0 \subset \A_f^\times$ small enough that $s(K_0) \subset K_\Phi$, and assume that $K_0$ factors as 
 \[
 K_0 = \Z_p^\times \cdot K_0^p.
 \]
   Let $F/\Q$ be the abelian extension of $\Q$ determined by 
\[
 \mathrm{rec} : \Q^\times_{>0} \backslash \A_f^\times/K_0  \iso \Gal(F  / \Q).
\]
Fix a prime $\mathfrak{p}\subset \co_F$ above $p$, and let $R\subset F$ be the localization of $\co_F$ at $\mathfrak{p}$.  Note that the above assumption on $K$ implies that $p$ is unramified in $F$.

\begin{proposition}
\label{prop:formal cusp nbhd}
If we set
\[
\widehat{\mathcal{T}}_\Phi(\sigma) =  \Spf \Big(  \Z_{(p)}  [[ q_\alpha]]_{  \substack{  \alpha \in \Gamma_\Phi^\vee(1)  \\ \langle \alpha, \sigma \rangle \ge 0   }  }   \Big),
\]
there is a unique  morphism 
\[
 \bigsqcup_{ a\in  \Q^\times_{>0} \backslash \A_f^\times / K_0  }   \widehat{\mathcal{T}}_ { \Phi} (\sigma)_{/R}
\to   \widehat{ \mathcal{S} }_K(G,\DD,\Sigma)_{/R}
\]
of formal algebraic spaces over  $R$ whose base change to $\C$ agrees with the morphism of Proposition \ref{prop:torsor splitting}.  Moreover, if $t$ is any point of the source and $s$ is its image in $\widehat{ \mathcal{S} }_K(G,\DD,\Sigma)_{/R}$, the induced map on \'etale local rings $\co^{et}_s\to \co^{et}_t$ is faithfully flat.
\end{proposition}
 
\begin{proof}
The uniqueness of such a morphism is clear. We have to show existence. The proof of this proceeds just as that of Proposition~\ref{prop:torsor splitting}, except that it uses Proposition~\ref{prop:strata descp} as input. The only additional observation required is that we have an isomorphism
\begin{equation}\label{eqn:gm integral points}
\bigsqcup_{a\in\Q^\times_{>0}\backslash \A_f^\times/K_0} \Spec(R)
\iso
\mathcal{S}_{K_0}(\mathbb{G}_m,\mathcal{H}_0)_{/R}
\end{equation}
of $R$-schemes, which realizes \eqref{eqn:gm shimura points} on $\C$-points. 
Here $\mathcal{S}_{K_0}(\mathbb{G}_m,\mathcal{H}_0)$ is defined as the normalization of $\Spec(\Z_{(p)})$ in $\Sh_{K_0}(\mathbb{G}_m,\mathcal{H}_0)$.

To see this, note that the defining property of canonical models provides an isomorphism
\[
\Spec(F) \iso  \Sh_{K_0}(\mathbb{G}_m,\mathcal{H}_0)
\] 
of $\Q$-schemes, and hence an isomorphism  $F$-schemes
\[
\bigsqcup_{a\in \Q^\times_{>0}\backslash \A_f^\times/K_0}\Spec(F) 
\iso \Sh_{K_0}(\mathbb{G}_m,\mathcal{H}_0)_{/F}.
\]
Using the fact that   $p$ is unramified in $F$, 
one can see that this isomorphism extends to  \eqref{eqn:gm integral points}.
\end{proof}

Suppose $\psi$ is a section of the line bundle $\bm{\omega}^{\otimes k}$ on $\Sh_K(G,\DD)_{/F}$.  It follows from  Proposition \ref{prop:q principle} that the $q$-expansion  (\ref{basic q}) of $\psi$ has coefficients in $F$ for every $a\in \A_f^\times$.     If we view $\psi$ as a rational section  on $\mathcal{S}_K(G,\DD,\Sigma)_{/R}$,  the following result gives a criterion for testing flatness of its divisor.

\begin{corollary}\label{cor:flatness by q}
Assume that the special fiber of $\mathcal{S}_K(G,\DD)_{/R}$ is geometrically normal, and for every  $a\in \A_f^\times$  the $q$-expansion (\ref{basic q})  satisfies
\[
\mathrm{FJ}^{(a)} ( \psi )   \in    R [[q_\alpha]]_{  \substack{    \alpha \in \Gamma_\Phi^\vee(1)     \\      \langle \alpha, \sigma \rangle \ge 0     }  }.
\]
 If this $q$-expansion is nonzero modulo $\mathfrak{p}$ for all $a$,  then  $\mathrm{div}(\psi)$  is $R$-flat.  
\end{corollary}

\begin{proof}
As  $\mathcal{S}_K(G,\mathcal{D},\Sigma)_{/R}$ is flat over $R$, to show that $\mathrm{div}(\psi)$ is $R$-flat  it is enough to show that its support  does not contain any irreducible components of the special fiber of $\mathcal{S}_K(G,\mathcal{D},\Sigma)_{/R}$.

Every connected component 
\[
C\subset \mathcal{S}_K(G,\DD,\Sigma)_{/R}.
\] 
 has irreducible special fiber.  Indeed, we have assumed that the special fiber of $\mathcal{S}_K(G,\DD)_{/R}$ is geometrically normal. 
  It therefore follows  from  \cite[Theorem 1]{mp:compactification} that the special fiber of $\mathcal{S}_K(G,\DD,\Sigma)_{/R}$ is also geometrically normal.  
 On the other hand,  \cite[Corollary 4.1.11]{mp:compactification} shows that $C$   has geometrically connected special fiber.   Therefore the special fiber of $C$ is both connected and normal, and hence is irreducible.

As in the proof of Proposition \ref{prop:q principle}, the closed stratum 
\[
 \mathcal{Z}_K^{(\Phi,\sigma)} ( G,\DD,\Sigma)_{/R} \subset \mathcal{S}_K(G,\DD,\Sigma)_{/R}
\]
meets every connected component.  Pick a closed point $s$ of this stratum lying on the connected component $C$.
By the definition of $\mathrm{FJ}^{(a)}(\psi)$, and from Proposition~\ref{prop:formal cusp nbhd}, our hypothesis on the $q$-expansion implies that the restriction of $\psi$ to the completed local ring $\mathcal{O}_s$ of $s$ defines a rational section of $\bm{\omega}^{\otimes k}$ 
whose divisor is an $R$-flat Cartier divisor on $\Spf( \mathcal{O}_s )$.

  It follows that $\mathrm{div}(\psi)$ does not contain the special fiber of $C$, and varying $C$ shows that $\mathrm{div}(\psi)$ contains no irreducible components of the special fiber of $\mathcal{S}_K(G,\mathcal{D},\Sigma)_{/R}$.
\end{proof}

\begin{remark}
If $V_{\Z_p}$ is almost self-dual, then $\mathcal{S}_K(G,\DD)$ is smooth over $\Z_{(p)}$, and hence has geometrically normal special fiber.
Without the assumption of almost self-duality,  Proposition~\ref{prop:normal fiber} tells us that the special fiber is geometrically normal whenever $n\geq 6$.
\end{remark}


\section{Borcherds products on integral models}
\label{s:integral borcherds}


Keep $V_\Z \subset V$ of signature $(n,2)$ with $n\ge 1$, and  let $(G,\DD)$ be the associated GSpin Shimura datum.  As in the introduction, let $\Omega$ be a finite set of prime numbers containing all primes at which $V_\Z$ is not maximal, and choose (\ref{K choice}) to be factorizable $K=\prod_p K_p$ with
\[
K_p= G(\Q_p) \cap C( V_{\Z_p} )^\times
\]
for all $p\not\in \Omega$.  Set $\Z_\Omega= \Z[ 1/p : p\in \Omega]$.


\subsection{Statement of the main result}


In  \S  \ref{ss:gspin integral} and \S \ref{ss:special divisors} we constructed, for every prime $p\not\in \Omega$, 
 an integral model over $\Z_{(p)}$  of the  Shimura variety $\Sh_K(G,\DD)$,  along with   a family of special divisors and a line bundle of weight one modular forms.
As explained in \cite[\S 2.4]{AGHMP-1} and \cite[\S 4.5]{AGHMP-2},  as $p$ varies these models arise as the localizations of a  flat and normal integral model
\[
\mathcal{S}_K(G,\DD) \to \Spec(\Z_\Omega), 
\]
 endowed with a family of special divisors $\mathcal{Z}(m,\mu)$ indexed by positive $m\in \Q$ and $\mu \in L^\vee / L$,  and a line bundle of weight one modular forms $\bm{\omega}$.

\begin{theorem}\label{thm:main borcherds}
 Suppose  
\[
f(\tau)  = \sum_{  \substack{ m\in \Q \\ m \gg -\infty} } c(m) \cdot q^m   \in M^!_{1- \frac{n}{2} }( \overline{\rho}_{V_\Z } )
\]
is a weakly holomorphic form as in   (\ref{input form}), and assume $f$ is integral in the sense of Definition \ref{def:integral form}.  
  After multiplying  $f$ by any sufficiently divisible positive integer, there is a rational section $\psi(f)$  of $\bm{\omega}^{ \otimes c(0,0) }$  over $\mathcal{S}_K(G,\DD)$ whose norm under the  metric  (\ref{better metric}) is related to the regularized theta lift   of \S \ref{ss:borcherds definition} by 
  \begin{equation}\label{norm match}
    -2\log \| \psi(f) \|       =  \regtheta (f) ,
\end{equation}
and whose divisor is 
   \begin{equation}\label{divisor match}
  \mathrm{div}( \psi (f) ) =  
  \sum_{  \substack{  m > 0 \\ \mu\in V_\Z^\vee / V_\Z }  }  c(-m,\mu) \cdot \mathcal{Z}(m,\mu). 
  \end{equation}
  \end{theorem}

The remainder of this subsection is devoted to proving Theorem \ref{thm:main borcherds} under some  restrictive hypotheses on the pair $V_\Z \subset V$.  
These  will allow us to deduce algebraicity of the Borcherds product  from Proposition \ref{prop:borcherds algebraic},  prove its descent to $\Q$ using the $q$-expansion principle of Proposition \ref{prop:q principle}, and  deduce the equality of divisors (\ref{divisor match}) from the flatness of both sides over $\Z_\Omega$.

\begin{proposition}\label{prop:almost there}
If  $n\ge 6$, and if  there exists an $h\in G(\A_f)$ and isotropic vectors 
 $
 \ell , \ell_*  \in h V_\Z
 $
  such that $[\ell ,\ell_*] =1$, then Theorem \ref{thm:main borcherds} holds.
\end{proposition}

\begin{proof}
It suffices to treat the case where 
\[
K = G(\A_f) \cap C(V_{\widehat{\Z}})^\times, 
\]
 for then we can pull back $\psi(f)$  to any smaller level structure.

The vectors  $\ell,\ell_* \in V$ satisfy the relation  (\ref{second isotropic}) with $k=\ell_*$.
Let $I$ and $I_*$ be the  isotropic lines  in $V$ spanned by $\ell$ and $\ell_*$, respectively.
 Let $P$ be the stabilizer of $I$, and let $\DD^\circ \subset \DD$ be a connected component.  This determines a cusp label representative 
 \[
 \Phi = (P,\DD^\circ, h).
 \]
 
Although we will not use this fact explicitly, the following lemma implies that the $0$-dimensional stratum of the Baily-Borel compactification $\Sh_K(G,\DD)^\mathrm{BB}$ indexed by $\Phi$ is geometrically connected.
In other words, Baily-Borel compactification has a cusp defined over $\Q$.

\begin{lemma}\label{lem:optimal section}
The complex orbifold $\Sh_K(G,\DD)(\C)$ is connected, and the  section (\ref{component section}) determined by $I_*$ satisfies  
\begin{equation}\label{optimal section}
\spl(\widehat{\Z}^\times) \subset K_\Phi.
\end{equation}
\end{lemma}

\begin{proof}
We first prove (\ref{optimal section}).
Consider the hyperbolic place 
\[
W= \Q\ell + \Q\ell_* \subset V.
\]
Its corresponding spinor similitude group $\GSpin(W)$ is just the unit group of the even Clifford algebra $C^+(W)$.
The natural inclusion $\GSpin(W) \to G$  takes values in the subgroup $Q_\Phi$, and the cocharacter (\ref{component section}) factors as
\[
\mathbb{G}_m \map{s} \GSpin(W) \to Q_\Phi 
\]
where the first arrow sends $a\in \Q^\times $ to  
\[
s(a)=a^{-1} \ell_*\ell+  \ell\ell_* \in C^+(W)^\times.
\]
From this explicit formula and the inclusion
\[
H_{ \widehat{\Z} }  =\widehat{\Z} \ell \oplus \widehat{\Z} \ell_* \subset h V_{\widehat{\Z}}, 
\]
 it is clear that   (\ref{component section}) satisfies
\[
s( \widehat{\Z}^\times) \subset
 C^+( W_{ \widehat{\Z} } )^\times \subset Q_\Phi (\A_f) \cap C(h V_{\widehat{\Z}})^\times =  K_\Phi.
\]

Now we prove the connectedness claim.
From (\ref{optimal section}) it follows that
\[
   \widehat{\Z}^\times =    \nu_\Phi(s ( \widehat{\Z}^\times ))    \subset     \nu_\Phi(K_\Phi) \subset \nu ( K) ,
\]
 and  hence the $0$-dimensional Shimura variety 
\[
\Sh_{\nu(K)}(\mathbb{G}_m , \mathcal{H}_0)(\C) = \Q^\times \backslash \mathcal{H}_0 \times \A_f^\times / \nu ( K)
\]
consists of  a single point.   The proof of Proposition \ref{prop:q principle} shows that the fibers of 
\[
\Sh_K(G,\DD)(\C) \to \Sh_{\nu(K)}(\mathbb{G}_m , \mathcal{H}_0)(\C)
\]
are connected, completing the proof.
\end{proof}

Applying Theorem \ref{thm:old borcherds} and  Proposition  \ref{prop:borcherds algebraic}  to the form $2f$ gives us a rational section 
\begin{equation}\label{almost there borcherds}
\psi(f) = (2\pi i)^{ c(0,0)} \Psi(2 f)
\end{equation}
of $\bm{\omega}^{ \otimes c(0,0)}$ over $\Sh_K(G,\DD)_{/\C}$.
We first prove that   $\psi(f)$ can be rescaled by a constant of absolute value $1$  to make it  defined over $\Q$.

Fix a neat compact open subgroup $\tilde{K} \subset K$ small enough  that there  is a $\tilde{K}$-admissible complete cone decomposition $\Sigma$ for $(G,\DD)$  satisfying the conclusion of Lemma \ref{lem:good decomposition}.
In particular, we have a top-dimensional rational polyhedral cone $\sigma \in \Sigma_\Phi$ whose interior is contained in a fixed Weyl chamber 
\[
\mathscr{W} \subset \mathrm{LightCone}^\circ(V_{0\R}) \iso C_\Phi.
\]

Let  $\tilde{\psi}(f)$ denote the pullback of $\psi(f)$ to  $\Sh_{\tilde{K}} (G,\DD,\Sigma)_{/\C}$.
Recalling the construction of $q$-expansions of (\ref{basic q}),
the  toroidal stratum representative 
\[
(\Phi , \sigma)  \in \mathrm{Strat}_{\tilde{K}}(G,\DD,\Sigma)
\]
determines  a collection of formal $q$-expansions
\begin{equation}\label{lifted q}
\mathrm{FJ}^{(a)} (\tilde{\psi}(f) ) \in  \C[[ q_\alpha]]_{  \substack{  \alpha \in \Gamma_\Phi^\vee(1)  \\ \langle \alpha, \sigma \rangle \ge 0   } }
\end{equation}
indexed by  $a \in \Q^\times_{>0} \backslash \A_f^\times / \tilde{K}_0$, where $\tilde{K}_0 \subset \mathbb{G}_m(\A_f)$ is chosen small enough that its image under  (\ref{component section}) is contained in $\tilde{K}_\Phi$.

We can read off these $q$-expansions from Proposition \ref{prop:product expansion}, which implies 
 \begin{equation}\label{lifted q form}
 \mathrm{FJ}^{(a)} (\tilde{\psi}(f) )
 =  \big( \kappa^{(a)} \cdot q_{ \alpha(\varrho) } \cdot \mathrm{BP}(f) \big)^2,
 \end{equation}
for an explicit 
 \begin{equation}\label{integral BP}
 \mathrm{BP}   (f) \in \Z [[ q_\alpha ]]_{  \substack{    \alpha \in \Gamma_\Phi^\vee(1) \\  \langle \alpha, \sigma\rangle \ge 0      }   },
\end{equation}
and some constants $\kappa^{(a)} \in \C$ of absolute value $1$.
Indeed,  the hypotheses on  $\ell , \ell_* \in V_\Z$  imply that the constants   $N$ and $A$ appearing in  (\ref{constant A}) are equal to  $1$, and our choice of $k=\ell_* \in h V_\Z$ implies that   $\zeta_\mu =1$
 for all $\mu \in hV_\Z^\vee / h V_\Z$.

 Moreover, it is clear from the presentation of   $\mathrm{BP}(f)$ as a product that its constant term  is equal to $1$.

The $q$-expansion (\ref{lifted q})  is actually independent of $a$. 
 Indeed, using the notation of (\ref{nbhd diagram}), with $K$ replaced by $\tilde{K}$ throughout,  these $q$-expansions can be computed in terms of the pullback of $\psi(f)$ to the upper left corner in 
\[
\xymatrix{
{  \bigsqcup\limits_{ a \in \Q^\times_{>0} \backslash \A_f^\times / \tilde{K}_0 } \tilde{\Gamma}_\Phi^{(a)} \backslash \DD^\circ }  \ar[d]  \ar[rrr]^{z\mapsto  ( z, s(a) h) } 
& &  & {   \Sh_{\tilde{K}} (G,\DD)(\C)  } \ar[d]   \\
 {  (  K_\Phi \cap U_\Phi(\Q)  )  \backslash \DD^\circ   } \ar[rrr]^{z\mapsto  ( z, h) }  & & & {  \Sh_{K} (G,\DD)(\C) }.
}
\]
Here we have chosen our coset representatives $a\in \widehat{\Z}^\times$.  This implies, by Lemma \ref{lem:optimal section}, that $s(a) \in K_\Phi \subset h K h^{-1}$, and so 
  \[
  \tilde{\Gamma}_\Phi^{(a)} = s(a) \tilde{K}_\Phi s(a)^{-1} \cap U_\Phi(\Q) \subset  K_\Phi \cap U_\Phi(\Q) 
  \]
  and $s(a)hK= hK$.  It follows that the pullback of $\psi(f)$ to the upper left corner is the same on every copy of $\DD^\circ$.

 Having proved that all of the $\kappa^{(a)}$ are equal, we may rescale $\psi(f)$ by a constant of absolute value $1$ to make all of them equal to $1$.
 The $q$-expansion principle of   Proposition \ref{prop:q principle} now implies that $\tilde{\psi}(f)$ is defined over $\Q$, and the same is therefore true of $\psi(f)$.
The equality (\ref{norm match}) follows from the equality (\ref{naive borcherds}).

It only remains to prove the equality of divisors (\ref{divisor match}). 
In the generic fiber, this follows from (\ref{norm match}) and  the analysis of the singularities of  $\regtheta(f)$ found in   \cite{Bor98} or \cite{Bruinier}.  
To prove equality on the integral model, it therefore suffices to prove that both sides of the desired equality are flat over $\Z_\Omega$.  
Flatness of the special divisors $\mathcal{Z}( m,\mu)$  is Proposition \ref{prop:nice divisors}.

To prove the flatness  of $\mathrm{div}(\psi(f))$ it suffices to show, for every prime $p\not\in \Omega$, that $\mathrm{div}(\psi(f))$ has no irreducible components supported in characteristic $p$.
This follows from Corollary \ref{cor:flatness by q} and the observation made above that (\ref{integral BP}) has nonzero reduction at $p$.

 The only technical point is that to apply  Corollary \ref{cor:flatness by q} to the integral model of $\Sh_{\tilde{K}}(G,\DD, \Sigma)$ over $\Z_{(p)}$,  we must choose $\tilde{K}$ to have $p$-component 
\[
\tilde{K}_p = G(\Q_p) \cap C(V_{\Z_p})^\times,
\] 
and similarly choose $\tilde{K}_0$ to have $p$-component $\Z_p^\times$.
As $p$ varies, this forces us to vary $\tilde{K}$.  As we need $\tilde{K}$ to satisfy the conclusion of Lemma \ref{lem:good decomposition}, this may require us to also vary both $\Sigma$ and the rational polyhedral cone $\sigma \in \Sigma_\Phi$.  Thus, having rescaled the Borcherds product to eliminate the constants $\kappa^{(a)}$ at one boundary stratum, we may be forced to apply Corollary \ref{cor:flatness by q} at a different boundary stratum of a different toroidal compactification at different level structure, at which we must deal with new  constants $\kappa^{(a)}$.

This is not really a problem. For a given $p$,  one can check using Remark \ref{rem:K shrink} that it is possible to choose $\tilde{K}$ (and hence $\Sigma$ and $\sigma \in \Sigma_\Phi$) as in Lemma \ref{lem:good decomposition}  by shrinking only the prime-to-$p$ part of $K$.  
Using Lemma \ref{lem:optimal section},  we may then choose $\tilde{K}_0$ to have $p$-component $\Z_p^\times$.   
Now pull back $\psi(f)$ via the resulting \'etale cover
\[
\mathcal{S}_{\tilde{K}}(G,\DD)_{/\Z_{(p)}} \to \mathcal{S}_K(G,\DD)_{/\Z_{(p)}} 
\]
over integral models over $\Z_{(p)}$ to obtain a section $\tilde{\psi}(f)$ whose $q$-expansion again has the form 
(\ref{lifted q form}) for some  constants $\kappa^{(a)}$ of absolute value $1$.

  The point is simply that our  $\psi(f)$, hence also $\tilde{\psi}(f)$,  has been rescaled so that it is defined over $\Q$.
This allows us to use the $q$-expansion principle of Proposition \ref{prop:q principle} to deduce that each $\kappa^{(a)}$ is rational, hence is $\pm 1$.  
Thus the power series (\ref{lifted q form})  has integer coefficients and nonzero reduction at $p$.
  Corollary \ref{cor:flatness by q} implies that the divisor of $\tilde{\psi}(f)$ has no irreducible components in characteristic $p$, so the same holds for $\psi(f)$.
\end{proof}


\subsection{Proof of Theorem \ref{thm:main borcherds}}
\label{ss:rational}


In this subsection we complete the proof  of Theorem \ref{thm:main borcherds}  by developing a purely algebraic analogue of the embedding trick of Borcherds.
This allows us to deduce the general case from the special case proved in Proposition \ref{prop:almost there}.

  According to \cite[Lemma 8.1]{Bor98} there exist self-dual $\Z$-quadratic spaces $\Lambda^{[1]}$ and $\Lambda^{[2]}$ of signature $(24,0)$  whose corresponding theta series
\[
\vartheta^{[i]}  (\tau) = \sum_{ x\in \Lambda^{[i]} } q^{Q(x)} \in M_{12}(\SL_2(\Z) , \C)
\]
 are related by 
 \begin{equation}\label{theta trick}
 \vartheta^{[2]}  - \vartheta^{[1]} = 24 \Delta.
 \end{equation}
 Here $\Delta$ is Ramanujan's modular discriminant, and $Q$ is the quadratic form on $\Lambda^{[i]}$.
    Denote by 
 \[
r^{[i]}(m) = \# \{ x\in \Lambda^{[i]} : Q(x)=m \}
\] 
the $m^\mathrm{th}$ Fourier coefficient of $\vartheta^{[i]}$. 
Set 
\begin{equation}\label{bumped lattices}
V_{\Z}^{[i]} = V_{\Z} \oplus \Lambda^{[i]}   \qquad\mbox{and}\qquad  V^{[i]} = V \oplus \Lambda_\Q^{[i]}.
\end{equation}

 In the notation of \S \ref{ss:weakly holomorphic forms}, the inclusion $V_\Z \hookrightarrow V_\Z^{[i]}$  identifies
 \[
 V_\Z^\vee / V_\Z  \iso  (V_\Z^{[i]} )^\vee/ V_\Z^{[i]}, 
 \]
and the induced isomorphism 
\[
 \sS_{V_\Z  }  \iso   \sS_{V_\Z^{ [i]} } 
\] 
is compatible with the Weil representations on source and target. The fixed weakly holomorphic form $f$ of (\ref{input form})  therefore determines a  form
\[
f^{[i]} (\tau)  =    \sum_{  \substack{ m\in \Q \\ m \gg -\infty} } c^{[i]}(m) \cdot q^m   \in M^!_{-11- \frac{n}{2} }( \overline{\rho}_{V^{[i]}_\Z } )
\]
 by setting  $f^{[i]}    =    f / (  24  \Delta)$.
The relation  
\[
f=  \vartheta^{[2]} f^{[2]} - \vartheta^{[1]} f^{[1]}
\]
  implies the equality of Fourier coefficients
\begin{align}\label{fourier splitting}
c(m,\mu)  & = \sum_{k \ge 0}   r^{[2]} (k)  \cdot c^{[2]} (m - k ,\mu )  \\
&\quad  -  \sum_{k \ge 0}  r^{[1]} (k)  \cdot c^{[1]} (m - k ,\mu ) .\nonumber
\end{align}

Each $V^{[i]}$ determines a GSpin Shimura datum $(G^{[i]} , \DD^{[i]})$.  By choosing 
\[
K^{[i]} = G^{[i]}(\A_f) \cap C(V^{[i]}_{\widehat{\Z}})^\times
\]
 for our compact open subgroups, we put ourselves in the situation of  \S \ref{ss:integral pullbacks}.
 Note that in \S \ref{ss:integral pullbacks} the integral models were over $\Z_{(p)}$, but everything extends verbatim to $\Z_\Omega$.
    In particular,  we have  finite morphisms of integral models
\[
\xymatrix{
&  {   \mathcal{S}_K(G,\DD) \ar[dl]_{j^{[1]}} \ar[dr]^{j^{[2]}  }    }  \\ 
  {  \mathcal{S}^{[1]}   } &   & { \mathcal{S}^{[2]}, } 
}
\]
over $\Z_\Omega$, where we abbreviate 
\[
 \mathcal{S}^{[i]}=\mathcal{S}_{K^{[i]}}(G^{[i]} , \DD^{[i]} ).
\]
Each $\mathcal{S}^{[i]}$ has its own line bundle of weight one modular forms $\bm{\omega}^{[i]}$ and its own family $\mathcal{Z}^{[i]}(m,\mu)$ of special divisors.

The following lemma shows that each $V_\Z^{[i]}  \subset V^{[i]}$ satisfies the hypotheses of Proposition \ref{prop:almost there}.
Thus, after replacing $f$  (and hence both $f^{[1]}$ and $f^{[2]}$)  by a positive integer multiple, we obtain  a Borcherds product $\psi(f^{[i]})$ on $\mathcal{S}^{[i]}$  with divisor
 \begin{equation}\label{embedded divisor}
 \mathrm{div}( \psi( f^{ [i] } ) =  \sum_{ \substack{  m >  0 \\ \mu \in V_\Z^\vee / V_\Z } }   c^{ [i] }(-m,\mu) \cdot \mathcal{Z}^{[i]} (m,\mu) . 
 \end{equation}

\begin{lemma}
There exist isotropic vectors 
$
\ell , \ell_* \in V_\Z^{[i]}
$
 with $[\ell , \ell_*]=1$.
\end{lemma}

\begin{proof}
Let $\mathbb{H}=\Z \ell \oplus \Z \ell_*$ be the integral hyperbolic plane, so that $\ell$ and $\ell_*$ are isotropic with $[\ell,\ell_*]=1$.  To prove the existence of an isometric embedding $\mathbb{H} \to V_\Z^{ [i] }$, we first prove the existence everywhere locally.   

At the archimedean place this is clear from the signature, so fix a prime $p$.
The $\Q_p$-quadratic space  $\Lambda^{ [i] } \otimes_\Z \Q_p$ has dimension $\ge 5$, so  admits an isometric embedding
\[
\mathbb{H}\otimes \Q_p \to  \Lambda^{ [i] } \otimes  \Q_p .
\]
Enlarging the image of $\mathbb{H} \otimes \Z_p$ to a maximal lattice, and invoking Eichler's theorem   that all maximal lattices in a $\Q_p$-quadratic space are isometric \cite[Theorem 8.8]{Ger}, we find that $\mathbb{H}\otimes \Z_p$  embeds into the  (self-dual, hence maximal)  lattice
 $\Lambda^{ [i] } \otimes \Z_p$.  A fortiori, it embeds into  $V_\Z^{[i]} \otimes_\Z \Z_p$.

The existence of  the desired embeddings everywhere locally implies   that there exist  isometric embeddings
\begin{equation}\label{global hyperbolic}
a : \mathbb{H}  \otimes \Q \to V_\Z^{[i]} \otimes \Q,
\end{equation}
and 
\[
\alpha :  \mathbb{H} \otimes \widehat{\Z}  \to V_\Z^{[i]} \otimes \widehat{\Z}.
\]
We may choose these in such a way that $a$ and $\alpha$ induce the same embedding of $\Q_p$-quadratic spaces
at all but finitely many  primes $p$.
All embeddings 
\[
\mathbb{H}\otimes \Q_p  \to V_\Z^{[i]} \otimes \Q_p
\]
 lie in a single $\mathrm{SO}( V^{[i]} )  ( \Q_p)$-orbit, and so there exists a
 \[
g  \in \mathrm{SO}( V^{[i]} )( \A_f)
 \]
  such that 
\begin{equation}\label{spinor adjust}
g a ( \mathbb{H} \otimes \widehat{\Z} ) =\alpha ( \mathbb{H} \otimes \widehat{\Z} )  .
 \end{equation}

Fix a subspace $W \subset V_\Z^{[i]} \otimes \Q$ of signature $(2,1)$ perpendicular to the image of (\ref{global hyperbolic}).  There exists an isomorphism $\SO(W) \iso \mathrm{PGL}_2$ identifying
 the spinor norm
\[
\mathrm{SO}(W)(\A_f) \to \A_f^\times / (\A_f^\times)^2
\]
with the determinant, and hence the spinor norm is surjective.
This allows us to modify $g$ by an element of $\SO(W)(\A_f)$, which does not change the relation (\ref{spinor adjust}), in order to arrange that $g$ has trivial spinor norm.
Now choose any lift 
\[
g \in \mathrm{Spin}( V^{[i]} )( \A_f),
\]
and note that (\ref{spinor adjust}) implies
 \[ 
 g a ( \mathbb{H} \otimes \widehat{\Z} ) \subset V_\Z^{[i]}\otimes \widehat{\Z}.
 \]

As the spin group is simply connected, it satisfies strong approximation.  
By choosing $\gamma \in \mathrm{Spin}(  V^{[i]} ) ( \Q)$ sufficiently close to $g$, we find an
 isometric embedding $\gamma a : \mathbb{H} \to V_\Z^{[i]}$.
\end{proof}

At least formally, we wish to define 
\[
\psi(f) = \frac{ (j^{[2]})^* \psi(f^{ [2]  } ) }{ (j^{[1]})^* \psi(f^{ [1] } )  } .
\]
As noted in \S \ref{ss:intro proof},  the image of $j^{[i]}$ will typically be contained in the support of the divisor of  $\psi(f^{ [i]  } )$, and so the quotient on the right will typically be either  $0/0$ or $\infty/\infty$.

The key to making sense of this quotient is to combine the following lemma, which is really just a restatement of (\ref{embedded divisor}),  with the pullback formula of Proposition \ref{prop:pullback}.  As in the pullback formula, we use $\mathcal{Z}^{[i]}(m,\mu)$ to denote both the special divisor and its corresponding line bundle, and extend the definition to $m\le 0$ by 
\[
\mathcal{Z}^{[i]} (m,\mu) = \begin{cases}
(\bm{\omega}^{ [i] } )^{-1} & \mbox{if } (m,\mu)=(0,0) \\
\co_{ \mathcal{S}^{[i]} } & \mbox{otherwise.}
\end{cases}
\]

\begin{lemma}\label{lem:borcherds iso}
The Borcherds product $\psi( f^{[i]} )$ determines an isomorphism of line bundles
\[
\co_{ \mathcal{S}^{[i]} }\iso  \bigotimes_{ \substack{  m\ge 0 \\ \mu \in V_\Z^\vee / V_\Z   } } \mathcal{Z}^{[i]} (m,\mu)^{  \otimes c^{[i]} (-m ,\mu ) }.
\]
\end{lemma}

\begin{proof}
 If $m>0$ there is a canonical section
 \[
 s^{[i]}(m,\mu) \in H^0 \big(   \mathcal{S}^{[i]} , \mathcal{Z}^{[i]} (m,\mu) \big)
 \]
with divisor  the Cartier divisor $\mathcal{Z}^{[i]} (m,\mu)$ of the same name.
This is just the constant function $1$, viewed as a section of 
\[
\co_{\mathcal{S}^{[i]}} \subset \mathcal{Z}^{[i]} (m,\mu).
\]
The equality of divisors (\ref{embedded divisor}) implies that there is a unique isomorphism
\[
(\bm{\omega}^{ [i] } ) ^{ \otimes c^{[i]}(0,0) } \iso  \bigotimes_{ \substack{  m> 0 \\ \mu \in V_\Z^\vee / V_\Z   } } 
\mathcal{Z}^{[i]} (m,\mu)^{  \otimes c^{[i]} (-m ,\mu ) }
\]
sending
\[
\psi(f^{[i]}) \mapsto   \bigotimes_{ \substack{  m> 0 \\ \mu \in V_\Z^\vee / V_\Z   } } s^{[i]} (m,\mu)^{  \otimes c^{[i]} (-m ,\mu ) },
\]
and so the claim is  immediate from the definition of $\mathcal{Z}^{[i]}(0,\mu)$.
\end{proof}

\begin{proof}[Proof of Theorem \ref{thm:main borcherds}]
If we pull back the isomorphism of Lemma \ref{lem:borcherds iso} via $j^{[i]}$ and use the pullback formula of Proposition \ref{prop:pullback}, we  obtain  isomorphisms of line bundles
\[
\co_{ \mathcal{S}_K(G,\DD) } \iso    \bigotimes_{ \substack{  m_1,m_2 \ge 0 \\ \mu \in V_\Z^\vee / V_\Z   } }
\mathcal{Z}(m_1,\mu)^{ \otimes r^{[i]} (m_2)  \cdot c^{[i]} (-m_1 - m_2 ,\mu )  }
\]
for $i\in\{ 1, 2\}$.  These two isomorphisms,  along with (\ref{fourier splitting}), determine an isomorphism
\[
\co_{ \mathcal{S}_K(G,\DD) } \iso    \bigotimes_{ \substack{  m \ge 0 \\ \mu \in V_\Z^\vee / V_\Z   } }
\mathcal{Z}(m ,\mu)^{ \otimes  c (-m ,\mu )  }.
\]

Now simply reverse the reasoning in the proof of Lemma \ref{lem:borcherds iso}.  Rewrite  the isomorphism above as
\[
\bm{\omega}^{ c(0,0) } \iso  \bigotimes_{ \substack{  m > 0 \\ \mu \in V_\Z^\vee / V_\Z   } }  
\mathcal{Z}(m ,\mu)^{ \otimes  c (-m ,\mu )  }.
\] 
Each line bundle on the right admits a canonical section $s(m,\mu)$ whose divisor is the Cartier divisor $\mathcal{Z}(m,\mu)$ of the same name, and so the rational section of $\bm{\omega}^{ c(0,0)}$ defined by  
\begin{equation}\label{subtle borcherds}
\psi(f) =   \bigotimes_{ \substack{  m > 0 \\ \mu \in V_\Z^\vee / V_\Z   } }  s(m ,\mu)^{ \otimes  c (-m ,\mu )  }
\end{equation}
has divisor 
\[
 \mathrm{div}( \psi(f) ) =  \sum_{ \substack{  m >  0 \\ \mu \in V_\Z^\vee / V_\Z } }   c(-m,\mu) \cdot \mathcal{Z}(m,\mu).
 \]
 To complete the proof of Theorem \ref{thm:main borcherds}, we need only prove that the section defined by (\ref{subtle borcherds}) satisfies the norm relation  (\ref{norm match}).

Fix a $g\in G(\A_f)$, and consider the complex  uniformizations
\[
\xymatrix{
& {   \DD^{[1] }  }  \ar[rr] &  &  {   \mathcal{S}^{[1]}(\C) }    \\
{  \DD  } \ar[rr] \ar[ur]^{ j^{[1]} } \ar[dr]_{ j^{[2]} }& &  { \mathcal{S}_K (G,\DD)  (\C) }  \ar[ur] ^{ j^{[1]} } \ar[dr] _{ j^{[2]} } \\
&  {    \DD^{[2]}  } \ar[rr]&  &  {   \mathcal{S}^{[2]}(\C)  } 
}
\]
in which all horizontal arrows send $z\mapsto (z,g)$.

Denote by $\psi_g(f)$ the pullback of $\psi(f)$ to $\DD$.
The similarly defined meromorphic sections  $\psi_g(f^{[i]})$ on $\DD^{[i]}$ are already assumed to satisfy the norm relation
\[
-2 \log \| \psi_g(f^{[i]} ) \| =\regtheta_g(f^{[i]} ) 
\]
on $\DD^{[i]}$,   where 
\[
\regtheta_g(f^{[i]} )=\regtheta(f^{[i]} ,g)
\]
 is the regularized theta lift of \S \ref{ss:borcherds definition}.

Recall from \S \ref{ss:deformation} that every $x \in  V^{[i]}$ with $Q(x)>0$ determines a global section 
\[
\mathrm{obst}^{an}_x \in H^0 \big( \DD^{[i]}  ,    \bm{\omega}_{\DD^{ [i] }}^{-1}  \big),
\]
with zero locus the analytic divisor
\[
\DD^{[i]}(x) = \{ z\in \DD^{[i]}  : [z, x] =0  \}.
\]
The pullback of $\mathcal{Z}^{ [i] }(m,\mu)(\C)$ to $\DD^{[i]}$  is  given  by the locally finite sum of analytic divisors
\[
\sum_{  \substack{   x\in g \mu+  g V^{[i]}_\Z \\ Q(x)=m     }   }\DD^{[i]}(x) .
\]
Define the \emph{renormalized} Borcherds product
\[
\tilde{\psi}_g(f^{[i]}) =
\psi_g(f^{[i]}) \otimes  \bigotimes_{ m>0 }  \bigotimes_{  \substack{  \lambda \in \Lambda^{[i]} \\ Q(\lambda)=m }   }   ( \mathrm{obst}^{an}_\lambda)^{ \otimes -c^{[i]}(-m,0)} .
\]
This is a meromorphic section of 
$
 \bigotimes_{ m \ge 0 }  \big(\bm{\omega}_{ \DD^{ [i] } } \big)^{ \otimes  r^{[i]} (m) c^{[i]}(-m,0)}  
$
with divisor
\[
\mathrm{div}\big( \tilde{\psi}_g(f^{[i]}) \big)  = \sum_{\substack {m>0 \\ \mu\in V_\Z^\vee / V_\Z }}
c^{[i]}(-m,\mu) \sum_{  \substack{   x\in g \mu+  g V^{[i]}_\Z \\ Q(x)=m   \\ x\not\in \Lambda^{[i]}  }   }   \DD^{[i]} (x)  .
\]

Note that each divisor $\DD^{[i]}(x)$ appearing on the right hand side intersects the subspace $\mathcal{D}\subset \DD^{[i]}$ properly.  Indeed,
If we decompose $x=y +\lambda$ with $y \in g\mu + g V_\Z$ and $\lambda\in \Lambda$, then   
\[
\DD^{[i]} (x)   \cap \DD = \begin{cases}
\DD(y)  & \mbox{if }Q(y)>0 \\
 \emptyset & \mbox{otherwise,}
 \end{cases}
\]
where $\DD(y)=\{z\in \DD : [ z,y] =0 \}$.

This is the point: by renormalizing the Borcherds product we have removed precisely the part of its divisor that intersects $\DD$ improperly, and so the renormalized Borcherds product has a well-defined pullback to $\DD$.  Indeed, using the relation (\ref{fourier splitting}), we see that 
\begin{equation}\label{analytic renorm}
\psi_g( f) = \frac{  ( j^{ [2] })^* \tilde{\psi}_g(f^{[2]})}{  ( j^{ [1] })^*   \tilde{\psi }_g(f^{[1]})}
\end{equation}
is a section of the line bundle $\bm{\omega}_\DD^{ \otimes c(0,0)}$. 
By directly comparing the algebraic and analytic constructions, which ultimately boils down to the comparison of algebraic and analytic obstructions found in Proposition \ref{prop:two obstructions}, one can check that it agrees with the $\psi_g( f)$ defined at the beginning of the proof.

Define the \emph{renormalized} regularized theta lift
\begin{align*}
\widetilde{\Theta}^\mathrm{reg}_g(f^{[i]}) & = \regtheta_g(f^{[i]}) 
 +  2 \sum_{m>0} c^{[i]}(-m,0) \sum_{ \substack{  \lambda\in \Lambda^{[i]} \\ Q(\lambda) =m }  }
  \log\|\mathrm{obst}^{an}_\lambda\|
\end{align*}
so that 
\begin{equation}\label{over-reg borcherds}
- 2 \log \| \tilde{\psi}_g( f^{[i]} )\| =  \widetilde{\Theta}^\mathrm{reg}_g(f^{[i]})  .
\end{equation}
Combining this with (\ref{analytic renorm}) yields
\[
-2\log\|\psi_g( f ) \|  =  ( j^{ [2] })^* \widetilde{\Theta}^\mathrm{reg}_g( f^{ [2] } )   -   ( j^{ [1] })^*  \widetilde{\Theta}^\mathrm{reg}_g(  f^{ [1]} )   .
\]

As was noted in Remark \ref{rem:miracle values}, 
the regularized theta lift $\regtheta_g( f^{ [i] } )$ is \emph{over-regularized}, in the sense that its definition makes sense at every point of $\DD^{[i] }$, even at points of the divisor along which $\regtheta_g( f^{ [i] } )$ has its logarithmic singularities.  As in  \cite[Proposition 5.5.1]{AGHMP-1}, its values along the discontinuity agree with the values of  $\widetilde{\Theta}^\mathrm{reg}_g( f^{ [i] } )$, and in fact we have 
\[
( j^{ [i] })^* \regtheta_g( f^{ [i] } )  = ( j^{ [i] })^* \widetilde{\Theta}^\mathrm{reg}_g( f^{ [i] } )
\]
as functions on $\DD$.  

On the other hand, for each $i\in \{1,2\}$, the regularized theta lift has the form
\[
\regtheta_g( f^{ [i] } ) (z) = \int_{   \SL_2(\Z) \backslash \mathcal{H}   }      f^{[i]} (\tau)   \vartheta^{[i]} (\tau,z,g)        \  \frac{ du\, dv}{v^2}
\]
 as in (\ref{theta integral}).  As in  \cite[(4.16)]{BY}, when we restrict the variable $z$ to $\DD \subset \DD^{[i]}$ the theta kernel factors as
\[
 \vartheta^{[i]} (\tau,z,g)   =  \vartheta  (\tau,z,g)   \cdot \vartheta^{[i]}(\tau),
\]
where $\vartheta  (\tau,z,g)$ is the theta kernel defining $\regtheta_g(f)$.  Thus 
\[
( j^{ [i] })^* \regtheta_g( f^{ [i] } ) 
=  \int_{   \SL_2(\Z) \backslash \mathcal{H}   }      f  (\tau)   \vartheta (\tau,z,g)     \cdot \frac{\vartheta^{[i]}(\tau)}{24 \Delta}    \  \frac{ du\, dv}{v^2}
\]

Combining this last equality with (\ref{theta trick}) proves the first equality in 
\begin{align*}
\regtheta_g(f)  & = 
  ( j^{ [2] })^* \regtheta_g( f^{ [2] } )  -  ( j^{ [1] })^*  \regtheta_g(  f^{ [1]} )    \\
& =
 ( j^{ [2] })^* \widetilde{\Theta}^\mathrm{reg}_g( f^{ [2] } )   -  ( j^{ [1] })^*  \widetilde{\Theta}^\mathrm{reg}_g(  f^{ [1]} ),
\end{align*}
which is just a more explicit statement of  \cite[Lemma 8.1]{Bor98}.  
Combining this with (\ref{analytic renorm}) and  (\ref{over-reg borcherds}) shows that $\psi(f)$ satisfies the norm relation (\ref{norm match}), and completes the proof of  Theorem \ref{thm:main borcherds}.
\end{proof}


\subsection{A remark on sufficient divisibility}
\label{ss:divisible remark}


In order to obtain a  Borcherds product $\psi(f)$ on the integral model $\mathcal{S}_K(G,\DD) \to \Spec(\Z_\Omega)$,   Theorem \ref{thm:main borcherds}  requires that we first multiply the integral form
\[
f(\tau)  = \sum_{  \substack{ m\in \Q \\ m \gg -\infty} } c(m) \cdot q^m   \in M^!_{1- \frac{n}{2} }( \overline{\rho}_{V_\Z } )
\]
by some unspecified positive integer $N$.     In fact, examination of the proof shows that $N=N(V_\Z)$ depends only on the quadratic lattice $V_\Z$, and not on  the choice of $f$, the finite set of primes $\Omega$, or the level subgroup $K$.

Indeed, one  first checks this in the situation of Proposition \ref{prop:almost there}.  
Thus we  assume that $n \ge 6$, and that there exists an $h\in G(\A_f)$ and isotropic vectors $\ell ,\ell_* \in h V_\Z$ with $[\ell,\ell_*]=1$.   As in the proof of that proposition, one can reduce to the case 
\[
K = G(\A_f) \cap C(  V_{ \widehat{\Z}} )^\times. 
\]
The only point in the proof of Proposition \ref{prop:almost there} where one must replace $f$ by $Nf$  is when Theorem \ref{thm:old borcherds} and Proposition  \ref{prop:borcherds algebraic}  are invoked to obtain the Borcherds product (\ref{almost there borcherds}) over the complex fiber $\Sh_K(G,\DD)_{/\C}$.
Thus we only need to require that $N$ be chosen divisible enough that the multipliers 
\[
\xi_g(f) : G(\Q)^\circ  \cap g K g^{-1} \to \C^\times
\]
of (\ref{multiplier}) satisfy $\xi_g(f)^N =1$, as 
$f$ varies over all integral forms as above  and $g\in G(\A_f)$ runs over the finite set of indices in
\[
\bigsqcup_g  ( G(\Q)^\circ  \cap g K g^{-1} ) \backslash \DD^\circ \iso \Sh_K(G,\DD)_{/\C}.
\]
This is possible,  as the natural map 
\[
G(\Q)^\circ  \cap g K g^{-1} \to \SO(g V_\Z)
\]
has kernel $\{ \pm 1\}$, and its image has  finite abelianization; see \cite{Bor:GKZcorrection}.

The general case follows by examining the constructions of \S \ref{ss:rational}.  
Applying the  special case above to the lattices in (\ref{bumped lattices}) yields positive integers $N(V_\Z^{[1]})$ and  $N(V_\Z^{[2]})$.  Any multiple of  
\[
N(V_\Z) = N(V_\Z^{[1]}) \cdot N(V_\Z^{[2]})
\]
 is then  ``sufficiently divisible" for the purposes of Theorem \ref{thm:main borcherds}.


\subsection{Modularity of the generating series}


For any positive $m\in \Q$ and any $\mu \in V_\Z^\vee/ V_\Z$,  we denote  again by  
\[
\mathcal{Z}(m,\mu) \in \mathrm{Pic} ( \mathcal{S}_K(G,\DD) )
\]
 the line bundle  defined by the Cartier divisor of the same name.   Extend the definition to $m=0$ by
\[
\mathcal{Z}(0,\mu) = \begin{cases}
\bm{\omega}^{-1} & \mbox{if }\mu=0 \\
\co_{ \mathcal{S}_K(G,\DD)  }  & \mbox{otherwise.}
\end{cases}
\]

Recall from \S \ref{ss:weakly holomorphic forms} the Weil representation
\[
\rho_{V_\Z }  : \widetilde{\SL}_2(\Z) \to \Aut_\C( \sS_{V_\Z} )
 \] 
 on $\sS_{V_\Z  } = \C [ V^\vee_\Z / V_\Z ]$.

\begin{theorem}\label{thm:naked modularity}
Let $\phi_\mu \in \sS_{V_\Z}$ be the characteristic function of the coset $\mu\in V_\Z^\vee / V_\Z$.
For any $\Z$-linear map 
$
\alpha: \mathrm{Pic}( \mathcal{S}_K(G,\DD) ) \to \C
$
we have  
\[
 \sum_{ \substack{ m \ge 0 \\   \mu \in V_\Z^\vee / V_\Z   }  }  \alpha ( \mathcal{Z}(m ,\mu ) ) \cdot  \phi_\mu   \cdot q^m \in M_{ 1+ \frac{n}{2} } ( \rho_{V_\Z} ).
\]
\end{theorem}

\begin{proof}
According to the modularity criterion of \cite[Theorem 3.1]{Bor:GKZ}, a formal $q$-expansion
\[
\sum_{\substack{ m \ge 0 \\ \mu \in V_\Z^\vee / V_\Z } } a (m,\mu) \cdot \phi_\mu \cdot q^m
\]
with coefficients in $\sS_{V_{\Z}}$ defines an element of $M_{ 1+ \frac{n}{2} } ( \rho_{V_\Z} )$ if and only if
\begin{equation}\label{no obstruction}
0= \sum_{ \substack{ m\ge 0 \\ \mu \in V_\Z^\vee / V_\Z } } c(-m,\mu) \cdot  a(m,\mu)
\end{equation}
for every 
\[
f(\tau)  = \sum_{  \substack{ m\in \Q\\ \mu\in V_\Z^\vee / V_\Z }}  c(m,\mu) \cdot q^m   \in M^!_{1- \frac{n}{2} }( \overline{\rho}_{V_\Z } ).
\]
By the main result of \cite{mcgraw}, it suffices to verify (\ref{no obstruction})  only for $f(\tau)$ that are  integral, in the sense of Definition 
\ref{def:integral form}.

For any  integral $f(\tau)$, Theorem \ref{thm:main borcherds} implies that 
\[
\bm{\omega}^{ c(0,0 )}  = \sum_{ \substack{ m> 0 \\ \mu \in V_\Z^\vee / V_\Z } } c(-m,\mu) \cdot  \mathcal{Z}(m,\mu) 
\]
up to a torsion element in $\mathrm{Pic}( \mathcal{S}_K(G,\DD) )$, and hence
\[
\sum_{ \substack{ m\ge 0 \\ \mu \in V_\Z^\vee / V_\Z } } c(-m,\mu) \cdot  \mathcal{Z}(m,\mu)  \in \mathrm{Pic}( \mathcal{S}_K(G,\DD) )
\]
is killed by any $\Z$-linear map $\alpha: \mathrm{Pic}( \mathcal{S}_K(G,\DD) ) \to \C$.  
Thus the claimed modularity follows from the result of Borcherds cited above.
\end{proof}


\subsection{Modularity of the arithmetic generating series}


Bruinier \cite{Bruinier} has defined a Green function $\regtheta(F_{m,\mu})$  for the divisor $\mathcal{Z}(m,\mu)$.
This Green function is constructed, as in (\ref{theta integral}),  as the regularized theta lift of a harmonic Hejhal-Poincare series
\[
F_{m,\mu} \in H_{1-\frac{n}{2}} ( \overline{\rho}_{V_\Z} )
\]
whose holomorphic part, in the sense of \cite[\S 3]{BruinierFunke}, has the form
\[
F_{m,\mu}^+(\tau) = \left( \frac{ \phi_{\mu}+ \phi_{-\mu}}{2}  \right) \cdot  q^{-m}  + O(1),
\]
where $\phi_\mu \in S_{V_\Z}$ is the characteristic function of the coset $\mu \in V_\Z^\vee/ V_\Z$.
See \cite[\S 3.2]{AGHMP-1} and the references therein.

This Green function determines a metric on the corresponding line bundle, and so determines a class
\[
\widehat{\mathcal{Z}}(m,\mu) = (  \mathcal{Z}(m,\mu) , \regtheta(F_{m,\mu}) )  \in \widehat{\mathrm{Pic}}( \mathcal{S}_K(G,\DD) )
\]
for every positive $m\in \Q$ and $\mu\in V_\Z^\vee / V_\Z$.  
Recall that that we have defined a metric (\ref{better metric}) on the line bundle $\bm{\omega}$, and so obtain a class
\[
\widehat{\bm{\omega}} \in \widehat{\mathrm{Pic}}( \mathcal{S}_K(G,\DD) ).
\]
 We  define
\[
\widehat{\mathcal{Z}}(0,\mu) = \begin{cases}
\widehat{\bm{\omega}}^{-1} & \mbox{if }\mu=0 \\
0 & \mbox{otherwise.}
\end{cases}
\]

\begin{theorem}\label{thm:metrized modularity}
Suppose  $n\ge 3$. For any $\Z$-linear functional 
\[
\alpha : \widehat{\mathrm{Pic}}( \mathcal{S}_K(G,\DD) ) \to \C
\]
we have
\[
\sum_{ \substack{ m \ge 0 \\   \mu \in V_\Z^\vee / V_\Z   }  } 
 \alpha \big( \widehat{\mathcal{Z}}(m ,\mu ) \big) \cdot \phi_\mu  \cdot q^m  \in M_{ 1+ \frac{n}{2} } ( \rho_{V_\Z} ).
\]
\end{theorem}

\begin{proof}
The assumption that $n\ge 3$ implies that any form 
\[
f(\tau)  = \sum_{  \substack{ m\in \Q\\ \mu\in V_\Z^\vee / V_\Z }}  c(m,\mu) \cdot q^m   \in M^!_{1- \frac{n}{2} }( \overline{\rho}_{V_\Z } ).
\]
has negative weight.  As in \cite[Remark 3.10]{BruinierFunke}, this implies that any such $f$
is a linear combination of the Hejhal-Poincare series $F_{m,\mu}$, and in fact
\[
f = \sum_{ \substack{m>0 \\ \mu \in V_\Z^\vee / V_\Z } } c(-m,\mu)  \cdot F_{m,\mu}.
\]
This last equality follows, as in the proof of \cite[Lemma 3.10]{BHY}, from the fact that the difference between the two sides is a harmonic weak Maass form whose holomorphic part is $O(1)$. 
In particular,  we have the equality of regularized theta lifts
\[
 \regtheta(f)   =    \sum_{ \substack{m>0 \\ \mu \in V_\Z^\vee / V_\Z } } c(-m,\mu)  \cdot    \regtheta(F_{m,\mu}) .
\]

Now assume that $f$ is integral.  After replacing $f$ by a positive integer multiple, Theorem \ref{thm:main borcherds} provides us with  a Borcherds product $\psi(f)$ with arithmetic divisor
\[
\widehat{\mathrm{div}}(\psi(f)) = ( \mathrm{div}(\psi(f) ) , - \log \| \psi (f) \|^2 )
= \sum_{ \substack{m>0 \\ \mu \in V_\Z^\vee / V_\Z } } c(-m,\mu)  \cdot   \widehat{\mathcal{Z}}(m,\mu).
\]
On the other hand, in the group of metrized line bundles we have 
\[
\widehat{\mathrm{div}}(\psi(f)) =
 \widehat{\bm{\omega}}^{\otimes c(0,0) }  = - c(0,0) \cdot  \widehat{\mathcal{Z}}(0,0).
\]

The above relations show that
\[
 \sum_{ \substack{m\ge 0 \\ \mu \in V_\Z^\vee / V_\Z } } c(-m,\mu)  \cdot   \widehat{\mathcal{Z}}(m,\mu) \in \widehat{\mathrm{Pic}}( \mathcal{S}_K(G,\DD) )
\]
is a torsion element  for any integral $f$.    Exactly as in the proof of Theorem \ref{thm:naked modularity}, the claim follows from the modularity criterion of Borcherds.
\end{proof}


\bibliographystyle{smfalpha}

\providecommand{\bysame}{\leavevmode\hbox to3em{\hrulefill}\thinspace}
\providecommand{\MR}{\relax\ifhmode\unskip\space\fi MR }
\providecommand{\MRhref}[2]{%
  \href{http://www.ams.org/mathscinet-getitem?mr=#1}{#2}
}
\providecommand{\href}[2]{#2}


\begin{thebibliography}{AGHMP00}

\bibitem[AGHMP17]{AGHMP-1}
F. Andreatta, E.~Z. Goren, B. Howard, and K. {Madapusi Pera},
  \emph{{Height pairings on orthogonal Shimura varieties}}, Compos. Math.
  \textbf{153} (2017), no.~3, 474--534.

\bibitem[AGHMP18]{AGHMP-2}
\bysame \emph{Faltings heights of abelian varieties with complex multiplication},  \emph{Ann. of Math.} \textbf{187} (2018), no.~2, 391-531.


\bibitem[BO78]{BO}
P.~Berthelot and A.~Ogus, \emph{Notes on Crystalline Cohomology}, Princeton University Press,1978. 

\bibitem[Bor95]{Bor95}
R.~Borcherds, \emph{Automorphic forms on {${\rm O}_{s+2,2}({\bf R})$}
  and infinite products}, Invent. Math. \textbf{120} (1995), no.~1, 161--213.

\bibitem[Bor98]{Bor98}
\bysame \emph{Automorphic forms with singularities on {G}rassmannians},
  Invent. Math. \textbf{132} (1998), no.~3, 491--562. 

\bibitem[Bor99]{Bor:GKZ}
\bysame, \emph{The {Gross-Kohnen-Zagier} theorem in higher dimensions}, Duke
  Math. J. \textbf{97} (1999), no.~2, 219--233.

\bibitem[Bor00]{Bor:GKZcorrection}
\bysame, \emph{Correction to: ``{T}he {G}ross-{K}ohnen-{Z}agier theorem in
  higher dimensions'' [{D}uke {M}ath. {J}. {\bf 97} (1999), no. 2, 219--233]}, Duke Math. J. \textbf{105} (2000), no.~1,
  183--184. 

\bibitem[Bru02]{Bruinier}
J.H. Bruinier, \emph{Borcherds products on {$O(2,l)$} and {C}hern classes of
  {H}eegner divisors}, Lecture Notes in Mathematics, Springer, 2002.


\bibitem[BBK]{BBK}
J.H. Bruinier, J.I. Burgos Gil, and Ulf K\"uhn, \emph{Borcherds products and arithmetic intersection theory on {H}ilbert modular surfaces}, Duke Math. J.
  \textbf{139} (2007), no.~1, 1--88.
  
  
\bibitem[BF04]{BruinierFunke}
J.H. Bruinier and J.~Funke, \emph{On two geometric theta lifts}, Duke Math. J.
  \textbf{125} (2004), no.~1, 45--90.

\bibitem[BHY15]{BHY}
J.H. Bruinier, B.~Howard, and T.H. Yang, \emph{Heights of {Kudla-Rapoport}
  divisors and derivatives of $L$-functions}, Invent. Math.  \textbf{201} (2015), no.~1, 1--95.


\bibitem[BY09]{BY}
J.H. Bruinier and T.H. Yang, \emph{Faltings heights of {CM} cycles
  and derivatives of $L$-functions}, Invent. Math. \textbf{177} (2009), no.~3,
  631--681.

\bibitem[Con14]{MR3362641}
B. Conrad, \emph{Reductive group schemes}, Autour des sch\'emas en groupes.
  {V}ol. {I}, Panor. Synth\`eses, vol. 42/43, Soc. Math. France, Paris, 2014,
  pp.~93--444. 

\bibitem[Del79]{Deligne:Shimura2}
P.~Deligne, \emph{Vari\'et\'es de {S}himura: interpr\'etation modulaire, et
  techniques de construction de mod\`eles canoniques}, Automorphic forms,
  representations and {$L$}-functions ({P}roc. {S}ympos. {P}ure {M}ath.,
  {O}regon {S}tate {U}niv., {C}orvallis, {O}re., 1977), {P}art 2, Proc. Sympos.
  Pure Math., ZZZIII, Amer. Math. Soc., Providence, R.I., 1979, pp.~247--289.

\bibitem[FC90]{FaltingsChai}
G.~Faltings and C.-L. Chai, \emph{Degeneration of abelian varieties},
  Ergebnisse der Mathematik und ihrer Grenzgebiete (3) [Results in Mathematics
  and Related Areas (3)], vol.~22, Springer-Verlag, Berlin, 1990, With an
  appendix by David Mumford.


\bibitem[Ger08]{Ger}
L.~Gerstein, \emph{Basic quadratic forms}, Graduate Studies in Mathematics,
  vol.~90, American Mathematical Society, Providence, RI, 2008. \MR{2396246
  (2009e:11064)}
  
  
\bibitem[GR84]{GR}
H. Grauert and R. Remmert, \emph{Coherent analytic sheaves},
  Grundlehren der Mathematischen Wissenschaften [Fundamental Principles of
  Mathematical Sciences], vol. 265, Springer-Verlag, Berlin, 1984.

\bibitem[Har84]{harris:automorphic_0}
M.~Harris, \emph{Arithmetic vector bundles on {S}himura varieties}, Automorphic
  forms of several variables ({K}atata, 1983), Progr. Math., vol.~46,
  Birkh\"auser Boston, Boston, MA, 1984, pp.~138--159. 

\bibitem[Har85]{harris:automorphic_1}
\bysame, \emph{Arithmetic vector bundles and automorphic forms on {S}himura
  varieties. {I}}, Invent. Math. \textbf{82} (1985), no.~1, 151--189.

\bibitem[Har86]{harris:automorphic_2}
\bysame, \emph{Arithmetic vector bundles and automorphic forms on {S}himura
  varieties. {II}}, Compositio Math. \textbf{60} (1986), no.~3, 323--378.

\bibitem[Har89]{harris:automorphic_3}
\bysame, \emph{Functorial properties of toroidal compactifications of locally
  symmetric varieties}, Proc. London Math. Soc. (3) \textbf{59} (1989), no.~1,
  1--22.

\bibitem[HZ94a]{HZ1}
M.~Harris and S.~Zucker, \emph{Boundary cohomology of {S}himura varieties. {I}.
  {C}oherent cohomology on toroidal compactifications}, Ann. Sci. \'Ecole Norm.
  Sup. (4) \textbf{27} (1994), no.~3, 249--344. 

\bibitem[HZ94b]{HZ2}
\bysame, \emph{Boundary cohomology of {S}himura varieties. {II}. {H}odge theory
  at the boundary}, Invent. Math. \textbf{116} (1994), no.~1-3, 243--308.

\bibitem[HZ01]{HZ3}
\bysame, \emph{Boundary cohomology of {S}himura varieties. {III}. {C}oherent
  cohomology on higher-rank boundary strata and applications to {H}odge
  theory}, M\'em. Soc. Math. Fr. (N.S.) (2001), no.~85, vi+116. 





\bibitem[Har70]{hart}
R.~Hartshorne. \emph{Ample subvarieties of algebraic varieties}, Lecture Notes in Mathematics, Vol.~156, 
Springer-Verlag, Berlin-New York, 1970.


\bibitem[H{\"o}r10]{hor:thesis}
F.~H{\"o}rmann, \emph{The arithmetic volume of {S}himura varieties of
  orthogonal type}, Ph.D. thesis, Humbolt-Universit\"at zu Berlin, 2010.

\bibitem[H{\"o}r14]{hor:book}
\bysame, \emph{The geometric and arithmetic volume of {S}himura varieties of
  orthogonal type}, CRM Monograph Series, vol.~35, American Mathematical
  Society, Providence, RI, 2014.


\bibitem[How19]{Ho3}
B.~Howard \emph{Linear invariance of intersections on unitary Rapoport-Zink spaces}, Forum. Math. 
\textbf{31} (2019), no.~5, 1265--1281.

\bibitem[HP17]{Howard-Pappas}
B.~Howard and G.~Pappas, \emph{{R}apoport-{Z}ink spaces for spinor groups},
  Compos. Math. \textbf{154} (2017), 1050--1118.


\bibitem[KM16]{KMP}
W. Kim and K. {Madapusi Pera}, \emph{2-adic integral canonical models},
  Forum Math. Sigma \textbf{4} (2016), e28, 34.
  
\bibitem[Kis10]{KisinJAMS}
M.~Kisin, \emph{Integral models for {S}himura varieties of abelian type}, J.
  Amer. Math. Soc. \textbf{23} (2010), no.~4, 967--1012. 


\bibitem[Kud97]{Ku97}
S.~S. Kudla, \emph{Algebraic cycles on {S}himura varieties of orthogonal type}, Duke Math. J.
  \textbf{86} (1997), no.~1, 39--78. 
  

\bibitem[Kud03]{KuBorcherds}
\bysame \emph{Integrals of {B}orcherds forms}, Compositio Math.
  \textbf{137} (2003), no.~3, 293--349. 

\bibitem[Kud04]{Ku-MSRI}
\bysame, \emph{Special cycles and derivatives of {E}isenstein series}, Heegner
  points and {R}ankin {$L$}-series, Math. Sci. Res. Inst. Publ., vol.~49,
  Cambridge Univ. Press, Cambridge, 2004, pp.~243--270.



\bibitem[KR99]{KR-hilbert}
S.~S. Kudla and M. Rapoport, \emph{Arithmetic {H}irzebruch-{Z}agier
  cycles}, J. Reine Angew. Math. \textbf{515} (1999), 155--244. 

\bibitem[KR00a]{KR-siegel}
\bysame, \emph{Cycles on {S}iegel threefolds and derivatives of {E}isenstein
  series}, Ann. Sci. \'Ecole Norm. Sup. (4) \textbf{33} (2000), no.~5,
  695--756.

\bibitem[KR00b]{KR-drinfeld}
\bysame, \emph{Height pairings on {S}himura curves and {$p$}-adic
  uniformization}, Invent. Math. \textbf{142} (2000), no.~1, 153--223.
%





\bibitem[Lan13]{Lan}
K.-W. Lan.
\newblock {\em Arithmetic compactifications of {PEL}-type Shimura varieties},  vol.~36 of
{\em London Mathematical Society Monographs,} Princeton University Press, 2013

\bibitem[{Mad}16]{mp:spin}
K. {Madapusi Pera}, \emph{{Integral canonical models for {S}pin {S}himura varieties}},
  Compos. Math. \textbf{152} (2016), no.~4, 769--824.

\bibitem[{Mad}19]{mp:compactification}
\bysame, \emph{Toroidal compactifications of integral models of
  {S}himura varieties of {H}odge type.}, Ann. Sci. {\'E}c. Norm. Sup\'er. \textbf{52} (2019), no.~1, 393--514


\bibitem[McG03]{mcgraw}
W.~J. McGraw, \emph{The rationality of vector valued modular forms
  associated with the {W}eil representation}, Math. Ann. \textbf{326} (2003),
  no.~1, 105--122.

\bibitem[Mil88]{milne:automorphic}
J.~S. Milne, \emph{Automorphic vector bundles on connected {S}himura
  varieties}, Invent. Math. \textbf{92} (1988), no.~1, 91--128.

\bibitem[Mil90]{milne:canonical}
\bysame, \emph{Canonical models of (mixed) {S}himura varieties and automorphic
  vector bundles}, Automorphic forms, {S}himura varieties, and {$L$}-functions,
  {V}ol.\ {I} ({A}nn {A}rbor, {MI}, 1988), Perspect. Math., vol.~10, Academic
  Press, Boston, MA, 1990, pp.~283--414.

\bibitem[Ogu79]{Ogus1979}
A.~Ogus, \emph{Supersingular $K3$ crystals in Journées de {G}éométrie {A}lgébrique de {R}ennes ({R}ennes, 1978)}, Asterisque \textbf{64} (1979), 3--86 

\bibitem[Ogu01]{Ogus2001-wy}
\bysame  \emph{Singularities of the height strata in the moduli of $K3$
  surfaces}, Moduli of abelian varieties, Progr. Math., vol. 195,
  Birkh{\"{a}}user, Basel, 2001, pp.~325--343.

\bibitem[Pin89]{pink}
R.~Pink, \emph{Arithmetical compactification of mixed {S}himura varieties},
  Ph.D. thesis, Bonn, 1989.

\bibitem[PS08]{PS}
C.~Peters and J.~Steenbrink, \emph{Mixed {H}odge structures}, Ergebnisse der
  Mathematik und ihrer Grenzgebiete. 3. Folge. A Series of Modern Surveys in
  Mathematics [Results in Mathematics and Related Areas. 3rd Series. A Series
  of Modern Surveys in Mathematics], vol.~52, Springer-Verlag, Berlin, 2008.

\bibitem[Wil00]{wildeshaus}
J.~Wildeshaus, \emph{Mixed sheaves on {S}himura varieties and their higher
  direct images in toroidal compactifications}, J. Algebraic Geom. \textbf{9}
  (2000), no.~2, 323--353. 

\end{thebibliography}

\end{document}